\documentclass{amsart}
\pdfoutput=1

\usepackage{mathrsfs}
\usepackage{enumerate}
\usepackage{amssymb}
\usepackage{microtype}
\usepackage{tikz}
\usepackage[pdftex,colorlinks,citecolor=black, linkcolor=black,urlcolor=black,bookmarks=false]{hyperref}
\def\MR#1{\href{http://www.ams.org/mathscinet-getitem?mr=#1}{MR#1}}
\def\arXiv#1{arXiv:\href{http://arXiv.org/abs/#1}{#1}}
\usepackage{doi}

\numberwithin{equation}{section}
\numberwithin{figure}{section}

\newtheorem{theorem}{Theorem}[section]
\newtheorem{proposition}[theorem]{Proposition}
\newtheorem{lemma}[theorem]{Lemma}
\newtheorem{corollary}[theorem]{Corollary}

\newtheorem*{claim}{Claim}
\theoremstyle{definition}
\newtheorem{definition}[theorem]{Definition}
\newtheorem{example}[theorem]{Example}
\theoremstyle{remark}
\newtheorem{remark}[theorem]{Remark}

\newcommand{\id}{\mathop{\textup{id}}}

\newcommand{\wth}{\widetilde{h}}
\newcommand{\jbl}{{\boxtimes}}

\DeclareMathOperator{\dsupp}{dsupp}

\DeclareMathOperator{\inj}{inj}
\DeclareMathOperator{\Poisson}{Poisson}

\newcommand{\eps}{\varepsilon}
\newcommand{\bOmega}{{\mathbf{\Omega}}}
\newcommand{\wOmega}{{\widetilde{\Omega}}}
\newcommand{\wB}{\widetilde{B}}
\newcommand{\wcF}{\widetilde{\mathcal{F}}}
\newcommand{\wmu}{\widetilde{\mu}}
\newcommand{\EE}{\mathbb{E}}
\newcommand{\RR}{\mathbb{R}}
\newcommand{\PP}{\mathbb{P}}

\newcommand{\NN}{\mathbb{N}}

\newcommand{\cP}{\mathcal{P}}
\newcommand{\cQ}{\mathcal{Q}}

\newcommand{\cW}{\mathbb{W}}
\newcommand{\cS}{\mathcal{S}}
\newcommand{\cF}{\mathcal{F}}
\newcommand{\cG}{\mathcal{G}}
\newcommand{\cU}{\mathbb{U}}
\newcommand{\cC}{\mathcal{C}}

\newcommand{\wcW}{\widetilde{\cW}}
\newcommand{\wW}{\widetilde{W}}

\newcommand{\wg}{\widetilde{g}}
\newcommand{\sP}{\mathscr{P}}
\newcommand{\sQ}{\mathscr{Q}}
\newcommand{\hfG}{\hat{\mathfrak G}}
\newcommand{\fG}{\mathfrak G}
\newcommand{\deltt}{\delta_{2\to 2}}
\def\d22{d_{2\to 2}}
\def\delGP{\delta_\diamond}
\def\tdelGP{\widetilde\delta_\diamond}
\def\tdel22{{\widetilde{\delta}}_{2\to 2}}
\def\djbl{d_{\jbl}}
\def\deljbl{\delta_{\jbl}}
\def\tdeljbl{\widetilde\delta_{\jbl}}

\title[Identifiability for graphexes and the weak kernel metric ]
{Identifiability for graphexes\\ and the weak kernel metric }

\author[Borgs]{Christian Borgs}
\address{Microsoft Research\\
One Memorial Drive\\
Cambridge, MA 02142} \email{borgs@microsoft.com}

\author[Chayes]{Jennifer T.\ Chayes}
\address{Microsoft Research\\
One Memorial Drive\\
Cambridge, MA 02142} \email{jchayes@microsoft.com}

\author[Cohn]{Henry Cohn}
\address{Microsoft Research\\
One Memorial Drive\\
Cambridge, MA 02142} \email{cohn@microsoft.com}

\author[Lov\'asz]{L\'aszl\'o Mikl\'os Lov\'asz}
\address{Mathematics Department\\
UCLA\\
Los Angeles, CA 90095} \email{lmlovasz@math.ucla.edu}

\begin{document}

\begin{abstract}
In two recent papers by Veitch and Roy and by Borgs, Chayes, Cohn, and
Holden, a new class of sparse random graph processes based on the concept
of graphexes over $\sigma$-finite measure spaces has been introduced. In
this paper, we introduce a metric for graphexes that generalizes the cut
metric for the graphons of the dense theory of graph convergence. We show
that a sequence of graphexes converges in this metric if and only if the
sequence of graph processes generated by the graphexes converges in
distribution. In the course of the proof, we establish a regularity lemma
and determine which sets of graphexes are precompact under our metric.
Finally, we establish an identifiability theorem, characterizing when two
graphexes are equivalent in the sense that they lead to the same process of
random graphs.
\end{abstract}

\maketitle

\tableofcontents

\section{Introduction} \label{sec:intro}

The theory of graph limits has been extensively developed for dense graph
sequences \cite{BCLSV06,ls-graphlimits,ls-analyst,BCLSV08,BCL10,BCLSV12}, but
the sparse case is not as well understood. In this paper, we study a model
introduced and studied in a sequence of papers
\cite{caron-fox,VR15,BCCH16,JANSON16,VR16,JANSON17,BCCV17} based on the
notion of graphexes. In contrast to the graphons of the dense theory, which
are symmetric two-variable functions defined over a probability space,
graphexes are defined over $\sigma$-finite measure spaces, and, in addition
to a graphon part $W$, contain two other components: a function $S$ taking
values in $\RR_+$, and a parameter $I\in \RR_+$. Formally, the graphex is
then the quadruple $\cW=(W,S,I,\bOmega)$, where $\bOmega=(\Omega,\cF,\mu)$ is
the underlying measure space.

A graphex then leads to a process $(G_T(\cW))_{T\geq 0}$ of random graphs as
follows: starting from a Poisson process over $\Omega$ with intensity $T\mu$,
one attaches $\Poisson(TS(x_i))$ leaves to each Poisson point $x_i$, and in
addition, joins two Poisson points $x_i,x_j$ with probability $W(x_i,x_j)$.
Finally, one adds $\Poisson(T^2I)$ isolated edges not connected to any of the
other points. Removing isolated vertices as well as the labels of the
remaining vertices gives a \emph{graphex process} $(G_T(\cW))_{T\geq 0}$ of
unlabeled graphs \emph{sampled from $\cW$}.

Several notions of convergence for graphexes were introduced in \cite{BCCH16}
and \cite{VR16} and further studied in \cite{JANSON17}. Among these notions,
we will be particularly interested in graphex process convergence
(GP-convergence), which was introduced in \cite{VR16}. A sequence of
graphexes is GP-convergent if the random graph processes generated by the
graphexes in the sequence converge. It was pointed out in \cite{JANSON17}
that GP-convergence can be metricized using the abstract theory of
probability measures over Polish spaces, but this does not give a very
explicit metric on graphexes; in fact, it does not even allow us to determine
when two graphexes lead to the \emph{same} random graph process.

In this paper, we introduce a concrete notion of distance for graphexes that
is equivalent to GP-convergence, which can be thought of as corresponding to
the ``cut distance'' in the dense case. For reasons we explain in the next
section, we call it the ``weak kernel distance''. We show that convergence in
this distance is equivalent to GP-convergence.

In general, the set of all graphexes is not compact. Indeed, it is not
difficult to show that for a set to be compact under GP-convergence, certain
uniform boundedness assumptions are necessary on the set of graphexes, which
we call ``tightness''. As a part of our proof that our weak kernel distance
metricizes GP-convergence, we develop a (Frieze-Kannan-type) regularity lemma
for graphexes and show that the sets that are precompact under the weak
kernel metric are precisely those that are tight.

Finally, we prove an identifiability theorem, showing to what extent a
graphex can be identified from its graphex process. Formulated differently,
we give a characterization of the equivalence classes of graphexes, where two
graphexes are called equivalent if they give rise to the same graphex
process. Generalizing a construction that was developed by Janson for the
dense case \cite{JANSON13}, we assign to each graphex $\cW$ a ``canonical
version'' $\widehat{\cW}$ such that $\cW$ is a pullback of $\widehat{\cW}$
and show that if two graphexes are equivalent, then their canonical versions
are isomorphic up to measure zero changes. This in turn will imply that two
graphexes $\cW_1$ and $\cW_2$ are equivalent if and only if there is a third
graphex $\cW_3$ (which can be taken to be their canonical graphex) such that
after restricting the two graphexes to their ``support'' (strictly speaking,
we have to restrict them to their ``degree support'', a notion we will define
in the next section) both $\cW_1$ and $\cW_2$ are pullbacks of $\cW_3$. We
note that this proves a conjecture of Janson; see Remark~5.4 in
\cite{JANSON16}.

We note that in this paper we treat graphexes slightly differently from the
definition in \cite{VR15,JANSON16,VR16,JANSON17,BCCV17}. Namely, as in
\cite{BCCH16}, we follow the convention from the theory of dense graph
limits, and define the graphex process corresponding to a graphex as a
process of graphs without loops. Indeed, we believe that a theory with loops
is most naturally embedded into a more general theory of graphex processes
with multi-edges and loops, which is beyond the scope of this paper.

Nonetheless, it is worth pointing out that the reader interested in the theory with
loops (but not multi-edges) can derive many results for this theory from
those developed here, even though some of the theorems will need to be
modified to accommodate additional technical complications. For the
identifiability theorem, this is done in
Appendix~\ref{sec:samplingwithloops}.

Finally, we note that while signed graphexes (i.e., graphexes for which $W$,
$S$ and $I$ are not necessarily non-negative) do not make much sense if we
want to use them to generate a random graph process, they are quite natural
from an analytic point of view. Indeed, we will prove several of our results
for signed graphexes. Still, the goal of this paper is to study unsigned
graphexes, and our results on signed graphexes should be considered more of
an aside at this point.

\section{Definitions and statements of main results}
\label{sec:defs}

\begin{definition}\label{def:graphex}
A graphex $\cW=(W,S,I,\bOmega)$ consists of a $\sigma$-finite measure space
$\bOmega=(\Omega,\cF,\mu)$, a symmetric measurable function $W\colon \Omega
\times \Omega \rightarrow [0,1]$, a measurable function $S\colon \Omega
\rightarrow \RR^+$, and a nonnegative real number $I$ such that the following
\emph{local finiteness} conditions hold:
\begin{enumerate}
\item $W(\cdot, x)$ is integrable for almost all $x\in\Omega$, and
\item there exists a measurable subset $\Omega' \subseteq \Omega$ such that
    $\mu(\Omega \setminus \Omega')<\infty$ and $\cW|_{\Omega'}$ is
 integrable.
\end{enumerate}
The quadruple will be called a \emph{signed graphex} if instead of taking
values in $[0,1]$ and $\RR_+$, $W$, $S$ and $I$ take values in $\RR$. The
graphex $\cW=(W,S,I,\bOmega)$ is called \emph{integrable} if
\[\|\cW\|_1 :=\int_{\Omega \times \Omega} |W(x,y)|\,d\mu(x)\,d\mu(y) + 2 \int_\Omega |S(x)| \,d\mu(x)
+2|I|<\infty,\] and the \emph{restriction} $\cW|_{\Omega'}$ of $\cW$ to
$\Omega' \subseteq \Omega$ is defined as the quadruple
$\cW'=(W',S',I',\bOmega')$ with $\bOmega'=(\Omega',\cF',\mu')$, where
$\cF'=\{A\in \cF : A\subseteq\Omega'\}$, $\mu'$ is the restriction of $\mu$
to $\cF'$, $S'$ is the restriction of $S$ to $\Omega'$, and $W'$ is the
restriction of $W$ to $\Omega'\times\Omega'$.
\end{definition}

We often refer to $\cW$ as a signed graphex over $\bOmega$, and we will refer
to the function $W$ as a \emph{graphon}, or the \emph{graphon part} of $\cW$.
Similarly, $S$ will be called a \emph{star intensity}, or the \emph{star
part} of $\cW$, and $I$ will be called a \emph{dust density}, or the
\emph{dust part} of $\cW$. (The reason for this terminology will become clear
when we discuss the random graph process generated by an unsigned graphex
$\cW$; as we will see, the star part of $\cW$ will lead to stars, and the
dust part will lead to isolated edges, which we call dust following
\cite{JANSON17}.) If two signed graphexes $\cW_1,\cW_2$ are defined on the
same space $\bOmega$, then we say that $\cW_1=\cW_2$ almost everywhere if
$W_1=W_2$ almost everywhere, $S_1=S_2$ almost everywhere, and $I_1=I_2$.

We define the marginal of a graphex $\cW=(W,S,I,\bOmega)$ over
$\bOmega=(\Omega,\cF,\mu)$ as the a.e.\ finite function $D_\cW\colon
\Omega\to\RR_+$ defined by
\[
D_\cW(x)=D_W(x)+S(x)\quad\text{ where }\quad D_W(x)=\int_{\Omega} W(x,y)\,d\mu(y).
\]
We say that $\cW$ has $D$-bounded marginals if $\|D_{\cW}\|_\infty \le D$.
Finally, we define its degree support as the set
\[
\dsupp \cW=\left\{x\in\Omega : D_{\cW}(x)>0\right\}.
\]
Note that $\cW$ is integrable if and only if its marginals are integrable.

Given a graphex $\cW$, we will define a stochastic process
$(\cG_T(\cW))_{T\geq 0}$ indexed by $T\in\RR_+$ and taking values in the set
of graphs with labels in $\RR_+$. To make this precise, we need to define a
$\sigma$-algebra over the set of countable graphs with vertices in $\RR_+$.
To this end, we first define the adjacency measure $\xi_G$ of a countable
graph $G$ with vertices in $\RR_+$ as the measure $\xi_G$ on $\RR_+^2$ given
by
\[
\xi_G=\sum_{t,t'\in V(G):\{t,t'\}\in E(G)}\delta_{(t,t')}.
\]
We call $\xi$ an \emph{adjacency measure} if there exists a countable graph
$G$ such that $\xi=\xi_G$. We then equip the set of adjacency measures with
the smallest $\sigma$-algebra such that the maps $\xi\mapsto \xi(A)$ are
measurable for all bounded Borel sets $A\subseteq\RR_+^2$, and the set of
countable graphs with vertices in $\RR_+$ with the smallest $\sigma$-algebra
such that the maps $G\mapsto \xi_G$ are measurable.

A graphex $\cW=(W,S,I,\bOmega)$ then generates a family $(\cG_T(\cW))_{T\geq
0}$ of random graphs as follows: we start with a Poisson point process with
intensity $\lambda\times\mu$ on $\RR_+\times\Omega$, where $\lambda$ is the
Lebesgue measure on $\RR_+$, and then connect two points $(t,x)$ and
$(t',x')$ of the Poisson process with probability $W(x,x')$, independently
for all pairs of points. For each point of the Poisson process $(t,x)$, we
take another Poisson point process on $\RR_+$ with intensity ${S(x)}\lambda$,
and connect $(t,x)$ to a vertex with ``birth time'' $t_i$ for each point
$t_i$ in the process. We also take a Poisson process with intensity
${I(\lambda\times\lambda)}$ on $\RR_+^{2}$, and for each point $(t_x,t_y)$ we
take an isolated edge between vertices with birth time $t_x$ and $t_y$. If we
ignore the labels in the feature space $\Omega$ and delete the vertices with
degree zero, this leads to an infinite graph $\cG_\infty(\cW)$ with vertices
labeled by their birth time $t\in \RR_+$. We then define $\cG_T(\cW)$ by
first taking the induced subgraph on the set of vertices which lie in $[0,T]$
and then deleting vertices whose neighbors in $\cG_\infty(\cW)$ all lie
outside the interval $[0,T]$.

We will refer to the part of $\cG_\infty(\cW)$ generated with the help of the
dust intensity $I$ as the \emph{dust part} of $\cG_\infty(\cW)$, and as the
part generated with the help of the star intensity $S$ as the \emph{stars} in
$\cG_\infty(\cW)$. While it may not be \emph{a priori} clear whether these
parts can be inferred from just observing the infinite graph
$\cG_\infty(\cW)$, this is actually the case, a fact which was first noted in
Remark~5.4 in \cite{JANSON16}: almost surely, the dust part consists of all
edges in $\cG_\infty(\cW)$ that are isolated, the star part consists of all
edges with one vertex of degree one and a second vertex of infinite degree,
and the remaining edges are generated by the graphon part of $\cW$ and have
two endpoints with infinite degree.

\begin{definition}\label{def:graphex-process}
Let $\cW$ be a graphex, let $(\cG_T(\cW))_{T\geq 0}$ be the random family of
graphs defined above, and let $\xi[\cW]$ be the random adjacency measure
$\xi_{\cG_\infty(\cW)}$. We call the stochastic process $(\cG_T(\cW))_{T\geq
0}$ the \emph{graphex process generated by} $\cW$, and the adjacency measure
$\xi[\cW]$ the \emph{adjacency measure generated by} $\cW$. We say two
graphexes are \emph{equivalent}, if the graphex processes generated by these
graphons are equal in law.
\end{definition}

\begin{remark}\label{rem:graphex-process}
(1) Following \cite{BCCV17}, we defined a graphex process as a stochastic
process taking values in a space of graphs with labels in $\RR_+$.
Alternatively, one might want to define a graphex process as a process taking
values in the space of unlabeled graphs without isolated vertices. In our
current context, this would correspond to ignoring the time labels of the
graphs in $\cG_T(\cW)$, leading to a graph which we denote by $G_T(\cW)$.
When it is important to distinguish them, we will refer to the process
$(G_T(\cW))_{T\geq 0}$ as the unlabeled graphex process corresponding to
$\cW$, and to the process $(\cG_T(\cW))_{T\geq 0}$ as the labeled graphex
process corresponding to $\cW$. Note that it is easy to recover $\cG_T(\cW)$
from $G_T(\cW)$: just assign i.i.d.\ labels chosen uniformly at random in
$[0,T]$ to all vertices. A related observation is the fact that $G_T(\cW)$
can be generated by first choosing $(x_i)_{i\geq 1}$ according to a Poisson
process with intensity $T \mu$ in $\Omega$, then connecting $i$ and $j$ with
probability $W(x_i,x_j)$, then adding a star whose number of leaves are
chosen as a Poisson random variable with mean $TS(x_i)$ to each point of the
process $(x_i)_{i\geq 1}$, and finally adding independent edges with rate
$IT^2$. Forgetting the labels then gives us $G_T(\cW)$. Relabeling each
vertex in the resulting graph independently by a uniform $t\in [0,T]$, we
obtain $\cG_T(\cW)$.

(2) It is sometimes convenient to assign a feature value to the endpoints of
the isolated edges generated from the dust part $I$ in the graphex, as well
as to the leaves of the stars generated using the function $S$. For our
purpose, we will say that these vertices have the feature label $\infty$, and
we will extend the marginal $D_\cW$ to $\Omega\cup\{\infty\}$ by setting
\[D_\cW(\infty)=\int_\Omega S(x) \,d \mu(x) +2I.
\]
Note that with this notation, $ \|\cW\|_1=\int_\Omega D_\cW(x)\, d\mu(x)
+D_\cW(\infty)$.

(3) In view of (2), one might want to equip the extended feature space
$\widetilde\Omega=\Omega\cup\{\infty\}$ with a $\sigma$-finite measure by
keeping the original measure on $\Omega$, and assigning some finite measure
$Q=\widetilde\mu(\infty)$ to the feature value $\infty$, giving a new
$\sigma$-finite measure space $\widetilde\bOmega_Q$.  On
$\widetilde\bOmega_Q$, one can then define a graphex of the form
$\widetilde\cW_Q=(\widetilde W_Q,0,0,\widetilde\bOmega)$ by setting
$\widetilde W_Q$ equal to $W$ on $\Omega\times\Omega$ and to $2I/Q^2$ on
$\{\infty\}\times\{\infty\}$, and by setting $\widetilde W(x,\infty)=S(x)/Q$
and $\widetilde W(\infty,y)=S(y)/Q$ if only one of the two features $x,y$
lies in $\Omega$.  With this construction, $D_{\widetilde
W_Q}(\infty)=D_\cW(\infty)/Q$, $D_{\widetilde W_Q}(x)=D_\cW(x)$ if
$x\in\Omega$, and $\|\cW\|_1=\|\widetilde W_Q\|_1=\int d\mu(x) D_{\widetilde
W_1}+QD_{\widetilde W}(\infty)$, with the notation in (2) corresponding to
the case $Q=1$. It is clear that the graphon process generated from
$\widetilde\cW$ cannot have exactly the same distribution as the one
generated from $\cW$ unless $I$ and $S$ are zero (to see this, note that in
$G_\infty(\widetilde\cW)$, all vertices have infinite degrees, while
$G_\infty(\cW)$ has vertices of degree one).  But one might wonder whether
the process generated from the ``pure graphon'' $\widetilde\cW$ approximates
the one generated from $\cW$. As we will see in
Remark~\ref{rem:W-approx-tildeW}, this is indeed the case, in the sense that
for any fixed $T$, the distribution of $G_T(\widetilde \cW_Q)$ converges to
$G_T(\cW)$.
\end{remark}

It is relatively easy to see that the local finiteness conditions (1) and (2)
from Definition~\ref{def:graphex} imply that the adjacency measure $\xi[\cW]$
is a.s.\ locally finite (i.e., $\xi[\cW](A)<\infty$ for all bounded Borel
sets $A\subset \RR_+^2$), or equivalently, that for all $T<\infty$, the
graphs $\cG_T(\cW)$ are a.s.\ finite. It turns out that these conditions are
also necessary for the local finiteness of $\xi[\cW]$. This is the main
statement of the following proposition, which we will prove in
Appendix~\ref{app:local-finite}. For graphexes over $\RR_+$ equipped with the
Lebesgue measure, a similar condition was established in \cite{VR15},
building on the work of \cite{Kal05} (the condition considered by \cite{VR15}
and \cite{Kal05} is the same as our condition (E) below, specialized to the
case $D=1$, even though it is clear that both \cite{VR15} and \cite{Kal05}
knew that for graphexes over $\RR_+$, conditions (D) and (E) are equivalent.)
To state the proposition, we use the notation $\{D_\cW> D\}$ for the set
$\{x\in\Omega : D_\cW(x)> D\}$, while $\{D_W> D\}$, $\{D_\cW\leq D\}$, and
$\{D_W\leq D\}$ are defined analogously.

\begin{proposition}
\label{prop:local-finite} Let $\cW=(W,S,I,\bOmega)$ be a $4$-tuple consisting
of a $\sigma$-finite measure space $\bOmega=(\Omega,\cF,\mu)$, a symmetric
measurable function $W\colon \Omega \times \Omega \rightarrow [0,1]$, a
measurable function $S\colon \Omega \rightarrow \RR^+$, and a nonnegative
real number $I$. Then the local finite conditions (1) and (2) from
Definition~\ref{def:graphex} are equivalent to the local finiteness of the
adjacency measure generated by $\cW$. If we assume condition (1), then
following are equivalent:
\begin{enumerate}[(A)]
\item The graphex $\cW$ obeys the local finiteness condition (2).
\item For all $D>0$, $\mu(\{D_\cW > D\}) <\infty$ and $\cW|_{\{D_\cW \le
    D\}}$ is integrable.
\item There exists a $D>0$ such that $\mu(\{D_\cW > D\}) <\infty$ and
    $\cW|_{\{D_\cW \le D\}}$ is integrable.
\item For all $D>0$, $\mu(\{D_W > D\}) <\infty$, and both $W|_{\{D_W \le
    D\}}$ and $\min\{S,1\}$ are integrable.
\item There exists a $D>0$ such that $\mu(\{D_W > D\}) <\infty$, and both
    $W|_{\{D_W\le D\}}$ and $\min\{S,1\}$ are integrable.
\end{enumerate}
\end{proposition}

Note that this proposition implies in particular that a graphex with bounded
marginals is integrable, since for graphexes with $\|D_\cW\|_\infty\leq D$
the graphex $\cW$ and the graphex $\cW|_{\{D_\cW\} \le D}$ are the same.

Having defined the graphex process associated with a graphex $\cW$, there are
several natural questions one might want to answer. In particular, one might
want to characterize when two graphexes lead to the same process, i.e., when
$\xi[\cW]$ and $\xi[\cW']$ have the same distribution. More generally, one
might want to define a metric on the set of graphexes such that the
distributions of $\xi[\cW]$ and $\xi[\cW']$ are close if $\cW$ and $\cW'$ are
close. Addressing these questions is one of the main goals of this paper.

Before discussing this further, it will be useful to embed the theory of
graphex processes into the general theory of locally finite point processes.
To this end, we first introduce the set $\mathcal N=\mathcal N(\RR_+^2)$ of
locally finite counting measures on $\RR^2_+$ (i.e., the set of measures
$\xi$ such that $\xi(A)$ is a finite, non-negative integer for all bounded
Borel sets $A\subset\RR_+^2$), and equip it with the Borel $\sigma$-algebra
inherited from the vague topology (defined as the coarsest topology for which
the maps $\xi\mapsto \int f \,d\xi$ are continuous for all continuous
functions $f\colon \RR_+^2\to \RR_+$ with compact support). As shown in,
e.g., \cite{DVJ03v1}, Appendix A2.6, the vague topology on $\mathcal N$ can
be metricized in such a way that $\mathcal N$ becomes a complete, separable
metric space, making $\mathcal N$ into a Polish space, and the Borel
$\sigma$-algebra inherited from this topology is the smallest
$\sigma$-algebra such that for all bounded Borel sets $A\subset\RR_+^2$ the
maps $\mu\mapsto \mu(A)$ are measurable.

As usual, a locally finite point process on $\RR_+^2$ is then defined as a
random measure on $\mathcal N(\RR^2_+)$ equipped with this Borel algebra, and
convergence in distribution is defined as weak convergence in the set of
probability measures on $\mathcal N$, so that convergence in distribution of
a sequence of locally finite point process $\xi_n$ on $\RR_+^2$ to a locally
finite point process $\xi$ is defined by the condition that $\EE[
F(\xi_n)]\to\EE [F(\xi)]$ for all continuous, bounded functions $F$, with
continuity defined with respect to the vague topology on $\mathcal N$. As
observed in \cite{JANSON17}, the fact that $\mathcal N$ is Polish makes the
set of probability distributions on $\mathcal N$ a Polish space as well (see,
e.g., \cite{Bil}, Appendix III for a proof), showing that convergence in
distribution for locally finite point processes on $\RR_+^2$ can be
metricized.

Next we consider the set $\hfG$ of simple graphs $G$ with vertices in $\RR_+$
such that (a) no vertex in $G$ is isolated, and (b) for all $T<\infty$, the
induced subgraph of $G$ on $V(G)\cap [0,T]$ is finite. We also consider the
subset $\hfG_0$ of finite graphs in $\hfG$. The map $G\mapsto \xi(G)$ then
gives a one-to-one map between graphs in $\hfG$ and adjacency measures. In
particular, $\hfG$ and its subset $\hfG_0$ inherit the vague topology and
corresponding Borel $\sigma$-algebra from $\mathcal N$. In this language, the
graphex process $(\cG_T(\cW))_{T\geq 0}$ then becomes a CADLAG stochastic
process with values in $\hfG_0$ indexed by a time $T\in \RR_+$.

Note that $\hfG$ should be distinguished from the set of unlabeled countable
graphs without isolated vertices, $\fG$. While we will not equip $\fG$ with
any topology, the set of finite unlabeled graphs without isolated vertices,
denoted by $\fG_0$, will be given the discrete topology. In this language,
the unlabeled graphex process $(G_T(\cW))_{T\geq 0}$ introduced in
Remark~\ref{rem:graphex-process} is then a CADLAG process with values in
$\fG_0$.

In \cite{JANSON17}, various notions of convergence for graphons and graphexes
(proposed originally in \cite{BCCH16} and \cite{VR16}) were studied. Here we
are most interested in what \cite{VR16} introduces as GP-convergence, where
GP stands for graphex process. This notion is closely related to the notion
of sampling convergence for graphs introduced in \cite{BCCV17}; see Lemma 5.4
in that paper, as well as the discussion at the end of this section.
 Janson showed that the following are equivalent if
$\cW,\cW_1,\cW_2,\dots$ are graphexes:
\begin{enumerate}
\item $\xi(G(\cW_n))\to \xi(G(\cW))$ in distribution.
\item For every $T<\infty$, $\xi(\cG_T(\cW_n))\to \xi(\cG_T(\cW_n))$ in
    distribution.
\item For every $T<\infty$, $G_T(\cW_n)\to G_T(W)$ in distribution.
\end{enumerate}
Following \cite{VR16} we call this notion of convergence GP-convergence, and
say that $\cW_n$ is GP-convergent to $\cW$ if one of these equivalent
conditions holds.

As already alluded to above, Janson also observed that by the abstract theory
of probability measures over Polish spaces, this notion of convergence can be
metricized, turning the set of graphexes into a complete, separable metric
space. But this abstract theory does not give a very explicit metric on the
space of locally finite graphexes; in fact, it does not even address the
question of when two graphexes are equivalent in the sense that the resulting
point processes are equal in law.

To discuss the second question, we define measure-preserving transformations,
pullbacks, and couplings. Given two $\sigma$-finite spaces
$\bOmega=(\Omega,\cF,\mu)$ and $\bOmega'=(\Omega',\cF',\mu')$, we say that a
map $\phi\colon \Omega'\to\Omega$ is a \emph{measure-preserving
transformation} if $\phi$ is measurable and $\mu'(\phi^{-1}(A))=\mu(A)$ for
all $A\in \cF$. If $\cW=(W,S,I,\bOmega)$ is a signed graphex over $\bOmega$,
we define its \emph{pullback} under $\phi$ to be the graphex
$\cW^\phi=(W^\phi,S^\phi,I,\bOmega')$ where
$W^\phi(x',y')=W(\phi(x'),\phi(x'))$ and $S^\phi(x')=S(\phi(x))$. It is clear
that for unsigned graphexes $\cW$ and $\cW^\phi$ give rise to the same
process of random graphs. Note that we can define the pullback even when
$\varphi$ is measurable but not measure-preserving, but in this case the two
graphexes do not necessarily give rise to the same random process.
Nevertheless, we will sometimes use pullbacks in this situation. If we do, we
will write $\cW^{\varphi,\mu'}$ to emphasize the dependence on the measure on
$\Omega'$. Given two $\sigma$-finite spaces
$\bOmega_1=(\Omega_1,\cF_1,\mu_1)$ and $\bOmega_2=(\Omega_2,\cF_2,\mu_2)$, we
say that $\mu$ is a coupling of $\mu_1$ and $\mu_2$ if $\mu$ is a measure on
$\cF_1\times\cF_2$ such that $\mu(\Omega_1\times S_2)=\mu_2(S_2)$ and
$\mu(S_1\times\Omega_2)=\mu_1(S_1)$ for all $S_1\in\cF_1$ and all
$S_2\in\cF_2$. Note that the existence of such a coupling implies that
$\mu_1(\Omega_1)=\mu(\Omega_1\times\Omega_2)=\mu_2(\Omega_2)$. It turns that
this condition is both necessary and sufficient for the existence of a
coupling; see \cite{BCCH16} for a proof.

Based on the known results for dense graphs, one might conjecture that two
graphexes are equivalent if and only if there exists a third graphex such
that both are pullbacks of this third graphex. It turns out that this is not
quite correct, but that it is correct once we remove the part of the
underlying space on which $D_\cW=0$. This is the statement of the following
theorem, which is one of the main results of this paper, and will be proved
in Section~\ref{sec:idnetify}.

\begin{theorem} \label{thm:identify}
Let $\cW_1=(W_1,S_1,I_1,\bOmega_1)$ and $\cW_2=(W_2,S_2,I_2,\bOmega_2)$ be
graphexes, where $\bOmega_i=(\Omega_i,\cF_i,\mu_i)$ are $\sigma$-finite
spaces. Then $G_t(\cW_1)$ and $G_t(\cW_2)$ have the same distribution for all
$t\in\RR_+$ if and only if there exists a third graphex $\cW=(W,S,I,\bOmega)$
over a $\sigma$-finite measure space $\bOmega=(\Omega,\cF,\mu)$ and measure
preserving maps $\phi_i\colon\dsupp \cW_i \rightarrow \Omega$ such that
$\cW_i|_{\dsupp \cW_i} =\cW^{\phi_i}$ almost everywhere.
\end{theorem}

\begin{remark}\label{rem:identify}
If the two graphexes are defined over Borel spaces, we can prove an analogous
statement where the measure-preserving maps are turned around. Specifically,
for the case where $\bOmega_1$ and $\bOmega_2$ are $\sigma$-finite Borel
spaces, we can prove that $G_t(\cW_1)$ and $G_t(\cW_2)$ have the same
distribution for all $t\in\RR_+$ if and only if there exists a
$\sigma$-finite Borel space $\bOmega=(\Omega,\cF,\mu)$ and measure preserving
maps $\pi_i\colon \Omega \rightarrow\dsupp \cW_i$ such that $(\cW_1|_{\dsupp
\cW_1})^{\pi_1}=(\cW_2|_{\dsupp \cW_2})^{\pi_2}$ almost everywhere. In this
case, the space $\bOmega$ can be chosen to be a coupling of $\bOmega_1$ and
$\bOmega_2$, with $\pi_i$ being the coordinate projections from
$\Omega_1\times\Omega_2$ to $\Omega_i$. See Theorem~\ref{coupling} in
Section~\ref{sec:idnetify}. For graphexes without a dust and star part, this
was independently established in \cite{BCCH16} (using a different proof); see
also \cite{JANSON16}, which establishes a similar result (with yet another
proof), this time giving a coupling of the two graphexes (again without dust
and star part) after trivially extending them rather than restricting them to
the support of their marginals.
\end{remark}

To address the first question, concerning the relationship between graphexes
and the point processes generated by them, we would like to define an
analogue of the cut distance for graphons between graphexes, so that two
graphexes are close if and only if their graphex processes are close. To this
end, we first define some norms of a function $U$ over $\Omega_1 \times
\Omega_2$ for two $\sigma$-finite spaces $\bOmega_1=(\Omega_1,\cF_1,\mu_1)$
and $\bOmega_2=(\Omega_2,\cF_2,\mu_2)$. We denote by $\|U\|_p$ the $L^p$ norm
of $U$ as a function over $\Omega_1 \times \Omega_2$ (so we forget the
product structure). Given two measurable functions $f\colon \Omega_1\to\RR$
and $g\colon \Omega_2\to\RR$, let
\[f \circ U(y)= \int_{\Omega_1} f(x)U(x,y)\,dy,\]
\[U \circ g(x)= \int_{\Omega_2} U(x,y)g(y)\,dy,\]
and
\[f \circ U \circ g= \int_{\Omega_1 \times\Omega_2} f(x)U(x,y)g(y)\,dy.\]
We will also use the notation $U_x$ for the function $y\mapsto U(x,y)$.

\begin{definition}\label{def:kernel-norm}
Given a function $U$ defined on $\Omega_1 \times \Omega_2$ for two
$\sigma$-finite measure spaces $\bOmega_i=(\Omega_i,\cF_i,\Omega_i)$ for
$i=1,2$, we define
\[\|U\|_{2 \rightarrow 2}=\sup_{f,g:\|f\|_2=\|g\|_2=1} f \circ U \circ g
=\sup_{g:\|g\|_2=1} \|U \circ g\|_2.\]
\end{definition}

Note that the norm $\|U\|_{2 \rightarrow 2} $ is simply the operator norm
when we consider $U$ the kernel of an operator $\widehat U$ from
$L^2(\bOmega_2)$ to $L^2(\bOmega_1)$. We will therefore call it the
\emph{kernel norm} of $U$. Our next norm is a modification of the standard
cut norm; in the dense graph setting, it was first systematically used in
\cite{KLS14}, where it was defined as a norm for functions defined over a
probability space.

\begin{definition}\label{def:jumble-norm}
Given a measurable function $U$ defined on $\Omega \times \Omega$ for a
$\sigma$-finite measure space $\Omega$, we define the \emph{jumble norm}
\[\|U\|_{\jbl}=\sup_{S,T \subseteq \Omega} \left| \frac{1}{\sqrt{\mu(S)\mu(T)}}\int_{S \times T} U(x,y) \,d\mu(x)
\,d\mu(y)
\right|.\] Here the supremum is over subsets with finite and nonzero measure.
\end{definition}

It is easy to show that these are norms; in particular, they satisfy the
triangle inequality, and are equal to $0$ if and only if $U$ is zero almost
everywhere. If we want to stress the dependence of these norms on the measure
$\mu$ and the function $U$, we write $\|U\|_{*,\mu}$ instead of $\|U\|_*$,
where $*$ is replaced by the appropriate norm.

We will see later that for graphexes with uniformly bounded marginals and
uniformly bounded $\|\cdot\|_1$ norms, the $\|\cdot\|_{2 \rightarrow 2}$ norm
and the $\|\cdot\|_{\jbl}$ norm are equivalent
(Lemma~\ref{lemmasquare2to2equiv}), implying in particular that they are
equivalent in the theory of dense graph limits (where $\Omega$ has bounded
measure). In the dense setting, the above two norms are also equivalent to
the standard cut norm, defined as
\[
\|U\|_{\square}=\sup_{S,T \subseteq \Omega} \left|\int_{S \times T} U(x,y)\,d\mu(x) \,d\mu(y) \right|
=
\sup_{f,g\colon \Omega\to [0,1]}| f \circ U \circ g|.
\]
Indeed, $\|U\|_\square \le \|U\|_{\jbl} \mu(\Omega)$ and $\|U\|_{\jbl} \le
\sqrt{\|U\|_\square\|U\|_\infty}$, where the second bound follows from the
fact that
\[
\left|\int_{S \times T} U(x,y) \,d\mu(x) \,d\mu(y)
\right|\leq \inf\{\lambda(S)\lambda(T)\|U\|_\infty,\|U\|_\square\}.
\]
Therefore, in the theory of dense graph limits all three norms are
equivalent. However, although the cut norm is the simplest to state, we
believe that the kernel norm $\|\cdot\|_{2 \rightarrow 2}$ norm is the
correct extension to graphexes.

We now define some distances between graphexes. First, we define the $\deltt$
distance, which will define a notion of convergence that is equivalent to
GP-convergence for graphexes with uniformly bounded marginals. The definition
of $\deltt$ will make sense for signed graphexes, provided both the graphon
parts and the absolute marginals are in $L^1\cap L^2$. We will therefore
define the $\deltt$ metric in this more general\footnote{To see that this
setting is indeed more general than the assumption of bounded marginals for
(unsigned) graphexes we recall that by Proposition~\ref{prop:local-finite}, a
graphex with bounded marginals is integrable. Using this, and the fact that
by definition, the graphon part of a graphex is bounded, the claim is easy to
verify.} setting.
\begin{definition}
A signed graphex $\cW=(W,S,I,\bOmega)$ over $\bOmega=(\Omega,\cF,\mu)$ is
said to be in $L^1\cap L^2$ if both $W$ and $D|_{\cW|}$ are in $L^1\cap L^2$.
Here $|\cW|$ is the graphex $|\cW|=(|W|,|S|,|I|,\bOmega)$.
\end{definition}

Suppose $\cW_1=(W_1,S_1,I_1,\bOmega)$ and $\cW_2= (W_2,S_2,I_2,\bOmega)$ are
defined on the same underlying space $\bOmega$. We then define their
$\d22$-distance as
\begin{equation}\label{d22-def}
\d22(\cW_1,\cW_2)= \max\left( \|W_1-W_2\|_{2 \rightarrow 2},
\sqrt{\|D_{\cW_1}-D_{\cW_2}\|_2},\sqrt[3]{\left|
\rho(\cW_1)-\rho(\cW_2)
\right|}\right),
\end{equation}
where $\rho(\cW_i)$ is the ``edge density'' of the signed graphex $\cW_i$,
\begin{equation}
\label{rho-of-W-def} \rho(\cW_i)=\int W_i +2 \int S_i + 2I.
\end{equation}
The reason we take the roots will become clearer later when we define the
general distance $\delGP$. Since $\sqrt{c_1+c_2} \le \sqrt{c_1}+\sqrt{c_2}$
and $\sqrt[3]{c_1+c_2} \le \sqrt[3]{c_1}+\sqrt[3]{c_2}$, this is indeed a
metric.

Next, suppose two signed graphexes in $L^1\cap L^2$,
$\cW_1=(W_1,S_1,I_1,\bOmega_1)$ and $\cW_2= (W_2,S_2,I_2,\bOmega_2)$, are
defined over two $\sigma$-finite spaces $\bOmega_1=(\Omega_1,\cF_1,\mu_1)$
and $\bOmega_2=(\Omega_2,\cF_2,\mu_2)$ with
$\mu_1(\Omega_1)=\mu_2(\Omega_2)$. Let $\pi_1\colon \Omega_1 \times \Omega_2
\rightarrow \Omega_1$ and $\pi_2\colon \Omega_1 \times \Omega_2 \rightarrow
\Omega_2$ be the projections. Then we define $\tdel22(\cW_1,\cW_2)$ as the
infimum
\begin{equation}\label{tdel22-def}
\tdel22(\cW_1,\cW_2)
=\inf_\mu\d22(\cW_1^{\pi_1,\mu},\cW_2^{\pi_2,\mu}),
\end{equation}
where the infimum is over all couplings $\mu$ of $\mu_1$ and $\mu_2$.

To define the $\deltt$-distance we need one more notion, that of a trivial
extension of $\cW=(W,S,I,\bOmega)$, where $\bOmega=(\Omega,\cF,\mu)$ is a
$\sigma$-finite measure space. It is defined as a quadruple
$\cW'=(W',S',I',\bOmega')$ where $\bOmega'=(\Omega',\cF',\mu')$ is a
$\sigma$-finite measure space such that $\Omega\in \cF'$, $\cF=\{A\in \cF' :
A\subseteq \Omega\}$, and $\mu$ is the restriction of $\mu'$ to $\cF$, while
$W'$ is the extension of $W$ that is $0$ on the complement of
$\Omega\times\Omega$, $S'$ is the extension of $S$ that is $0$ on the
complement of $\Omega$, and $I'=I$. It is easy to see that taking a trivial
extension of a graphex has no effect on $\cG_T$ or $\cG_\infty$ (since
Poisson points sampled in the complement of $\Omega$ will be isolated for all
$T$).

\begin{definition}\label{def:del22}
Let $\cW_1$ and $\cW_2$ be signed graphexes in $L^1\cap L^2$. Then we define
\begin{equation}\label{del22-def}
\deltt(\cW_1,\cW_2)=
\tdel22(\cW_1',\cW_2'),
\end{equation}
where $\cW_1'$ and $\cW_2'$ are trivial extensions of $\cW_1$ and $\cW_2$ to
measure spaces of infinite total mass. We refer to $\deltt(\cW_1,\cW_2) $ as
the \emph{kernel distance} of $\cW_1$ and $\cW_2$ and call $\deltt$ the
\emph{kernel metric}.
\end{definition}

The existence of these extensions is trivial, since we can always append an
interval equipped with the Lebesgue measure. Nevertheless, it is not clear
that $\deltt(\cW_1,\cW_2)$ is well defined, since the right side of
\eqref{del22-def} could depend on the particular choice of the extensions
$\cW_1'$ and $\cW_2'$. In a similar way, while it is clear that $\deltt$ is
symmetric and that $\deltt(\cW,\cW)=0$, it is not clear that it is a metric
(even after factoring out the null space), since it is not clear that it
satisfies the triangle inequality. The following theorem addresses both
questions, and will be proved in Section~\ref{sec:prelim}.

\begin{theorem} \label{thm:deltt-metric}
Let $\cW_1$ and $\cW_2$ be signed graphexes in $L^1\cap L^2$. Then the right
side of \eqref{del22-def} does not depend on the choice of the trivial
extensions $\cW_1'$ and $\cW_2'$. Furthermore, given three signed graphexes
$\cW_1,\cW_2,\cW_3$ in $L^1\cap L^2$,
\[
\deltt(\cW_1,\cW_3) \le \deltt(\cW_1,\cW_2)+\deltt(\cW_2,\cW_3)
.
\]
Therefore, $\deltt$ is a well-defined pseudometric.
\end{theorem}

\begin{remark}
In \cite{BCCH16}, when defining the cut distance between two graphons, it was
only necessary to extend the smaller space to the larger one, and it was not
necessary to extend further. It is natural to ask whether a trivial extension
to a space of infinite metric is necessary, or, equivalently, whether for two
graphexes $\cW_1,\cW_2$ defined on spaces with the same (finite) measure,
$\tdel22(\cW_1,\cW_2)=\deltt(\cW_1,\cW_2)$. In contrast to the cut distance
discussed in \cite{BCCH16}, for the kernel metric it is sometimes necessary
to take trivial extensions of both spaces, not just an extension of the
smaller space to one of the same measure as the larger one. See
Example~\ref{ex1} in Section~\ref{sec:prelim}.
\end{remark}

Our next theorem states that on sets with uniformly bounded marginals, the
topology induced by the kernel metric $\deltt$ is equivalent to the topology
of GP-convergence. We will prove it in Sections~\ref{sec:subgraph-counts} and
\ref{sec:sampling}.

\begin{theorem}
\label{thm:deltt-GP} For any $D>0$, $\deltt$-convergence is equivalent to
GP-convergence on the space of graphexes with $D$-bounded marginals.
\end{theorem}

In general, $\deltt$-convergence implies GP-convergence, but the reverse is
not true. This is because if we do not assume bounded marginals, it is
possible to have a very small measure set with very large degree. This will
have a non-negligible effect on $\deltt$ distance; however, for a fixed $T$,
the chances of obtaining a vertex in the small set is small, and thus has a
small effect on sampling. To give a more concrete example, let $W_n$ be equal
to $1$ on $[0,1/n] \times [1,1+n]$ and $[1,1+n] \times [0,1/n]$, and zero
everywhere else. Let $\cW_n=(W_n,0,0,\RR_+)$. Then for any fixed $T$, the
probability of seeing a single edge in $G_T(\cW_n)$ converges to $0$, and
therefore $\cW_n$ is GP-convergent to $0$. However, it is easy to see that
$\deltt(\cW_n,0)$ does not converge to $0$. To address this issue, we will
define a new distance such that two graphexes whose graphex processes can be
obtained from each other by removing a small set of vertices are close in the
new metric. Our construction is loosely motivated by the construction of the
usual metric of weak convergence. For that reason, we will refer to the new
metric as the weak kernel metric.

Before defining this distance, we introduce the notation $\mu-r \le \mu' \le
\mu$ whenever $\mu,\mu'$ are two measures over the same measurable space
$(\Omega,\cF)$ such that
\[\mu(B)-r \le \mu'(B) \le \mu(B)\]
for all measurable sets $B$. Note that this property is equivalent to the
existence of a function $h\colon \Omega \rightarrow [0,1]$ such that
$\mu'(B)=\int_{B}h \,d\mu$ and $\|1-h\|_{1,\mu} \le r$. An example of such a
function, which we will often use, is the indicator function of a set $\Omega'\subseteq\Omega$
such that $\mu(\Omega\setminus\Omega') \le r$.

We will define the weak kernel metric for arbitrary graphexes (removing the
condition that they are in $L^1\cap L^2$), and in fact will again allow for
signed graphexes. We will assume that the graphon parts of these signed
graphexes are bounded in the $L^\infty$ norm, a condition which is true for
unsigned graphexes, since for these, the graphon part takes values in
$[0,1]$.

\begin{definition}\label{def:delGP}
Let $\cW_1=(W_1,S_1,I_1,\bOmega_1)$ and $\cW_2=(W_2,S_2,I_2,\bOmega_2)$ be
signed graphexes, where $\bOmega_i=(\Omega_i,\cF_i,\mu_i)$ and
$\|W_i\|_\infty<\infty$ for $i=1,2$. We define $\delGP(\cW_1,\cW_2)$ as the
infimum of the set of real numbers $c$ such that there exist two measures
$\widetilde\mu_1$ and $\widetilde\mu_2$ over $(\Omega_1,\cF_1)$ and
$(\Omega_2,\cF_2)$ that satisfy the following: the signed graphexes
$\widetilde \cW_1$ and $\widetilde \cW_2$ obtained from $\cW_1$ and $\cW_2$
by replacing $\mu_1$ and $\mu_2$ by $\widetilde\mu_1$ and $\widetilde\mu_2$,
respectively, are in $L^1\cap L^2$, and
\begin{enumerate}
\item for $i=1,2$, we have $\mu_i-c^2 \le \widetilde\mu_i \le \mu_i$, and
\item $\deltt(\widetilde \cW_1,\widetilde \cW_2)\leq c$.
\end{enumerate}
We refer to $\delGP(\cW_1,\cW_2)$ as the \emph{weak kernel distance} between
$\cW_1$ and $\cW_2$ and call $\delGP$ the \emph{weak kernel metric}.
\end{definition}

Note that for unsigned graphexes, the weak kernel distance is well defined
and finite. Indeed, given $0<D<\infty$, choose $\widetilde\mu_i$ as the
restriction of $\mu_i$ to $\{D_{\cW_i}\leq D\}$.
Proposition~\ref{prop:local-finite} then implies that $\{D_{\cW_i}>D\}$ has
finite measure, and $\cW_i|_{\{D_{\cW_i}\leq D\}}$ is integrable and hence in
$L^1\cap L^2$. The fact that $\delGP(\cW_1,\cW_2)$ is well defined for signed
graphexes with bounded graphon part follows from
Proposition~\ref{prop:local-finite} and further arguments, and is deferred to
Section~\ref{sec:prelim}; see in particular
Lemma~\ref{lem:signed-delGP<infty} in that section.

We will show that $\delGP$ is a pseudometric. It is clear that it is
symmetric, and that $\delGP(\cW,\cW)=0$. It is not obvious that it satisfies
the triangle inequality. We will prove this fact in Section \ref{sec:prelim}.

\begin{theorem} \label{thm:dmetric}
Given three signed graphexes $\cW_1,\cW_2,\cW_3$ with bounded graphon part,
\[\delGP(\cW_1,\cW_3) \le \delGP(\cW_1,\cW_2)+\delGP(\cW_2,\cW_3) .\]
Therefore, $\delGP$ is a pseudometric.
\end{theorem}

\begin{remark} \label{rem:pullbacksamedistance}
Given a signed graphex $\cW=(W,S,I,\bOmega)$ with $\bOmega=(\Omega,\cF,\mu)$
and a measure-preserving map $\varphi\colon \bOmega' \rightarrow \bOmega$,
let $\cW'=\cW^\varphi$ almost everywhere. We can take a coupling
$\widetilde{\mu}$ on $\bOmega' \times \bOmega$ defined by $\widetilde{\mu}(A
\times B)= \mu'(A \cap \varphi^{-1}(B))$. It is easy to see that then the
pullbacks of the two signed graphexes to $\bOmega' \times \bOmega$ will be
equal almost everywhere, which implies that
$\deltt(\cW,\cW')=\tdel22(\cW',\cW)=\delGP(\cW',\cW)=\tdelGP(\cW',\cW)=0$.
\end{remark}

With this new metric, we now have a definition of distance for any pair of
graphexes. Note that in general, the metrics $\deltt$ and $\delGP$ are not be
the same, even if both are finite. However, we will show that for graphexes
with uniformly bounded marginals, the two metrics provide the same topology.
This is the content of the next proposition, which will be proved in
Section~\ref{sec:prelim}.

\begin{proposition} \label{propboundedequivmetrics}
Fix $D<\infty$. Then $\delGP$ and $\deltt$ give an equivalent topology on the
space of graphexes with $D$-bounded marginals.
\end{proposition}

We will also show that convergence in the weak kernel metric $\delGP$ is
indeed equivalent to GP-convergence. This is the statement of the next
theorem, and is one of the two main results of this paper. It will be proved
using three main ingredients: a compactness statement stemming from a
suitable analogue of the Frieze-Kannan regularity lemma, a counting lemma
showing that subgraph counts in the graphs $G_T(\cW)$ are close if the
corresponding graphexes are close in the metric $\deltt$ (and the graphexes
have uniformly bounded marginals), and a sampling lemma showing that as
$T\to\infty$, the suitably rescaled graphex process $G_T(\cW)$ converges to
$\cW$ in probability. These techniques are developed in
Sections~\ref{sec:reg}, \ref{sec:subgraph-counts}, and \ref{sec:sampling},
and are combined to prove the theorem at the end of
Section~\ref{sec:sampling}, where we will also prove
Theorem~\ref{thm:deltt-GP}.

\begin{theorem}\label{thm:delGP-GP}
Given a sequence of graphexes $\cW_n$ and a graphex $\cW$, $\cW_n$ is
GP-convergent to $\cW$ if and only if $\delGP(\cW_n,\cW) \rightarrow 0$.
\end{theorem}

\begin{remark}
The reader might wonder whether instead of building our metric for
GP-convergence around the kernel norm $\|\cdot\|_{2\to 2}$, one could
equivalently build it around the cut norm, $\|\cdot\|_{\square}$. Concretely,
one might want to define $d_\square$ by replacing the kernel norm in
\eqref{d22-def} by the cut norm and the $L^2$ norm by the $L^1$ norm, then
proceed as in \eqref{tdel22-def} and \eqref{del22-def} to obtain a cut
distance $\delta_\square$ between graphexes with bounded marginals, and
finally proceed as in Definition~\ref{def:delGP} to obtain a ``weak cut
metric'' for arbitrary graphexes.

The following example shows that this approach does not work, in that it will
not metricize GP-convergence. Define $W_n$ to be the graphex that is constant
and equal to $n^{-2}$ over $[0,n]^2$ and $0$ everywhere else, and set
$\cW_n=(W_n,0,0,\RR_+)$, where $\RR_+$ is equipped with the Lebesgue measure.
The marginal $D_{\cW_n}$ of $\cW_n$ is then equal to $1/n$ times the
indicator function of the interval $[0,n]$, and its $L^1$ norm is equal to
$1$. It is then not hard to check that $\cW_n$ converges to the pure dust
graphex $(0,0,1,\RR_+)$ in the metric $\delGP$. Indeed, $\|W_n\|_{2\to2}\to
0$ and $\|D_{\cW_n}\|_2=n^{-1/2}\to 0$, while $\|\cW_n\|=1\to 1=\|\cW\|_1$,
which immediately implies convergence in the metric $\delGP$ and hence
GP-convergence (based on the proof of equivalence in this paper, though for
this specific case it is simple to check GP-convergence directly). By
contrast, $\|W_n\|_\square =\|W_n\|_1 =1$ stays bounded away from zero,
showing in particular that $\cW_n$ does not converge to $\cW$ in the cut
metric $\delta_\square$. Since changing the Lebesgue measure to a measure
$\mu_n$ such that $\lambda-\eps_n\leq\mu_n\leq \lambda$ with $\eps_n \to0$
will asymptotically not change the cut norm of $W_n$, the graphexes $\cW_n$
do not converge to $\cW$ in the weak cut metric either. Note that this can't
be cured by choosing a different norm for the marginal difference
$D_{\cW_1}-D_{\cW_2}$, e.g., by keeping the $L^2$ norm for that part, since
the above counter example works independently of the norm used for that part.
\end{remark}

In studying the general topology of graphexes, we define a notion of
tightness for sets of graphexes. Tight sets play an important role, in
particular, they are the precompact sets in our topology: any sequence that
is tight has a convergent subsequence, and any convergent sequence must be
tight.

\begin{definition}
\label{def:tight} A set $\cS$ of graphexes is \emph{tight} if for every
$\varepsilon>0$, there exist $C$ and $D$ such that for every $\cW \in \cS$,
$\cW=(W,S,I,\bOmega)$ with $\bOmega=(\Omega,\cF,\mu)$, there exists
$\Omega_\varepsilon \subseteq \Omega$ such that $\mu(\Omega_\varepsilon) \le
\varepsilon$ and the graphex $\cW'=\cW|_{\Omega \setminus
\Omega_\varepsilon}$ is $(C,D)$-bounded. Here a graphex $\cW'$ is called
$(C,D)$-\emph{bounded} if its marginals are $D$-bounded and $\|\cW'\|_1 \le
C$.
\end{definition}

Note that Proposition~\ref{prop:local-finite} implies that every finite set
of graphexes is tight. In Section~\ref{sec:tight}, we will prove that a set
$\cS$ of graphexes is tight if and only if for all fixed $T$, the
corresponding set $\{G_T(\cW)\}_{\cW\in \cS}$ of unlabeled graphex processes
at time $T$ is tight (which will also be equivalent to the existence of some
$T>0$ such that $\{G_T(\cW)\}_{\cW\in \cS}$ is tight; see
Theorem~\ref{theoremtightequiv} below). Here, as usual, a collection
$\mathscr S$ of distributions on finite graphs is called tight if for every
$\varepsilon>0$, there exists a finite set $T$ of graphs such that for each
of the random graphs in $\mathscr S$, the probability that the random graph
is not isomorphic to a graph in $T$ is at most $\varepsilon$. This is
equivalent to the set of random measures being tight under the discrete
topology on the set of isomorphism classes of finite graphs, or the set of
distributions of the number of edges being tight.

Our main theorem concerning tightness is the following theorem. It will be
proved in Section~\ref{sec:reg}, where we will establish a version of the
weak (or Frieze-Kannan) regularity lemma for graphexes. Note that while our
regularity lemma will hold for signed graphexes, the following is only stated
for unsigned graphexes. The reason is that our proof relies heavily on the
notion of tightness, which we only develop for unsigned graphexes; see also
Remark~\ref{rem:signed-tightness} in Section~\ref{sec:tight}.

\begin{theorem}\label{thm:complete}
The space of all graphexes is complete under the topology induced by the weak
kernel metric $\delGP$. A subset is relatively compact if and only if it is
tight. In particular, for any $C$ and $D$, the set of graphexes with
$\|\cW\|_1 \le C$ is compact under $\delGP$, and the set of $(C,D)$-bounded
graphexes is compact under both $\delGP$ and $\deltt$.
\end{theorem}

\begin{remark}\label{rem:signed-compactness}
As mentioned above, we only develop the theory of tightness for unsigned
graphexes. In particular, we don't characterize the set of precompact signed
graphexes. That notwithstanding, some of our compactness results do hold for
signed graphexes. Here we only mention that the analogue of the statement for
the set of graphexes with $\|\cW\|_1 \le C$ holds for signed graphexes as
well, provided we restrict the $L^\infty$ norm of the graphon part (which by
definition is bounded by $1$ for unsigned graphexes). To be explicit, any
sequence of signed graphexes $\cW_n=(W_n,S_n,I_n,\bOmega_n)$ with
$\|W_n\|_\infty\leq B$ and $\|\cW_n\|_1\leq C$ has a subsequence converging
to a signed graphex $\cW=(W,S,I,\bOmega)$ with $\|W\|_\infty\leq B$ and
$\|\cW\|_1\leq C$. See Remark~\ref{rem:signed-compactness-pf} in
Section~\ref{sec:reg} below.
\end{remark}

The advantage of $(C,D)$-bounded (unsigned) graphexes is that although there
is no \emph{a priori} bound on the size of $G_T(\cW)$ at any given time $T$,
for any finite graph $F$, the expected number of copies of $F$ in $G_T(\cW)$
is finite. Furthermore, it turns out that under the assumption of
$(C,D)$-boundedness, if two graphexes have the same subgraph densities, then
they are equivalent, i.e., have $\deltt$ distance $0$. In this way, we can
heuristically think of these subgraph densities as being analogous to moments
of random variables: it is well known that moments determine the distribution
of random variables, provided the moments do not grow too quickly.

To make these statements precise, we will define homomorphism densities for a
graphex $\cW$. To this end, we first consider a finite, labeled graph $F$ and
a graphon $W$, and define
\[t(F,W)=\int_{\Omega^{V(F)}} \prod_{(i,j) \in E(F)} W(x_i,x_j) \prod_{i \in V(F)} d\mu(x_i).
\]
Given a connected multigraph $F=(V,E)$ on $k \ge 2$ vertices with no loops,
and a graphex $\cW=(W,S,I,\bOmega)$, we define $t(F,\cW)$ as follows. First,
if $F$ consists of a single edge, we define
\[t(F,\cW)=\int_{\Omega^2} W(x,y)\,d\mu(x)\,d\mu(y) + 2 \int_{\Omega} S(x) \,d\mu(x) + 2I=
\rho(\cW).
\]
Otherwise, let $V_{\ge 2}$ be the set of vertices of $F$ with degree at least
$2$, and for each such vertex $v$, let $d_1(v)$ be the number of neighbors of
$v$ that have degree $1$. Then
\[t(F,\cW)=\int_{\Omega^{V_{\ge 2}}} \prod_{\{v,w\} \in E(F(V_{\geq 2}))}
W(z_v,z_w) \prod_{v \in V_{\ge 2}} D_{\cW}(z_v)^{d_1(v)}
d\mu(z_v).
\]
Finally, for any multigraph $F$ with no isolated vertices, let
$F_1,F_2,\dots,F_k$ be the components of $F$. Then we define the
\emph{homomorphism density of $F$ in $\cW$} as
\[t(F,\cW)=\prod_{i=1}^k t(F_i,\cW)
.\] As we will see in Proposition~\ref{prop:t(F,W)}, these homomorphism
densities are defined in such a way that for a simple graph $F$ and a graphex
$\cW$, they are equal to the expected number of injective homomorphisms from
$F$ into $G_T(\cW)$ times $T^{-|V(F)|}$.

Having defined the subgraph densities $t(F,\cW)$, we can summarize the main
relationship between convergence in the metric $\deltt$, convergence of
subgraph counts, and GP-convergence in the following theorem. Its proof will
also be given at the end of Section~\ref{sec:sampling}.

\begin{theorem} \label{thm:D-bounded-convergence}
Assume that $\cW$ and $\cW_n$ for $n\geq 1$ are graphexes whose marginals are
$D$-bounded for some finite $D$. Then the following are equivalent.
\begin{enumerate}
\item $\deltt(\cW_n,\cW)\to 0$.\label{graphexconvdeltt-d}
\item For every graph $F$ with no isolated vertices, $t(F,\cW_n)
    \rightarrow t(F,\cW)$. \label{graphexconvnoisol-d}
\item For every connected graph $F$, $t(F,\cW_n) \rightarrow t(F,\cW)$.
    \label{graphexconvconn-d}
\item $G_T(\cW_n)\rightarrow G_T(\cW)$ in distribution for every $T$.
    \label{graphexconvallT-d}
\item $G_T(\cW_n)\rightarrow G_T(\cW)$ in distribution for some $T$.
    \label{graphexconvoneT-d}
\end{enumerate}
\end{theorem}

\begin{remark}
The above theorem implies in particular that in order to check whether a
sequence $\cW_n$ of graphexes with uniformly bounded marginals is
GP-convergent, it is enough to check convergence of $G_T(\cW_n)$ for a single
$T>0$. In a similar way, several other properties of sequences or sets of
graphexes can be equivalently stated for all $T>0$ or some $T>0$ (see, in
particular, the already mentioned Theorem~\ref{theoremtightequiv} about
tightness and Theorem~\ref{thmunifintegequiv} about uniform integrability).
But for general sequences of graphexes, we do not know whether GP-convergence
is equivalent to the convergence of $G_T(\cW_n)$ for just one $T>0$.
\end{remark}

It is instructive to compare our notions of convergence to the notions of
graph convergence introduced in \cite{BCCH16} and \cite{BCCV17}. Before
defining these notions, we first introduce the notion of a dilated empirical
graphon corresponding to a finite graph $G$. It involves a ``dilation
parameter'' $\rho\in\RR_+$ and is defined as the graphex $\cW(G,\rho)$
consisting of a zero dust part, a zero star part, a measure space consisting
of the vertex set $V(G)$ where each vertex has measure $\rho$, and a graphon
$W(G,\rho)$ which is simply the adjacency matrix of $G$. The usual way to
embed graphs into the space of graphons in the dense case corresponds to
$\rho=1/|V(G)|$.

By contrast, in \cite{BCCH16}, $\rho$ was chosen to be $1/\sqrt{2|E(G)|}$;
the resulting dilated empirical graphon was called the stretched empirical
graphon, and a sequence was said to \emph{converge in the stretched cut
metric} if the graphons $W(G,1/\sqrt{2|E(G)|})$ converge in the cut metric
$\delta_\square$. It was then shown that this leads to completeness (every
Cauchy sequence has a limit), that convergence implies a certain condition
called uniform tail regularity, and that any uniformly tail regular sequence
has a convergent subsequence.

The notion of convergence in \cite{BCCV17} is slightly different. It does not
start from a metric, and instead tries to emulate the notion of subgraph
convergence from dense graphs. Roughly speaking, it asks that certain random
subgraphs of the graphs in the sequence converge in distribution to some
well-defined distribution over finite graphs. More precisely, given a
parameter $p\in [0,1]$, define $\text{Smpl}(G,p)$ as the unlabelled graph
obtained by first taking each vertex i.i.d.\ with probability $p$, then
removing all isolated vertices in the resulting subgraph, and finally
discarding all the labels. A sequence $G_n$ is then said to be \emph{sampling
convergent} if for all $t>0$, the samples
\[\text{Smpl}(G_n,\min\{1,t/\sqrt{2|E(G_n)|}\})\]
converge in distribution. It was then shown that any sequence of finite
graphs has a convergent subsequence, and that the limiting distribution can
be expressed as $G_t(\cW)$ for some integrable graphex $\cW$ with
$\|\cW\|_1\leq 1$. It was also shown that this inequality holds with equality
if and only if the sequence has a property called \emph{uniform sampling
regularity}.

It is instructive to relate the results and notions from \cite{BCCV17} to
those developed in this paper. To this end, we first note that---as already
observed in \cite{BCCV17}---a sequence of graphs is sampling convergent if
and only if the stretched canonical graphexes $\cW(G_n,1/\sqrt{2|E(G_n)|})$
are GP-convergent. Since by definition, the stretched canonical graphex has
$L^1$ norm $1$, this sequence is tight. By our compactness theorem,
Theorem~\ref{thm:complete}, it therefore has a convergent subsequence.

To relate some of the other notions and results from \cite{BCCH16} and
\cite{BCCV17} to those of this paper, we introduce a couple of definitions.
The first notion is that of uniform integrability. Recall that a set $S$ of
random variables with values in $\RR$ is called uniformly integrable if for
every $\eps>0$, there exists $K \in \RR$ such that for every $X \in S$,
\[\EE[|X|1|_{X|>K}]<\eps
.\] Note that this implies that $\EE[|X|] \le \eps+K$, so the set of random
variables consists of integrable variables with uniformly bounded integrals.
This motivates the following definition.

\begin{definition}\label{def:UI}
A set of graphexes $\mathcal{S}$ is called \emph{uniformly integrable} if the
graphexes in $\cS$ have uniformly bounded $\|\cdot\|_1$-norms, and for every
$\varepsilon>0$, there exists a $D$ such that for all $\cW \in \mathcal{S}$,
$\|D_{\cW} 1_{D_{\cW}>D}\|_1<\varepsilon$.
\end{definition}

As we will see in Theorem~\ref{thmunifintegequiv} below, uniform
integrability of a set $\cS$ of graphexes is equivalent to uniform
integrability of the random variables $\{E(G_T(\cW)):\cW \in \cS\}$ for all
$T>0$ (which is also equivalent to uniform integrability of this set of
random variables for some $T>0$).

The notion of uniform sampling regularity from \cite{BCCV17} is then simply
uniform integrability of the stretched empirical graphexes, and the following
theorem is a more or less straightforward generalization of Corollary 3.10 in
\cite{BCCV17}, which states that the limiting graphex of a sampling
convergent sequence of graphs has norm $1$ if and only if it is uniformly
sampling regular. We will prove the theorem in Section~\ref{sec:UIandUTR}.

\begin{theorem}\label{thm:UI-norm-convergence}
Suppose $\cW_n$ is a sequence of integrable graphexes with uniformly bounded
$\|\cdot\|_1$-norms that converges to a graphex $\cW$ in the weak kernel
metric. Then
\[
\|\cW\|_1 \le \liminf_{n \to \infty} \|\cW_n\|_1.
\]
In particular, $\cW$ is integrable. We furthermore have that
\[
\lim_{n \to \infty} \|\cW_n\|_1 = \|\cW\|_1
\]
if and only if the sequence $\cW_n$ is uniformly integrable.
\end{theorem}

Our next set of theorems relates the notion of sampling convergence from
\cite{BCCV17} to the notion of convergence in the cut metric from
\cite{BCCH16}. We start by recalling the definition of uniform tail
regularity from \cite{BCCH16} (see also Lemma~\ref{lemuniftailregequiv} in
Section~\ref{sec:UIandUTR} for other, equivalent definitions).

\begin{definition}[\cite{BCCH16}]
Given a set of signed, integrable \emph{graphons} $\cS$, we say that they are
\emph{uniformly tail regular} if for any $\varepsilon>0$, there exists $M$
such that for each $W\in \cS$ with the usual notation, there exists $\Omega_0
\subseteq \Omega$ such that $\mu(\Omega_0) \le M$ and $\|W\|_1 -
\|W|_{\Omega_0}\|_1 \le \varepsilon$.
\end{definition}

Note that uniform tail regularity is more restrictive than uniform
integrability (for the set of graphexes obtained by setting the dust and star
part to zero). For sequences of graphs, the corresponding result was shown in
\cite{BCCV17}, but it holds in our more general setting as well, with
essentially the same proof; see Lemma~\ref{lem:tail-reg-implies-UI} in
Section~\ref{sec:UIandUTR} below. More interestingly, any sequence of
graphons that is convergent in cut metric has uniformly regular tails, and
any sequence of graphons with uniformly regular tails has a subsequence that
converges in cut metric. Motivated by this (which is one of the central
results of \cite{BCCH16}), here we prove the following.

\begin{theorem} \label{thmcutmetricvsweakkernel}
Given a sequence of integrable graphexes of the form
\[\cW_n=(W_n,0,0,\bOmega_n),
\]
the following are equivalent.
\begin{enumerate}
\item The sequence $W_n$ converges to a graphon $W$ in cut metric.
\item The sequence $W_n$ is uniformly tail regular, and in the weak kernel
    metric, the sequence $\cW_n$ converges to a graphex of the form
    $\cW=(W,0,0,\bOmega)$.
\item The sequence $W_n$ is uniformly tail regular, and in the weak kernel
    metric, the sequence $\cW_n$ converges to some graphex $\cW$.
\end{enumerate}
\end{theorem}

This theorem, as well as our next theorem, will also be proved in
Section~\ref{sec:UIandUTR}.

\begin{theorem} \label{thmwhenconvergetographon}
Given a sequence of uniformly integrable graphexes
\[
\cW_n=(W_n,S_n,I_n,\bOmega_n),
\]
which converge to a graphex $\cW=(W,S,I,\bOmega)$ in the weak kernel metric,
the following are equivalent.
\begin{enumerate}
\item The graphex $\cW$ is of the form $\cW=(W,0,0,\bOmega)$.
\item $\|S_n\|_1 \to 0$ and $I_n \to 0$, and the sequence of graphons $W_n$
    has uniformly regular tails.
\item $\|S_n\|_1 \to 0$ and $I_n \to 0$, and the sequence of graphons $W_n$
    converges in the cut metric.
\end{enumerate}
\end{theorem}

Recall that a sequence of graphs is sampling convergent if and only if the
stretched canonical graphexes are GP-convergent. This fact,
Theorem~\ref{thm:UI-norm-convergence}, and
Theorem~\ref{thmwhenconvergetographon} together imply that given a sequence
of graphs that is sampling convergent to a graphex of norm one, the sequence
converges to a pure graphon if and only if the sequence is uniformly tail
regular. We have therefore given a characterization of when the notion of
sampling convergence from \cite{BCCV17} reduces to the notion of convergence
in the stretched cut metric from \cite{BCCH16}.

\begin{remark}\label{rem:UI-is-needed}
In the above theorem, the assumption of uniform integrability is necessary.
To see this, let $W_n$ be the graphon defined by being $1$ on the set
$[0,1/n] \times [1,n+1]$ and its transpose, and $0$ otherwise, and let
$\cW_n=(W_n,0,0,\RR_+)$. Then the sequence $\cW_n$ converges to the zero
graphex in the weak kernel metric, which is a pure graphon. However, the
sequence is clearly not uniform tail regular (or even uniformly integrable).
We can also let $S_n(x)=1/n$ on the set $[0, n]$, and $0$ everywhere else. In
that case the sequence still converges to $0$ in the weak kernel metric, but
$\|S_n\|_1$ does not converge to $0$.
\end{remark}

We close this section by discussing possible extensions of our theory. First,
as already discussed in the introduction, it would be natural to extend the
theory of graphexes to a theory that naturally generates multi-graphs. Note
that \emph{a priori}, this falls plainly in the framework of exchangeable
random measures on $\RR_+^2$ as developed by Kallenberg; in fact, it falls
into the framework of exchangeable random counting measures. Generalizing the
approach of \cite{VR15,BCCH16}, this will give a natural notion of
``multi-graphexes'' characterizing all exchangeable multi-graphs with
vertices labelled by $\RR_+$. But extending the current work to
multi-graphexes is beyond the scope of this paper, in particular given that
it would require to generalize at least some of the results from
\cite{VR15,JANSON16,VR16,JANSON17,BCCV17} to this setting in a first step.
See \cite{BCDS18} for some very preliminary steps in this direction.

The next extension one might want to consider is the extension of our
analytical results (i.e., those of our results which do not refer to the
graphex process generated by a graphex) to signed graphexes. In contrast to
the theory of cut metric convergence for graphons over $\sigma$-finite
measure spaces developed in \cite{BCCH16}, which works as well for signed,
unbounded graphons as for graphons with values in $[0,1]$, here we focused
most of the theory of graphex convergence on unsigned graphexes (with graphon
parts taking values in $[0,1]$). While several of our technical proofs and
results hold for signed graphexes (in particular, all of
Section~\ref{sec:prelim}, as well as parts of Sections~\ref{sec:reg} and
\ref{sec:subgraph-counts} are formulated in this language), the core analytic
concepts and results such as tightness, precompactness, etc., have only been
formulated for unsigned graphexes. Indeed, we believe that the generalization
of Theorem~\ref{thm:complete} to signed graphexes requires modifications to
either our topology or our notions of tightness; see
Remark~\ref{rem:signed-tightness} below. In a similar way, while the
identification theorem of \cite{BCCH16} works for signed graphons, our
identification theorem requires non-negative graphexes, even though one might
conjecture that when stated as a characterization of the equivalence classes
under the weak kernel metric, it should hold for signed graphexes as well, at
least when suitably formulated.

Finally, one might want to consider graphexes where the graphon part $W$ is
unbounded, whether non-negative or signed. For non-negative graphexes, one
could, for example, follow the approach in \cite{BCDS18} and use such a
graphex to generate multi-graphs by adding $\text{Pois}(W(x_i,y_i))$ many
edges to a pair of Poisson points with features $x_i$ and $x_j$, or one could
try to generalize the approach of \cite{LP1} to the setting of graphexes, by
taking a decreasing ``dilution probability'' $p_t$, and then connect two
Poisson points with features $x_i$ and $x_j$ with probability $\min\{1,p_t
W(x_i,x_j)\}$ (see \cite{BCCH2} for a related approach). But is far from
obvious what the analogue of the weak kernel metric should be, and how to
generalize our other results to this setting. As the other open questions
discussed here, we leave these questions as open research problems.

\section{Preliminaries} \label{sec:prelim}

In this section, we study the metrics $\deltt$ and $\delGP$. In particular,
we will prove Theorems~\ref{thm:deltt-metric} and \ref{thm:dmetric}, as well
as Proposition~\ref{propboundedequivmetrics} relating the two for graphexes
with bounded marginals (in fact, we will prove its generalization to signed
graphexes, stated as Proposition~\ref{prop:signed-met-equiv} below).

In addition, we study the metric $\deljbl$ obtained from $\deltt$ by
replacing $\d22$ as defined in \eqref{d22-def} by
\begin{equation}\label{djbl-def}
\djbl(\cW_1,\cW_2)= \max\left\{ \|W_1-W_2\|_{\jbl},
\|D_{\cW_1}-D_{\cW_2}\|_\jbl,
\left|\rho(\cW_1)-\rho(\cW_2)\right|
\right\}.
\end{equation}
Here we use the norm $\|\cdot\|_\jbl$ both for functions from $\Omega^2\to
\RR$ (see Definition~\ref{def:jumble-norm}) and functions from
$\Omega\to\RR$, where it is defined as
\[
\|F\|_\jbl=\sup_{S \subseteq \Omega} \left| \frac{1}{\sqrt{\mu(S)}}\int_{S} F(x)d\mu(x)\right|.
\]
We will in particular show that the analogue of
Theorem~\ref{thm:deltt-metric} holds for this metric.

\begin{theorem} \label{thm:deltws-metric}
Let $\cW_1$ and $\cW_2$ be signed graphexes in $L^1\cap L^2$. Define
$\deljbl$ by replacing the right side of \eqref{del22-def} with $\tdeljbl$,
which in turn is obtained from $\tdel22$ by replacing $\d22$ with $\djbl$.
Then the value of $\deljbl(\cW_1,\cW_2) $ does not depend on the choice of
the trivial extensions $\cW_1'$ and $\cW_2'$, and $\deljbl$ obeys the
triangle inequality, making it a well-defined pseudometric.
\end{theorem}

This \emph{jumble metric} will be particularly useful when establishing the
regularity lemma for graphexes, which takes a nicer form when stated in terms
of the distance $\djbl$ instead of the distance $\d22$, both because of the
absence of the various roots, and because the proof of the regularity lemma
leads more naturally to bounds in term of $\|\cdot\|_{\jbl}$ rather than
$\|\cdot\|_{2\rightarrow2}$. To obtain our compactness results for the metric
$\delGP$ (which is derived from $\d22$), we will then need to compare the
two. We will do this in Proposition~\ref{prop:kernel-jumbl-equiv} and
Remark~\ref{rem:kernel-jumbl-equiv} below.

We will also establish a simple lemma relating the kernel norm to $4$-cycle
counts (Lemma~\ref{lemmasquare2to2equiv}). Finally, we will prove that the
homomorphism densities $t(F,\cW)$ indeed describe the expected number of
injective homomorphisms from $F$ into $G_T(\cW)$
(Proposition~\ref{prop:t(F,W)}).

Except for the last result, all results in this section are as easily derived
for signed graphexes as for unsigned graphexes. We therefore formulate
everything in this section in the language of signed graphexes. To do so, we
need very little extra notation, except for the following.

First, we define the \emph{absolute marginal} of a signed graphex $\cW$ as
$D_{|\cW|}$, and say that $\cW$ has $D$-bounded absolute marginals if
$\|D_{|\cW|}\|_\infty \le D$. We say that $\cW$ is $(C,D)$-bounded if in
addition $\|\cW\|_1 \le C$. Furthermore, we introduce the notion of
$(B,C,D)$-boundedness of a graphex $\cW=(W,S,I,\bOmega)$ by requiring that
\begin{equation}
\label{BCD-bounded}
\|W\|_\infty\leq B,
\qquad
\|\cW\|_1\leq C,
\qquad\text{and}\qquad
\|D_{|\cW|}\|_\infty\leq D,
\end{equation}
and finally, we say that $\cW$ has a bounded graphon part if
$\|W\|_\infty<\infty$.

We use the following standard facts about measure-preserving transformations,
which we prove for completeness.

\begin{lemma} \label{lemmapullback}
Suppose $\phi\colon \Omega \rightarrow \Omega'$ is a measure preserving map
between two $\sigma$-finite measure spaces $(\Omega, \cF,\mu)$ and $(\Omega',
\cF',\mu')$.
\begin{enumerate}
\item For any $\sigma$-algebra $\cG \subseteq \cF$, $L^1(\Omega,\cG,\mu)$
    is a closed subspace in $L^1(\Omega,\cF,\mu)$.
\item If $f\in L^1(\Omega,\phi^{-1}(\cF'),\mu)$, then there exists a
    function $f'\in L^1(\Omega',\cF',\mu')$ such that $f=f'^{\phi}=f' \circ
\phi$ almost everywhere.
\item The map $\phi^*\colon L^1(\Omega',\cF',\mu') \rightarrow
    L^1(\Omega,\phi^{-1}(\cF'),\mu)$ with $f'\mapsto {f'}^\phi$ is an
isometric isomorphism, implying that $\phi^*$ and its inverse are
continuous and hence Borel measurable.
\end{enumerate}
\end{lemma}

\begin{proof}
\begin{enumerate}
\item $L^1(\Omega,\cG,\mu)$ is clearly a subspace of $L^1(\Omega,\cF,\mu)$.
    Since they are both Banach spaces, they are both complete. Since
    $L^1(\Omega,\cG,\mu)\subseteq L^1(\Omega,\cF,\mu)$ is an isometric
    embedding, the only way we can have a {complete} subset of a complete
    space is if the subset itself is closed.
\item The map sending $A \subseteq \Omega$ to $\int_A f$ defines a finite,
    signed measure $\nu$ on $\Omega$, absolutely continuous with respect to
  $\mu$. This measure pushes forward to a signed measure $\nu'$ on
  $\Omega'$. If $B \subseteq \Omega'$ has $\mu'(B)=0$, then
  $\mu(\phi^{-1}(B))=0$, so $\nu'(B)=\nu(\phi^{-1}(B))=0$. Therefore $\nu'$
  is absolutely continuous with respect to $\mu'$, so it has a
  Radon-Nikodym derivative $f'$. It is straightforward to check that, since
  $f$ is $\phi^{-1}(\cF')$ measurable, we have $f=f'^{\phi}$ almost
  everywhere.
\item This follows from the previous parts.
\end{enumerate}
\end{proof}

We will also need the following lemma.

\begin{lemma}\label{lem:auxil1}
Suppose that $\bOmega=(\Omega,\cF,\mu)$ with $\mu(\Omega)<\infty$, and
suppose that $g \in L^2(\Omega)$. Let
\begin{equation} \label{eqnmaxsetjumble}
S=\sup_{f \in L^\infty(\Omega),0 \le f \le 1} \left| \frac{\int_\Omega fg \,d\mu}{\sqrt{\int_\Omega f \,d\mu}}\right|
.\end{equation}
Then $S\leq\|g\|_2<\infty$, and there exists $X \subseteq \Omega$ such that
\[S=\frac{\left|\int_X g \,d\mu\right|}{\sqrt{\mu(X)}}
.\]
Note that this expression is the same as taking $f=1_X$ in (\ref{eqnmaxsetjumble}).
\end{lemma}

\begin{proof}
First, note that if $0 \le f \le 1$, then
\[
 \left|\frac{\int_\Omega fgd \mu}{\sqrt{\int_\Omega f d \mu}} \right|
 \le \sqrt{\frac{\int_\Omega f^2 d \mu\int_\Omega g^2 d \mu}{\int_\Omega f \,d\mu}}
 = \|g\|_2 \sqrt{\frac{\int_\Omega f^2 \,d\mu}{\int_\Omega f \,d\mu}} \le \|g\|_2<\infty
.\] Next, we will show that there exists a $\delta>0$ such that in
(\ref{eqnmaxsetjumble}), it suffices to consider $f$ with $\|f\|_1 \ge
\delta$. Let $\wOmega=\Omega \times [0,1]$ and
$\widetilde\mu=\mu\times\lambda$, where $\lambda$ is the Lebesgue measure,
and let $\wg(x,t)=g(x)$. Then the expression in (\ref{eqnmaxsetjumble}) is
the same as taking the supremum of
\[\frac{\left|\int_X \wg \,d\widetilde\mu\right|}{\sqrt{\widetilde\mu(X)}}
\]
over $X \subseteq \wOmega$. Indeed, given $X \subseteq \wOmega$, we can plug
in the function $f(x)=\lambda(\{t:(x,t) \in X\})$ into
(\ref{eqnmaxsetjumble}), which is defined almost everywhere, and given $f$ as
in (\ref{eqnmaxsetjumble}), we can take $X=\{(x,t):t \le f(x)\}$. It is
straightforward to check that (\ref{eqnmaxsetjumble}) and the above
expression give the same value. Note that we have
\[
\frac{\left|\int_X \wg \,d\widetilde\mu\right|}{\sqrt{\widetilde\mu(X)}} \le \|\wg|_{X}\|_2
.\] Since $\|\wg\|_2<\infty$, for any $\eps>0$, there exists a $K$ such that
\[
\int_X \wg^2 \,d\mu
\leq \int \wg^21_{\wg>\sqrt K}+ K\widetilde\mu(X) \leq \frac{\eps^2}2 + K\widetilde\mu(X).
\]
Thus, given $\eps>0$, we can find a $\delta>0$ such that if $\mu(X)<\delta$,
then $\|\wg|_{X}\|_2\leq\eps$. Taking $\eps = S/2$ (note that unless $g=0$,
$S>0$), we obtain a $\delta$ such that if $\int_\Omega f \,d\mu<\delta$, then
the expression in (\ref{eqnmaxsetjumble}) is less than $S/2$. This means that
it suffices to take the supremum over $f$ with $\|f\|_1 \ge \delta$.

Recall that we assumed that $\Omega$ has finite measure. The set of $f \in
L^\infty(\Omega)$ with $0 \le f \le 1$ and $\|f\|_1 \ge \delta$ is weak-$*$
closed and therefore weak-$*$ compact. The expression in
(\ref{eqnmaxsetjumble}) is weak-$*$ continuous; therefore, there exists an
$f$ which maximizes the expression.

Clearly such an $f$ is supported either on the set $\{g>0\}$ or $\{g<0\}$.
Assume without loss of generality that it is supported on $\{g>0\}$. Suppose
that $f$ is not equal to $0$ or $1$ almost everywhere. Then we can find $0
\le h \le \wth \le 1$, supported on $\{g>0\}$, such that $f=(\wth+h)/2$, and
it is not the case that $\wth=h$ almost everywhere. This implies that
\[\int_\Omega (\wth - h) \,d\mu >0
\]
and
\[\int_\Omega (\wth - h)g \,d\mu > 0.
\]
Let $h_t=t\wth+(1-t)h$. Note that for $t \in [0,1]$, $0 \le h_t \le 1$
almost everywhere. Let
\[p(t)= \frac{\int_\Omega h_t g}{\sqrt{\int_\Omega h_t}}.
\]
We have
\[\frac{d}{dt}p(t)=\frac{\int_\Omega (\wth-h)g \,d\mu \int_\Omega(t\wth+(1-t)h) \,d\mu
- \frac{1}{2} \int_\Omega( \wth-h)\,d\mu \int_\Omega(t\wth+(1-t)h)g \,d\mu}{\left(\int_\Omega(t\wth+(1-t)h)\,d\mu\right)^{3/2}}.
\]
Notice that the denominator above is always positive, and the numerator above
is of the form $At+B$, where
\begin{align*}
A&=\int_\Omega (\wth-h)g \,d\mu \int_\Omega (\wth-h) \,d\mu - \frac{1}{2} \int_\Omega(\wth-h) \,d\mu \int_\Omega (\wth-h)g \,d\mu\\
&=\frac{1}{2} \int_\Omega(\wth-h) \,d\mu \int_\Omega (\wth-h)g \,d\mu >0
.
\end{align*}
This means that there are three possibilities for $\frac{d}{dt}p(t)$: it can
be positive for every $t \in (0,1)$, it can be negative for every $t \in
(0,1)$, or it can be negative and then positive. Either of these cases
implies that the maximum of $p$ on $[0,1]$ is attained at one or both of the
endpoints, and therefore either $p(0)$ or $p(1)$ is strictly greater than
$p(1/2)$. This contradicts the assumption that $f$ was maximal, completing
the proof of the lemma.
\end{proof}

Using the previous two lemmas, we prove the following proposition. We will
use it for functions which arise as the difference of two graphons with
bounded marginals.

\begin{proposition} \label{proppullbacknormsame}
Let $\bOmega=(\Omega, \cF,\mu)$ and $\bOmega'=(\Omega',\cF',\mu')$ be
$\sigma$-finite measure spaces, and let $\varphi\colon\Omega' \rightarrow
\Omega$ be measurable. If $U\colon\Omega \times \Omega\to\RR$ and
$F\colon\Omega\to \RR$ are square integrable, then
\[
\|U\|_{2 \rightarrow 2}=\|U^\varphi\|_{2 \rightarrow 2},
\qquad\|F\|_{{\jbl}}=\|F^\varphi\|_\jbl,
\qquad\text{and}\qquad
\|U\|_{{\jbl}}=\|U^\varphi\|_{{\jbl}}.
\]
If instead of square integrability, we assume that $U$ is integrable, then
\[
\|U\|_{\square}=\|U^\varphi\|_{\square}.
\]
\end{proposition}

\begin{proof}
For any $f,g \in L^2(\bOmega)$, it is easy to see that $f \circ U \circ g =
f^\varphi \circ U^\varphi \circ g^\varphi$, and we furthermore have that
$f^\phi,g^\phi\in L^2(\bOmega')$ with $\|f^\phi\|_2=\|f\|_2$ and
$\|g^\phi\|_2=\|g\|_2$. This implies that
\[\|U\|_{2 \rightarrow 2}\le \|U^\varphi\|_{2 \rightarrow 2}
.\]

To prove the opposite inequality, let $\widehat{f},\widehat{g} \in
L^2(\bOmega')$, and assume first that $\widehat{f},\widehat{g} \in
L^1(\bOmega')$ as well. Let $f'=\EE[\widehat{f}|\varphi^{-1}(\cF)]$ and
$g'=\EE[\widehat{g}|\varphi^{-1}(\cF)]$. That is, $f'$ is a
$\varphi^{-1}(\cF)$-measurable function such that for any
$\varphi^{-1}(\cF)$-measurable set $S'\subseteq \Omega'$, $\int_{S'}
f'=\int_{S'} \widehat{f}$, and same for $g'$. These functions exist by the
Radon-Nikodym theorem (since all measures are $\sigma$-finite). Then
$\|f'\|_2 \le \|\widehat{f}\|_2$, $\|g'\|_2 \le \|\widehat{g}\|_2$. We claim
that $f' \circ U^\varphi \circ g'=\widehat{f}\circ U^\varphi \circ
\widehat{g}$. Indeed, for any $x' \in \Omega'$,
\[
(U^\varphi)_{x'}(y')=U^\varphi(x',y')
=U(\varphi(x'),\varphi(y'))
=U_{\varphi(x')}(\varphi(y'))
=(U_{\varphi(x')})^\varphi (y'),
\]
showing that $(U^\varphi)_{x'}$ is the pullback of an $\cF$-measurable
function, and is thus $\varphi^{-1}(\cF)$-measurable. Therefore, for every
$x'\in\Omega'$,
\[
\int_{\Omega'} U^\varphi(x',y')\widehat{g}(y')\,d\mu'(y')=
\int_{\Omega'}U^\varphi(x',y')g'(y')\,d\mu'(y'),
\]
which shows that $\widehat{f} \circ U^\varphi \circ \widehat{g}=\widehat{f}
\circ U^\varphi \circ g'$. We can analogously show that $\widehat{f} \circ
U^\varphi \circ g'=f' \circ U^\varphi \circ g'$. Then, since $f'$ and $g'$
are $\varphi^{-1}(\cF)$-measurable, there exist by Lemma \ref{lemmapullback}
$f,g \in L^{{1}}(\Omega)$ with $f'=f^\varphi$ and $g' = g^\varphi$. This
implies that $\|f\|_2 \le \|\widehat{f}\|$, $\|g\|_2 \le \|\widehat{g}\|_2$,
and $f \circ U \circ g = \widehat{f} \circ U^\varphi \circ \widehat{g}$,
which shows that
\[
\widehat{f} \circ U^\varphi \circ \widehat{g}\leq \|U\|_{2 \rightarrow 2}
\]
whenever $\widehat{f},\widehat{g} \in L^2(\bOmega')\cap L^1(\bOmega')$ and
$\|\widehat{f}\|_2,\|\widehat{g}\|_2\leq 1$. Since $\bOmega'$ is
$\sigma$-finite, any function in $L^2(\bOmega')$ can be written as a limit of
functions in $L^2(\bOmega')\cap L^1(\bOmega')$. A dominated convergence
argument then shows that the above bound holds whenever
$\widehat{f},\widehat{g} \in L^2(\bOmega')$ and
$\|\widehat{f}\|_2,\|\widehat{g}\|_2\leq 1$, proving that
\[ \|U^\varphi\|_{2 \rightarrow 2}\le\|U\|_{2 \rightarrow 2}.\]

To prove the statement for the cut norm, we use the representation $
\|U\|_{\square}=\sup_{f,g\colon\Omega\to [0,1]}| f \circ U \circ g|$. Using
this representation, the proof for the cut norm proceeds along the same lines
as the proof for the $\|\cdot\|_{2 \rightarrow 2}$ norm.

Next, let us prove that $\|F\|_\jbl=\|F^\varphi\|_\jbl$. First, for any
measurable $X \subseteq \Omega$ with $\mu(X)<\infty$, since $\mu$ is the
pushforward of $\mu'$,
\[\frac{\int_{\varphi^{-1}(X)}F^\varphi \,d\mu'}{\sqrt{\mu'(\varphi^{-1}(X))}}=\frac{\int_X F \,d\mu}{\sqrt{\mu(X)}}
.\] This shows that $\|F\|_\jbl \le \|F^\varphi\|_\jbl$. Suppose now that $X'
\subseteq \Omega'$, and let $f'=\EE[1_X|\varphi^{-1}(\cF)]$. Then there
exists $f \in L^1(\Omega)$ so that $f'=f^\varphi$. Since $F^\varphi$ is
$\varphi^{-1}(\cF)$-measurable,
\[\frac{\int_{X'} F^\varphi \,d\mu'}{\sqrt{\mu'(X)}}= \frac{\int_{\Omega'} f' F^\varphi \,d\mu'}{\sqrt{\int_\Omega f' \,d\mu'}}= \frac{\int_\Omega f F \,d\mu}{\sqrt{\int_\Omega f \,d\mu}}
.\] Fix $\varepsilon>0$. Since $\Omega$ is $\sigma$-finite, there exists
$\Omega_0 \subseteq \Omega$ with $\mu(\Omega_0)<\infty$ and
\[
\frac{\int_{\Omega_0} f F \,d\mu}{\sqrt{\int_{\Omega_0} f \,d\mu}}
\ge (1-\varepsilon) \frac{\int_\Omega f F \,d\mu}{\sqrt{\int_\Omega f \,d\mu}}
.\]
By the previous lemma, there exists a measurable set $X \subseteq \Omega_0$ so that
\[\frac{\int_X F \,d\mu}{\sqrt{\mu(X) }} \ge \frac{\int_{\Omega_0} f F \,d\mu}{\sqrt{\int_{\Omega_0} f \,d\mu}}
\ge (1-\varepsilon) \frac{\int_\Omega f F \,d\mu}{\sqrt{\int_\Omega f
\,d\mu}} .\] Since this holds for any $\varepsilon>0$, this proves that
$\|F\|_\jbl \ge \|F^\varphi\|_\jbl$.

Finally, we show that $\|U\|_\jbl=\|U^\varphi\|_\jbl$. As before, for any
measurable $X,Y \subseteq \Omega$ with $\mu(X),\mu(Y)<\infty$, since $\mu$ is
the pushforward of $\mu'$,
\[\frac{\int_{\varphi^{-1}(X) \times \varphi^{-1}(Y)}U^\varphi \,(d\mu')^2}{\sqrt{\mu'(\varphi^{-1}(X))}\sqrt{\mu'(\varphi^{-1}(Y))}}=\frac{\int_{X \times Y} U \,(d\mu)^2}{\sqrt{\mu(X)} \sqrt{\mu(Y)}}
.\] This shows that $\|U\|_\jbl \le \|U^\varphi\|_\jbl$. For the other
direction, let $X',Y' \subseteq \Omega'$ with finite measure. By the previous
argument, there exist functions $f,g\colon\Omega \to [0,1]$ such that
\[\frac{\int_{\Omega^2}f(x)U(x,y)g(y) \,d\mu(x)\,d\mu(y)}{\sqrt{\int_\Omega f \,d\mu \int_\Omega g \,d\mu}} = \frac{\int_{X' \times Y'}U^\varphi \,(d\mu')^2}{\sqrt{\mu'(X')\mu'(Y')}}
.\] Fix $\varepsilon>0$. By the previous argument, there exists $X \subseteq
\Omega$ such that
\[\frac{\int_{X \times \Omega}U(x,y)g(y) \,d\mu(x)\,d\mu(y)}{\sqrt{\mu(X) \int_\Omega g \,d\mu}} \ge
(1-\varepsilon)\frac{\int_{\Omega^2}f(x)U(x,y)g(y) \,d\mu(x)\,d\mu(y)}{\sqrt{\int_\Omega f \,d\mu \int_\Omega g \,d\mu}}
.\]
Then, applying it again, there exists $Y \subseteq \Omega$ such that
\[\frac{\int_{X \times Y}U(x,y) \,d\mu(x)\,d\mu(y)}{\sqrt{\mu(X)\mu(Y)}} \ge (1-\varepsilon)\frac{\int_{X \times \Omega}U(x,y)g(y) \,d\mu(x)\,d\mu(y)}{\sqrt{\mu(X) \int_\Omega g}}
.\]
Combining these, we obtain
\[\frac{\int_{X \times Y}U(x,y) \,d\mu(x)\,d\mu(y)}{\sqrt{\mu(X)\mu(Y)}} \ge (1-\varepsilon)^2\frac{\int_{X' \times Y'}U^\varphi \,d\mu'}{\sqrt{\mu'(X')\mu'(Y')}}
.\] Since this holds for any $\varepsilon>0$, we obtain that $\|U\|_\jbl \ge
\|U^\varphi\|_\jbl$.
\end{proof}

Next we establish a sequence of lemmas leading to the proof of
Theorem~\ref{thm:deltt-metric}, which states that $\deltt$ is a pseudometric.
To state the first lemma, we define the vectors
\begin{align*}
\Delta_{2\to 2}(\cW_1,\cW_2)
&= \left( \|W_1-W_2\|_{2 \rightarrow 2}, \|D_{\cW_1}-D_{\cW_2}\|_2,
\left|
\rho(\cW_1)-\rho(\cW_2)\right|\right)
\quad\text{and}
\\
\Delta_{\jbl}(\cW_1,\cW_2)
&= \left( \|W_1-W_2\|_{\jbl}, \|D_{\cW_1}-D_{\cW_2}\|_\jbl,
\left|\rho(\cW_1)-\rho(\cW_2)\right|\right),
\end{align*}
where again we require the signed graphexes $\cW_1$ and $\cW_2$ to be in
$L^1\cap L^2$. Note that each coordinate satisfies the triangle inequality.
We will also use the following property.

\begin{lemma} \label{lemmapullbacksame}
Let $\bOmega=(\Omega,\cF,\mu)$ and $\bOmega'=(\Omega',\cF',\mu')$ be
$\sigma$-finite measure spaces, let $\varphi\colon \Omega \rightarrow
\Omega'$ be a measure-preserving map, and let $\cW_1,\cW_2$ be
signed graphexes in $L^1\cap L^2$.
Then
\[
\Delta_{2\to 2}(\cW_1,\cW_2)=\Delta_{2\to 2}(\cW_1^\varphi,\cW_2^\varphi)
\quad\text{and}\quad
\Delta_{\jbl}(\cW_1,\cW_2)=\Delta_{\jbl}(\cW_1^\varphi,\cW_2^\varphi).
\]
\end{lemma}

\begin{proof}
Clearly $\rho(\cW_1)=\rho(\cW_1^\varphi)$ and
$\rho(\cW_2)=\rho(\cW_2^\varphi)$, which means that their differences are
equal too. We also have for every $x \in \Omega$, $
D_{\cW_i}(\varphi(x))=D_{\cW_i^\varphi}(x), $ which implies that $
\|D_{\cW_1}-D_{\cW_2}\|_2=\|D_{\cW_1^\varphi}-D_{\cW_2^\varphi}\|_2. $
Therefore it remains to show that
\[
\|W_1-W_2\|_{2 \rightarrow 2}=\|W_1^\varphi-W_2^\varphi\|_{2 \rightarrow 2}
\]
as well as
\[\|D_{\cW_1}-D_{\cW_2}\|_\jbl=\|D_{\cW_1^\varphi}-D_{\cW_2^\varphi}\|_\jbl
\quad\text{and}\quad
\|W_1-W_2\|_{\jbl}=\|W_1^\varphi-W_2^\varphi\|_{\jbl}.
\]
Note that $D_{\cW_1^\varphi}-D_{\cW_2^\varphi}=\left(D_{\cW_1}- D_{\cW_2}\right)^\varphi$
and
$W_1^\varphi-W_2^\varphi=(W_1-W_2)^\varphi$, so by Proposition
\ref{proppullbacknormsame}, we are done.
\end{proof}

To prove the triangle inequality, we would like to take a coupling of
$\Omega_1$ and $\Omega_2$ and a coupling of $\Omega_2$ and $\Omega_3$, and
use them to obtain a coupling of $\Omega_1$ and $\Omega_3$. Unfortunately
this cannot be done for general signed graphexes. We can, however do it if
the signed graphexes involved are step graphexes. To define these, we first
define a \emph{subspace partition} of a measure space
$\bOmega=(\Omega,\cF,\mu)$ as a partition of a measurable subset
$\Omega'\subseteq\Omega$ into countably many measurable subsets. Such a
subspace partition is called finite if it is a partition into finitely many
sets of finite measure. A signed graphex $\cU$ is then called a \emph{step
graphex} over the subspace partition $\cP=(P_1, \dots, P_m)$ if $\cP$ is a
finite subspace partition, $\dsupp \cU\subseteq P_1 \cup P_2 \cup \dots \cup
P_m$, and for all $x,x' \in P_i$, $S(x)=S(x')$ and $W_x=W_{x'}$, where, as
before, $W_x$ is the function $y\mapsto W(x,y)$.

\begin{remark} \label{rem:stepgraphexesequiv}
Given a signed step graphex $\cW=(W,S,I,\bOmega)$ over a finite subspace
partition $\cP=(P_1,P_2,\dots,P_m)$, we can define another signed graphex
$\cW'=(W',S',I',\bOmega')$ with $\bOmega'=(\Omega',\cF',\mu')$, where
$\Omega'=[m]$, $\cF'$ consists of all subsets, and the measure is defined by
$\mu'(\{i\})=\mu(P_i)$. Setting $I'=I$, $S'(i)=S(x)$ for any $x \in P_i$
(they are all equal), and $W'(i,j)=W(x,y)$ for $x \in P_i,y \in P_j$ (again
the choice of $x$ and $y$ does not matter), we obtain that
$\cW=(\cW')^\varphi$ where $\varphi\colon\bOmega \rightarrow \bOmega'$ is the
map with $\varphi(x)=i$ for $x \in P_i$. In particular, by Remark
\ref{rem:pullbacksamedistance}, the distance between $\cW$ and $\cW'$ is zero
(for any of the distance notions). Suppose now that
$\cW''=(W'',S'',I'',\bOmega'')$ with $\bOmega''=(\Omega'',\cF'',\mu'')$ is
another signed step graphex over a finite subspace partition
$\cQ=\{Q_1,Q_2,\dots,Q_m\}$, with $I''=I$, $\mu''(Q_i)=\mu(P_i)$, and
$S''(x'')=S(x)$, $W''(x'',y'')=W(x,y)$ for $x \in P_i,x'' \in Q_i,y \in
P_j,y'' \in Q_j$. Then $\bOmega''$ can also be mapped to $\bOmega'$ so that
$\cW''$ is the pullback of $\cW'$ (by mapping $Q_i$ to $i$). This implies
that the distance of both $\cW$ and $\cW''$ from $\cW'$ is $0$, which (by the
still to be proven triangle inequality) implies that their distance from each
other is $0$ (again for any of the notions of distance).
\end{remark}

Returning to the proof of the triangle inequality, we will in fact consider
signed graphexes that are countable step graphexes, i.e., the number of
``steps'' is countable, and each step has finite measure. First, however, we
need the following technical lemma:

\begin{lemma}
Let $\cW_1=(W_1,S_1,I_1,\bOmega)$ and $\cW_2=(W_2,S_2,I_2,\bOmega)$ be signed
graphexes in $L^1\cap L^2$. Assume that both are countable step graphexes on
$\bOmega=(\Omega,\cF,\mu)$ with common refinement
$\cP=\{P_1,P_2,\dots,P_m,\dots\}$, suppose $\mu'$ is another measure on
$\Omega$ with $\mu(P_i)=\mu'(P_i)$, and let $\bOmega'=(\Omega,\cF,\mu')$ and
$\cW_i'=(W_i,S_i,I_i,\bOmega')$. Then
\[
\Delta_{2\to 2}(\cW_1,\cW_2)=\Delta_{2\to 2}(\cW_1',\cW_2')
\quad\text{and}\quad
\Delta_{\jbl}(\cW_1,\cW_2)=\Delta_{\jbl}(\cW_1',\cW_2')
.
\]
\end{lemma}

\begin{proof}
Let $\bOmega_\cP=(\Omega_\cP,\cF_\cP,\mu_\cP)$ where
$\Omega_\cP=\{x_1,x_2,\dots,x_m,\dots\}$, $\cF_\cP$ is the set of all subsets
of $\Omega_\cP$, and $\mu_\cP(x_i)=\mu(P_i)$. Let
$\varphi,\,\varphi'\colon\Omega \rightarrow \Omega_\cP$ with
$\varphi(x)=\varphi'(x)=x_i$ for $x \in P_i$. Then $\varphi\colon\bOmega
\rightarrow \bOmega_\cP$ and $\varphi'\colon\bOmega' \rightarrow \bOmega_\cP$
are both measure preserving (these are the same function on $\Omega$ but as
maps between measure spaces are different). Define
$\cW_{i,\cP}=(W_{i,\cP},S_{i,\cP},I_i)$ with $S_{i,\cP}(x_j)=S_{i}(x)$ for
any $x \in P_j$ (they are all equal), and $W_{i,\cP}(x_j,x_k)=W(x,y)$ for $x
\in P_j,y \in P_k$ (again, they are all equal). Then
$\cW_{i,\cP}^{\varphi}=\cW_i$ and $\cW_{i,\cP}^{\varphi'}=\cW_i'$. Therefore,
by Lemma \ref{lemmapullbacksame}, we have
\[\Delta_{2\to 2}(\cW_1,\cW_2)=\Delta_{2\to 2}(\cW_{1,\cP},\cW_{2,\cP})=\Delta_{2\to
2}(\cW_1',\cW_2'),\]
and similarly for $\Delta_{\jbl}$.
\end{proof}

To state the next lemma, we use the symbol $\pi_{ij,k}$ to denote the
coordinate projection from a product space $\Omega_i\times\Omega_j$ to
$\Omega_k$, where $k=i$ or $k=j$.

\begin{lemma} \label{lemmacouplingsequencecountablesteps}
Let $\cW_i=(W_i,S_i,I_i,\bOmega_i)$, for $i=1,2,3$, be countable step
graphexes in $L^1\cap L^2$. Let $\mu_{12}$ be a coupling measure on $\Omega_1
\times \Omega_2$, and $\mu_{23}$ be a coupling measure on $\Omega_2 \times
\Omega_3$. Then there exists a coupling measure $\mu_{13}$ on $\Omega_1$ and
$\Omega_3$ such that
\begin{align*}
\Delta_{2\to 2}&(\cW_1^{\pi_{13,1},\mu_{13}},\cW_3^{\pi_{13,3},\mu_{13}})\\
&\le\Delta_{2\to 2}(\cW_1^{\pi_{12,1},\mu_{12}},\cW_3^{\pi_{12,2},\mu_{12}})
+\Delta_{2\to 2}(\cW_1^{\pi_{23,2},\mu_{23}},\cW_3^{\pi_{23,3},\mu_{23}})
\end{align*}
\and
\begin{align*}
\Delta_{\jbl}&(\cW_1^{\pi_{13,1},\mu_{13}},\cW_3^{\pi_{13,3},\mu_{13}})\\
&\le\Delta_{\jbl}(\cW_1^{\pi_{12,1},\mu_{12}},\cW_3^{\pi_{12,2},\mu_{12}})
+\Delta_{\jbl}(\cW_1^{\pi_{23,2},\mu_{23}},\cW_3^{\pi_{23,3},\mu_{23}}),
\end{align*}
where the inequalities hold coordinate-wise.
\end{lemma}

\begin{proof}
Let the steps of $\cW_1$ be $A_1,A_2,\dots$, the steps of $\cW_2$ be
$B_1,B_2,\dots$, and the steps of $\cW_3$ be $C_1,C_2,\dots$. Without loss of
generality, we may assume that each $\mu_1(A_p)>0$, each $\mu_2(B_q)>0$, and
each $\mu_3(C_r)>0$.
First, take the measure
$\mu_{123}$ on $\Omega_1 \times\Omega_2 \times \Omega_3$ where
\[
{\mu_{123}}(E) =\sum_{p,q,r}\frac{\mu_{12}(A_p \times B_q)\mu_{23}(B_q \times
C_r)}{\mu_1(A_p)\mu_2(B_q)^2 \mu_3(C_r)} \left(\mu_1 \times \mu_2 \times
\mu_3\right)(E \cap A_p \times B_q \times C_r) .\] Then
\begin{align*}
\mu_{123}&(A_{p_0} \times B_{q_0} \times \Omega_3)=
\sum_{p,q,r}\frac{\mu_{12}(A_p \times B_q)\mu_{23}(B_q \times C_r)}{\mu_1(A_p)\mu_2(B_q)^2 \mu_3(C_r)}
\\&\qquad \qquad \phantom{} \cdot
\left(\mu_1 \times \mu_2 \times \mu_3\right)(A_{p_0}\times B_{q_0} \times \Omega_3 \cap A_p \times B_q \times C_r)
\\
&= \mu_{12}(A_{p_0} \times B_{q_0}) \sum_r \frac{\mu_{23}(B_{q_0} \times C_r)}{\mu_1(A_{p_0})\mu_2(B_{q_0})^2\mu_3(C_r)}
\mu_1(A_{p_0})\mu_2(B_{q_0})\mu_3(C_r)\\
&= \mu_{12}(A_{p_0} \times B_{q_0}) \sum_r \frac{\mu_{23}(B_{q_0} \times C_r)}{\mu_2(B_{q_0})}
= \mu_{12}(A_{p_0} \times B_{q_0}) .
\end{align*}
In other words, if $\pi_{123,12}$ is the projection from
$\Omega_1\times\Omega_2\times\Omega_3$ to $\Omega_1\times\Omega_2$ and
$\mu_{12}'=\mu_{123}^{\pi_{123,12}}$, then $\mu_{12'}(A_p \times
B_q)=\mu_{12}(A_p\times B_q)$. Analogously, if
$\mu_{23}'=\mu_{123}^{\pi_{123,23}}$, then $\mu_{23'}(A_p \times
B_q)=\mu_{23}(A_p\times B_q)$. Furthermore, for any $F \subseteq \Omega_1$,
\begin{align*}
&\mu_{123}(F \times \Omega_2 \times \Omega_3)\\
&=\sum_{p,q,r}\frac{\mu_{12}(A_p \times B_q)\mu_{23}(B_q \times C_r)}{\mu_1(A_p)\mu_2(B_q)^2 \mu_3(C_r)} \left(\mu_1 \times \mu_2 \times \mu_3\right)(F\times \Omega_2 \times \Omega_3 \cap A_p \times B_q \times C_r)\\
&=\sum_{p,q,r}\frac{\mu_{12}(A_p \times B_q)\mu_{23}(B_q \times C_r)}{\mu_1(A_p)\mu_2(B_q)^2 \mu_3(C_r)} \mu_1(F \cap A_p) \mu_2(B_q) \mu_3(C_r)\\
&=\sum_{p,q,r}\frac{\mu_{12}(A_p \times B_q)\mu_{23}(B_q \times C_r)}{\mu_1(A_p)\mu_2(B_q)} \mu_1(F \cap A_p)
=\sum_{p,q}\frac{\mu_{12}(A_p \times B_q)}{\mu_1(A_p)} \mu_1(F \cap A_p)
\\&
=\sum_{p}\frac{\mu_{12}(A_p \times \Omega_2)}{\mu_1(A_p)} \mu_1(F \cap A_p)
=
\sum_{p} \mu_1(F \cap A_p)=\mu_1(F).
\end{align*}
Analogously, for any $G \subseteq \Omega_3$, $\mu_{123}(\Omega_1 \times
\Omega_2 \times G)=\mu_3(G)$. Therefore, $\mu_{13}=\mu_{123}^{\pi_{123,13}}$
is a coupling measure on $\Omega_1 \times \Omega_3$ of $\mu_1$ and $\mu_3$.
By Lemma~\ref{lemmapullbacksame} and the triangle inequality for the
coordinates of $\Delta_{2\to 2}$, we then have
\begin{align*}
\Delta_{2\to 2}&(\cW_1^{\pi_{13,1},\mu_{13}},\cW_3^{\pi_{13,3},\mu_{13}})\\
&=\Delta_{2\to 2}(\cW_1^{\pi_{123,1},\mu_{123}},\cW_3^{\pi_{123,3},\mu_{123}}) \\
&\le \Delta_{2\to 2}(\cW_1^{\pi_{123,1},\mu_{123}},\cW_2^{\pi_{123,2},\mu_{123}})
+\Delta_{2\to 2}(\cW_2^{\pi_{123,2},\mu_{123}},\cW_3^{\pi_{123,3},\mu_{123}})\\
&=
\Delta_{2\to 2}(\cW_1^{\pi_{12,1},\mu_{12}'},\cW_2^{\pi_{12,2},\mu_{12}'})
+\Delta_{2\to 2}(\cW_2^{\pi_{23,2},\mu_{23}'},\cW_3^{\pi_{23,3},\mu_{23}'})\\
&=\Delta_{2\to 2}(\cW_1^{\pi_{12,1},\mu_{12}},\cW_2^{\pi_{12,2},\mu_{12}})
+\Delta_{2\to 2}(\cW_2^{\pi_{23,2},\mu_{23}},\cW_3^{\pi_{23,3},\mu_{23}}).
\end{align*}
The proof for $\Delta_{\jbl}$ is the same.
\end{proof}

We are now ready to prove that $\tdel22$ and the distance
$\tdeljbl$ (obtained by replacing $\d22$ in \eqref{tdel22-def} with $\djbl$)
obey the triangle inequality.

\begin{lemma} \label{lemmacouplingsequence}
Suppose that $\cW_i=(W_i,S_i,I_i,\bOmega_i)$ with
$\bOmega_i=(\Omega_i,\cF_i,\mu_i)$, for $i=1,2,3$, are signed graphexes in
$L^1\cap L^2$, and assume that
 $\mu_1(\Omega_1)=\mu_2(\Omega_2)=\mu_3(\Omega_3)$. Then
\[
\tdel22(\cW_1,\cW_3)\leq\tdel22(\cW_1,\cW_2) + \tdel22(\cW_2,\cW_3)
\]
and
\[
\tdeljbl(\cW_1,\cW_3)\leq\tdeljbl(\cW_1,\cW_2) + \tdeljbl(\cW_2,\cW_3).
\]
\end{lemma}

\begin{proof}
We first claim that it is enough to prove that for any coupling measure
$\mu_{12}$ on $\Omega_1 \times \Omega_2$, any coupling measure $\mu_{23}$ on
$\Omega_2 \times \Omega_3$ and any $\varepsilon>0$, there exists a coupling
measure $\mu_{13}$ on $\Omega_1 \times \Omega_3$, such that
\begin{multline}
\label{Delta22-triangle}
\Delta_{2\to 2}(\cW_1^{\pi_{13,1},\mu_{13}},\cW_3^{\pi_{13,3},\mu_{13}})\\
\le
\Delta_{2\to 2}(\cW_1^{\pi_{12,1},\mu_{12}},\cW_2^{\pi_{12,2},\mu_{12}})
+\Delta_{2\to 2}(\cW_2^{\pi_{23,2},\mu_{23}},\cW_3^{\pi_{23,3},\mu_{23}})
+(\varepsilon,\varepsilon,\varepsilon)
\end{multline}
and
\begin{multline}
\label{Delta-jbl-triangle}
\Delta_{\jbl}(\cW_1^{\pi_{13,1},\mu_{13}},\cW_3^{\pi_{13,3},\mu_{13}})\\
\le
\Delta_{\jbl}(\cW_1^{\pi_{12,1},\mu_{12}},\cW_2^{\pi_{12,2},\mu_{12}})
+\Delta_{\jbl}(\cW_2^{\pi_{23,2},\mu_{23}},\cW_3^{\pi_{23,3},\mu_{23}})
+(\varepsilon,\varepsilon,\varepsilon).
\end{multline}
Indeed, given that $\eps>0$ is arbitrary, \eqref{Delta-jbl-triangle} clearly
implies the triangle inequality for $\tdeljbl$. To see that
\eqref{Delta22-triangle} implies the triangle inequality for $\tdel22$,
observe that $(x+y)^{1/k} \le x^{1/k}+y^{1/k}$ whenever $k \ge 1$.

Next, we claim that for any $\varepsilon>0$, any $\cW=(W,S,I,\bOmega)$ in
$L^1\cap L^2$ can be approximated by a signed step graphex $\cW'$ such that
\[
\Delta_{\jbl}(\cW,\cW') \le \Delta_{2\to 2}(\cW,\cW') \le (\varepsilon,\varepsilon,\varepsilon).
\]
Indeed, let $\bOmega=(\Omega,\cF,\mu)$, let
$\Omega_1\subseteq\Omega_2\subseteq\dots\subseteq\Omega$ be such that
$\mu(\Omega_n)<\infty$ and $\Omega=\bigcup_n\Omega_n$, and let
$\cW_n=(W_n,S_n,I,\bOmega)$, where $W_n=W1_{\Omega_n\times\Omega_n}$ and
$S_n=S1_{\Omega_n}$. Using the dominated convergence theorem and the
assumption that $\cW$ is in $L^1\cap L^2$, we then have that
\[
\left|\rho(\cW)-\rho(\cW_n)\right|\leq \|W-W_n\|_1
=\|\,|W|(1-1_{\Omega_n\times\Omega_n})\|_1\to 0,
\]
\[
\|W-W_n\|_{2\to 2}\leq \|W-W_n\|_2=\|\,W(1-1_{\Omega_n\times\Omega_n})\|_2\to 0,
\]
and
\[
\|S-S_n\|_2=\|\,S(1-1_{\Omega_n})\|_2\to 0
\]
as $n\to\infty$. Next, defining $\chi_n$ by
$\chi_n(x,y,z)=(1-1_{\Omega_n\times\Omega_n}(x,y))
(1-1_{\Omega_n\times\Omega_n}(y,z))$, we bound
\begin{align*}
\|&D_W-D_{W_n}\|_2
&\leq\sqrt{\int|W(x,y)||W(y,z)| \chi_n(x,y,z) \,d\mu(x)\,d\mu(y) \,d\mu(z)}.
\end{align*}
Since $\chi_n$ goes to zero pointwise and $D_{|W|}$ is in $L^2$, the right
side again goes to zero by the dominated convergence theorem. Therefore,
\[\|D_{\cW}-D_{\cW_n}\|_2 \le \|D_W-D_{W_n}\|_2+\|S-S_n\|_2 \to 0
.\] This shows that for $n$ large enough $\Delta_{2\to 2}(\cW,\cW_{n})\leq
\eps /4$.

Fixing $n$ such that this holds, we now define $W^{(k)}=W_n1_{|W_n|\leq k}$
and $S^{(k)}=S_n 1_{|S_n|\leq k}$. Another application of the dominated
convergence theorem then shows that for $k$ large enough, $\Delta_{2\to
2}(\cW^{(k)},\cW_{n})\leq \eps /4$, giving us a graphon
$\cW''=(W'',S'',I,\bOmega)$ such that the degree support of $\cW''$ has
finite measure, both $W''$ and $D_{\cW''}$ are bounded, and $\Delta_{2\to
2}(\cW'',\cW)\leq \eps /2$. But such a graphex can be approximated to
arbitrary precision by a step graphex with finitely many steps, proving the
claim for $\Delta_{2\to 2}$. Since on two variable functions,
$\|\cdot\|_{\jbl}$ is bounded by $\|\cdot\|_{2\to 2}$, and on functions of
one variable it is bounded by $\|\cdot\|_2$, the claim for $\Delta_{\jbl}$
follows as well.

Fix $\varepsilon>0$, and let $\cW_1',\cW_2',\cW_3'$ be approximations of
$\cW_1,\cW_2,\cW_3$ by signed step graphexes such that for $k=1,2,3$,
\[
\Delta_{2\to 2}(\cW_k,\cW_k') \le (\varepsilon/6,\varepsilon/6,\varepsilon/6).
\]
If $\mu_{ij}$ is a coupling measure on $\Omega_i\times\Omega_j$ and
$\pi_{ij,k}$ is the projection onto $\Omega_k$, $k=i$ or $j$, then
$\pi_{ij,k}$ is measure preserving. Therefore, by
Lemma~\ref{lemmapullbacksame},
\[
\Delta_{2\to 2}\Bigl((\cW_k)^{\pi_{ij,k},\mu_{ij}},(\cW_k')^{\pi_{ij,k},\mu_{ij}}\Bigr) \le (\varepsilon/6,\varepsilon/6,\varepsilon/6).
\]
Combined with Lemma~\ref{lemmacouplingsequencecountablesteps}, we conclude
that there exists a coupling measure $\mu_{13}$ on $\Omega_1\times\Omega_3$
such that
\begin{align*}
\Delta_{2\to 2}&(\cW_1^{\pi_{13,1},\mu_{13}},\cW_3^{\pi_{13,3},\mu_{13}})
\le
\Delta_{2\to 2}\Bigl((\cW_1')^{\pi_{13,1},\mu_{13}},(\cW_3')^{\pi_{13,3},\mu_{13}}\Bigr)
+(\eps/3,\eps/3,\eps/3)
\\
&\le
\Delta_{2\to 2}\Bigl((\cW_1')^{\pi_{12,1},\mu_{12}},(\cW_2')^{\pi_{12,2},\mu_{12}}\Bigr)
+\Delta_{2\to 2}\Bigl((\cW_2')^{\pi_{23,2},\mu_{23}},(\cW_3')^{\pi_{23,3},\mu_{23}}\Bigr)
\\
&\qquad\phantom{}+(\eps/3,\eps/3,\eps/3)
\\
&\le
\Delta_{2\to 2}(\cW_1^{\pi_{12,1},\mu_{12}},\cW_2^{\pi_{12,2},\mu_{12}})
+\Delta_{2\to 2}(\cW_2^{\pi_{23,2},\mu_{23}},\cW_3^{\pi_{23,3},\mu_{23}})
+(\varepsilon,\varepsilon,\varepsilon)
\end{align*}
proving \eqref{Delta22-triangle} and hence the first statement of the lemma.
The proof of \eqref{Delta-jbl-triangle} and the second statement follows in
the same way.
\end{proof}

The proof of Theorems~\ref{thm:deltt-metric} and \ref{thm:deltws-metric}
 will be an easy corollary
of Lemma~\ref{lemmacouplingsequence} and the following extension lemma.

\begin{lemma}\label{lem:extension}
Let $\cW=(W,S,I,\bOmega)$ be a signed graphex
in $L^1\cap L^2$, with possibly unbounded graphon parts,
and let $\bOmega=(\Omega,\cF,\mu)$.
\begin{enumerate}
\item If $\cW'$ and $\cW''$ are trivial extensions of $\cW$ by
    $\sigma$-finite spaces of infinite measure, then
\[
\tdeljbl(\cW',\cW'')=\tdel22(\cW',\cW'')=0.
\]
\item If $\mu(\Omega)=\infty$ and $\widetilde\cW=(\widetilde W,\widetilde
S,I,\widetilde\bOmega)$ is obtained from $\cW$ by appending an arbitrary
$\sigma$-finite space of infinite measure, then
\[
\tdel22(\cW,\widetilde\cW)=\tdeljbl(\cW,\widetilde\cW) =0.
\]
\end{enumerate}
\end{lemma}

\begin{proof}
To prove the first statement, let $\bOmega'=(\Omega',\cF',\mu')$ and
$\bOmega''=(\Omega'',\cF'',\mu'')$ be the spaces $\bOmega$ has been extended
by. Let $\widehat\mu$ be the measure on $\Omega\times\Omega$ which couples
$\mu$ to itself along the diagonal, choose an arbitrary coupling
$\widetilde\mu$ of $\mu'$ and $\mu''$, and let $\widehat\mu'$ be the measure
on $(\Omega\cup\Omega')\times(\Omega\cup\Omega'')$ defined by
\[\widehat\mu'(A)=\widehat\mu(A\cap(\Omega\times\Omega))+\widetilde\mu(A\cap(\Omega'\times\Omega'')).
\]
Using the fact that $\widehat\mu'(\Omega\times\Omega'')=\widehat\mu'(\Omega'\times\Omega)=0$,
it is easy to see that
\[
\tdel22(\cW',\cW'')\leq \d22((\cW')^{\pi_1,\widehat\mu'},(\cW'')^{\pi_2,\widehat\mu'})
=\d22(\cW^{\pi_1,\widehat\mu},\cW^{\pi_2,\widehat\mu})=0.
\]
This proves the first statement for the metric $\tdel22$. The proof for the
metric $\tdeljbl$ is identical.

To prove, the second statement, let
$\widetilde\bOmega=(\widetilde\Omega,\widetilde\cF,\widetilde\mu)$. Since
$\bOmega$ is $\sigma$-finite, we can find a sequence of measurable subsets
$\Omega_n\subseteq\Omega$ such that $\Omega=\bigcup\Omega_n$ and each
$\Omega_n$ has finite measure. Replacing $\Omega_n$ by
$\Omega_1\cup\dots\cup\Omega_n$, we may further assume that $\Omega_n$ is an
increasing sequence of sets. Let $W_n$ be equal to $W$ on
$\Omega_n\times\Omega_n$ and $0$ everywhere else, and let $S_n=S$ on
$\Omega_n$ and $0$ outside of $\Omega_n$. Let $\cW_n$ be the corresponding
graphex on $\bOmega$ (with the same value $I$), and $\widetilde\cW_n$ be its
trivial extension to $\widetilde\bOmega$. By monotone convergence, $W_n\to W$
in both $L^1$ and $L^2$, and $S_n\to S$ in $L^1$ and $L^2$, implying that
$\tdel22(\cW_n,\cW)\leq \d22(\cW_n,\cW)\to 0$. For the same reason,
$\tdel22(\widetilde \cW_n,\widetilde\cW)\to 0$. But since $\widetilde\cW_n$
and $\cW_n$ can both be obtained from the restriction of $\cW_n$ to
$\Omega_n$ by appending a space of infinite total measure, we have
$\tdel22(\cW_n,\widetilde\cW_n)=\tdel22(\cW_n,\cW_n)=0$ by the first
statement of the lemma. Using the triangle inequality for $\tdel22$, this
proves the second statement for the distance $\tdel22$. The proof for the
metric $\tdeljbl$ follows from the fact that the jumble norm is bounded by
the kernel norm, which in turn implies that $\djbl(\cW_n,\cW)\to 0$ whenever
$\d22(\cW_n,\cW)\to 0$.
\end{proof}

We are now ready to prove Theorems~\ref{thm:deltt-metric} and
\ref{thm:deltws-metric}.

\begin{proof}[Proof of Theorems~\ref{thm:deltt-metric} and \ref{thm:deltws-metric}]
The first statement of Lemma~\ref{lem:extension} implies that if $\cW_1'$ and
$\cW_2'$ are trivial extensions of $\cW_1$ and $\cW_2$ obtained by
\emph{appending} two $\sigma$-finite spaces of infinite measure, then
$\tdel22(\cW_1',\cW_2')$ and $\tdeljbl(\cW_1',\cW_2')$ do not depend on the
choice of these extensions, and the second (combined with the triangle
inequality) allows us to conclude that this remains true for extensions
\emph{to} spaces of infinite measure, which completes the proof of the first
statements of the two theorems.

Since clearly $\deltt$ is symmetric and $\deltt(\cW,\cW)=0$ for all
integrable graphexes, all that remains to be proved is the triangle
inequality for $\deltt$, which follows from the (already established)
triangle inequality for $\tdel22$. The same holds for $\tdeljbl$.
\end{proof}

The following example shows that the extension to infinite spaces in the
definition of $\deltt$ is really needed.

\begin{example}\label{ex1}
Let $\cW_1=(W_1,0,0,\bOmega_1)$ where $\bOmega_1$ consists of just two atoms
$a$ and $b$, with weight $p$ and $1-p$, where $0<p<1/2$, and
$W_1(a,a)=W_1(b,b)=0$, $W_1(a,b)=W_1(b,a)=1$. Furthermore, let
$\cW_2=(W_2,0,0,\bOmega_2)$ where $\bOmega_2$ consists of just one atom $c$
with weight $1$, and $W_2$ is the constant $a=\sqrt{p(1-p)}$. Then we have
just one choice of coupling. For this coupling, $W_1^{\pi_1}-W_2^{\pi_2}$
will have two atoms, and it will be equal to $-\sqrt{p(1-p)}$ on the
diagonal, and $1-\sqrt{p(1-p)}$ off the diagonal. It is then not difficult to
see that $\|W_1-W_2\|_{2 \rightarrow 2}$ is equal to the largest eigenvalue
(in absolute value) of the matrix
\[\begin{pmatrix}
-pa & (1-p)(1-a)\\
p(1-a) & -(1-p)a
\end{pmatrix}
.
\]
The trace of this matrix is $-a$, and the determinant is
\[
p(1-p)a^2-p(1-p)(1-a)^2=p(1-p)(2a-1)<0
.
\]
Here we used that $a=\sqrt{p(1-p)}<1/2$. We then have that the two
eigenvalues of the matrix have opposite signs, and their sum is $-a$, which
implies that the negative one must be less than $-a$; i.e., it must have
larger absolute value than $a$.

On the other hand, clearly $\|W_1\|_{2 \rightarrow 2}=a$, and $\|W_2\|_{2\to
2}$ is equal to the largest eigenvalue (in absolute value) of the matrix
\[\begin{pmatrix}
0 & 1-p\\
p & 0
\end{pmatrix}
,
\]
which can easily be seen to be equal to $\sqrt{p(1-p)}=a$.

Therefore, if we extend $W_1$ and $W_2$ by spaces of total measure at least
$1$, and couple each $\Omega_i$ to the extension, then
\[
\|\wW_1-\wW_2\|_{2 \rightarrow 2}=\max\{\|\wW_1\|_{2\to2},\|\wW_2\|_{2\to2}\}=a
.
\]
Therefore, by extending, we can obtain a better coupling.

Finally, note that in $(\ref{d22-def})$, if we multiply the measure of the
underlying space by $c$, then the first term is multiplied by $c$, the second
by $c^{3/4}$, and the third by $c^{2/3}$. We can therefore take $c$ large
enough so that the maximum in the term is dominated by $\|W_1-W_2\|_{2\to2}$.
Therefore, we obtain that also for minimizing $d_{2 \to 2}$, we obtain a
better coupling if we extend by trivial extensions than if we do not.
\end{example}

Next, we would like to prove Theorem \ref{thm:dmetric}. Before doing so, we
note that the distance $\delGP$ is well defined and finite for signed
graphexes as well, provided the graphon part is bounded in the $L^\infty$
norm. As the reader may easily verify, this immediately follows from the
following lemma, which is an easy corollary to
Proposition~\ref{prop:local-finite}.

\begin{lemma}
\label{lem:signed-delGP<infty} Let $\cW$ be a signed graphex with bounded
graphon part, and let $0<D<\infty$. Then the set $\{D_{|\cW|}>D\}$ has finite
measure, and $\cW|_{\{D_{|\cW|}\leq D\}}$ is integrable and hence in $L^1\cap
L^2$.
\end{lemma}

\begin{proof}
Let $\cW=(W,S,I,\bOmega)$ with $\bOmega=(\Omega,\cF,\mu)$, and assume that
$\|W\|_\infty\leq K$. If $K\leq 1$, the lemma follows from
Proposition~\ref{prop:local-finite} applied to $\cW'=|\cW|$. Otherwise, we
define $\cW'=(W',S',I',\bOmega)$ where $W'=|W|/K$, $S'=|S|$, and $I'=|I|$.
Applying Proposition~\ref{prop:local-finite} to this graphex, and noting that
$\frac 1K D_{|\cW|}\leq D_{\cW'}\leq D_{|\cW|}$, we see that
$\mu(\{D_{|\cW|}>D\})\leq \mu(\{D_{|\cW'|}>D/K\})<\infty$ and
$\|\cW_{|D_{|\cW|}\leq D}\|_1\leq \|\cW_{|D_{\cW'}\leq D}\|_1\leq
K\|\cW'_{|D_{\cW'}\leq D}\|_1<\infty$, as claimed.
\end{proof}

To prove Theorem~\ref{thm:dmetric}, we need to establish the triangle
inequality. The reason it is not obvious is because when we decrease the
measure on the underlying set, the $\deltt$ distance can increase. In the
following lemma, we show that although it can increase under restrictions to
subsets or decreasing of the underlying measure, it cannot increase too much.

\begin{lemma} \label{lemmaremovedsetdistancenotworse}
Let $\cW_1=(W_1,S_1,I_1,\bOmega)$ and $\cW_2=(W_2,S_2,I_2,\bOmega)$ be signed
graphexes in $L^1\cap L^2$, let $\bOmega=(\Omega,\cF,\mu)$, and let
\begin{enumerate}
\item $\|W_1-W_2\|_{2 \rightarrow 2,\mu} =a$,
\item $\|D_{\cW_1}-D_{\cW_2}\|_{2,\mu} =b$, and
\item $|\rho(\cW_1)-\rho(\cW_2)| =c$.
\end{enumerate}
If $\cW_1'=(W_1,S_1,I_1,\bOmega')$ and $\cW_2'=(W_2,S_2,I_2,\bOmega')$ where
$\bOmega'=(\Omega,\cF,\mu')$ for some measure $\mu'$ such that $\mu-r \le
\mu' \le \mu$ for some $r<\infty$, then
\begin{enumerate}
\item $\|W_1-W_2\|_{2 \rightarrow 2,\mu'} \le a$, \label{removed2to2}
\item $\|D_{\cW_1'}-D_{\cW_2'}\|_{2,\mu'} \le b+a\sqrt{r}$, and
    \label{removedD2}
\item $\Bigl| |\rho(\cW_1')-\rho(\cW_2')| - c\Bigr|\le 2b\sqrt{r}+a r$.
    \label{removedL1}
\end{enumerate}
\end{lemma}

We recall that $\mu-r \le \mu' \le \mu$ if and only if there exists a
measurable function $h\colon\Omega\to [0,1]$ such that $\mu'(B)=\int_{B}h
\,d\mu$ for all measurable sets $B$, and $\|1-h\|_{1,\mu} \le r<\infty$. An
interesting special case is the case where $h$ is the characteristic function
of $\Omega \setminus R$ for a set $R$ of measure $r$, in which case $\cW_i'$
is the restriction of $\cW$ to $\Omega \setminus R$, after neglecting points
outside the degree support.

\begin{proof}
To show property (\ref{removed2to2}), let $f,g$ be such that
$\|f\|_{2,\mu'}=\|g\|_{2,\mu'}=1$. In other words,
\[1=\int_\Omega f^2 \,d\mu' = \int_{\Omega} h f^2 \,d\mu,
\] and
\[1=\int_\Omega g^2 \,d\mu' = \int_{\Omega} h g^2 \,d\mu
.\] Let $U=W_1-W_2$. We then have
\begin{align*}
\left|\int_{\Omega \times \Omega} f(x) U(x,y) g(y) \,d(\mu' \times \mu') \right| &= \left|\int_{\Omega \times \Omega} f(x) U(x,y) g(y) h(x) h(y) \,d(\mu \times \mu) \right|\\
& \le \|U\|_{2 \rightarrow 2,\mu} \|f h \|_{2,\mu} \|g h \|_{2,\mu}
.
\end{align*}
We also have
\[\|fh\|_{2,\mu}^2=\int_\Omega f(x)^2h(x)^2 \,d\mu(x) \le \int_\Omega f(x)^2h(x) \,d\mu(x) = 1
.\] Similarly, $\|gh\|_{2,\mu} \le 1$. Therefore,
\[\left|\int_{\Omega \times \Omega} f(x) U(x,y) g(y) \,d(\mu' \times \mu') \right| \le a
.\] Since this holds for any $f,g$ such that
$\|f\|_{2,\mu'}=\|g\|_{2,\mu'}=1$, we have $\|U\|_{2 \rightarrow 2,\mu'} \le
\|U\|_{2 \rightarrow 2,\mu}=a$.

For (\ref{removedD2}), let
\[D_i(x)=D_{\cW_i}(x)-D_{\cW_i'}(x)=\int_{\Omega}W_i(x,y) (1-h(y))\,d\mu(y)
.\] Then
\begin{align*}
\|D_{\cW_1'}-D_{\cW_2'}\|_{2,\mu'} &\le \|D_{\cW_1'}-D_{\cW_2'}\|_{2,\mu}\\
& =\sup_{g:\|g\|_{2,\mu}=1} \int_{\Omega} (D_{\cW_1'}-D_{\cW_2'})(x)g(x)\,d\mu(x) \\
&= \sup_{g:\|g\|_{2,\mu}=1} \int_{\Omega}\left( (D_{\cW_1}-D_{\cW_2})(x)g(x)-(D_1-D_2)(x)g(x)\right)\,d\mu(x)\\
&=
\sup_{g:\|g\|_{2,\mu}=1} \Bigg(
\int_{\Omega} (D_{\cW_1}-D_{\cW_2})(x)g(x) \,d\mu(x)\\
& \qquad \qquad \phantom{} - \int_{\Omega \times \Omega} g(x)(W_1-W_2)(x,y)(1-h(y)) \,d(\mu \times \mu)
\Bigg)\\
&\le b+ a \|g\|_{2,\mu}\|1-h\|_{2,\mu}
\le b + a \sqrt{r}
.\end{align*}

To prove (\ref{removedL1}), we use that
\begin{align*}
\rho(\cW_i') &= \int_{\Omega \times \Omega} h(x) W_i(x,y) h(y) \,d\mu(x) \,d\mu(y) + 2 \int_{\Omega} S_i(x) h(x) \,d\mu(x) + I_i\\
&= \rho(\cW_i) -\int_{\Omega \times \Omega} (1-h(x)) W_i(x,y) \,d\mu(x) \,d\mu(y)\\
& \qquad \phantom{} - \int_{\Omega \times \Omega} W_i(x,y) (1-h(y)) \,d\mu(x) \,d\mu(y)\\
& \qquad \phantom{}+\int_{\Omega \times \Omega} (1-h(x)) W_i(x,y) (1-h(y)) \,d\mu(x) \,d\mu(y)\\
& \qquad \phantom{}- 2\int_{\Omega} (1-h(x))S(x) \,d\mu(x)\\
&=\rho(\cW_i) - 2 \int_\Omega (1-h(x))D_{\cW_i}(x) \,d\mu(x)\\
& \qquad \phantom{}+\int_{\Omega \times \Omega} (1-h(x)) W_i(x,y) (1-h(y)) \,d\mu(x) \,d\mu(y).
\end{align*}
Therefore,
\begin{align*}
\Bigl|\bigl|\rho(\cW'_1)-&\rho(\cW'_2)\bigr|-\bigl|\rho(\cW_1)-\rho(\cW_2)\bigr|\Bigr|\\
&\le \biggl|2\int_{\Omega} (1-h(x))(D_{\cW_1}-D_{\cW_2})(x) \,d\mu(x)\biggr.\\
&\qquad
-
\biggl.\int_{\Omega \times \Omega}(1-h(x))(W_1-W_2)(x,y)(1-h(y))\,d\mu(x)\,d\mu(y) \biggr|\\
&\le 2b\sqrt{r}+ ar.
\end{align*}
Here we used the fact that $\|1-h\|_2 \le \sqrt{\|1-h\|_1} \le \sqrt{r}$.
This implies the claim.
\end{proof}

We also use the following equivalent representation of the weak kernel
distance $\delGP$.

\begin{lemma}\label{lem:tilde-delGP}
For $i=1,2$, let $\cW_i$ be graphexes over
$\bOmega_i=(\Omega_i,\cF_i,\mu_i)$, and let $\cW_i'$ be trivial extensions of
$\cW_i$ to $\sigma$-finite measure spaces
$\bOmega_i'=(\Omega_i',\cF_i',\mu_i')$ with $\mu'_i(\Omega_i')=\infty$. Then
$\delGP(\cW_1,\cW_2)=\tdelGP(\cW_1',\cW_2')$, where $\tdelGP(\cW_1',\cW_2')$
is defined as the infimum over all $c$ such that there exists a measure
$\mu'$ over $\Omega_1'\times\Omega_2'$ obeying the conditions
\begin{enumerate}
\item[(1)] $\mu_i'-c^2\leq(\mu')^{\pi_i}\leq\mu_i'$ for $i=1,2$, and
\item[(2)] $\d22((\cW_1')^{\pi_1,\mu'},(\cW_2')^{\pi_2,\mu'})\leq c$.
\end{enumerate}
\end{lemma}

\begin{proof}
Let $\mu'$ and $c$ be such that they obey the conditions in the statement of
the lemma. For $i=1,2$, let $\widetilde\mu_i$ be the restriction of
$(\mu')^{\pi_i}$ to $\Omega_i$, and let $\widetilde\cW_i$ be obtained from
$\cW_i$ by replacing $\mu_i$ by $\widetilde\mu_i$. Then $\mu_i-c^2 \le
\widetilde\mu_i \le \mu_i$. Furthermore, $(\mu')^{\pi_i}$ extends
$\widetilde\mu_i$ to $\Omega_i'$, and
$(\mu')^{\pi_i}(\Omega_i'\setminus\Omega_i)\geq
\mu_i'(\Omega_i'\setminus\Omega_i)-c^2=\infty$, showing that this defines an
extension by a space of infinite measure. Finally, $\mu'$ is a coupling of
$(\mu')^{\pi_1}$ and $(\mu')^{\pi_2}$. Together, these facts imply that
$\deltt(\widetilde\cW_1,\widetilde\cW_2)\leq c$, proving that
$\delGP(\cW_1,\cW_2)\leq c$. This shows that $\delGP(\cW_1,\cW_2)\leq
\tdelGP(\cW_1',\cW_2')$.

To prove the reverse inequality, assume that $c$ is such that there are
measures $\widetilde\mu_1$ and $\widetilde\mu_2$ over $\Omega_1$ and
$\Omega_2$ such that $\deltt(\widetilde\cW_1,\widetilde\cW_2)\leq c$ and
$\mu_i-c^2\leq\widetilde\mu_i\leq\mu_i$ for $i=1,2$, where $\widetilde\cW_i$
is again obtained from $\cW_i$ by replacing $\mu_i$ by $\widetilde\mu_i$. If
we transform $\bOmega_i'$ into a space $\widetilde\bOmega_i'$ by setting
$\widetilde\mu_i'$ to $\widetilde\mu_i$ on $\Omega_i$, and to $\mu'_i$ on
$\Omega_i\setminus\Omega_i'$, and define $\widetilde\cW_i'$ as the trivial
extension of $\widetilde\cW_i$ to $\widetilde\bOmega_i'$, then Theorem
\ref{thm:deltt-metric} implies that
$\tdel22(\widetilde\cW_1',\widetilde\cW_2')=\deltt(\widetilde\cW_1,\widetilde\cW_2)\leq
c$. This means that for all $\eps>0$, there is a coupling $\mu'$ of
$\widetilde\mu_1'$ and $\widetilde\mu_2'$ such that
$\d22((\cW_1')^{\pi_1,\mu'},(\cW_2')^{\pi_2,\mu'})\leq c+\varepsilon$.
Observing that the bound $\mu_i-c^2\leq\widetilde\mu_i\leq\mu_i$ and our
construction of $\widetilde\mu_i'$ imply that
$\mu_i'-c^2\leq\widetilde\mu_i'=(\mu')^{\pi_i}\leq\mu_i'$, and that
$(\widetilde\cW_i')^{\pi_i,\widetilde\mu'}=(\cW_i')^{\pi_i,\widetilde\mu'}$,
this shows that $\tdelGP(\cW_1',\cW_2')\leq c+\varepsilon$. Since
$\varepsilon$ was arbitrary, this shows that $\delGP(\cW_1,\cW_2)\geq
\tdelGP(\cW_1',\cW_2')$.
\end{proof}

We are now ready to prove Theorem \ref{thm:dmetric}.

\begin{proof}[Proof of Theorem \ref{thm:dmetric}]
It is clear that $\delGP$ is symmetric, and that $\delGP(\cW,\cW)=0$. So we
have to prove the triangle inequality. By Lemma \ref{lem:tilde-delGP}, taking
trivial extensions of each graphex to a space of infinite measure, it
suffices to prove the triangle inequality for $\tdelGP$.

Let $\cW_1,\cW_2,\cW_3$ be three graphexes with the usual notation, defined
over measure spaces which all have infinite measure. Let $\mu_{12}$ be a
measure on $\Omega_1 \times \Omega_2$ that shows that
$\tdelGP(\cW_1,\cW_2)\leq c_1$, let $\mu_{23}$ be a measure on $\Omega_2
\times \Omega_3$ that shows that $\tdelGP(\cW_1,\cW_2)\leq c_2$, let $\mu'_1$
and $\mu'_2$ be the marginals of $\mu_{12}$, and let $\mu_2''$ and $\mu_3''$
be the marginals of $\mu_{23}$. We would like to use
Lemma~\ref{lemmacouplingsequence} to create a coupling of $\mu'_1$ and
$\mu_3''$, but unfortunately, the conditions of the lemma require that
$\mu_2'=\mu_2''$, which we cannot guarantee. To deal with this problem, we
will slightly decrease $\mu_{12}$ and $\mu_{23}$ so that after this
perturbation, the second marginal of the first is equal to the first marginal
of the second.

Let $\pi_{ij,i}$ be the projection map from $\Omega_i\times\Omega_j$ to
$\Omega_i$ for $i,j \in [3]$, and let
$\cW_{ij,i}=\cW_i^{\pi_{ij,i},\mu_{ij}}$. Let
$\mu_2'={\mu_{12}}^{\pi_{12,2}}$ and $\mu_2''={\mu_{23}}^{\pi_{23,2}}$. Let
$h'=\frac{d\mu_2'}{d\mu_2}$ and $h''=\frac{d\mu_2''}{d\mu_2}$. Then we can
assume that $0 \le h',h'' \le 1$, $\|1-h'\|_{1,\mu_2} \le c_1^2$, and
$\|1-h''\|_{1,\mu_2} \le c_2^2$. Let $\widetilde{h}(x)=\min(h'(x),h''(x))$,
and let $\widetilde{\mu}_2$ be the measure defined by
\[\widetilde{\mu}_2(A) =\int_A \widetilde{h} \,d\mu_2
.\] Then $\|h'-\widetilde{h}\|_{1,\mu_2} \le \|1-h''\|_{1,\mu_2} \le c_2^2$.
For $x \in \Omega_1 \times \Omega_2$, let
$h_{12}(x)=\frac{\widetilde{h}(\pi_{12,2}(x))}{h'(\pi_{12,2}(x))}\le 1$, and
let $\widetilde{\mu}_{12}$ be the measure defined by
\[\widetilde{\mu}_{12}(A)=\int_A h_{12}(x) \,d\mu_{12}
.\] Note that $\widetilde{\mu}_{12}^{\pi_{12,2}}=\widetilde{\mu}_2$.
Furthermore, since $\tilde h(x)\le h'(x)$,
\begin{align*}
\int_{\Omega_1
\times \Omega_2} \bigl(1-h_{12}(x)\bigr)
\,d\mu_{12}(x)
&=\int_{\Omega_2}\left(1-\frac{\widetilde{h}(x)}{h'(x)}\right)
\,d\mu_2'(x)
\\
&=\int_{\Omega_2}\left(1-\frac{\widetilde{h}(x)}{h'(x)}\right)
h'(x)\,d\mu_2(x)
= \|h'-\widetilde{h}\|_{1,\mu_2}
\le c_2^2 .
\end{align*}
This means that for any set $A \subseteq \Omega_1 \times \Omega_2$,
\[\mu_{12}(A)-c_2^2 \le \widetilde{\mu}_{12}(A) \le \mu_{12}(A).\]
This implies that for any $A \subseteq \Omega_1$,
\[\mu_1(A) -c_1^2-c_2^2 \le \mu_{12}^{\pi_{12,1}}(A)-c_2^2 \le \widetilde{\mu}_{12}^{\pi_{12,1}}(A) \le \mu_{12}^{\pi_{12,1}}(A) \le \mu_1(A)
.\] We similarly construct $\widetilde{\mu}_{23}$ and $\Omega_2 \times
\Omega_3$ so that $\widetilde{\mu}_{23}^{\pi_{23,2}}=\widetilde{\mu}_2$ and
for any set $A \subseteq \Omega_2 \times \Omega_3$,
\[\mu_{23}(A)-c_1^2 \le \widetilde{\mu}_{23}(A) \le \mu_{23}(A),\]
which implies that for any $A \subseteq \Omega_3$,
\[\mu_3(A) -c_1^2-c_2^2 \le \mu_{23}^{\pi_{23,3}}(A)-c_1^2 \le \widetilde{\mu}_{23}^{\pi_{23,3}}(A) \le \mu_{23}^{\pi_{23,3}}(A) \le \mu_3(A)
.\] Let $\widetilde{\mu}_1=\widetilde{\mu}_{12}^{\pi_{12,1}}$ and
$\widetilde{\mu}_3=\widetilde{\mu}_{23}^{\pi_{23,3}}$, and note that
$\widetilde{\mu}_{12}$ is a coupling of $\widetilde{\mu}_{1}$ and
$\widetilde{\mu}_{2}$, $\widetilde{\mu}_{23}$ is a coupling of
$\widetilde{\mu}_{2}$ and $\widetilde{\mu}_{3}$, and
$\mu_1-c_1^2-c_2^2\leq\widetilde\mu_1\leq\mu_1$ and
$\mu_3-c_1^2-c_2^2\leq\widetilde\mu_3\leq\mu_3$.

Let $\widetilde \cW_i$ be equal to $\cW_i$ but with the measure $\mu_i$
replaced by $\widetilde{\mu}_i$. Fix $\varepsilon>0$. By Lemma
\ref{lemmacouplingsequence}, there exists a measure $\widetilde{\mu}_{13}$ on
$\Omega_1 \times \Omega_3$ such that
$\widetilde{\mu}_{13}^{\pi_{13,1}}=\widetilde{\mu}_{1}$ and
$\widetilde{\mu}_{13}^{\pi_{13,3}}=\widetilde{\mu}_{3}$, and we have
\[\Delta_{2\to 2}(\widetilde{\cW}_1^{\pi_{13,1}},\widetilde{\cW}_3^{\pi_{13,3}}) \le \Delta_{2\to 2}(\widetilde{\cW}_1^{\pi_{12,1}},\widetilde{\cW}_2^{\pi_{12,2}})
+\Delta_{2\to
2}(\widetilde{\cW}_2^{\pi_{23,2}},\widetilde{\cW}_3^{\pi_{23,3}})+(\varepsilon,\varepsilon,\varepsilon)
.\] Note that by the above inequalities,
\[
\mu_1-(c_1+c_2)^2\leq\mu_1-c_1^2-c_2^2\leq\widetilde\mu_1=\widetilde{\mu}_{13}^{\pi_{13,1}}\leq\mu_1\]
and
\[
\mu_3-(c_1+c_2)^2\leq\mu_3-c_1^2-c_2^2\leq\widetilde\mu_3=\widetilde{\mu}_{13}^{\pi_{13,3}}\leq\mu_3.\]

By Lemma \ref{lemmaremovedsetdistancenotworse},
\[\|\widetilde W_1^{\pi_{12,1}}-\widetilde W_2^{\pi_{12,2}}\|_{2 \rightarrow 2,\widetilde{\mu}_{12}} \le \|W_{12,1}-W_{12,2}\|_{2 \rightarrow 2,\mu_{12}} \le c_1,\]
\begin{align*}
\|D_{\widetilde \cW_1^{\pi_{12,1}}}-D_{\widetilde \cW_2^{\pi_{12,2}}}\|_{2,\widetilde{\mu}_{12}} &\le \|D_{\cW_{12,1}}-D_{\cW_{12,2}}\|_{2,\mu_{12}}+\|W_{12,1}-W_{12,2}\|_{2 \rightarrow 2,\mu_{12}} c_2\\
& \le c_1^2 + c_1c_2,
\end{align*}
and finally
\begin{align*}
\left|\rho(\widetilde \cW_1^{\pi_{12,1}})- \rho(\widetilde \cW_2^{\pi_{12,2}}) \right| &\le \left|\rho(\cW_{12,1})-\rho(\cW_{12,2})\right| + 2 \|D_{\cW_{12,1}}-D_{\cW_{12,2}}\|_{2,\mu_{12}} c_2
\\
& \quad \phantom{} + \|W_{12,1}-W_{12,2}\|_{2 \rightarrow 2,\mu_{12}} c_2^2\\
& \le c_1^3 + 2c_1^2c_2 + c_1c_2^2.
\end{align*}
To summarize, this means that
\[
\Delta_{2\to 2}(\widetilde{\cW}_1^{\pi_{12,1}},\widetilde{\cW}_2^{\pi_{12,2}}) \le (c_1,c_1^2+c_1c_2,c_1^3+2c_1^2c_2+c_1c_2^2).
\]
Similarly,
\[
\Delta_{2\to 2}(\widetilde{\cW}_2^{\pi_{23,2}},\widetilde{\cW}_3^{\pi_{23,3}}) \le (c_2,c_2^2+c_1c_2,c_2^3+2c_1c_2^2+c_1^2c_2).
\]
Therefore,
\[\|\widetilde W_1^{\pi_{13,1}}-\widetilde W_3^{\pi_{13,3}}\|_{2 \rightarrow 2,\widetilde{\mu}_{13}} \le c_1+c_2+\varepsilon,\]
\[\|D_{\widetilde\cW_1}-D_{\widetilde\cW_3}\|_{2,\widetilde{\mu}_{13}} \le c_1^2+ c_1c_2+ c_2^2 + c_1c_2+\varepsilon=(c_1+c_2)^2+\varepsilon, \]
and finally
\[\left|
\rho(\widetilde\cW_1)- \rho(\widetilde\cW_3) \right| \le c_1^3 + 2c_1^2c_2 + c_1c_2^2+ c_2^3 + 2c_1c_2^2 + c_1^2c_2+\varepsilon=(c_1+c_2)^3+\varepsilon.\]
Since this can be done for any $\varepsilon>0$, this completes the proof that
$\tdelGP$ is a metric. With the help of Lemma~\ref{lem:tilde-delGP} the
triangle inequality for $\tdelGP$ implies that for $\delGP$.
\end{proof}

Next we prove Proposition~\ref{propboundedequivmetrics}, as well the
following version for signed graphexes.

\begin{proposition} \label{prop:signed-met-equiv}
Fix $B,C,D<\infty$. Then $\delGP$ and $\deltt$ define the same topology on
the space of $(B,C,D)$-bounded signed graphexes.
\end{proposition}

To prove these propositions, we need a lemma complementing the bounds from
Lemma~\ref{lemmaremovedsetdistancenotworse}. Recall that in
Lemma~\ref{lemmaremovedsetdistancenotworse}, we showed that the distance
between two graphexes defined on the same measure space cannot increase too
much when we decrease of the underlying measure. Our next lemma shows that if
the graphexes involved are signed graphexes that are $(B,C,D)$-bounded, we
can also go in the other direction.

\begin{lemma} \label{lemmaremovednotmuchsmaller}
Let $\cW_i=(W_i,S_i,I_i,\bOmega)$, for $i=1,2$, be $(B,C,D)$-bounded signed
graphexes on the same measure space $\bOmega$, and let
 $\mu'$, $r$, $\cW_1'$, and $\cW_2'$ be as in Lemma~\ref{lemmaremovedsetdistancenotworse}.
Then
\begin{enumerate}
\item $\|W_1-W_2\|_{2 \rightarrow 2,\mu'}\le \|W_1-W_2\|_{2 \rightarrow
    2,\mu} \le \|W_1-W_2\|_{2 \rightarrow 2,\mu'} + {4}\sqrt{{B}D r}$,
\item $\Bigl|\|D_{\cW_1}-D_{\cW_2}\|_{2,\mu}^2 -
    \|D_{\cW_1'}-D_{\cW_2'}\|_{2,\mu'}^2 \Bigr|\leq (4D^2+8BC)r $, and
\item
    $\Bigl|\bigl|{\rho(\cW_1')-\rho(\cW_2')}\bigr|-\bigl|{\rho(\cW_1)-\rho(\cW_2)}\bigr|\Bigr|\leq
 4 Dr.$
\end{enumerate}
\end{lemma}

\begin{proof}
Let $U=W_1-W_2$. Then for any $f,g$ with $\|f\|_{2,\mu}=\|g\|_{2,\mu}=1$,
\begin{align*}\int_{\Omega \times \Omega}&f(x)U(x,y)g(y) \,d\mu(x)\,d\mu(y)\\
&=\int_{\Omega \times \Omega}f(x)h(x)U(x,y)h(y)g(y)\,d\mu(x)\,d\mu(y)\\
&\quad\phantom{}+\int_{\Omega \times \Omega}f(x)((1-h(x))U(x,y)h(y))g(y)\,d\mu(x)\,d\mu(y)\\
&\quad\phantom{}+\int_{\Omega \times \Omega} f(x)U(x,y)(1-h(y))g(y)\,d\mu(x)\,d\mu(y)
.\end{align*}
We have
\begin{align*}
\int_{\Omega \times \Omega}&f(x)h(x)U(x,y)h(y)g(y)\,d\mu(x)\,d\mu(y)\\
&=\int_{\Omega \times \Omega}f(x)U(x,y)g(y)\,d\mu'(x)\,d\mu'(y)
\le \|f\|_{2,\mu'} \|U\|_{2 \rightarrow 2,\mu'} \|g\|_{2,\mu'}
\\ &\le
\|f\|_{2,\mu} \|U\|_{2 \rightarrow 2,\mu'} \|g\|_{2,\mu}
\le \|U\|_{2 \rightarrow 2,\mu'}.
\end{align*}
Furthermore,
\begin{align*}
\int_{\Omega \times \Omega} f(x)U(x,y)(1-h(y))&g(y)\,d\mu(x)\,d\mu(y)\\
& \le \|f\|_{2,\mu} \int_{\Omega}\|U_y\|_{2,\mu} (1-h(y))|g(y)| \,d\mu(y)\\
&\le \|f\|_{2,\mu} {2 \sqrt{BD}} \int_\Omega (1-h(y))|g(y)| \,d\mu(y)\\
& \le{2 \sqrt{BD}} \|f\|_{2,\mu}\|1-h\|_{2,\mu} \|g\|_{2,\mu}
\le {2\sqrt{BDr}}.
\end{align*}
Here we used the fact that $\|U\|_\infty \le
{2B}$ and $\|D_{|U|,\mu}\|_\infty \le 2D$, which implies that $\|U_y\|_{2,\mu}
\le {2 \sqrt{BD}}$. Analogously, we have
\begin{align*}
\int_{\Omega \times\Omega}
f(x)((1-h(x))U(x,y)&h(y))g(y) \,d\mu(x)\,d\mu(y)\\
 &\le {2 \sqrt{BD}}
\|f\|_{2,\mu} \|1-h\|_{2,\mu} \|hg\|_{2,\mu}
\le {2 \sqrt{BDr}}.\
\end{align*}
Adding this all up, we have
\[\int_{\Omega \times \Omega}f(x)U(x,y)g(y) \,d\mu(x) \,d\mu(y) \le \|U\|_{2 \rightarrow 2,\mu'}
+4 \sqrt{BDr}
.\]
This proves the upper bound in the first claim. The lower bound follows from Lemma~\ref{lemmaremovedsetdistancenotworse}.

To prove the second claim, we observe that
\begin{align*}
0&\le
\int_{\Omega} (D_{\cW_1}(x)-D_{\cW_2}(x))^2 \,d\mu(x) - \int_{\Omega} (D_{\cW_1}(x)-D_{\cW_2}(x))^2 \,d\mu'(x)
\\
&=\int_{\Omega} (D_{\cW_1}(x)-D_{\cW_2}(x))^2 (1-h(x)) \,d\mu(x)
\leq
\int_{\Omega}{4}D^2(1-h(x))\,d\mu(x)
\le {4}D^2r.
\end{align*}
Furthermore, since $\|U\|_\infty \le 2B$, we have that for any $x\in \Omega$,
\begin{align*}
&\biggl|\left(D_{\cW_1}(x)-D_{\cW_1'}(x)\right)-\left(D_{\cW_2}(x)-D_{\cW_2'}(x)\right)\biggr|
\\
&\qquad\qquad\qquad\qquad= \biggl|\int_{\Omega} (1-h(y)) U(x,y) \,d\mu(y)\biggr|\leq 2Br.
\end{align*}
Therefore,
\begin{align*}
&\left|\int_{\Omega}\left(D_{\cW_1}(x)-D_{\cW_2}(x)\right)^2 \,d\mu'(x)
 -\int_{\Omega}\left(D_{\cW_1'}(x)-D_{\cW_2'}(x)\right)^2 \,d\mu'(x) \right|
\\
&\quad
=\bigg|\int_{\Omega} \Bigl(\left(D_{\cW_1}(x)-D_{\cW_2}(x)\right)-\left(D_{\cW_1'}(x)-D_{\cW_2'}(x)\right)\Bigr)
\\
 &\qquad\qquad\Bigl(\left(D_{\cW_1}(x)-D_{\cW_2}(x)\right)+\left(D_{\cW_1'}(x)-D_{\cW_2'}(x)\right)\Bigr) \,d\mu'(x) \bigg|
\\
&\quad
\le 2Br\int_{\Omega}
\Bigl(\left|D_{\cW_1}(x)\right|+\left|D_{\cW_2}(x)\right|
   +\left|D_{\cW_1'}(x)\right|+\left|D_{\cW_2'}(x)\right|\Bigr) \,d\mu'(x)
\le
8BCr
.
\end{align*}
Combining these two inequalities proves the second claim.

To prove the third claim, we use the following bound, where the first
inequality was already established when proving the last claim of
Lemma~\ref{lemmaremovedsetdistancenotworse}:
\begin{align*}
\Bigl|\bigl|{\rho(\cW'_1)}-&{\rho(\cW'_2)}\bigr|-\bigl|{\rho(\cW_1)-\rho(\cW_2)}\bigr|\Bigr|\\
&\le \biggl|2\int_{\Omega} (1-h(x))(D_{\cW_1}-D_{\cW_2})(x) \,d\mu(x)\biggr.\\
&\qquad
-
\biggl.\int_{\Omega \times \Omega}(1-h(x))(W_1-W_2)(x,y)(1-h(y))\,d\mu(x)\,d\mu(y) \biggr|
\\
&=\biggl|2\int_{\Omega} (1-h(x))(S_1-S_2)(x) \,d\mu(x)
\biggr.
\\
&\qquad
+
\biggl.\int_{\Omega \times \Omega}(1-h(x))(W_1-W_2)(x,y)(2-1+h(y))\,d\mu(x)\,d\mu(y) \biggr|\\
&\le 2\int_\Omega (1-h(x))(D_{|\cW_1|}+D_{|\cW_2|})(x)\leq 4Dr. \qedhere
\end{align*}
\end{proof}

Our next lemma is an easy corollary to
Lemma~\ref{lemmaremovednotmuchsmaller}, and in turn immediately implies
Proposition~\ref{propboundedequivmetrics} and \ref{prop:signed-met-equiv}.

\begin{lemma}\label{lem:delGP-deltt}
Suppose $\cW_i=(W_i,S_i,I_i,\bOmega_i)$, for $i=1,2$, are signed graphexes
with $\|W_i\|_\infty\leq B$ and $\|D_{|\cW|}\|_\infty\leq D$, and let
$\delGP(\cW_1,\cW_2)\leq \eps$. Then $|\rho(W_1)-\rho(W_2)|\leq
\eps^3+4\eps^2D$. If, in addition, $\|W\|_i\leq C$, then
$\deltt(\cW_1,\cW_2)\leq f(\eps)$, where
\[
f(\eps)=\max
\left\{
\eps+4\eps\sqrt{BD},\left(\eps^2+2\eps\sqrt{D^2+2BC}\right)^{1/2},\left(\eps^3+4\eps^2D\right)^{1/3}
\right\}
\]
\end{lemma}

\begin{proof}
We first note that by Lemma~\ref{lem:signed-delGP<infty},
$\|\cW_i\|_1<\infty$, so even without the assumption that $\|W_i\|_1\leq C$,
we always have that $\|W_i\|_1\leq C$ for some $C<\infty$.

Next, let $\bOmega_i=(\Omega_i,\cF_i,\mu_i)$, for $i=1,2$, and let
$c>\varepsilon$. By the definition of $\delGP$, there exist measures
$\widetilde\mu_i$ such that $\widetilde\mu_i \le \mu_i$ and
$\delta_i=\mu_i(\Omega_i)-\widetilde \mu_i(\Omega_i)\leq c^2$ and such that
$\deltt(\widetilde \cW_1,\widetilde \cW_2)< c$ for the signed graphexes
$\widetilde \cW_i$ obtained from $\cW_i$ by replacing $\mu_i$ by
$\widetilde\mu_i$.

Consider an arbitrary $\sigma$-finite space $\bOmega_i''$ of infinite
measure, and two intervals $J_i$ of length $c-\delta_i$. Define $\bOmega_i'$
by appending $\bOmega_i''$ and the interval $J_i$ equipped with the Lebesgue
measure to $\bOmega_i$, and define $\widetilde\bOmega_i'$ by appending the
same spaces, except that we equip $J_i$ with the zero measure. By
Lemma~\ref{lem:extension} and the definition of $\deltt$, we then have that
$\tdel22(\widetilde\cW_1',\widetilde\cW_2')=\deltt(\widetilde
\cW_1,\widetilde \cW_2)< c$, where $\widetilde\cW_i'$ are the trivial
extensions to $\widetilde\bOmega_i'$. Furthermore, by our construction,
$\bOmega_i'=(\Omega_i,\cF_i',\mu_i')$ and
$\widetilde\bOmega_i'=(\widetilde\Omega_i,\widetilde\cF_i',\widetilde\mu_i')$
are such that $(\Omega_i',\cF_i')=(\widetilde\Omega_i,\widetilde\cF_i')$,
$\widetilde\mu_i' \le \mu_i'$, and $\mu_i'(\Omega_i')-\widetilde
\mu_i'(\Omega_i')= c^2$.

Given a coupling $\widetilde\mu'$ of $\widetilde\mu_1'$ and
$\widetilde\mu_2'$, let $\widetilde\cU_i'$ be the pullback of
$\widetilde\cW_i'$ under the coordinate projections onto $\Omega_i'$. By the
definition of the distance $\tdel22$, we can find a coupling $\widetilde\mu'$
such that $\d22(\widetilde\cU_1',\widetilde\cU_2')\leq c^2$. Choose a
coupling $\mu'$ of $\mu_1'$ and $\mu_2'$ by coupling $\mu_i'-\widetilde
\mu_i'$ arbitrarily. Then $\mu'-c\leq\widetilde\mu'\leq\mu'$. Defining
$\cU_i'$ to be the pullbacks of $\cW_i'$ under the coordinate projections
onto $\Omega_i'$, we may then apply Lemma~\ref{lemmaremovednotmuchsmaller}
with $r=c^2$ to conclude that
$|\rho(\cW_1)-\rho(\cW_2)|=|\rho(\cU_1')-\rho(\cU_2')|\leq c^3+4c^2D$ and
\begin{align*}
\deltt(\cW_1,\cW_2)&\leq \d22(\cU_1',\cU_2')
\\
&\leq\max
\left\{
c+4c\sqrt{BD},\left(c^2+2c\sqrt{D^2+2BC}\right)^{1/2},\left(c^3+4c^2D\right)^{1/3}
\right\}
\end{align*}
Since $c>\eps$ was arbitrary, this concludes the proof.
\end{proof}

\begin{proof}[Proof of Propositions~\ref{propboundedequivmetrics} and \ref{prop:signed-met-equiv}]
By the definition of $\delGP$, we clearly have that $\delGP\leq\deltt$. For
signed graphexes that are $(B,C,D)$-bounded, a bound in the opposite
direction follows immediately from Lemma~\ref{lem:delGP-deltt}, proving
Proposition~\ref{prop:signed-met-equiv}. To prove
Proposition~\ref{propboundedequivmetrics} we note that if
$\delGP(\cW_n,\cW)\to 0$ and both $\cW_n$ and $\cW$ are (unsigned) graphexes
with $D$-bounded marginals, then $\|\cW\|_1\leq C$ for some $C<\infty$ by
Proposition~\ref{prop:local-finite}. Since $\|\cW_n\|_1=\rho(\cW_n)$
converges to $\rho(\cW)=\|W\|_1\leq C$ by the first statement of the lemma,
we must have that $\|\cW_n\|_1\leq \widetilde C$ for some $\widetilde
C<\infty$, at which point the proof proceeds as the proof for the signed
case.
\end{proof}

Our next lemma relates the kernel norm $\|\cdot\|_{2\rightarrow2}$ of a two
variable function $U$ to the $4$-cycle counts of $U$.

\begin{lemma} \label{lem:2t2C4equiv}
Let $U\colon\Omega\times\Omega\to\RR$ be a measurable function. Then
\[\|U\|_{2 \rightarrow 2}^4 \le t(C_4,U) \le \|U\|_{2 \rightarrow 2}^2 \|U\|_2^2
.\]
\end{lemma}

\begin{proof}
For any $f,g$ with $\|f\|_2=\|g\|_2=1$, we have (using Cauchy's inequality)
\begin{align*}
\|f \circ U \circ g\|_2^4 &\le \|U \circ g\|_2^4
=(g \circ U \circ U \circ g)^2\\
&=\left(\int_{\Omega^2}g(x)U \circ U(x,y)g(y)\,d\mu(x)\,d\mu(y)\right)^2\\
&\le \left(\int_{\Omega^2}g(x)^2g(y)^2\,d\mu(x)\,d\mu(y)\right) \left( \int_{\Omega^2} (U \circ U(x,y))^2 \,d\mu(x)\,d\mu(y)\right)\\
&=t(C_4,U).
\end{align*}
This proves the first inequality. For the second, we have
\begin{align*}
t(C_4,U)&= \int_{\Omega^4} U(x,y)U(y,z)U(z,w)U(w,x)\,d\mu(x)\,d\mu(y)\,d\mu(z)\,d\mu(w)\\
& =
\int_{\Omega} \int_{\Omega}(U\circ U_z)(x)^2 \,d\mu(x)\,d\mu(z)
= \int_{\Omega} d\mu(z)\|U \circ U_z\|_2^2
\\
&\le \int_\Omega \|U\|_{2 \rightarrow 2}^2 \|U_z\|_2^2=\|U\|_{2 \rightarrow 2}^2\|U\|_2^2. \qedhere
\end{align*}
\end{proof}

Our next goal is to relate the jumble and the kernel distances. The next
proposition shows that for $(B,C,D)$-bounded graphexes, they are equivalent.
As we will see, similarly to the proof of
Propositions~\ref{propboundedequivmetrics} and \ref{prop:signed-met-equiv},
the proof also gives that for (unsigned) graphexes, the metrics $\deltt$ and
$\deljbl$ are equivalent on the space of graphexes with $D$-bounded
marginals; see Remark~\ref{rem:kernel-jumbl-equiv} below.

\begin{proposition}\label{prop:kernel-jumbl-equiv}
Given $B,C,D<\infty$, there exists a constant $c<\infty$ such that if $\cW_1$
and $\cW_2$ are $(B,C,D)$-bounded signed graphexes, then the following hold.
\begin{enumerate}
\item If $\d22(\cW_1,\cW_2)\leq \widetilde\eps$, then $\djbl(\cW_1,\cW_2)\leq \max\{\widetilde\eps,{\widetilde\eps}^3\}$.
\item If $\djbl(\cW_1,\cW_2)\leq \varepsilon$, then
    $\d22(\cW_1,\cW_2)\leq\max\{\sqrt[3]\eps,c\sqrt[4]{\eps}\}$.
\end{enumerate}
If the graphexes are such that $\rho(\cW_1)=\rho(\cW_2)$, then these bounds
can be replaced by $\djbl(\cW_1,\cW_2)\leq \widetilde\eps$ and
$\d22(\cW_1,\cW_2)\leq c\sqrt[4]{\eps}$.
\end{proposition}

To prove the proposition, we establish three preliminary lemmas.

\begin{lemma}\label{lem:sets2equiv-nonneg}
Given a bounded, nonnegative, measurable function $f$ on some measure space
$(\Omega,\cF,\mu)$, we have
\[\frac{\|f\|_2^2}{\sqrt{\|f\|_1\|f\|_\infty}}\le \sup_{S \subseteq \Omega} \frac{1}{\sqrt{\mu(S)}} \int_{S}f \,d\mu \le \|f\|_2.
\]
In particular, the second and third term define equivalent norms, for any
$C,D$, on the space of nonnegative functions with $\|f\|_1 \le C,\|f\|_\infty
\le D$.
\end{lemma}

\begin{proof}
The second inequality follows from Cauchy's inequality. For the first one,
let $\|f\|_\infty=K$. First, note that
\[\int_0^K d c\int_{\{f \ge c\}}f \,d \mu = \int_{\Omega} d\mu(x)f(x)\int_{0}^{ f(x)} dc =\int_\Omega f^2 \,d \mu.
\]
We then have
\begin{align*}
\|f\|^4_2&=
\left( \int_\Omega f^2 \,d\mu \right)^2
=\left( \int_0^K \int_{\{f \ge c\}}f\,d \mu \,dc \right)^2
\\ &
=\left( \int_0^K dc \sqrt{\mu(\{f\ge c\})}
\frac{\int_{\{f \ge c\}}f\,d \mu }{\sqrt{\mu(\{f\ge c\})}} \right)^2
\\
&\leq
\int_0^K \frac{\left(\int_{\{f \ge c\}}f \,d \mu \right)^2}{\mu(\{f \ge c\})} dc \int_0^K\mu(\{f \ge c\})\,dc
.
\end{align*}
Using the fact that
\[ \int_0^K\mu(\{f \ge c\})\,dc =\int_\Omega f \,d \mu,
\]
we have that
\[\int_0^K \frac{\left(\int_{\{f \ge c\}}f \,d \mu \right)^2}{\mu(\{f \ge c\})} \,dc \ge
\frac{\|f\|_2^4}{\|f\|_1},
\]
which means that there exists some $c$ such that
\[\frac{\left(\int_{\{f \ge c\}}f \,d \mu \right)^2}{\mu(\{f \ge c\})} \ge \frac{\|f\|_2^4}{\|f\|_1K}=\frac{\|f\|_2^4}{\|f\|_1\|f\|_\infty}.
\]
Taking $S$ to be $\{f \ge c\}$, the lemma is proved.
\end{proof}

The following lemma is an easy corollary of Lemma~\ref{lem:sets2equiv-nonneg}.

\begin{lemma} \label{lem:sets2equiv}
Given a bounded, measurable, not necessarily nonnegative function $f$ on some
measure space $(\Omega,\mu)$, we have
\[\frac{\|f\|_2^2}{\sqrt{2\|f\|_1\|f\|_\infty}}\le \sup_{S \subseteq \Omega} \frac{1}{\sqrt{\mu(S)}} \left|\int_{S}f \,d\mu \right| \le \|f\|_2.
\]
\end{lemma}

\begin{proof}
The second inequality again follows from Cauchy's inequality, so we just have
to prove the first one. Note that the left term is not affected by replacing
$f$ with $|f|$. Let $S$ be any subset of $\Omega$, let $S^+$ consist of the
points in $S$ where $f$ is nonnegative, and let $S^-$ be the rest. Then
\begin{align*}
\frac{1}{\sqrt{\mu(S)}} &\int_{S}|f| \,d\mu
=\frac{\left|\int_{S^+} f \,d\mu \right|}{\sqrt{\mu(S^+)}} \frac{\sqrt{\mu(S^+)}}{\sqrt{\mu(S)}}
 + \frac{\left|\int_{S^-} f \,d\mu \right|}{\sqrt{\mu(S^-)}} \frac{\sqrt{\mu(S^-)}}{\sqrt{\mu(S)}} \\
&\le \max\left(\frac{\left|\int_{S^+} f \,d\mu \right|}{\sqrt{\mu(S^+)}} , \frac{\left|\int_{S^-} f \,d\mu \right|}{\sqrt{\mu(S^-)}}\right)\left(\frac{\sqrt{\mu(S^+)}}{\sqrt{\mu(S)}}+ \frac{\sqrt{\mu(S^-)}}{\sqrt{\mu(S)}}\right)\\
&\le \max\left(\frac{\left|\int_{S^+} f \,d\mu \right|}{\sqrt{\mu(S^+)}} , \frac{\left|\int_{S^-} f \,d\mu \right|}{\sqrt{\mu(S^-)}}\right) \sqrt{2}
\leq\sqrt 2 \sup_{S \subseteq \Omega}
\frac{1}{\sqrt{\mu(S)}} \left|\int_{S}f \,d\mu \right|
.
\end{align*}
Therefore, we have
\[ \sup_{S \subseteq \Omega} \frac{1}{\sqrt{\mu(S)}} \left|\int_{S}f \,d\mu \right|
\ge \sup_{S \subseteq \Omega} \frac{1}{\sqrt{2}} \frac{1}{\sqrt{\mu(S)}} \int_{S}|f| \,d\mu \ge \frac{\|f\|_2^2}{\sqrt{2\|f\|_1\|f\|_\infty}}.
\]
This completes the proof.
\end{proof}

\begin{lemma} \label{lemmasquare2to2equiv}
For any measurable $U\colon\Omega \times \Omega \rightarrow \RR$,
\begin{align*}
\|U\|_{2 \rightarrow 2}
&\ge \|U\|_{\jbl}
\ge \frac{\|U\|_{2 \rightarrow 2}^4}{8\|U\|_\infty^{3/4}\|D_{|U|}\|_\infty^{3/4}\|D_{|U|}\|_2^{3/2}}
\ge \frac{\|U\|_{2 \rightarrow2}^4}{8\|U\|_\infty^{3/4}\|D_{|U|}\|_\infty^{3/2}\|U\|_1^{3/4}}.
\end{align*}
\end{lemma}

\begin{proof}
Fix $f$ and $g$ with $\|f\|_2=\|g\|_2=1$ and recall that we use $U_x$ to
denote the function $y\mapsto U(x,y)$. First, we have
\[\|U \circ g\|_\infty \le \sup_x \|U_x\|_2\le \sqrt{\|U\|_\infty
\|D_{|U|}\|_\infty}.
\]

We also have
\begin{align*}
\|U \circ g\|_1 &\le 2 \sup_{S \subseteq \Omega} \left| \int_{S}dx \int_\Omega U(x,y) g(y)\,dy \right|
=2 \sup_{S \subseteq \Omega} \left| \int_\Omega dy \,g(y) \int_S U(x,y)\,dx \right|
\\
&\le 2 \int_\Omega |g(y)| D_{|U|}(y) \le 2\|D_{|U|}\|_2.
\end{align*}

Combined with the first bound from Lemma~\ref{lem:sets2equiv} and the fact
that $ |f\circ U\circ g|\leq \|U\circ g\|_2$, this shows that
\begin{align*}
\sup_{S\subseteq\Omega}\frac{1}{\sqrt{\mu(S)}}
\left|\int_{S} U \circ g(x) \,d\mu(x)\right|
&\ge \frac{\|U \circ g\|^2_2}{\sqrt{2\|U \circ g\|_\infty\|U \circ g\|_1}}\\
&\ge \frac{(f \circ U \circ g)^2}{2\|U\|_\infty^{1/4}\|D_{|U|}\|_\infty^{1/4}\|D_{|U|}\|_2^{1/2} }.
\end{align*}

Analogously, defining $g_S$ as the function $x\mapsto\frac
1{\sqrt{\mu(S)}}1_{x\in S}$, and observing that $\frac
1{\sqrt{\mu(S)}}\left|\int_S (U\circ g)\,d\mu\right|=\left|g\circ U\circ g_S
\right|\leq \|U\circ g_S\|_2$, we have
\begin{align*}
&\sup_{T \subseteq \Omega}
\left|\frac{1}{\sqrt{\mu(T)\mu(S)}}\int_{S\times T}\,d\mu(x) d\mu(y)U(x,y)\right|\\
&\qquad=\sup_{T \subseteq \Omega}
\frac{1}{\sqrt{\mu(T)}}\left|\int_T(U\circ g_S)(y) \,d\mu(y)\right|
\\
&\qquad\ge \frac{\| U \circ g_S\|_2^2}{2 \|U\|_\infty^{1/4} \|D_{|U|}\|_\infty^{1/4}\|D_{|U|}\|_2^{1/2} }\\
&\qquad\ge \frac{(f \circ U \circ g)^4}{8\|U\|_\infty^{3/4}\|D_{|U|}\|_\infty^{3/4}\|D_{|U|}\|_2^{3/2}}.
\end{align*}
Therefore,
\[\|U\|_{\jbl} \ge \frac{\|U\|_{2 \rightarrow2}^4}{8\|U\|_\infty^{3/4}\|D_{|U|}\|_\infty^{3/4}\|D_{|U|}\|_2^{3/2}}
\geq \frac{\|U\|_{2 \rightarrow2}^4}{8\|U\|_\infty^{3/4}\|D_{|U|}\|_\infty^{3/2}\|U\|_1^{3/4}},
\]
where in the last step we used that $\|D_{|U|}\|_2^2 \leq \|D_{|U|}\|_1\|D_{|U|}\|_\infty=\|U\|_1\|D_{|U|}\|_\infty$.
\end{proof}

\begin{proof}[Proof of Proposition~\ref{prop:kernel-jumbl-equiv}]
The proposition follows immediately from Lemmas~\ref{lem:sets2equiv} and
\ref{lemmasquare2to2equiv} and the definition of the distances $\d22$ and
$\djbl$, with $c$ being the constant
$c=\max\{\sqrt[8]{8CD},\sqrt[4]{64B^{3/4}D^{3/2}C^{3/4}}\}$.
\end{proof}

\begin{remark}\label{rem:kernel-jumbl-equiv}
The above proof can easily be modified to see that for any $0<D<\infty$ the
metrics $\deltt$ and $\deljbl$ are equivalent on the space of (unsigned)
graphexes with $D$-bounded marginals. Indeed, if $\cW$ has bounded marginals,
it is integrable, and if either $\deltt(\cW_n,\cW)\to 0$ or
$\deljbl(\cW_n,\cW)\to 0$, then
$\|\cW_n\|_1=\rho(\cW_n)\to\rho(\cW)=\|\cW\|_1$. This shows that we can
assume that the sequences are $(C,D)$-bounded for some $C$, which means they
are $(B,C,D)$-bounded for $B=1$.
\end{remark}

We close this section with the (straightforward) proof that the homomorphism
densities $t(F,\cW)$ indeed describe the expected number of injective
homomorphisms from $F$ into $G_T(\cW)$.

\begin{proposition}\label{prop:t(F,W)}
For any simple graph $F$ with no isolated vertices and graphex $\cW$,
\[
\EE\left[\inj(F,G_T(\cW))\right]=T^{|V(F)|} t(F,\cW).
\]
If the marginals of $\cW$ are bounded, then the right side is finite, with
\[
t(F,\cW)\leq \prod_i\|\cW\|_1\|D_{\cW}\|_\infty^{v(F_i)-2},
\]
where the product runs over the components of $F$ and $v_i$ is the number
of vertices in $F_i$.
\end{proposition}

\begin{proof}
Recall that we extended the feature space $\Omega$ to include an additional
point $\infty$, and that we labeled the vertices corresponding to the leaves
of a star generated by $S$, as well as the two endpoints of the isolated
edges coming from $I$, by $\infty$. Let $k=|V(F)|$. First, suppose that
$\Omega$ has finite measure. Then the probability that there are $n$ points
sampled from $\Omega$ is $e^{-\mu(\Omega)T}\frac{(T\mu(\Omega))^n}{n!}$. Let
$V_2$ be the set of vertices of $F$ of degree at least $2$, let $V_1$ be the
set of vertices of degree $1$ whose neighbor is in $V_2$, and let $V_0$ be
the set of vertices that belong to an isolated edge. Each vertex in $V_2$
must be mapped to a vertex with feature label in $\Omega$. Each vertex in
$V_1$ must be mapped either to a vertex with feature label in $\Omega$ or a
vertex coming from the leaves of a the star attached to such a vertex (in
which case its feature label is $\infty$). For an isolated edge, there are
three possibilities: either it is mapped to two vertices with feature label
in $\Omega$, one endpoint is mapped to such a vertex and the other to a leaf
of a star whose center is the first vertex, or it is mapped to an isolated
edge generated by $I$ (in which case both feature labels are $\infty$). Let
us fix for each vertex in $V_1$ and $V_0$ whether its feature label is
$\infty$ or lies in $\Omega$, noting that this uniquely determines all the
choices we just discussed. Let $V_0' \subseteq V_0$ and $V_1' \subseteq V_1$
be the sets of vertices mapped to a vertex coming from $\Omega$, and let
$V'=V_0'\cup V_1' \cup V_2$. Let $U$ be the set of remaining vertices. Let
$J$ be the set of isolated edges, and $J''$ the set of isolated edges where
we have fixed that they are mapped to an edge generated by $I$; i.e., they
have both endpoints in $U$. For a vertex $i \in V'$, let $d_U(i)$ be its
degree to $U$. Conditioned on $V'$, we have
\begin{align*}
\EE[\inj(F,G_T)|V']
&=
\sum_{n=0}^\infty e^{-\mu(\Omega)T}\frac{(T\mu(\Omega))^n}{n!} \frac{(n)_{|V'|}}{\mu(\Omega)^{|V'|}} \\
& \qquad \qquad \int_{\Omega^{V'}}\prod_{\{i,j\} \in E(F|_{V'})} W(x_i,x_j) \prod_{i \in V'} (TS(x_i))^{d_U(i)} (2T^2I)^{|J''|}\\
&=
\sum_{n=|V'|}^\infty T^{|V'|} e^{-\mu(\Omega)T}\frac{(T\mu(\Omega))^{n-|V'|}}{(n-|V'|)!} \\
& \qquad \qquad \int_{\Omega^{V'}}T^{|U|}\prod_{\{i,j\} \in E(F|_{V'})} W(x_i,x_j) \prod_{i \in V'} (S(x_i))^{d_U(i)} (2I)^{|J''|}\\
&=
T^{|V|} \left(\sum_{n=|V'|}^\infty e^{-\mu(\Omega)T} \frac{(T\mu(\Omega))^{n-|V'|}}{(n-|V'|)!} \right)\\
& \qquad \qquad \left(\int_{\Omega^{V'}}\prod_{\{i,j\} \in E(F|_{V'})} W(x_i,x_j) \prod_{i \in V'} (S(x_i))^{d_U(i)} (2I)^{|J''|}\right)\\
&=
T^{|V|}\int_{\Omega^{V'}}\prod_{\{i,j\} \in E(F|_{V'})} W(x_i,x_j) \prod_{i \in V'} (S(x_i))^{d_U(i)} (2I)^{|J''|}
.
\end{align*}
Therefore,
\[
\EE[\inj(F,G_T)]=T^{|V|} \sum_{\substack{V_0' \subseteq V_0\\V_1' \subseteq
V_1}}\int_{\Omega^{V'}}\prod_{\{i,j\} \in E(F|_{V'})} W(x_i,x_j) \prod_{i \in
V'} (S(x_i))^{d_U(i)} (2I)^{|J''|} .\] Now, it is not difficult to check that
this is multiplicative over connected components of $F$. Indeed, each term
with fixed $V_0',V_1'$ is multiplicative, and the choice of which vertices to
put in $V_0',V_1'$ from each of the components is independent. Therefore, we
may assume that $F$ is connected.

If $F$ consists of a single edge $\{i,j\}$, then the above expression gives
\[\EE[\inj(F,G_T)]=T^2\left(\int_{\Omega^2}W(x,y)\,dx\,dy + 2 \int_\Omega S(x) \,dx +
2I\right)=T^2t(F,\cW),
\]
as required. Otherwise, $F$ has no isolated edges, so $V_0$ is empty (and so
is $J''$). We then have
\begin{align*}
\EE[\inj(F,G_T)]&=T^{|V|}\sum_{V_1' \subseteq V_1} \int_{\Omega^{V'}} \prod_{\{i,j\} \in E(F|_{V'})}W(x_i,x_j) \prod_{i \in V_2} (S(x_i))^{d_U(i)}\\
&=T^{|V|}\int_{\Omega^{V_2}} \prod_{\{i,j\} \in E(F_{V_2})} W(x_i,x_j) \cdot \phantom{}\\
& \qquad \qquad\qquad \prod_{i \in V_2} \left( \sum_{T_i \subseteq N_{V_1}(i)} \int_{\Omega^{T_i}}\prod_{j \in T_i} W(x_i,x_j) S(x_i)^{|N_{V_1}(i)|-|T_i|} \right)\\
&= T^{|V|}\int_{\Omega^{V_2}} \prod_{\{i,j\} \in E(F_{V_2})} W(x_i,x_j) \cdot \phantom{}\\
& \qquad \qquad \qquad \prod_{i \in V_2} \left( \int_\Omega W(x_i,x_j) dx_j + S(x_i) \right)^{|N_{V_1}(i)|}\\
&=T^{|V|}t(F,\cW).
\end{align*}
This completes the proof of the first statement if $\Omega$ has finite
measure. The general $\sigma$-finite case follows by monotone convergence,
with both sides being possibly infinite in the limit.

To prove the second statement we consider the components of $F$ separately.
Furthermore, given a component $F'$ of $F$, we use the fact that
$\|W\|_\infty\leq 1$ to delete edges from $F'$ until $F'$ becomes a tree. At
this point, we can remove the leaves of the tree at the cost of a factor $D$
for each leaf, getting a new tree with less edges. We continue until we are
left with a single edge, at which point we bound the remaining integral by
$\|\cW\|_1$.
\end{proof}

\section{Tightness}
\label{sec:tight}

The goal of this section is to establish various equivalent notions of
tightness, and to then use tightness to relate convergence in the kernel and
the weak kernel metric. In particular, we will relate the convergence of a
sequence of graphexes in the weak kernel metric $\delGP$ to convergence of a
``regularized'' sequence in the kernel metric $\deltt$, where the regularized
sequence is obtained from the original one by discarding the part of the
space which has large marginals; see
Proposition~\ref{propgeneralconvergenceequiv} below. In contrast to the last
section, in this section we restrict ourselves to unsigned graphexes since we
believe that the obvious generalization of the notion of tightness to signed
graphexes will not be the right notion of tightness for the metric $\delGP$;
see Remark~\ref{rem:signed-tightness} at the end of this section.

We start by establishing the equivalence of various formulations of tightness.

\begin{theorem} \label{theoremtightequiv}
Given a set of graphexes $\cS$, the following are equivalent:
\begin{enumerate}
\item $\cS$ is tight. In other words, for every $\varepsilon>0$, there
    exist $C$ and $D$ such that for every graphex $\cW \in \cS$,
 $\cW=(W,S,I,\bOmega)$ with $\bOmega=(\Omega,\cF,\mu)$, there exists
$\Omega_\varepsilon \subseteq \Omega$ such that $\mu(\Omega_\varepsilon)
\le \varepsilon$ and $\cW'=\cW|_{\Omega \setminus \Omega_\varepsilon}$ has
$\|\cW'\|_1 \le C$, $\|D_{\cW'}\|_\infty \le D$. \label{tightCD}
\item For every $\varepsilon>0$, there exist $C$ such that for every
    graphex $\cW \in \cS,\cW=(W,S,I,\bOmega)$, there exists
 $\Omega_\varepsilon \subseteq \Omega$ such that $\mu(\Omega_\varepsilon)
 \le \varepsilon$ and $\cW'=\cW|_{\Omega \setminus \Omega_\varepsilon}$ has
 $\|\cW'\|_1 \le C$. \label{tightC}
\item For every $\varepsilon$, there is a $D$ and $C$ such that for any
    $\cW\in \cS$, taking $\Omega_{\le D}$ to be the set of points with
 $D_\cW(x)\le D$, $\mu(\Omega\setminus\Omega_{\leq D}) \le \varepsilon$,
 and $\|\cW|_{\Omega_{\leq D}}\|_1 \le C$. \label{tightD}
\item For every $T>0$, the set of random unlabeled finite graphs $G_T(\cW)$
    with $\cW \in \cS$ is tight. \label{tightallsamples}
\item There exists $T>0$ such that the set of random unlabeled finite
    graphs $G_T(\cW)$ with $\cW \in \cS$ is tight. \label{tightsomesamples}
\end{enumerate}
\end{theorem}

\begin{corollary}\label{cor:C-bounded-tight}
Let $\cS$ be a set of graphexes.
\begin{enumerate}
\item If there exists a $C<\infty$ such that $\|\cW\|_1\leq C$ for all
    $\cW\in\cS$, then $\cS$ is tight.

\item If $\cS$ is tight and has uniformly bounded marginals, then there
    exist $C,D<\infty$ such that $\cS$ is $(C,D)$-bounded.
\end{enumerate}
\end{corollary}
\begin{proof}
(1) Taking $\Omega_\varepsilon=\emptyset$ for any $\varepsilon$, the set
$\cS$ clearly satisfies condition (\ref{tightC}) from the theorem.

(2) Choose $\varepsilon$ arbitrarily, say $\varepsilon=1$, and let $C'$, $D'$
be such that (3) from Theorem~\ref{theoremtightequiv} holds. Furthermore, let
$D$ be such that the marginals of all graphexes in $\cS$ are bounded by $D$.
Then
\begin{align*}
\|\cW\|_1&\leq\|\cW|_{\Omega_{\leq D'}}\|_1+2\int W(x,y)1_{D_\cW(x) >D'}\,d\mu(x) d\mu(y)+{2}\int_{D_\cW >D'}S(x)\,d\mu(x)
\\
&\leq C'+2\int_{D_{\cW}>D'} D_{\cW(x)}\,d\mu(x)
\leq C'+2 D\varepsilon=C'+2D=:C,
\end{align*}
proving the claim.
\end{proof}

In order to prove Theorem~\ref{theoremtightequiv}, we will use the following
lemma.

\begin{lemma} \label{lem:edgebound}
The probability that $G_T(\cW)$ has more than $KT^2\|\cW\|_1$ edges is at
most
\[\frac{T^2\|\cW\|_1/2+T^3\|D_\cW\|_2^2}{(K-1/2)^2T^4\|\cW\|_1^2},
\]
and the probability that it has less than $T^2\|\cW\|_1/4$ edges is at most
\[\frac{16(T^2\|\cW\|_1/2+T^3\|D_\cW\|_2^2)}{T^4\|\cW\|_1^2}
\]
\end{lemma}

\begin{proof}
Let $X_T$ be the number of edges of $G_T(\cW)$. By
Proposition~\ref{prop:t(F,W)}, $X_T$ has expectation $T^2 \|\cW\|_1/2$. To
calculate the variance, note that we have
\[X_T^2=\frac{\inj(F_1,G_T)}{2}+\frac{\inj(F_2,G_T)}{4}+\inj(F_3,G_T),\]
where $F_1$ consists of a single edge, $F_2$ consists of a pair of disjoint
edges, and $F_3$ consists of two edges joined at one vertex. Therefore, we
can again use Proposition~\ref{prop:t(F,W)} to conclude that
\begin{align*}
\text{Var}(X_T)&=\EE[X_T^2]-\EE[X_T]^2\\
&=\frac{T^2\|\cW\|_1}{2}+\frac{T^4
\|\cW_1\|^2}{4}+T^3\|D_\cW\|_2^2-\left(\frac{T^2\|\cW\|_1}{2}\right)^2\\
&=\frac{T^2\|\cW\|_1}{2}+T^3\|D_\cW\|_2^2 .
\end{align*}
The bounds on the probabilities of having too many or too few edges
follow from Chebyshev's inequality.
\end{proof}

\begin{proof}[Proof of Theorem \ref{theoremtightequiv}]
$(\ref{tightCD}) \Rightarrow (\ref{tightC})$ is obvious.

$(\ref{tightC}) \Rightarrow (\ref{tightCD})$: Suppose $\cS$ satisfies $(2)$,
and let $\varepsilon>0$. Take $C$ from property $(2)$ for $\varepsilon/2$,
and take $D=2C/\varepsilon$. For each $\cW \in \cS$ with underlying space
$\Omega$, there is a set $\Omega' \subseteq \Omega$ with $\mu(\Omega
\setminus \Omega') \le \varepsilon/2$ so that the restriction
$\cW'=\cW|_{\Omega}$ has $\|\cW'\|_1 \le C$. Suppose $\mu(x \in
\Omega':D_{\cW'}(x)>D)
>\varepsilon/2$. Then we would have $\|W'\|_1 >D \varepsilon/2 =C$, a
contradiction. Therefore, removing the set of points with $D_{\cW'}(x)>D$, we
have removed points with total measure at most $\varepsilon$, and the
restricted graphex is $(C,D)$-bounded.

$(\ref{tightD}) \Rightarrow (\ref{tightCD})$ is obvious.

$(\ref{tightCD}) \Rightarrow (\ref{tightallsamples})$: Fix $T>0$ and
$\varepsilon>0$. Take $\varepsilon'$ such that
$e^{-T\varepsilon'}>1-\varepsilon/2$, and take $C,D$ for $\cS$ from the
definition of tightness. Given $\cW \in \cS$, there exists
$\Omega_{\varepsilon'} \subseteq \Omega$ such that
$\mu(\Omega_{\varepsilon'}) \le \varepsilon'$ and $\cW'=\cW|_{\Omega
\setminus \Omega_{\varepsilon'}}$ has $\|\cW'\|_1 \le C$,
$\|D_{\cW'}\|_\infty \le D$. The probability that $G_T(\cW)$ samples a point
in $\Omega \setminus {\Omega_{\varepsilon'}}$ during the Poisson process is
at most $1-e^{-T\varepsilon'}<\varepsilon/2$. Conditioned on this not
happening, the sample is equivalent to a sample from $G_T(\cW')$. For this,
we have that $\|D_{\cW'}\|_2^2 \le CD$; therefore we can take $K=K(C,D)$
large enough so that the probability that there are more than $KT^2C$ edges
in $G_T(\cW')$ is at most $\varepsilon/2$ (independently of $\cW'$).
Therefore, the probability that there are more than $KT^2C$ edges in
$G_T(\cW)$ is at most $\varepsilon$.

$(\ref{tightallsamples}) \Rightarrow (\ref{tightsomesamples})$ is obvious.

$(\ref{tightsomesamples}) \Rightarrow (\ref{tightD})$: Let $\varepsilon>0$.
First, we show that there exists a $D$ so that for every $\cW \in \cS$, the
measure of the set $\{D_\cW>D\}$ is at most $\eps$. Suppose not. We will show
that this implies that for each $M$ we can find a $\cW\in \cS$ such with
probability at least $\frac 12(1-e^{-\varepsilon T/2})$, the number of edges
in $G_T(\cW)$ is at least $M$. This contradicts the assumption that the set
of random graphs $G_T(\cW)$ is tight.

Assume thus that for every $D$, there exists a $\cW=\cW(D)\in \cS$ such that
the set $\{D_\cW> D\}$ has measure larger than $\eps$. Take $G_T(\cW)$ and
randomly color the vertices red and blue. With probability at least
$1-e^{-\varepsilon T/2}$, there exists at least one blue point whose feature
label falls into the set $\{D_\cW > D\}$. Conditioned on this, taking a blue
point with feature label $x\in \{D_\cW > D\}$, the number of red neighbors it
has is a Poisson random variable with mean $TD_\cW(x)/2$. Given $M<\infty$,
choose $D=D(M,T)$ in such a way that a Poisson random variable with mean at
least $TD/2$ has probability at least $1/2$ of being greater than $M$. As a
consequence, given $T$ and an arbitrary large $M$ and we can find a $D$ and
$\cW=\cW(D)\in \cS$ such that with probability at least $\frac 12
(1-e^{-\varepsilon T/2})$, the number of edges in $G_T(\cW)$ is at least $M$,
contradicting tightness.

We claim that $\|\cW|_{\Omega_{\leq D}}\|_1$ can't be arbitrarily large. Set
$\cW'=\cW|_{\Omega_{\leq D}}$ and assume that $\|\cW'\|_1=C$. Then the
probability that $G_T(\cW')$ has less than $T^2C/4$ edges is at most
\[\frac{8+16TD}{T^2C}
.\] If $C$ is large enough, this is less than $1/2$. But then for large
enough $C$ with probability at least $1/2$, the number of edges is at least
$T^2C/4$, contradicting the assumption of tightness. This means that $C$
can't be arbitrarily large.

This completes the proof of the theorem.
\end{proof}

\begin{remark}\label{rem:DC-mon}
It will sometimes be useful to transform a graphex $\cW$ over an arbitrary
$\sigma$-finite space $\bOmega=(\Omega,\cF,\mu)$ into a graphex over an
atomless space by mapping $\bOmega$ to the product space $\Omega\times [0,1]$
equipped with the measure $\mu\times\lambda$, with $\lambda$ denoting the
Lebesgue measure, and mapping $\cW$ to $\Phi(\cW)=\cW^\phi$, with
$\phi\colon\Omega\times[0,1]\to\Omega$ denoting the coordinate projection
onto $\Omega$. It is easy to see that $\delGP(\cW,\Phi(\cW))=0$, which
together with the triangle inequality implies that the map $\Phi$ does not
change distances between graphons. It is also easy to check that if $\cS$ is
a tight set of graphexes, then the set of graphexes obtained by mapping each
graphex $\cW\in\cS$ to the corresponding atomless graphex $\Phi(\cW)$ is
tight as well.
\end{remark}

Let us analyze when graphexes converge under $\delGP$. To this end, we first
prove a few lemmas.

\begin{lemma}
Given $C,D,M\in(0,\infty)$, there exists a function $f\colon[0,\infty)^2\to
[0,\infty)$ such that $f(x)\to 0$ as $x\to 0$ and such that the following
holds
\begin{enumerate}
\item Let $\cW$ be a graphex over $(\Omega,\cF,\mu)$, and let
    $\widetilde\mu$ be a second measure over $(\Omega,\cF)$ such that
 $\mu-r\leq\widetilde\mu\leq\mu$. If $\widetilde\cW$ is obtained from $\cW$
 by replacing $\mu$ with $\widetilde\mu$ then
\[
\mu(D_\cW>D+r)-r\leq \widetilde\mu(D_{\widetilde\cW}>D)\leq \mu(D_\cW>D)
\]
for all $D>0$.

\item Let $\cW_1$ and $\cW_2$ be graphexes with bounded marginals, defined
    over the same space $(\Omega,\cF,\mu)$. Suppose that
 $\d22(\cW_1,\cW_2)<\varepsilon$. Then
\[\mu(\{|D_{\cW_1}-D_{\cW_2}|\ge \varepsilon\})<\varepsilon^{2}
.\]

\item Let $\widetilde\cW_1$ and $\widetilde\cW_2$ be graphexes with bounded
    marginals, defined over $(\Omega_1,\cF_1,\widetilde\mu_1)$ and
 $(\Omega_2,\cF_2,\widetilde\mu_2)$. If
 $\deltt(\widetilde\cW_1,\widetilde\cW_2)< \varepsilon$ and
 $D>\varepsilon$, then
\[
\widetilde\mu_1(\{D_{\widetilde\cW_1}>D+\varepsilon\})-\varepsilon^2
\leq\widetilde\mu_2(\{D_{\widetilde\cW_2}>D\})
\leq\widetilde\mu_1(\{D_{\widetilde\cW_1}>D-\varepsilon\})+\varepsilon^2.
\]

\item For $i=1,2$, let $\cW_i$ be graphexes defined over
    $(\Omega_i,\cF_i,\mu_i)$, and let $\cW_{i,\leq D}$ be the restriction
  of $\cW_i$ to the subset $\{D_{\cW_i}\leq D\}$ of $\Omega_i$. Assume that
  $\eps+\eps^2<D$, $\mu_1(\{D_{\cW_1}>D\})+\mu_2(\{D_{\cW_2}>D\})\leq M$,
  $\|\cW_{2,\leq D}\|_1\leq C$, $\delGP(\cW_1,\cW_2)\leq\varepsilon$, and
  $\mu_2(\{D-\eps-\eps^2<D_{\cW_2}\leq D+\varepsilon+\eps^2\})\leq \delta$.
  Then
\[
\deltt(\cW_{1,\leq D},\cW_{2,\leq D})
\leq f(\varepsilon,\delta).
\]
\end{enumerate}
\end{lemma}

\begin{proof}
(1) The assumption $\mu-r\leq\widetilde\mu\leq\mu$ clearly implies that for
all $x\in\Omega$,
\[
D_\cW(x)-r\leq D_{\widetilde\cW}(x)\leq D_\cW(x).
\]
As a consequence,
\[\mu(D_\cW>D+r)-r\leq\widetilde\mu(D_\cW>D+r)\leq\widetilde\mu(D_{\widetilde\cW}>D)\]
and
\[\widetilde\mu(D_{\widetilde\cW}>D)\leq \widetilde\mu(D_\cW>D)\leq\mu(D_\cW>D)
.\]

(2) By the definition of $\d22$, $ \|D_{\cW_1}-D_{\cW_2}\|_2^2<
\varepsilon^4$, which clearly implies that
\[
\mu(\{|D_{\cW_1}-D_{\cW_2}|\geq \varepsilon\})<\varepsilon^2.
\]

(3) For $i=1,2$, let
$(\widetilde\Omega'_i,\widetilde\cF'_i,\widetilde\mu'_i)$ be a measure space
obtained from $(\Omega_i,\cF_i,\widetilde\mu_i)$ by appending some space of
infinite total measure, and let $\widetilde\cW_i'$ be the trivial extension
of $\widetilde\cW_i$ onto
$(\widetilde\Omega'_i,\widetilde\cF'_i,\widetilde\mu'_i)$. Furthermore, let
$\mu'$ be a coupling of $\widetilde\mu_1'$ and $\widetilde\mu_2'$ such that
$\d22(\cW'_1,\cW_2')\leq \varepsilon$, where
$\cW_i'=(\widetilde\cW_i')^{\pi_i,\mu'}$ for $i=1,2$. Then by (2),
\begin{align*}
\widetilde\mu_1(\{D_{\widetilde\cW_1}>D+\varepsilon\})
&=\widetilde\mu_1'(\{D_{\widetilde\cW_1'}>D+\varepsilon\})\\
&=\mu'(\{D_{\cW'_1}>D+\varepsilon\})\\
&\le \mu'(\{D_{\cW_2'}>D\})+\mu'(\{|D_{\cW_1'}-D_{\cW_2'}| \ge \varepsilon\})\\
&\leq \mu'(\{D_{\cW'_2}>D\})+\varepsilon^2\\
&=\widetilde\mu_2(\{D_{\cW_2}>D\})+\varepsilon^2,
\end{align*}
proving the first bound in (2). The second is proved analogously.

(4) For $i=1,2$, let $\cW_i'$ be the trivial extension of $\cW_i$ to a space
$(\Omega_i',\cF_i',\mu_i')$ obtained from $(\Omega_i,\cF_i,\mu_i)$ by
appending some $\sigma$-finite space of infinite total mass. Recalling
Lemma~\ref{lem:tilde-delGP}, we can use the assumption
$\delGP(\cW_1,\cW_2)<\varepsilon$ to infer the existence of a measure $\mu'$
over $\Omega_1'\times\Omega_2'$ such that
$\d22((\cW_1')^{\pi_1,\mu'},(\cW_2')^{\pi_2,\mu'})<\varepsilon$ and
$\mu_i'-\varepsilon^2\leq(\mu')^{\pi_i}\leq\mu_i'$, $i=1,2$. For $i=1,2$,
define $\widetilde\mu_i=(\mu')^{\pi_i}$, $\cU_i'= (\cW_i')^{\pi_{i},\mu'}$,
and $\Omega'_{i,\leq D}=\{x\in\Omega_i': D_{\cW_i'}(x)\leq D\}$. Then
$\cW'_{{1,\leq D}}:=(\cW_1')|_{\Omega'_{1,\leq D}}$ and $\cW'_{{2,\leq
D}}:=(\cW_2')|_{\Omega'_{2,\leq D}}$ are extensions of $\cW_{1,\leq D}$ and
$\cW_{2,\leq D}$ by spaces of infinite measure. Let $\mu'_{1,D}$ and
$\mu'_{2,D}$ be the marginals of the measure $\mu'|_{\Omega'_{1,\le D} \times
\Omega'_{2,\le D}}$. We then have that $\mu_{1,D}' \le \mu_1'$. Observing
that $ D_{\cW_i'}(\pi_i(x))-\varepsilon^2\leq D_{\cU_i'}(x)\leq
D_{\cW_i'}(\pi_i(x))$, we furthermore have that
\begin{align*}
0 &\le (\mu_1'-\mu_{1,D}')(\Omega'_{1,\le D})
\leq\varepsilon^2+(\widetilde\mu_1'-\mu_{1,D}')(\Omega'_{1,\le D})
\\
&=\varepsilon^2+\mu'(\pi_1^{-1}(\{D_{\cW_1'}\leq D\})\times\pi_2^{-1}(\{D_{\cW_2'}>D\}))
\\
&\leq\varepsilon^2+\mu'(\{ D_{\cU'_1}\le D  \}\cap\{D_{\cU'_2}>D-\varepsilon^2\})
\\
&\le 2\varepsilon^2 + \mu'(\{ D-\varepsilon^2<D_{\cU'_2}\leq D+\varepsilon \})
\\
&\leq 2\varepsilon^2 + \mu_2(\{ D-\varepsilon^2<D_{\cW_2} \le D+\varepsilon+\varepsilon^2\})
\leq 2\varepsilon^2+\delta=:\widetilde\delta.
\end{align*}
Here we used the fact that $\d22(\cU_1',\cU_2')<\varepsilon$, which meant
that we could apply (2). Similarly, $\mu_{2,D}' \le \mu'_2$ and
\begin{align*}
0 &\le (\mu'_2-\mu_{2,D}')(\wOmega'_{2,\le D})\\
&\leq\varepsilon^2+\mu'(\{ D_{\cU'_1}>D-\varepsilon^2  \}\cap\{D_{\cU'_2}\le D\})
\\
&\le 2\varepsilon^2 + \mu'(\{ D-\varepsilon-\varepsilon^2<D_{\cU'_2}\leq D \})\\
&\leq 2\varepsilon^2 + \mu_2(\{ D-\varepsilon-\varepsilon^2<D_{\cW_2} \le D+\varepsilon^2\})
\leq \widetilde\delta.
\end{align*}
Next we claim that we may assume without loss of generality that
\[
(\mu_1' - \mu'_{1,D})(\Omega_{1,\leq D}) =
(\mu'_2-\mu_{2,D}')(\Omega'_{2,\leq D})\leq \widetilde\delta.
\]
Indeed, we can trivially extend either $ \cW_1'$ or $\cW_2'$ by appending a
space of total measure $\delta'\leq \widetilde\delta$ (e.g., the interval
$[0,\delta')$), setting $\mu'$ to zero on the additional set. This
corresponds to trivially extending both ${\cU_{1}'}$ and ${\cU_{2}'}$ by
either $[0,\delta']\times\Omega_2'$ or $\Omega_1'\times [0,\delta']$. Since
$\mu'=0$ on the extension, this does not change $\d22({\cU_{1}'},{\cU_{2}'})
$.

Note also that
\begin{align*}
\mu'(\Omega_1'\times\Omega_2'\setminus\Omega'_{1,\leq D}\times \Omega'_{2,\leq D})
&\leq
\mu'(\{D_{\cW_1'}>D\}\times\Omega_2')+\mu'(\Omega_1'\times\{D_{\cW_2'}>D\})
\\
&\leq\mu_1(\{D_{\cW_1}>D\})+\mu_2(\{D_{\cW_2}>D\})
\leq M.
\end{align*}
If $\cU_1''$ and $\cU_2''$ are the restrictions of $\cU_1'$ and $\cU_2'$ to
$\Omega'_{1,\leq D}\times \Omega'_{2,\leq D}$, and $\mu''$ is the restriction
of $\mu'$, then by Lemma \ref{lemmaremovedsetdistancenotworse},
\[
\|\cU_1''-\cU_2''\|_{2 \rightarrow 2,\mu''} \le \varepsilon,
\]
\[
\|D_{\cU_1''}-D_{\cU_2''}\|_{2,\mu''} \le \varepsilon^2 + \sqrt{M}\varepsilon,
\]
and
\[
|\|\cU_1''\|_1 - \|\cU_2''\|_1| \le \varepsilon^3 +2\sqrt{M} \varepsilon^2 + M\varepsilon.
\]

Next, we increase $\mu''$ to a measure $\mu$ on $\Omega'_{1,\leq D}\times
\Omega'_{2,\leq D}$ by coupling $\mu_1' - \mu_{1,D}'$ and $\mu_2'-\mu_{2,D}'$
arbitrarily. Then $\mu$ has marginals $\mu_1'|_{\Omega'_{1,\leq D}}$ and
$\mu_2'|_{\Omega'_{2,\leq D}}$ and $\mu-\widetilde\delta\leq \mu''\leq \mu$.
If we apply Lemma \ref{lemmaremovednotmuchsmaller}, we then obtain a coupling
of $\cW'_{1,\le D}$ and $\cW'_{2,\le D}$ such that the pullbacks $\cU_1'''$
and $\cU_2'''$ obey the bounds
\[
\|\cU_1'''-\cU_2'''\|_{2 \rightarrow 2,\mu} \le \varepsilon + 2\sqrt{ 2 D\widetilde\delta},
\]
and
\[
|\|\cU_1'''\|_1 - \|\cU_2'''\|_1| \le \varepsilon^3 +2\sqrt{M} \varepsilon^2 + M\varepsilon + 2 D\widetilde\delta.
\]
Setting
\[
\widetilde C=C+\varepsilon^3 +2\sqrt{M} \varepsilon^2 + M\varepsilon + 2D\widetilde\delta,
\]
we then have $\max\{\|\cU_1'''\|_1 ,\|\cU_2'''\|_1\}\leq\widetilde C$ and hence
\[
\|D_{\cU_1'''}-D_{\cU_2'''}\|_{2,\mu}^2 \le \Bigl(\varepsilon^2 +
\sqrt{M}\varepsilon\Bigr)^{2} + 4 \widetilde C\widetilde\delta +
D^2\widetilde\delta.\] This completes the proof of (4).
\end{proof}

Suppose $(\cW_n)_{n=1}^\infty$ and $\cW$ are graphexes over the
$\sigma$-finite measure spaces $\bOmega_n=(\Omega_n,\cF_n,\mu_n)$ and
$\bOmega=(\Omega,\cF,\mu)$. Define, for any $D>0$,
\[
\Omega_{n,\leq D}=\{x \in \Omega_n: D_{\cW_n}(x) \le D\}
\qquad\text{and}\qquad
\Omega_{n,>D}=\{x \in \Omega_n: D_{\cW_n}(x) > D\}.
\]
Recall that we have $\mu_n(\Omega_{n, >D})<\infty$. Let $\cW_{n,\leq D}$
consist of $\cW_n$ restricted to $\Omega_{n,{\leq D}}$. Define $\Omega_{\leq
D}$, $\Omega_{>D}$, and $\cW_{\leq D}$ similarly. We then have the following.
\begin{proposition} \label{propgeneralconvergenceequiv}
Given a sequence of graphexes $\cW_n$ and a graphex $\cW$ over the
$\sigma$-finite measure spaces $\bOmega_n=(\Omega_n,\cF_n,\mu_n)$ and
$\bOmega=(\Omega,\cF,\mu)$, respectively, define $\cW_{n,{\leq D}}$ and
$\cW_{\leq D}$ as above. Then the following are equivalent.
\begin{enumerate}
\item\label{convequivatmostDboth} For all $D>0$ such that
    $\mu(\{D_{\cW}=D\})=0$, we have $\deltt(\cW_{n,{\leq D}},\cW_{\leq D})
\rightarrow 0$ and $\mu_n(\Omega_{n,>D}) \rightarrow \mu(\Omega_{>D})$.
\item\label{convequivatmostDtight} The sequence is tight, and for all $D>0$
    such that $\mu(\{D_{\cW}=D\})=0$, we have $\deltt(\cW_{n,{\leq
D}},\cW_{\leq D}) \rightarrow 0$.
\item \label{convequivremoveeps} For every $\varepsilon>0$ and $n\in\NN$
    there exist subsets $\Omega_{n,\varepsilon} \subseteq \Omega_n$ and a
subset $\Omega_{\varepsilon} \subseteq\Omega$ with $\mu_{n}(\Omega_n
\setminus \Omega_{n,\varepsilon}) \le\varepsilon$ and $\mu(\Omega \setminus
\Omega_{\varepsilon}) \le \varepsilon$, such that $\deltt(\cW_n',\cW')
\rightarrow 0$, where $\cW_n'=(\cW_n)|_{\Omega_n,\eps}$ and
$\cW'=\cW|_{\Omega_\eps}$.
\item\label{convequivepsilondistance} $\delGP(\cW_n,\cW) \rightarrow 0$.
\end{enumerate}
\end{proposition}

\begin{proof}
$(\ref{convequivatmostDboth}) \Rightarrow (\ref{convequivatmostDtight})$: We
have to prove tightness. For any $\varepsilon$, there exists a $D$ such that
$\mu(\Omega_{>D}) \le \varepsilon/2$. Assuming without loss of generality
that $D$ is chosen in such a way that $\mu(\{D_{\cW}=D\})=0$, we further have
$\mu_{n}(\Omega_{n,>D}) \rightarrow \mu(\Omega_{> D})$. This means that for
all but a finite set of $n$, $\mu_{n}(\Omega_{n,>D}) \le \varepsilon$. By
increasing $D$, we can guarantee this for all $n$. Since $\deltt$ convergence
implies in particular that $\|\cW_{n,{\leq D}}\|_1 \rightarrow \|\cW_{\leq
D}\|_1$, we have that $\|\cW_{n,{\leq D}}\|_1$ is bounded. This proves
property \ref{tightD} from Theorem \ref{theoremtightequiv}.

(\ref{convequivatmostDtight}) clearly implies (\ref{convequivremoveeps}),
because tightness implies that for any $\varepsilon$, there exists a $D$ such
that the measure of points with degree greater than $D$ is at most
$\varepsilon$, and we can increase $D$ to make sure that
$\mu(\{D_\cW=D\})=0$.

To show that (\ref{convequivremoveeps}) implies
(\ref{convequivepsilondistance}), we note that the conditions in
(\ref{convequivremoveeps}) imply that
$\limsup_{n\to\infty}\delGP(\cW_n,\cW)\leq \sqrt\varepsilon$. Since
$\varepsilon$ is arbitrary, this gives (\ref{convequivepsilondistance}).

It remains to show that (\ref{convequivepsilondistance}) implies
(\ref{convequivatmostDboth}). The assumption $\delGP(\cW_n,\cW) \rightarrow
0$ implies that for all $\varepsilon>0$ there exists an $n_0$ such that for
$n\geq n_0$, $\delGP(\cW_n,\cW)< \varepsilon$. Recalling
Definition~\ref{def:delGP} and combining statements (1) and (3) of the
previous lemma, this implies that for $D>\varepsilon+\varepsilon^2>0$ and
$n\geq n_0$,
\begin{align*}
\mu(D_{\cW}&>D+\varepsilon+\varepsilon^2)-2\varepsilon^2
\leq
\mu_n(D_{\cW_n}>D)
\leq \mu(D_\cW>D-\varepsilon^2-\varepsilon)+2\varepsilon^2.
\end{align*}
By our assumption that $\mu(\{D_{\cW}=D\})=0$, we have that $D$ is a
continuity point of the function $x\mapsto\mu(\{D_\cW>x\})$, showing that the
upper and lower bound converge to $\mu(\{D_\cW>D\})$ as $\varepsilon\to 0$.
This shows that $\mu_n(D_{\cW_n}>D)\to \mu(\{D_\cW>D\})$ as $n\to\infty$.

Next we define
\[M=\sup_n\mu_n\left(\Omega_{n,>D}\right)+\mu(\Omega_{>D})
\quad\text{ and }\quad \delta(\eps)=\mu(\{D-\varepsilon-\eps^2<D_\cW\leq
D+\eps+\eps^2\}).\] Note that $M$ is finite by the fact that
$\mu_n(\Omega_{n,>D}) \rightarrow \mu(\Omega_{>D})$, and that
$\delta(\eps)\to 0$ as $\varepsilon\to 0$ by the fact that
$\mu(\{D_\cW=D\})=0$. We now apply the previous lemma (with $W_1$ replaced by
$\cW_n$ and $\cW_2$ replaced by $\cW$) to conclude the proof.
\end{proof}

The following proposition is an easy corollary of
Proposition~\ref{propgeneralconvergenceequiv}.

\begin{proposition} \label{propequivgraphexes}
Given two graphexes $\cW_1,\cW_2$, let $\cW_{i,{\leq D}}$ be the graphex
$\cW_i$ restricted to $\Omega_{i,\leq D}=\{x \in \Omega_i: D_{\cW_i}(x) \le
D\}$. Then the following are equivalent.
\begin{enumerate}
\item \label{graphexequiveachDsameboth} For any $D>0$,
    \[\mu_{{1}}(\Omega_1\setminus\Omega_{1,{\leq
    D}})=\mu_{{2}}(\Omega_2\setminus\Omega_{2,{\leq D}})\text{ and }
 \deltt(\cW_{1,{\leq D}},\cW_{2,{\leq D}})=0.\]
\item For any $D>0$, $\deltt(\cW_{1,{\leq D}},\cW_{2,{\leq D}})=0$.
    \label{graphexequiveachDsame}
\item For any $\varepsilon>0$, there exist subsets $\Omega_{1,\varepsilon}
    \subseteq \Omega_1$, $\Omega_{2,\varepsilon} \subseteq \Omega_2$ with
$\mu_{{i}}(\Omega_i \setminus \Omega_{i,\varepsilon}) \le \varepsilon$,
such that if $\cW_i'$ is the restriction to $\Omega_i \setminus
\Omega_{i,\varepsilon}$, then $\deltt(\cW_1',\cW_2') \le \varepsilon$.
\label{graphexequivepsilonD}
\item $\delGP(\cW_1,\cW_2)=0$. \label{graphexequivonlyepsilon}
\end{enumerate}
\end{proposition}

It is easy to see that this is an equivalence relation.

\begin{proof}
Observing that the functions $f_i\colon\RR_+\to\RR+$ defined by
$D\mapsto\mu_i(\Omega_i\setminus\Omega_{i,\leq D})$ for $i=1,2$ are equal if
and only if $f_1(D)=f_2(D)$ for all continuity points of $f_1$, this follows
by applying Proposition \ref{propgeneralconvergenceequiv} with $\cW=\cW_1$
and each $\cW_n=\cW_2$.
\end{proof}

\begin{remark}\label{rem:signed-tightness}
It is not \emph{a priori} clear how to generalize the notion of tightness to
signed graphexes, even if we restrict ourselves to the case where the graphon
parts are uniformly bounded, for example by taking graphons that take values
in $[-1,1]$. Indeed, recalling Lemma~\ref{lem:signed-delGP<infty} and the
role it played in showing that $\delGP$ is finite for signed graphexes with
bounded graphon part, one might want to modify Definition~\ref{def:tight}
 for such graphexes by replacing the notion of
$D$-bounded marginals by that of $D$-bounded absolute marginals, since this
would, in particular, guarantee that a finite set of signed graphexes with
graphon parts in $[-1,1]$ is tight. It would also make the generalization of
several of our results straightforward, since this definition just reduces
the notion of tightness of a set of graphexes $\cS$ to the set of graphexes
$\cS'=\{|\cW| : \cW\in \cS\}$.

The following example shows that this straightforward generalization of
Definition~\ref{def:tight} to signed graphexes does not give a
characterization of precompact sets with respect to the metric $\delGP$, as
it did for unsigned graphexes; see Theorem~\ref{thm:complete}.

Let $W_n$ be equal to $n^{-3/4}$ on $[0,1)\times[1,n+1)\cup [1,n+1)\times
[0,1)$, equal to $-n^{-3/4}$ on $[0,1)\times[n+1,2n+2)\cup [n+1,2n+1)\times
[0,1)$, and zero everywhere else on $\RR_+^2$. Define $\cW_n$ to be the
graphex with graphon part $W_n$ and zero star and dust part. Then $D_{\cW_n}$
is equal to $n^{-3/4}$ on $[1,n+1)$, equal to $-n^{-3/4}$ on $[n+1,2n+1)$,
and $0$ everywhere else. Finally, $\rho(\cW_n)= 0$ for all $n$. Since
$\|W_n\|_2=2n^{-1/2}$ and $\|D_{\cW_n}\|_2=\sqrt 2 n^{-1/2}$, $\cW_n$ tends
to the zero graphex in the metric $\deltt$ and hence also in $\delGP$. But
$\|\cW_n\|_1=\|W\|_1=4 n^{1/4}\to\infty$, a fact which can't be changed by
removing just a part of measure $\varepsilon$ from the underlying space,
$\RR_+$. This shows that with the obvious generalization of
Definition~\ref{def:tight} to signed graphexes not all sequences of signed
graphexes that are convergent in $\deltt$ or $\delGP$ are tight.

We therefore believe that a complete theory of signed graphexes, even in the
simplified case where all graphons take values in $[-1,1]$, requires either a
modification of the metric, or modification of the notion of tightness. We
leave this problem as an open research problem.
\end{remark}

\section{Regularity lemma and compactness} \label{sec:reg}

In this section, we will prove a regularity lemma
(Theorem~\ref{regularity-lemma} below), and use it to prove
Theorem~\ref{thm:complete}, which in turn is an important ingredient in our
proof that GP-convergence and $\delGP$-convergence are equivalent. To state
the regularity lemma, we recall that a finite subspace partition of a measure
space $\bOmega=(\Omega,\cF,\mu)$ is a partition of a measurable subset of
$\Omega$ into finitely many measurable subsets of finite measure. Throughout
this section, we will use the notation $\sP=(\Omega_{\sP},\cP)$ for a finite
subspace partition, with $\Omega_{\sP}$ denoting the subset of $\Omega$, and
$\cP=(P_1,\dots,P_m)$ denoting the partition of $\Omega_{\sP}$. We will also
need the notion of refinement.
\begin{definition}
Given two subspace partitions $\mathscr{P} = (\Omega_{\sP},\cP)$ and
$\sQ=(\Omega_{\cQ},\cQ)$, we say that $\sP$ \emph{refines} $\sQ$ if
$\Omega_{\cQ} \subseteq \Omega_{\sP}$ and $\cP$ is a refinement of $\cQ \cup
\{\Omega_{\sP} \setminus \Omega_{\cQ}\}$.
\end{definition}

Given an integrable, signed graphex $\cW$, a subspace partition $\sP$
naturally generates a step function $\cW_{\sP}$ by ``averaging''. The precise
definition is as follows.

\begin{definition}
Given a signed graphex $\cW=(W,S,I,\bOmega)$ and a finite subspace partition
$\sP=(\Omega_\sP,\cP)$, take $\cW_\sP$ to be the signed graphex
$\cW_\sP=(W_\sP,S_\sP,I_\sP,\bOmega)$ defined by
\[
I_\sP= \frac{1}{2}\int_{(\Omega \setminus \Omega_{\sP})\times (\Omega \setminus \Omega_{\sP})} W(x,y)\,d\mu(x) \,d\mu(y) + \int_{\Omega \setminus \Omega_{\sP}}S(x)\,d\mu(x)+I,
\]
\[
S_{\sP}(x)=\frac{1}{\mu(P_i)}\int_{P_i} \left(S(x)+\int_{\Omega \setminus \Omega_{\sP}}W(x,y)\,d\mu(y) \right)\,d\mu(x)
\quad\text{if}\quad x \in P_i
\]
for some $i\in\{1,\dots,k\}$, and $0$ everywhere else,
and
\[
W_{\sP}(x,y)=\frac1{\mu(P_i)\mu(P_j)}\int_{P_i \times P_j}W(x',y')\,d\mu(x')\,d\mu(y')\quad\text{if}\quad(x,y)\in P_i\times P_j
\]
for some $i,j\in \{1,\dots,k\}$ and $W_{\sP}(x,y)=0$ everywhere else.
\end{definition}
Note that with this definition, for $x \in P_i$,
\[D_{\cW_\sP}(x)=\frac{1}{\mu(P_i)} \int_{P_i} D_\cW(x) \,d\mu(x)
.\] We also have $\rho(\cW_\sP)=\rho(\cW)$, as well as
$\|\cW_\cP\|_1\leq\|\cW\|_1$, $\|D_{\cW_\sP}\|_\infty\leq \|D_\cW\|_\infty$,
and $\|W_\sP\|_\infty\leq \|W\|_\infty$.

\begin{theorem} \label{regularity-lemma}
For any $B,C,D<\infty$, and $\varepsilon>0$, there exists an $M(\varepsilon)$
and $N(\varepsilon)$ such that for any signed graphex $\cW$ that is
$(B,C,D)$-bounded there exists a partition $(\Omega_{\sP},\cP)$ with
$\cP=\{P_1,\dots,P_m\}$, $m \le M(\varepsilon)$, and $\mu(\Omega_{\sP}) \le
N(\varepsilon)$ such that
\[
d_\boxtimes(\cW,\cW_{\sP})\leq \eps.
\]
We can take
\[
 M(\eps)=2^{(2BC+CD)/\eps^2}\qquad\text{and}\qquad
 N(\varepsilon)=
(4C^3D+8BC^2D)/{\eps^4}.
 \]
Given any finite subspace partition $\sQ=(\Omega_{\sQ},\cQ)$, we can require
the subspace partition $\sP=(\Omega_{\sP},\cP)$ to be a refinement of $\sQ$.
In this case, the bound on the number of parts is $|\cQ|M(\varepsilon)$ and
the bound on $\mu(\Omega_{\sP})$ is $\mu(\Omega_{\cQ})+N(\varepsilon)$.
\end{theorem}

\begin{proof}
Motivated by the original proof of the weak regularity lemma (see in
particular the proof of Theorem 12 in \cite{FK99}, which to our knowledge is
the first place where a weak regularity lemma for functions
$W\colon[0,1]^2\to \RR$ was established), we construct a sequence of
partitions $\sP_0,\sP_1,\dots,\sP_\ell$ such that eventually, we must have
that $d_\boxtimes(\cW,\cW_{\sP_\ell})\leq \eps$. We start with the trivial
partition $\sP_0=(\emptyset,\emptyset)$ (so that $\cW_{\sP_0}$ is the graphex
with zero graphon and star part, and dust part $\rho(\cW)$), and then
construct a sequence of refinements $\sP_0,\sP_1,\dots,\sP_\ell$.

In a preliminary step, we claim that for any partition $\sP$,
\[\langle W-W_\sP,W_\sP\rangle= \int_{\Omega \times \Omega} (W(x,y)-W_\sP(x,y))W_\sP(x,y) \,dx \,dy =0
\] and
\[\langle D_{\cW}-D_{\cW_\sP},D_{\cW_\sP}\rangle=\int_\Omega (D_\cW(x)-D_{\cW_{\sP}}(x))D_{\cW_{\sP}}(x) \,dx =0
\]
This follows from the fact that for any pair of finite parts $P_i,P_j$,
$W_\sP$ and $D_{\cW_\sP}$ are constant, and the integral of $W-W_\sP$ and
$D_{\cW}-D_{\cW_\sP}$ is zero. Since $W_\sP$ is zero between pairs of parts
where at least one is non-finite, and $D_{\cW_\sP}$ is zero on non-finite
parts, this implies the claim. Therefore, we have
\[\|W\|_2^2=\|W_\sP\|_2^2+\|W-W_\sP\|_2^2\]
and
\[\|D_\cW\|_2^2=\|D_{\cW_\sP}\|_2^2+\|D_{\cW}-D_{\cW_\sP}\|_2^2.\]
If we have a finite subspace partition $\sP=(\Omega_{\sP},\cP)$ and a
refinement $\sP'$, it is then easy to check that $(\cW_{\sP'})_\sP=\cW_\sP$;
therefore, the same properties hold for $\cW_\sP$ and $\cW_{\sP'}$.

Suppose now that we have constructed a sequence of refinements
$\sP_0,\sP_1,\dots,\sP_i$ such that $d_\boxtimes(\cW,\cW_{\sP_j}) >\eps$ for
all $j\leq i$. Then we in particular have that $d_\boxtimes(\cW,\cW_{\sP_i})
>\eps$, which implies that $\|D_\cW-D_{\cW_{\sP_i}}\|_\boxtimes >
\varepsilon$ or $\|W-W_{\sP_i}\|_{\boxtimes} > \varepsilon$. If the former
holds, then there exists a set $S \subseteq \Omega$ (of finite measure) such
that
\begin{equation}
\label{D-jbl>eps}
\frac{1}{\sqrt{\mu(S)}}\left|\int_{S} D_\cW(x)\,d\mu(x)-\int_{S} D_{\cW_{\sP_{{i}}}}(x)\,d\mu(x)\right|=A>\eps.
\end{equation}
Let $\sP_{i+1}=(\Omega_{{\sP_{i+1}}},\cP_{i+1})$ with
$\Omega_{{\sP_{i+1}}}=\Omega_{\sP_{i}} \cup S$ and let $\cP_{i+1}$ be the
partition that refines each part of $\cP_{i}$ by the intersection with $S$;
in particular, this divides each part into at most $2$ parts, and
$\mu(\Omega_{{\sP_{i+1}}}) \le \mu(\Omega_{\sP_{i}})+\mu(S)$. We also have
\[\frac{1}{\sqrt{\mu(S)}}\left|\int_{S} D_{\cW_{{\sP_{i+1}}}}(x)\,d\mu(x)-\int_{S}
D_{\cW_{\sP_{i}}}(x)\,d\mu(x)\right|=A
.\]
Therefore,
\[\|D_{\cW_{{\sP_{i+1}}}}-D_{\cW_{\sP_{i}}}\|_2^2 \ge \left\langle D_{\cW_{{\sP_{i+1}}}}-D_{\cW_{\sP_{i}}},\frac{\chi_{S}}{\sqrt{\mu(S)}}
\right\rangle^2=A^2
.\]
Overall, this implies that we have
\[\|D_{\cW_{\sP_{i}}}\|_2^2+{\eps^2} \le \|D_{\cW_{{\sP_{i+1}}}}\|_2^2 \le \|D_{\cW}\|_2^2
.\]
We also have
\[
\|W_{\sP_i}\|_2^2 \le \|W_{\sP_{i+1}}\|_2^2 \le \|W\|_2^2 .\] If
$\|W-W_{\sP_i}\|_{\boxtimes} > \varepsilon$, then there exist sets $S,T
\subseteq \Omega$ (of finite measure) such that
\begin{equation}
\label{W-jbl>eps}
\frac{1}{\sqrt{\mu(S)\mu(T)}}\left|\int_{S \times T} W(x,y)\,d\mu(x)\,d\mu(y)-\int_{S \times T}
W_{\sP_{i}}(x,y)\,d\mu(x)\,d\mu(y)\right|>\eps.
\end{equation}
Let $\sP_{{i+1}}=(\Omega_{\sP_{i+1}},\cP_{i+1})$ with
$\Omega_{\sP_{i+1}}=\Omega_\sP \cup S \cup T$ and let $\cP_{i+1}$ be the
partition that refines each part of $\cP_{i}$ by the intersection with $S$
and $T$, in particular, this refines each part into at most $4$ parts, and
$\mu(\Omega_{\sP_{i+1}}) \le \mu(\Omega_{\sP_{i}})+\mu(S)+\mu(T)$. Proceeding
as before, we have
\[\|W_{\sP_{i}}\|_2^2+\varepsilon^2 \le \|W_{\sP_{i+1}}\|_2^2 \le \|W\|_2^2
,\] and furthermore,
\[\|D_{\cW_{\sP_i}}\|_2 \le \|D_{\cW_{\sP_{i+1}}}\|_2 \le
\|D_{\cW}\|_2 .\] The first step can occur at most
${\|D_\cW\|_2}/{\varepsilon^2}$ times, and the second at most ${\|W\|_2^2
}/{\varepsilon^2}$ times. Since in the first step, the number of partition
classes at most doubles, and in the second it goes up by at most a factor of
four, this proves that there exists a partition $\sP$ with at most
\[
2^{(2\|W\|_2^2+\|D_\cW\|_2^2)/\eps^2}\leq2^{(2BC+CD)/\eps^2}=M(\eps)
\]
classes such that $d_\boxtimes(\cW,\cW_{\sP})\leq \eps$.

To prove that $\mu(\Omega_{\sP}) \le N(\varepsilon)$, we claim that
in each step, \eqref{D-jbl>eps} implies that
\[\mu(S) \le \frac{4C^2}{\eps^2},
\]
and \eqref{W-jbl>eps} implies that
\[\mu(S),\mu(T) \le \frac{4CD}{\eps^2}
.\] Indeed, for any $S \subseteq \Omega$,
\[\left|\int_S D_\cW(x)-D_{\cW_{\cP_i}}(x) \right| \le 2C .\]
This implies that if
\[
\frac{1}{\sqrt{\mu(S)}} \left|\int_S D_\cW(x) \,d\mu(x)-\int_S D_{\cW_{\sP_i}}(x) \,d\mu(x) \right|\ge
\eps
\]
then
\[\mu(S) \le \frac{4C^2}{\eps^2}.
\]
On the other hand,
\[\left|\int_{S \times T}W-W_\sP \right| \le 2C
\qquad\text{and}\qquad
\left|\int_{S \times T}W-W_\sP \right| \le 2 D\mu(T).
\]
Therefore,
\[\left|\int_{S \times T}W-W_\sP \right| \le 2\sqrt{C D\mu(T)}.
\]
This implies that if
\[\frac{1}{\sqrt{\mu(S)\mu(T)}}\left|\int_{S \times T} W- W_{\sP_i}\right| \ge
\eps
\] then
\[\mu(S) \le \frac{4CD}{\eps^2}
.\] The bound for $\mu(T)$ follows similarly.

Since first step can occur at most ${\|D_\cW\|_2}/{\varepsilon^2}$ times, and
the second at most ${\|W\|_2^2 }/{\varepsilon^2}$ times, this shows that
\[\mu(\Omega_{\sP}) \le
\frac{4C^2}{\eps^2}\frac{\|D_\cW\|_2^2}{\varepsilon^2}+ 2\frac{4CD}{\eps^2}\frac{\|W\|_2^2 }{\varepsilon^2}
\leq\frac 4{\eps^4}\Bigl(C^3D+2BC^2D\Bigr)=N(\varepsilon).
\]
The second statement follows by choosing $\sP_0=\sQ$.
\end{proof}

\begin{remark}
With the help of Proposition~\ref{prop:kernel-jumbl-equiv},
Theorem~\ref{regularity-lemma} can immediately be transformed into a similar
statement for the kernel distance $\d22(\cW,\cW_\sP)$, provided $N(\eps)$ and
$M(\eps)$ are replaced by bounds of the form $M(\eps)=2^{c/\eps^8}$ and
$N(\eps)=d/\eps^{16}$ where $c$ and $d$ are constants depending on $B$, $C$
and $D$.
\end{remark}

We would like to prove a version of this ``regularity lemma'' for $\d22$
where the parts have equal size. We first show some preliminary lemmas.

\begin{lemma} \label{lem:partitionsmallerdistance}
Let $\cW_1,\cW_2$ be two graphexes on the same space $\bOmega$, and let $\sP$
be a finite subspace partition of $\bOmega$. Then
\[\|W_{1,\sP}-W_{2,\sP}\|_{2 \rightarrow 2} \le \|W_1-W_2\|_{2 \rightarrow 2}
\] and
\[
\|D_{\cW_{1,\sP}}-D_{\cW_{2,\sP}}\|_2
\le \|D_{\cW_1}-D_{\cW_2}\|_2
.\]
\end{lemma}

\begin{proof}
Note that
\begin{align*}
\|W_{1,\sP}-W_{2,\sP}\|_{2 \rightarrow 2}&=\sup_{\substack{f,g \in L^2(\Omega)\\ \|f\|_2=\|g\|_2=1}}
f \circ (W_{1,\sP}-W_{2,\sP}) \circ g\\
&=\sup_{\substack{f,g \in L^2(\Omega_\sP)\\ \|f\|_2=\|g\|_2=1}} f \circ (W_{1,\sP}-W_{2,\sP}) \circ
g.
\end{align*}
If we let $f_\sP$ and $g_\sP$ consist of the average values of $f$ and $g$ on
each part of $\cP$ and zero outside $\Omega_\sP$, then
\[f \circ (W_{1,\sP}-W_{2,\sP}) \circ g=f_\sP \circ (W_{1,\sP}-W_{2,\sP}) \circ g_\sP=f_\sP \circ (W_{1}-W_{2}) \circ g_\sP,
\]
which implies the first claim. The second claim follows similarly.
\end{proof}

\begin{corollary}\label{corrolarystepvsaverage}
Suppose that $\cW$ is a graphex, and that $\cU$ is a step graphex over the
same space as $\cW$. If $\sP=(\Omega_\sP,\cP)$ is a finite subspace partition
such that $\cU$ is constant on each part of $\cP$ and zero outside
$\Omega_\sP$, then
\[\|W_\sP-W\|_{2 \rightarrow 2} \le 2\|U-W\|_{2 \rightarrow 2}
\]
and
\[\|D_{\cW_\sP}-D_\cW\|_2 \le 2\|D_\cU-D_\cW\|_2.
\]
\end{corollary}

\begin{proof}
Note that $\cU_{\sP}=\cU$. We have
\begin{align*}
\|W_\sP-W\|_{2 \rightarrow 2} &\le \|W_\sP-U\|_{2 \rightarrow 2} +\|U-W\|_{2 \rightarrow 2}\\
&= \|W_\sP-U_\sP\|_{2 \rightarrow 2} +\|U-W\|_{2 \rightarrow 2}
\le 2\|U-W\|_{2 \rightarrow
2}.
\end{align*}
Similarly,
\begin{align*}
\|D_{\cW_\sP}-D_\cW\|_2 &\le\|D_{\cW_\sP}-D_\cU\|_2+\|D_{\cU}-D_\cW\|_2\\
&=\|D_{\cW_\sP}-D_{\cU_\sP}\|_2+\|D_{\cU}-D_\cW\|_2
\le
2\|D_\cU-D_\cW\|_2
. \qedhere
\end{align*}
\end{proof}

\begin{theorem} \label{regularityequalparts}
Given $B$, $C$, $D$, and $\varepsilon>0$, there exists
$\rho_0=\rho_0(\varepsilon,B,C,D)>0$ and $N_0=N_0(\varepsilon,B,C,D)$ such
that for any $\rho<\rho_0$, any $m \ge N_0/\rho$, and any $(B,C,D)$-bounded
signed graphex $\cW$ on an atomless space with infinite measure, there exists
a subspace partition $\sP=(\Omega_\sP,\cP)$ with exactly $m$ parts of size
$\rho$ such that $\d22(\cW,\cW_\sP)\leq \eps$. If
$\sP_0=(\Omega_{\sP_0},\cP_0)$ is an arbitrary finite subspace partition, we
can require $\sP$ to refine $\sP_0$, as long as each part of $\cP_0$ is
divisible by $\rho$ (and increasing the bound on $N_0(\varepsilon)$ and
decreasing $\rho_0$ appropriately depending on $|\cP_0|$ and
$\mu(\Omega_{\sP_0})$).
\end{theorem}

\begin{proof}
Apply Theorem~\ref{regularity-lemma} and
Proposition~\ref{prop:kernel-jumbl-equiv} to obtain a subspace partition
$\sP'$ with at most $M(\varepsilon/3)$ parts and size at most
$N(\varepsilon/3)$ such that $\d22(\cW,\cW_{\sP'})\leq \eps/3$.
 We first construct a refinement $\sQ=(\Omega_\sQ,\cQ)$ as
follows. Add a part from $\Omega \setminus \Omega_{\sP'}$ so that the total
measure of $\Omega_\sQ$ is equal to $m\rho$ (this will require $m\geq
N(\varepsilon/3)/\rho$), and divide each part of
$\cP'\cup\{\Omega_\sQ\setminus\Omega_{\sP'}\}$ into parts of size $\rho$,
with perhaps one part remaining of smaller size. We then define $\sP$ by
combining the parts of $\cQ$ that have size smaller than $\rho$, including
the part added from $\Omega \setminus \Omega_{\sP'}$, into a single set
$\Omega'$, and then dividing $\Omega'$ into parts of size $\rho$ (and keeping
the remaining parts of $\sQ$). Then $\cW_\cQ$ and $\cW_\cP$ differ only on
$\Omega'$, which has size at most $(M(\varepsilon/3)+1)\rho$, implying that
\begin{align*}
 \|W_\sQ-W_\sP\|_{2 \rightarrow 2}^2
 &\le \|W_\sQ-W_\sP\|_{2}^2\\
 &\le 2\int_{\Omega \times \Omega'}(W_\sQ-W_\sP)^2
 \le
 2\int_{\Omega\times\Omega'}B(|W_\sQ|+|W_\sP|)\\
 &\leq 4B\int_{\Omega'} D_{|W|} \le 4(M(\varepsilon/3)+1)\rho B D
 .
\end{align*}
We also have
\[
 \|D_{\cW_\sQ}-D_{\cW_\sP}\|_2 \le 2\sqrt{\int_{\Omega'} D_{|\cW|}(x)^2\,d\mu(x)} \le 2\sqrt{(M(\varepsilon/3)+1)\rho}D
.\]
Furthermore, we know that $\cW_{\sP'}$ is constant on each part of $\sQ$.
Therefore, applying Corollary \ref{corrolarystepvsaverage}, we have
\[\|W-W_\sQ\|_{2 \rightarrow 2} \le 2\|W-\cW_{\sP'}\|_{2 \rightarrow 2}\le 2\varepsilon/3\]
and
\[\|D_\cW-D_{\cW_\sQ}\|_2 \le 2 \|D_\cW-D_{\cW_{\sP'}}\|_2 \le 2\varepsilon^2/9
.\] Therefore, if $\rho$ is small enough, then
\[\|W-W_\sP\|_{2 \rightarrow 2} \le \varepsilon\]
and
\[\|D_\cW-D_{\cW_\sP}\|_2 \le \varepsilon^2
.\] We can add parts of zero measure and the above argument still works. If
we want $\sP$ to be a refinement of a starting partition $\sP_0$, we apply
Theorem~\ref{regularity-lemma} and Proposition~\ref{prop:kernel-jumbl-equiv}
and make sure to combine the leftover parts so that they are within the same
part of $\cP_0$, this can be done since each part of $\cP_0$ is divisible by
$\rho$.
\end{proof}

We close this section by proving Theorem~\ref{thm:complete}. To this end, we
will first establish two lemmas.

\begin{lemma} \label{lemmapartitionsequence}
Let $\cS$ be a set of signed graphexes over atomless spaces of infinite
measure such that the following condition holds:
\begin{enumerate}[(a)]
\item For every $\varepsilon>0$, there exists $B,C,D$ such that for any
    $\cW=(W,S,I,(\Omega,\cF,\mu))\in \cS$, taking $\Omega_{\le D}$ to be
 the set of points with $D_{|\cW|}(x)\le D$, we have that $
 \mu(\Omega\setminus\Omega_{\leq D}) \le \varepsilon$, $\|W|_{\Omega_{\leq
 D}}\|_\infty \le B$, and $\|\cW|_{\Omega_{\leq D}}\|_1 \le C$.
\label{sign-tightD}
\end{enumerate}
Then there exist strictly increasing sequences of integers $a_k \ge 2k$ and
$b_k\geq k$ and sequences of positive real numbers $B_k$, $C_k$, and $D_k$
such that the following holds. For any graphex $\cW \in \cS$,
$\cW=(W,S,I,\bOmega)$ with $\bOmega=(\Omega,\cF,\mu)$, there exists a
sequence of subsets $P_{k,0}\subseteq\Omega$ and subspace partitions
$\sP_k=(\Omega_{\sP_k},\cP_k)$, $\cP_k=\{P_{k,1},\dots,P_{k,m_k}\}$ with
$m_k=2^{a_k+b_k}$, such that for all $k$, $\Omega_{\sP_k}$ is disjoint from
$P_{k,0}$ and
\begin{enumerate}
\item $P_{k+1,0} \subseteq P_{k,0}$ and $\mu(P_{k,0})=2^{-2k}$,
    \label{partitionsequenceremoved}
\item$\cW_k=\cW|_{\Omega\setminus P_{k,0}}$ is $(B_k,C_k,D_k)$-bounded,
    \label{partitionsequencebounded}
\item $\sP_{k+1}$ refines $\sP_k$, \label{partitionsequencerefinement}
\item $P_{k,i}$ for $i \ge 1$ has measure $2^{-a_k}$, and
    \label{partitionsequencesize}
\item $\d22(\cW_k,(\cW_k)_{\sP_k})\leq 2^{-k}$.
    \label{partitionsequenceregular}
\end{enumerate}
\end{lemma}
Note that by property (\ref{tightD}) from Theorem \ref{theoremtightequiv},
every tight set of unsigned graphexes obeys the condition \eqref{sign-tightD}
(with $B=1$), showing that the conclusions of the theorem hold for any tight
set $\cS$ of graphexes over atomless spaces of infinite measure.

\begin{proof}
Define
\[
D_k=\inf\{D : \text{for all $\cW\in\cS$, $\mu(\Omega\setminus\Omega_{\leq D})\leq 2^{-2k}$}\},
\]
and let $B_k=\sup_{\cW\in\cS}\|W|_{\Omega_{\leq D_k}}\|_\infty$ and
$C_k=\sup_{\cW\in\cS}\|\cW|_{\Omega_{\leq D_k}}\|_1$. By the condition
\eqref{sign-tightD} these are finite, and by construction they are monotone
non-decreasing functions of $k$. Given a graphex, we then first set each
$P_{k,0}'$ to be the set of points with degree greater than $D_{k}$. In this
way we have each $P_{k+1,0}' \subseteq P_{k,0}'$; however, they may be
strictly smaller than the required size. We therefore extend them one by one,
starting with $P_{0,0}$, and make sure that we still have each $P_{k+1,0}
\subseteq P_{k,0}$. Taking $\cW_k$ to be the restrictions, properties
(\ref{partitionsequenceremoved}) and (\ref{partitionsequencebounded}) are
satisfied.

Next, apply Theorem \ref{regularityequalparts} with $\varepsilon=1$ to obtain
$N_{0}$ and $\rho_0$. Increasing $N_0$ (if needed) so that it is of the form
$2^{b_0}$ for some nonnegative integer $b_0$, we then choose $a_0$ such that
$2^{-a_0}<\rho_0$. For any graphex, take $\sP_0$ according to the theorem
with $\rho=2^{-a_0}$. Keep iterating the theorem, in each step applying
Theorem \ref{regularityequalparts} with $B=B_k$, $C=C_{k}$, $D=D_{k}$,
$\varepsilon=2^{-k}$, and $\sP_0=(\Omega_{\sP_{k-1}} \cup (P_{k-1,0}
\setminus P_{k,0}),\cP_k \cup \{P_{k-1,0} \setminus P_{k,0}\}$), ensuring in
each step that $a_k \ge \max\{2k,a_{k-1}+1\}$, $2^{-a_k}<\rho_k$, and $b_k >
b_{k-1}$.
\end{proof}

\begin{lemma}
\label{lem:tight-relcomp} Every tight sequence of graphexes has a subsequence
that is $\delGP$-convergent. More generally, every sequence of signed
graphexes obeying the condition \eqref{sign-tightD} from
Lemma~\ref{lemmapartitionsequence}
 has a $\delGP$-convergent subsequence. If condition \eqref{sign-tightD} holds
with one or more of the constants $B,C,D$ not depending on $\eps$, then the subsequential
limit inherits the corresponding bound.
\end{lemma}

During the proof, we will use the following.

\begin{claim}
Suppose that $\cW=(W,S,I,\bOmega)$ is a signed graphex, $\sP=(\Omega_{\sP},
\cP)$ is a finite subspace partition, and $\cP_0 \subseteq \cP$. Let
$\Omega_0=\bigcup_{P \in \cP_0} P$, $\cP'=\cP \setminus \cP_0$, and
$\sP'=(\Omega'_{\sP'},\cP')$, where $\Omega'_{\sP'}=\Omega_{\sP} \setminus
\Omega_0$. Then $(\cW_{\sP})|_{\Omega'}=(\cW|_{\Omega'})_{\sP'}$.
\end{claim}

\begin{proof}
Note that $\Omega \setminus \Omega_\sP=\Omega' \setminus \Omega'_{\sP'}$
(since $\Omega_0$ is disjoint from both). First, we have
\begin{align*}
I_\sP&= \frac{1}{2}\int_{(\Omega \setminus \Omega_{\sP})\times (\Omega \setminus \Omega_{\sP})} W(x,y)\,d\mu(x) \,d\mu(y) + \int_{(\Omega \setminus \Omega_{\sP})}S(x)\,d\mu(x)+I\\
&=\frac{1}{2}\int_{(\Omega' \setminus \Omega'_{\sP'})\times (\Omega' \setminus \Omega'_{\sP'})} W(x,y)\,d\mu(x) \,d\mu(y) + \int_{(\Omega' \setminus \Omega'_{\sP'})}S(x)\,d\mu(x)+I=I_{\sP'}.
\end{align*}
We also have for $x \in \Omega'$, if $x \in P_i$, then
\begin{align*}
S_{\sP}(x)&=\frac{1}{\mu(P_i)}\int_{P_i} \left(S(x)+\int_{\Omega \setminus \Omega_{\sP}}W(x,y)\,d\mu(y) \right)\,d\mu(x)\\
&=\frac{1}{\mu(P_i)}\int_{P_i} \left(S(x)+\int_{\Omega' \setminus \Omega'_{\sP}}W(x,y)\,d\mu(y) \right)\,d\mu(x)=S_{\sP'}(x),
\end{align*}
and $S_{\sP}(x)=0=S_{\sP'}(x)$ if $x \in \Omega' \setminus \Omega'_{\sP'}$.
Finally, if $x,y \in \Omega'$ and $x \in P_i,y \in P_j$ with $P_i,P_j \in
\cP'$, then
\[
W_{\sP}(x,y)=\frac1{\mu(P_i)\mu(P_j)}\int_{P_i \times P_j}W(x',y')\,d\mu(x')\,d\mu(y')=(W|_{\Omega'})_{sP'},
\] and
$0$ otherwise.
\end{proof}

\begin{proof}[Proof of Lemma \ref{lem:tight-relcomp}]
Let $\cW_1,\cW_2,\dots,\cW_n,\dots$ be a sequence of signed graphexes obeying
the condition \eqref{sign-tightD}, and let $\widetilde \cW_i$ be obtained
from $\cW_i$ by appending an arbitrary $\sigma$-finite space of infinite
measure. Then $\widetilde\cW_1,\widetilde\cW_2,\dots,\widetilde\cW_n,\dots$
obeys the condition \eqref{sign-tightD} as well, and arguing as in
Remark~\ref{rem:DC-mon}, we can assume without loss of generality that
$\widetilde\cW_1,\widetilde\cW_2,\dots,\widetilde\cW_n,\dots$ are all defined
over atomless spaces. We want to show that they have a subsequence that
converges to a graphex.

We can take for each $n$ and $k$ sets $\widetilde P_{n,k,0}$ and subspace
partitions $\widetilde\sP_{n,k}$ as in Lemma \ref{lemmapartitionsequence},
defining in particular $\widetilde\cW_{n,k}$ as the restriction of
$\widetilde\cW_n$ to ${\widetilde\Omega_n\setminus \widetilde P_{n,k,0}}$.
For $k\geq k_0$, we will also define $\widetilde\cW_{n,k,k_0}$ as the
restriction of $(\widetilde\cW_{n,k})_{\widetilde\sP_{n,k}}$ to
${\widetilde\Omega_n\setminus \widetilde P_{n,k_0,0}}$. By the above claim,
$\widetilde\cW_{n,k,k_0}=(\widetilde\cW_{n,k_0})_{\widetilde\sP_{n,k,k_0}}$,
where $\widetilde\sP_{n,k,k_0}$ consists of the classes in
$\widetilde\sP_{n,k}$ which are subsets of ${\widetilde\Omega_n\setminus
\widetilde P_{n,k_0,0}}$. This implies in particular that
$\widetilde\cW_{n,k,k_0}$ is $(B_{k_0},C_{k_0},D_{k_0})$-bounded.

Furthermore, in view of Remark~\ref{rem:stepgraphexesequiv}, we can replace
each $(\widetilde\cW_{n,k})_{\sP_{n,k}}$ by an equivalent step function
$\cW_{n,k}$ over $\RR_+$, where the first part is $P_{n,k,0}:=[0,2^{-2k})$ (which is disjoint from $\dsupp \cW_{n,k}$),
the remaining parts $P_{n,k,i}$ for $i \ge 1$ are of the form
$[\ell/2^{a_k},(\ell+1)/2^{a_k})$ for some nonnegative integer $\ell$, and we
extend $\cW_{n,k}$ to zero above $N_{k}+2^{-2k}$, where $N_k=2^{b_k}$.

Let $\sP_k'=\left([2^{-2k},N_{k}+2^{-2k}),\cP_k'\right)$, where $\cP_k'$
partitions $[2^{-2k},N_{k}+2^{-2k})$ into intervals of length $2^{-a_k}$.
Note that after this change, the bound \eqref{partitionsequenceregular} from
Lemma \ref{lemmapartitionsequence} becomes the bound
$\deltt(\widetilde\cW_{n,k},\cW_{n,k})\leq 2^{-k}$. In this way, we have
mapped the ``steps'' of each step graphex to $\cP_k'$. For each $k$, the
graphex $\cW_{{n,k}}$ then just depends on a finite number of parameters,
each bounded. We can then use a diagonalization argument to take a
subsequence so that for every $k$, $\cW_{{n,k}}$ converges to some $\cW^k$ as
$n \rightarrow \infty$, in the sense that $\cW^k$ is a step graphex with the
same parts and each value of the function converges, which also implies that
$\|\cW_{{n,k}}\|_1\to\|\cW^k\|_1$ and $\cW_{{n,k}}\to\cW^k$ in the metric
$\deltt$ (since there are a finite number of steps).

Given $k \ge k_0$, let $\cW^{k,k_0}$ be equal to $\cW^k$ restricted to
$[2^{-2k_0},\infty)$, let $\cP_{k,k_0}$ consist of those intervals in
$\cP_k'$ which are above $2^{-2k_0}$, and finally, let
$\sP_{k,k_0}=(\Omega_{k,k_0},\cP_{k,k_0})$ where
$\Omega_{k,k_0}=[2^{-2k_0},N_k+2^{-2k})$. We claim that
$\cW^{k,k_0}=\cW^{k+1,k_0}_{\sP_{k,k_0}}$. This follows from the fact that
for any $n$, if $\cW_{n,k,k_0}$ is $\cW_{n,k}$ restricted to
$[2^{-2k_0},\infty)$, then
$\left(\cW_{n,k+1,k_0}\right)_{\sP_{k,k_0}}=\cW_{n,k,k_0}$, by the above
claim.

Next, recalling that $\widetilde\cW_{n,k,k_0}$ is
$(B_{k_0},C_{k_0},D_{k_0})$-bounded for each $n$ and $k$, we have that
$\cW^{k,k_0}$ is $(B_{k_0},C_{k_0},D_{k_0})$-bounded as well, implying in
particular that $\|\cW^{k,k_0}\|_1 \le C_{k_0}$. Also note that if $P_i,P_j
\in \cP_{k,k_0}$,  then
\[\int_{P_i \times P_j} |W^{k,k_0}(x,y)| \,d\mu(x) \,d\mu(y)
\le \int_{P_i \times P_j} |W^{k+1,k_0}(x,y)| \,d\mu(x) \,d\mu(y).\] Since $W^{k,k_0}$ is supported on the union of $P_i \times P_j$ for all choices of $P_i$ and $P_j$, this
implies that $\|W^{k,k_0}\|_1$ cannot decrease as $k$ increases. Together,
these observations imply that the limit $\lim_{k' \to \infty}
\|W^{k',k_0}\|_1$ exists and is at most $C_{k_0}$. Given $\eps>0$, we can
therefore find $k(\eps,k_0)<\infty$ such that $\lim_{k' \to \infty}
\|W^{k',k_0}\|_1 - \|W^{k,k_0}\|_1 < \eps$ for all $k\geq k(\eps,k_0)$.
Furthermore, because $\Omega_{k,k_0}^2$ is the support of $W^{k,k_0}$,
\[\int_{\Omega_{k,k_0}^2} |W^{k,k_0}(x,y)| \,d\mu(x)\,d\mu(y)
 \le \int_{\Omega_{k,k_0}^2} |W^{k+1,k_0}(x,y)| \,d\mu(x)\,d\mu(y)
.\] As a consequence, for all $k' \ge k(\eps,k_0)$,
\[\int_{\RR_+^2 \setminus \Omega_{k,k_0}^2} |W^{k',k_0}(x,y)|\,d\mu(x) \,d\mu(y) < \eps
.\] Therefore, the random variables $W^{k,k_0}$ are uniformly integrable. By
the martingale convergence theorem (applied to each
$[2^{-2k_0},2^{-2k'}+N_{k'}] \times [2^{-2k_0},2^{-2k'}+N_{k'}]$ with
$k'\geq k_0$), the graphon part $W^{k,k_0}$ of $\cW^{k,k_0}$ is pointwise
convergent almost everywhere to a function $\wW^{k_0}$ defined on
$[2^{-k_0},\infty)^2$, and it also converges to $\wW^{k_0}$ in $L^1$. Since
$\|W^{k,k_0}\|_\infty\leq B_{k_0}$, this implies convergence in $L^2$, and
hence in the kernel metric $\|\cdot\|_{2\to 2}$. Furthermore, the graphon
marginals converge in $L^1$, because
\[
\|D_{W^{k,k_0}}-D_{\wW^{k_0}}\|_1\leq
\|W^{k,k_0}-\wW^{k_0}\|_1 \xrightarrow[k \to \infty]{} 0.
\]

We also have
\[\int_{P_i} |D_{\cW^{k,k_0}}(x)| \,d\mu(x)
 \le \int_{P_i} |D_{\cW^{k+1,k_0}}(x)| \,d\mu(x)
.\] This implies that $\|D_{\cW^{k,k_0}}\|_1$ cannot decrease as $k$
increases. Since $\Omega_{k,k_0}$ is the support of $D_{\cW^{k,k_0}}$,
\[\int_{\Omega_{k,k_0}} |D_{\cW^{k,k_0}}(x)| \,d\mu(x)
 \le \int_{\Omega_{k,k_0}} |D_{\cW^{k+1,k_0}}(x)| \,d\mu(x)
.\] As before, this implies that the functions $D_{\cW^{k,k_0}}$ are
uniformly integrable and uniformly bounded (by $D_{k_0}$). We can again use the
martingale convergence theorem (applied to each $[2^{-2k_0},2^{-2k'}+N_{k'}]$)
to show that $D_{\cW^{k,k_0}}$ converges pointwise and in $L^1$, and
therefore in $L^2$, to a function $D_{\cW^{k_0}}$ taking values in
$[-D_{k_0},D_{k_0}]$. Define
\[\widetilde{S}^{k_0}(x)=D_{\wcW^{k_0}}(x) - D_{\wW^{k_0}}(x).\]
Since we also have
\[S^{k,k_0}(x)=D_{\cW^{k,k_0}}(x)-D_{W^{k,k_0}}(x),\] and since both
terms in the difference converge in $L^1$, $S^{k,k_0}$ converges in $L^1$ to
$\widetilde{S}^{k_0}$. By Theorem~3.12 in \cite{Rudin} this implies that some
subsequence converges pointwise almost everywhere, showing in particular that
in the unsigned case, $\widetilde{S}^{k_0} \geq 0$ almost everywhere.

We also have that $\rho :=\rho(\cW_{k,k_0})$ is constant in $k$. Define
\[I^{k_0}=\rho - 2\int \widetilde{S}^{k_0} - \int \wW^{k_0}
.\] Since for each $k$,
\[ I^{k,k_0}=\rho -2\int S^{k,k_0} - \int W^{k,k_0}
,\] and since $S^{k,k_0}$ and $W^{k,k_0}$ converge in $L^1$ to
$\widetilde{S}^{k_0}$ and $\wW^{k_0}$, respectively, this implies that
$I^{k,k_0}$ converges to $\widetilde{I}^{k_0}$. Because $W^{k,k_0}$,
$S^{k,k_0}$, and $D_{\cW}^{k,k_0}$ converge in $L^1$ (and hence pointwise
almost everywhere on some subsequence), the limit inherits $(B_{k_0},
C_{k_0}, D_{k_0})$-boundedness from $\cW^{k,k_0}$.

Since for each $k$, $\cW^{k,k_0+1}$, when restricted to $[2^{-2k_0},\infty)$,
is equal to $\cW^{k,k_0}$, and since $\wcW^{k_0+1}$ and $\wcW^{k_0}$ are
pointwise limits along some subsequence, $\wcW^{k_0}$ is also the restriction
of $\wcW^{k_0+1}$ to $[2^{-2k_0},\infty)$. Therefore, we can define a signed
graphex $\cW$ on $\RR^+$ as the ``union'' of the signed graphexes
$\wcW^{k_0}$. Note that the limit $\cW$ is locally finite by the fact that
$\widetilde \cW^{k_0}$ is $(B_{k_{0}},C_{k_0},D_{k_0})$-bounded. This also
implies that $\cW$ inherits any of these bounds from $\widetilde \cW^{k_0}$
that do not depend on $k_0$.

We claim that on the subsequence where $\cW_{{n,k}}$ converges to $\cW^k$ in
$\deltt$ for each $k$, $\cW_n\to \cW$ in the weak kernel metric $\delGP$. To
see this, we first note that $\wcW^{k_0}$ is obtained from $\cW$ by removing
a set of measure $2^{-2k_0}$, and $\cW^{k,k_0}$ is obtained from $\cW^k$ by
removing a set of the same measure, showing that for any $k_0$,
\[
\delGP(\cW, \cW^k)\leq \max\{2^{-{k_0}},\deltt(\wcW^{k_0},\cW^{k,k_0})\}.
\]
Given $\eps>0$, choose $k_0$ such that $2^{-{k_0}}\leq \eps/2$. Since
$\deltt(\cW^{k,k_0},\widetilde\cW^{k_0}) \to 0$ for each $k_0$ as $k \to
\infty$, this shows that for $k$ large enough, $\delGP(\cW, \cW^k)\leq
\eps/2$. In a similar way,
\[
\delGP(\cW_n,\cW^k)=\delGP(\widetilde\cW_n,\cW^k)
\leq \max\{2^{-k},\deltt(\widetilde\cW_{n,k},\cW^k)\}.
\]
Since $\deltt(\widetilde\cW_{n,k},\cW_{n,k})\leq 2^{-k}$ and
$\deltt(\cW_{n,k},\cW^k) \to 0$ for each $k$ as $n \to \infty$, we can first
choose $k$ and then $n$ large enough to guarantee that the right side is
smaller than $\eps/2$. Combined with the triangle inequality for $\delGP$,
this shows that for all $\eps>0$ we can find an $n_0$ such that for $n\geq
n_0$, we have $\delGP(\cW_n,\cW)\leq\eps$, as claimed.
\end{proof}

\begin{remark}\label{rem:signed-compactness-pf}
It is not hard to see that a sequence of signed graphexes
$\cW_n=(W_n,S_n,I_n,\bOmega_n)$ with $\|W_n\|_\infty\leq B$ and
$\|\cW_n\|_1\leq C$ obeys the condition \eqref{sign-tightD} from
Lemma~\ref{lemmapartitionsequence}. If $B\leq 1$, this follows by applying
Corollary~\ref{cor:C-bounded-tight} (1) to the sequence $|\cW_n|$, and for
$B>1$ it follows by applying Corollary~\ref{cor:C-bounded-tight} (1) to the
sequence $(|W_n|/B,|S_n|/B,|I_n|/B,\bOmega_n)$. Choosing a convergent
subsequence, it is clear from the last proof that the limiting graphex
$\cW=(W,S,I,\bOmega)$ must obey the bound $\|W\|_\infty\leq B$. The statement
from Remark~\ref{rem:signed-compactness} therefore is a direct consequence of
Lemma~\ref{lem:tight-relcomp}.
\end{remark}

\begin{proof}[Proof of Theorem~\ref{thm:complete}]
By Corollary~\ref{cor:C-bounded-tight} (1), any set of graphexes whose $L^1$
norms are bounded by $C$ is tight, so by Lemma~\ref{lem:tight-relcomp} any
such sequence of graphexes has a subsequence with a limit $\cW$ in the metric
$\delGP$. Lemma~\ref{lem:tight-relcomp} also implies that the limit inherits
the bound on the $L^1$ norm, and therefore the set of graphexes whose $L^1$
norms are bounded by $C$ is compact. The same proof gives the statement of
the theorem for $(C,D)$-bounded graphexes.

Suppose now that $\cW_1,\cW_2,\dots,\cW_n,\dots$ is a Cauchy sequence in
$\delGP$. We first claim that it must be tight. Indeed, for any
$\varepsilon>0$, there exists $n$ such that for any $m>n$,
$\delGP(\cW_n,\cW_m)<\varepsilon$. Fix such an $n$ and an $m>n$. By
Lemma~\ref{lem:tilde-delGP}, we can then decrease the measures $\mu_n,\mu_m$
by at most $\varepsilon^2$ such that the $\deltt$ distance of the resulting
graphexes $\widehat W_n$ and $\widehat W_m$ is less than $\varepsilon$. Let
$\cW_n'$, $\cW_m'$ be trivial extensions of the modified graphexes by spaces
of infinite measure, and let $\widetilde\cW_n'$, $\widetilde\cW_m'$ be
pullbacks according to a coupling so that
$\d22(\widetilde{\cW}_n',\widetilde{\cW}_m')<\varepsilon$, which exists by
the definition of $\deltt$. Furthermore, since every finite set of graphexes
is tight, we can find $(C,D)$ such that we can remove a set of measure
$\varepsilon^2$ from $\widehat\cW_n$ to make it $(C,D)$-bounded (independent
of $m$). If we remove the pullback of this set from the underlying space of
$\widetilde\cW_n'$ and $\widetilde\cW_m'$ and replace them by the
restrictions, then the two graphexes will still have $\deltt$ distance at
most $2\varepsilon$ by Lemma \ref{lemmaremovedsetdistancenotworse}. In
particular, this means that $\|\widetilde\cW_m'\|_1 \le C+8\varepsilon^3$. We
also have
\[\|D_{\widetilde\cW_m'}\|_2 \le \|D_{\widetilde\cW_n'}\|_2+\|D_{\widetilde\cW_m'}-D_{\widetilde\cW_n'}\|_2 \le \sqrt{CD}+4\varepsilon^2 \le
2\sqrt{CD} .\] Therefore, the measure of the points $x$ for which
$D_{\widetilde\cW_m'}(x)>\sqrt{CD/\varepsilon}$ is at most $4\varepsilon$.
Taking $C'=C+\varepsilon$ and $D'=\sqrt{CD/\varepsilon}$, we have obtained
that for any $m>n$ we can remove a set of measure at most $6\varepsilon$ so
that the remainder is $(C',D')$-bounded. Since any finite set is tight, and
the union of two tight sets is tight, this means that the entire sequence is
tight. Therefore, it must have a convergent subsequence that converges to a
graphex $\cW$. But then because the original sequence was a Cauchy sequence,
the entire sequence must converge to $\cW$. This proves that the space of
graphexes is complete.

The above lemma implies that every tight set is relatively compact, and the
fact that any Cauchy sequence must be tight implies that every relatively
compact set is tight.
\end{proof}

\section{Subgraph counts}
\label{sec:subgraph-counts}

In this section we will prove that convergence in the weak kernel metric
implies GP-convergence. The main technical tool for this proof will be the
following counting lemma, which says that given any $C,D<\infty$, two
$(C,D)$-bounded graphexes that are close in kernel metric $\deltt$ must have
close subgraph counts.

While this lemma and its corollary are formulated only for unsigned
graphexes, we note that both have natural generalizations to signed
graphexes.  See Remark~\ref{rem:Fclose} at the end of
Section~\ref{sec:counting} below.

\begin{lemma} \label{lemmatFclose}
Let $F$ be a simple, connected graph with $m$ edges and $n\geq 3$ vertices,
and let $C,D <\infty$. Suppose that $\cW_1$ and $\cW_2$ are graphexes on the
same underlying space $\bOmega$, with $\|\cW_i\|_1 \le C$ and
$\|D_{{\cW}_i}\|_\infty \le D$ for $i=1,2$, and let
$\varepsilon=\max\{\|W_1-W_2\|_{2 \rightarrow 2},\|D_{\cW_1-\cW_2}\|_2\}$.
Then
\[\
|t(F,\cW_1)-t(F,\cW_2)| \le m\varepsilon \widetilde CD^{n-3},
\]
where $\widetilde C=\max\{C,\sqrt{CD}\}$.
\end{lemma}

\begin{corollary} \label{corrolaryFclose}
Suppose $\cW_n$, $\cW$ have uniformly bounded marginals.
If
\[
\deltt(\cW_n,\cW) \rightarrow 0,
\]
then for any finite graph $F$ with no isolated vertices, $t(F,\cW_n)$
converges to $t(F,\cW)$.
\end{corollary}

\begin{proof}
By the definition of $\deltt$, if $F$ is an edge, then $t(F,\cW_n)
\rightarrow t(F,\cW)$. This implies in particular that $\|\cW\|_1$ and
$\|\cW_n\|_1$ are uniformly bounded. By Lemma \ref{lemmatFclose}, it follows
that $t(F,\cW_n) \rightarrow t(F,\cW)$ for any connected graph $F$. Since
homomorphism densities factor over connected components of the graph $F$,
this means that if $F$ is a finite graph without isolated vertices, then
$t(F,\cW_n) \rightarrow t(F,\cW)$.
\end{proof}

In a addition to the above counting lemma, we will need to show that
convergence of subgraph counts implies GP-convergence. Thinking of the
subgraph counts as the moments of a graphex, this result is similar to
standard moment theorems for random variables that show that under suitable
growth conditions, the distribution of a random variable is determined by its
moments.

\begin{theorem} \label{theoremgraphexconveq}
Assume that the marginals of $\cW_n$ and $\cW$ are bounded by some finite
constant $D$. Then the following are equivalent:
\begin{enumerate}
\item $G_T(\cW_n)\rightarrow G_T(\cW)$ in distribution for every $T$. \label{graphexconvallT}
\item $G_T(\cW_n)\rightarrow G_T(\cW)$ in distribution for some $T$. \label{graphexconvoneT}
\item For every graph $F$ with no isolated vertices, $t(F,\cW_n) \rightarrow t(F,\cW)$. \label{graphexconvnoisol}
\item For every connected graph $F$, $t(F,\cW_n) \rightarrow t(F,\cW)$. \label{graphexconvconn}
\end{enumerate}
\end{theorem}

We will prove the counting lemma in Subsection~\ref{sec:counting} below, and
Theorem~\ref{theoremgraphexconveq} in
Subsection~\ref{sec:samplingandsubgraphs}. In the final subsection,
Subsection~\ref{sec:metric-implies-GP}, we use these results to first show
that under the assumption of uniformly bounded marginals,
$\deltt$-convergence implies GP-convergence
(Theorem~\ref{thm:deltt-implies-GPconv} below). With the help of the results
about tightness established in Section~\ref{sec:tight}, this in turn allows
us to show that without any assumption on the marginals, $\delGP$-convergence
implies GP-convergence (Theorem~\ref{theoremdimpliessampling} below).

\subsection{Proof of the counting lemma} \label{sec:counting}

In order to prove the counting lemma, it will be convenient to consider
several variants of the homomorphism densities. For these variants, it will
be natural to consider signed graphexes, since we will need to consider
differences of graphexes for the proof of the counting lemma anyway. Note
that our proof of the counting lemma can easily be generalized to signed
graphons; see Remark~\ref{rem:Fclose} below.

\begin{definition}
Suppose we have a connected multigraph $F$, and signed graphexes $\cW_e$
assigned to each edge $e \in E(F)$ (refer to this vector of graphexes as
$\cW_F$), each with the same feature space $\Omega$. Let $V_{\ge 2}$ be the
set of vertices with degree at least $2$. If $V_{\ge 2}$ is nonempty (i.e.,
$F$ does not consist of a single edge), then we define
\begin{multline*}
t(F,\cW_F) = \int_{\Omega^{V_{\ge 2}}}dz_{V_{\ge 2}} \prod_{\{v,w\} \in E(F(V_{\ge 2}))}
W_{\{v,w\}}(z_v,z_w)\\
\phantom{} \cdot \prod_{v \in V_{\ge 2}} \prod_{\substack{w \in V \setminus V_{\geq 2}:\\\{v,w\} \in E(F)}} D_{\cW_{\{v,w\}}}(z_v).
\qquad\end{multline*}
If $F$ consists of just a single edge $f$, then
\[t(F,\cW_F)=
\rho(\cW_f),
\]
with $\rho(\cW_f)$ as in \eqref{rho-of-W-def}. Note that for signed graphexes
$t(F,\cW_F)$ is in general only well defined if the integrals are absolutely
convergent, a condition which can, e.g., be guaranteed by requiring that
$t(F,\cW_F^{\text{abs}})<\infty$, where $\cW_F^{\text{abs}}$ is obtained from
$\cW_F$ by replacing the graphexes $\cW_f$ by $|\cW_f|$.
\end{definition}

Note that $t(F,\cW_F)$ is a multilinear function of the signed graphexes
$\cW_e$ in $\cW_F$. If each $\cW_e$ is equal to some fixed $\cW$, then this
is just the previous definition.

We also define a conditional density where we fix the image of a single
vertex.

\begin{definition}
Suppose we have a connected multigraph $F$ with a labeled vertex $v_0$, more
than one edge, and signed graphexes $\cW_e$ assigned to each edge as before.
Let $x \in \Omega$, and define $z_{v_0}=x$. Take, furthermore, $V_{\geq
2}'=V_{\geq 2} \setminus \{v_0\}$ and $\widetilde{V}_{\geq 2}=V_{\geq 2} \cup
\{v_0\}$. Then we define $t_x(F,\cW_F)$ to be
\[
\int_{\Omega^{V_{\geq 2}'}} dz_{V_{\geq 2}'} \prod_{\{v,w\} \in
E(F(\widetilde V_{\geq2}))} W_{\{v,w\}}(z_v,z_w)
\prod_{v \in
{V_{\ge 2}}}
\prod_{\substack{w \in V \setminus (\widetilde V_{\geq 2}):\\ \{v,w\} \in E(F)}} D_{\cW_{\{v,w\}}}(z_v) .
\]
If $F$ consists of just a single edge $f$ adjacent to $v_0$, then we set
$t_x(F,\cW)=D_{\cW}(x)$. Again, for signed graphexes, this is in general only
well defined if the integrals are absolutely convergent.
\end{definition}

Note that if $v_0\in V_{\geq2}$, then $ t_x(F,\cW_F)$ is obtained from $
t(F,\cW_F)$ by simply fixing the feature corresponding to $v_0$ to be $x$,
implying in particular that $\int t_x(F,\cW_F)\, dx=t(F,\cW_F)$ (assuming the
integrals defining these are absolutely convergent). If $v_0\notin
V_{\geq2}$, i.e., if $v_0$ has degree $1$, and if $F$ has more than one edge,
then the situation is slightly more complicated, since the ``feature'' of the
image of $v_0$ could be either an element of $\Omega$, or the special value
$\infty$, interpreted earlier as the feature of the leaves of the star part
of a graphon process. With this reinterpretation, $ t_x(F,\cW_F)$ is still
obtained from $ t(F,\cW_F)$ by fixing the feature corresponding to $v_0$ to
be $x$, but the integral $\int t_x(F,\cW_F)\, dx$ now misses the contribution
of $x=\infty$, and hence is in general only bounded above by $t(F,\cW_F)$. It
is, however, equal to $t(F,\cW'_F)$, where $\cW'_F$ is obtained from $\cW_F$
by setting the star part of the graphex corresponding to the edge containing
$v_0$ to $0$.

\begin{lemma}
\label{lem:tx-bound} Suppose $F$ is a connected multigraph with no loops and
a labeled vertex $v_0$, $T$ is a spanning tree, and $\cW_e$ is a signed
graphex corresponding to each edge $e$, each with the same feature space
$\Omega$. Let $f \in T$ be an edge adjacent to $v_0$, and $x \in \Omega$.
Then
\[|t_x(F,\cW_{F})| \le D_{|\cW_f|}(x) \prod_{e \in T \setminus f} \|D_{|\cW_e|}\|_\infty \prod_{e \in E(F) \setminus T} \|W_e\|_\infty.
\]
\end{lemma}

\begin{proof}
Replacing all signed graphexes $\cW_e$ by the non-negative versions
$|\cW_e|$, and noting that $|t_x(F,\cW_F)| \le t_x(F,\cW_F^{\text{abs}})$, we
may without loss of generality assume that $W_e$ and $S_e$ are non-negative.

Next, assume that $F$ is a tree, i.e., $F=T$. We then prove the claim by
induction on the number of edges. If $F$ consists of a single edge, then by
definition $t_x(F,\cW)=D_{\cW}(x)$, which is exactly the bound in the lemma.
Otherwise, we can find an edge $\{v,w\}$ not equal to $f$, such that $v$ has
degree at least two, and $w$ has degree $1$ and is different from $v_0$. The
edge $\{v,w\}$ then contributes $D_{\cW_{\{v,w\}}}(z_v)$ to the second
product in the integral representing $t_x(F,\cW_F)$. For each $z_v$, this is
at most $\|D_{\cW_{\{v,w\}}}\|_\infty$. Taking a factor
$\|D_{\cW_{\{v,w\}}}\|_\infty$ out of the integral, and defining $F'$ to be
the restriction of $F$ to $V(F)\setminus\{w\}$, we therefore have that
\[t_x(F,\cW_F) \le \|D_{\cW_{\{v,w\}}}\|_\infty t_x(F',\cW_{F'})
.\] Note that this bound is actually weaker than necessary if the vertex $v$
becomes a vertex of degree one in $F'$, in which case we could have obtained
a contribution of $D_{W_{\{v,v'\}}}(z_{v'})$ for its neighbor $v'$ instead of
the contribution $D_{\cW_{\{v,v'\}}}(z_{v'})$ implicit in the above bound.

Suppose now that $F$ has edges outside $T$. Let $\{v,w\}$ be such an edge.
Note that both $v$ and $w$ must be in $V_{\ge 2}$. Therefore, this edge
contributes $W_{\{v,w\}}(z_v,z_w)$ to the product, which is at most
$\|W_{\{v,w\}}\|_\infty$. We can therefore conclude the lemma by induction on
the number of edges of $F$ outside $T$.
\end{proof}

\begin{lemma}\label{lem:tmultigraph}
Suppose $F$ is a connected multigraph with no loops, and we have a signed
graphex $\cW_e$ corresponding to each edge $e \in F$. Let $T$ be any spanning
tree in $F$, and $f \in T$. Then
\[
|t(F,\cW_{F})| \le \|\cW_f\|_1 \prod_{e \in T \setminus f} \|D_{|\cW_e|}\|_\infty \prod_{e \in E(F) \setminus T} \|W_e\|_\infty.
\]
\end{lemma}

\begin{proof}
If $F$ consists of a single edge $f$, then $|t(F,\cW_F)|=|\rho(\cW_f)|\leq
\|\cW_f\|_1$ by definition, and if $F$ has more than one edge, then
\begin{align*}
|t(F,\cW_F)|&\leq
\int_\Omega| t_x(F,\cW_{F})| \,d\mu(x)\\
 &\le \int_\Omega D_{|\cW_f|}(x) \prod_{e \in T \setminus f} \|D_{|\cW_e|}\|_\infty \prod_{e \in E(F) \setminus T} \|W_e\|_\infty \,d\mu(x)\\
&\leq
\|\cW_f\|_1 \prod_{e \in T \setminus f} \|D_{|\cW_e|}\|_\infty \prod_{e \in E(F) \setminus T} \|W_e\|_\infty. \qedhere
\end{align*}
\end{proof}

\begin{proof}[Proof of Lemma~\ref{lemmatFclose}]
Let $f=\{u,v\}$ be an edge in $F$, and let $\cW_{F,f}$ be a vector of
graphexes where we assign one of $\cW_1$ or $\cW_2$ to each edge $e \ne f$,
and $(W_1-W_2,S_1-S_2,0,\bOmega)$ to $f$ (since $F$ is a connected graph with
at least $2$ edges, the dust parts of $\cW_1$ and $\cW_2$ don't contribute to
$t(F,\cW_1)-t(F,\cW_2)$ and can be set to $0$). We would like to bound
$|t(F,\cW_{F,f})|$.

First, assume both endpoints of $f$ have degree at least $2$. Let the
components of $F$ restricted to $V(F) \setminus \{u,v\}$ be
$C_1,C_2,\dots,C_k$, with corresponding vertex sets $V_1,\dots,V_k$. There
can be three types of components: those with at least one edge to $u$ but
none to $v$, those with at least one edge to $v$ but none to $u$, and those
with at least one edge to both. Let $\cC_u$ be the set of components
connected to $u$, $\cC_v$ the set of those connected to $v$, and $\cC_{uv}$
the set connected to both. For each $i \in \cC_u$, let $F_i$ be the labeled
graph where we add $u$ back to $C_i$ as the labeled vertex, and for each $i
\in \cC_y$, let $F_i$ be the labeled graph where we add $v$ back to $C_i$ as
the labeled vertex. Furthermore, for each $i\in \cC_u\cup \cC_v$, choose an
additional vertex $v_i\in V_i$ such that $v_i$ is incident to an edge in
$F_i$. Let $V_{uv}$ consist of vertices that belong to a component in
$\cC_{uv}$.

Given a set of vertices $U$, let $U'$ be the set of vertices in $U$ that have
degree at least $2$ in $F$, and let
\[
W_{F,U,u}(z_u,z_{U'})
=\prod_{\substack{w \in {U'}\\ \{u,w\} \in E(F)} }
W_{\{u,w\}}(z_u,z_w)
\]
and
\[
W_{F,U,v}(z_v,z_{U'})=\prod_{\substack{w \in {U'}\\ \{v,w\} \in E(F)\\}} W_{\{v,w\}}(z_v,z_w)
.
\]
Let us also use the notation
\[
\cW_{{F,U}}(z_{U'})=
\prod_{\{w,w'\} \in E(F(U'))}W_{\{w,w'\}}(z_{w},z_{w'})
{\prod_{w\in U'}}
\prod_{\substack{w' \in U \setminus U'\\\{w,w'\} \in E(F)}} D_{\cW_{\{w,w'\}}}(z_w).
\]
Observe that if a vertex in $V_{uv}$ is adjacent to $u$ or $v$, it must be in
$V_{uv}'$. We therefore express $t(F,\cW_{F,f})$ as
\begin{align*}
\int_{\Omega^{{V_{uv}'}}} &\cW_{{F,V_{uv}}}(z_{V_{uv}'})
\int_{\Omega^2} \Biggl(W_{F,V_{uv},v}(z_v,z_{{V_{uv}'}}) W_{F,V_{uv},u}(z_u,z_{{V_{uv}'}})W_f(z_u,z_v)
\Biggr.
\\
&\Biggl.\prod_{i \in \cC_u} t_{z_u}(F_i,\cW_{F,f})
\prod_{i \in \cC_v} t_{z_v}(F_i,\cW_{F,f})\Biggr),
\end{align*}
and bound the inner integral by
\begin{align*}
&
\left\|W_{F,V_{uv},u}({\cdot},z_{V_{uv}'})
\prod_{i \in \cC_u} t_{\cdot}(F_i,\cW_{F,f})\right\|_2\\
& \qquad \qquad
\left\|W_{F,V_{uv},v}({\cdot},z_{V_{uv}'})
\prod_{i \in \cC_v} t_{\cdot}(F_i,\cW_{F,f})\right\|_2\|W_f\|_{2 \rightarrow 2}
\\
&\leq
\left\|W_{F,V_{uv},u}({\cdot},z_{V_{uv}'})
\prod_{i \in \cC_u} D_{\cW_{uv_i}}\right\|_2\\
& \qquad \qquad
\left\|W_{F,V_{uv},v}({\cdot},z_{V_{uv}'})
\prod_{i \in \cC_v} D_{\cW_{vv_i}}\right\|_2
\|W_f\|_{2 \rightarrow 2}\prod_{i\in \cC_v\cup\cC_u} D^{|V_i|-1},
\end{align*}
where in the last step we used Lemma~\ref{lem:tx-bound} and the fact that the
number of edges in a spanning tree for $F_i$ is $|V(F_i)|-1=|V_i|$. Inserting
this bound into the outer integral, an application of the Cauchy-Schwartz
inequality then gives the bound
\begin{align*}
t(F,\cW_{F,f})&\le
\|W_f\|_{2 \rightarrow 2}
\sqrt{\int_{\Omega^{V_{uv}'}} \cW_{{F,V_{uv}}}(z_{V_{uv}'})
\left\|W_{F,V_{uv},u}({\cdot},z_{V'_{uv}})
\prod_{i \in \cC_u} D_{\cW_{uv_i}}\right\|_2^2}
\\
&\cdot
\sqrt{\int_{\Omega^{{V_{uv}'}}}\cW_{{F,V_{uv}}}(z_{V'_{uv}})
\left\|W_{F,V_{uv},v}({\cdot},z_{V'_{uv}})
\prod_{i \in \cC_v} D_{\cW_{vv_i}}\right\|_2^2}
\prod_{i\in \cC_v\cup\cC_u} D^{|V_i|-1}
\end{align*}
We claim that the expressions under the square roots can be written as
$t(F'_u,\cW_{F,f})$ and $t(F'_v,\cW_{F,f})$ for some suitable multigraphs
$F'_u$ and $F'_v$. Indeed, starting from $F(V_{uv}\cup\{u\})$, we first
duplicate every edge in this graph that joins $u$ to some vertex in $V_{uv}$,
keeping the edges between vertices in $V_{uv}$ as simple edges. The graph
$F'_u$ is obtained from this graph by adding two more edges for each
component $i\in\cC_u$: the edge $uv_i$, and a second edge $uv_i'$, with
$v_i'$ being a new vertex we should think of as a twin of $v_i$ (in $F'_u$,
they both have degree one and are connected to $u$). This gives a connected
multigraph on $V_u'=V(F_u')$ with $|V_{uv}|+1+2k_u$ many vertices and
$|E(F(V_{uv}))|+2d_u+2k_u$ many edges where $d_u$ is the number of vertices
$v'\in V_{uv}$ such that $uv'$ is an edge in $F$, and $k_u=|\cC_u|$. Define
$F'_v$ (as well as $d_v$ and $k_v$) analogously. We then can reexpress the
above bound as
\begin{align*}
t(F,W_{F,f})&\le
\|W_f\|_{2 \rightarrow 2}
\sqrt{t(F'_u,\cW_{F,f})t(F'_v,\cW_{F,f})}
\prod_{i\in \cC_v\cup\cC_u} D^{|V_i|-1}
\\
&\le
\|W_f\|_{2 \rightarrow 2}C D^{|V_{uv}|+k_u+k_v-1}
\prod_{i\in \cC_v\cup\cC_u} D^{|V_i|-1}
\\
&=\|W_f\|_{2 \rightarrow 2}C
D^{n-3}.
\end{align*}

If $f$ has one endpoint with degree $1$, and the other endpoint with degree
at least $2$, then let $v$ be the endpoint with degree at least $2$, and let
$u'\neq u$ be a neighbor of $v$. Let $\cW_{F-f}$ be the graphex assignment
restricted to the edges in $F-f$. Then
\begin{align*}
|t(F,\cW_{F,f})|&=\bigg|\int_{\Omega} t_{x_v}(F-f,\cW_{F-f})D_{{\cW_1-\cW_2}}(x_v)\,d\mu(x_v)\bigg|\\
& \le \|D_{\cW_1-\cW_2}\|_2 \|t_{\cdot}(F-f,\cW_{F-f})\|_2
\\
&\leq \|D_{\cW_1-\cW_2}\|_2\|D_{\cW_{vu'}}\|_2D^{n-3}
\leq \sqrt {CD}D^{n-3}
\|D_{\cW_1-\cW_2}\|_2.
\end{align*}
Now, let $e_1,e_2,\dots,e_m$ be the edges of $F$. Let $\cW_{F,i}$ be the
vector of graphexes where we assign $\cW_1-\cW_2$ to $e_i$, $\cW_1$ to $e_j$
with $j<i$, and $\cW_2$ to $e_j$ with $j>i$. Then,
\[
\Bigl|t(F,\cW_1)-t(F,\cW_2)\Bigr|\leq
\sum_{i=1}^m \Bigl|t(F,\cW_{F,i})\Bigr|
\le m\eps\widetilde CD^{n-3}
. \qedhere\]
\end{proof}

\begin{remark}
It is instructive to note that the bound in Lemma~\ref{lemmatFclose} can be
tightened to give the constant
\[
\widetilde C=\max_i\max\{\|W_i\|_2,\|D_{\cW_i}\|_2\}
\]
instead of the constant $\widetilde C=\max\{C,\sqrt{CD}\}$. To see this, we
first note that near the end of the proof, we bounded $\|D_{\cW_{vu'}}\|_2$
by $\sqrt {CD}$, even though it is possible that the first term is finite
while the second is infinite. In a similar way, bounding the integral
representing $t(F'_u,\cW_{F,f})$ and $t(F'_v,\cW_{F,f})$ with the help of
Lemma~\ref{lem:tmultigraph} is suboptimal. Indeed, the multi-graph $F'_v$
always contains at least one double edge, or contain at least one edge $uv_i$
and its twin $uv_i'$, with both $v_i$ and $v_i'$ having degree one. In the
first case, Lemma~\ref{lem:tmultigraph} can be improved to extract a factor
$\|W_f\|_2$ instead of a factor $\|\cW_f\|_1$, in the second it can be
improved to extract a factor $\|D_{\cW_{f}}\|_2$. Inserted into the proof of
Lemma~\ref{lemmatFclose}, this gives the claimed improvement.
\end{remark}

\begin{remark}\label{rem:Fclose}
As the reader can easily verify, the above proof immediately generalizes to
signed graphexes, showing that Lemma~\ref{lemmatFclose} holds for $(B,C,D)$
bounded graphexes, provided we include a factor of $B^{m-(n-1)}$ on the right
side. As a consequence, Corollary~\ref{corrolaryFclose} holds for sequences
of $(B,C,D)$-bounded graphexes that converge in the kernel distance $\deltt$.
\end{remark}

\subsection{GP-convergence and subgraph counts} \label{sec:samplingandsubgraphs}

In this subsection, we prove Theorem~\ref{theoremgraphexconveq}. We start by
establishing the following theorem.

\begin{theorem} \label{theoremconveq}
Let $\mathbb{G}$ and $\mathbb{G}_n$, for $n\geq 1$, be random finite graphs
with no isolated vertices, and let $X$ and $X_n$, for $n\geq 1$, be the
random variables that correspond to the number of vertices in $\mathbb{G}$
and $\mathbb{G}_n$, respectively. If, for every $t>0$, $\EE[e^{t X}]$ and
$\EE[e^{t X_n}]$ are finite and uniformly bounded, then the following are
equivalent:
\begin{enumerate}
\item For any graph $G$, the probability that $\mathbb{G}_n$ is isomorphic
    to $G$ converges to the probability that $\mathbb{G}$ is isomorphic to
$G$. \label{convdist}
\item For every graph $F$, $\EE[\inj(F,\mathbb{G}_n)] \rightarrow
    \EE[\inj(F,\mathbb{G})]$. \label{convexpinjany}
\item For every graph $F$ with no isolated vertices,
    $\EE[\inj(F,\mathbb{G}_n)] \rightarrow \EE[\inj(F,\mathbb{G})]$.
    \label{convexpinjnoisol}
\end{enumerate}
\end{theorem}

To prove the theorem, we will first establish a couple of lemmas. As a
preparation, note that if $\EE[e^{tX} ]\le C$, then for any graph $F$ on $k$
vertices,
\[\EE[\inj(F,\mathbb{G})] \le \sum_{n=0}^{\infty} P(X=n) n^k
\le
\sum_{n=0}^{\infty} P(X=n)
{\frac{k!}{t^k}} e^{tn}
=
{\frac{k!}{t^k}}\EE[e^{tX}]\leq C'.
\]
Here we use the fact that $e^{tn} \ge \frac{(tn)^k}{k!} $. The same bound
holds for $\mathbb{G}_n$. In other words, for any graph $F$ the values
$\EE[\inj(F,\mathbb{G}_n)]$ and $\EE[\inj(F,\mathbb{G})]$ are bounded
uniformly in $n$.

Our first lemma roughly says that if a random graph model does not have
isolated vertices, and the number of vertices is not too large with high
probability, then the expected number of counts of finite graphs without
isolated vertices determines the expected number of counts of all graphs.

\begin{lemma} \label{lemmanoisolallequal}
Suppose we have two random finite graphs with no isolated vertices,
$\mathbb{G}$ and $\mathbb{G}'$. Suppose that for every finite graph $F$ with
no isolated vertices, $\EE[\inj(F,\mathbb{G})]=\EE[\inj(F,\mathbb{G}')]$. Let
$X$ be the random variable that gives the number of vertices in $\mathbb{G}$,
and suppose that for every $t$, $\EE[e^{tX}]$ is finite. Then
$\EE[\inj(F,\mathbb{G})]$ and $\EE[\inj(F,\mathbb{G}')]$ are equal for
\emph{every} finite graph $F$.
\end{lemma}

\begin{proof}
We prove this by induction on the number of isolated vertices in $F$. If $F$
has zero isolated vertices, the claim is true by the assumptions of the
lemma. Otherwise, let $F$ consist of $F'$ plus an isolated vertex $w_0$. For
every $k$, and graph $G$, let $\inj^*(F',k,G)$ be equal to the number of ways
we can take an injective image of $F'$ in $G$, take a vertex $v_0$ not in the
image of $F'$, and take a $k$-term sequence of distinct neighbors of $v_0$
(which may or may not be in the image of $F'$). If $G$ has no isolated
vertices, then
\[
\inj(F,G)=\sum_{k=1}^\infty \frac{(-1)^{k-1}}{k!} \inj^*(F',k,G).
\]
This follows from the fact that for each injective copy of $F$, if $v_0$ is
the image of $w_0$, then this contributes $(d(v_0))_k$ to $\inj^*(F',k,G)$.
Since $G$ is finite and has no isolated vertices, we must have
$0<d(v_0)<\infty$; therefore,
\[\sum_{k=1}^\infty \frac{(-1)^{k-1}}{k!}(d(v_0))_k = \sum_{k=1}^{d(v_0)}
(-1)^{k-1}\binom{d(v_0)}{k}=1
.\]
Next, we claim that
\[
\EE[\inj(F,\mathbb{G})]=\sum_{k=1}^\infty \frac{(-1)^{k-1}}{k!}\EE[\inj^*(F',k,\mathbb{G})].
\]
To show this, it suffices to show that
\[
\sum_{k=1}^\infty \frac{\EE[\inj^*(F',k,\mathbb{G})]}{k!}<\infty.
\]
Fix $k$. Then $\inj^*(F',k,G)$ is the sum of terms of the form $\inj(F'',G)$.
The terms are obtained as follows. Let $L=V(F') \cup z_0$, where $z_0$ is
disjoint from $V(F')$. Let $\ell=(\ell_1,\ell_2,\dots,\ell_k)$ be a sequence
of $k$ elements of $L$. Suppose further that any vertex in $V(F')$ appears at
most once in $\ell$. Then we define $F''(\ell)$ by taking a copy of $F'$, a
disjoint vertex $w_0$, for each $\ell_i \in V(F')$, we add an edge from $w_0$
to $\ell_i$, for each other $\ell_i$ we add an edge going to a new vertex. For example, if $\ell=(z_0,z_0,z_0,\dots,z_0)$, then $F''(\ell)$ is the
disjoint union of $F'$ and a star with $k$ edges. It is then not difficult to
see that
\[
\inj^*(F',k,G)=\sum_\ell \inj(F''(\ell),G).
\]

Let $a=|L|$. For each $k$, the number of such sequences is at most $a^k$.
Furthermore, for each such $F''(\ell)$, the number of vertices is at most
$a+k$. Thus,
\[\inj^*(F',k,G) \le a^k (|V(G)|)_{(a+k)}.
\]

Therefore, recalling that $X=|\mathbb{G}|$, we have that
\begin{align*}
\sum_{k=1}^\infty &\frac{\EE[\inj^*(F',k,\mathbb{G})]}{k!}
\le \sum_{k=1}^\infty \frac{a^k \EE[(X)_{(a+k)}]}{k!}\\
&\qquad=\sum_{k=1}^\infty \EE \left[ a^k X_{(a)} \binom{X-a}{k}\right]
=\EE \left[X_{(a)} (a+1)^{X-a}\right]
<\infty.
\end{align*}

Finally, we claim that
$\EE[\inj^*(F',k,\mathbb{G})]=\EE[\inj^*(F',k,\mathbb{G}')]$ for every $k$.
This follows from the fact that the graphs $F''(\ell)$ above each have fewer
isolated vertices than $F$, so
$\EE[\inj(F''(\ell),\mathbb{G})]=\EE[\inj(F''(\ell),\mathbb{G}')]$ for each
$F''(\ell)$, and $\inj^*(F',k, \mathbb{G})$ and $\inj^*(F',k, \mathbb{G}')$
are each a finite sum of such terms. Therefore, for $\mathbb{G}'$,
\[\sum_{k=1}^\infty \frac{\EE[\inj^*(F',k,\mathbb{G}')]}{k!}<\infty.
\]
We therefore have
\begin{align*}
\EE[\inj(F,\mathbb{G})]&=\sum_{k=1}^\infty \frac{(-1)^{k-1}}{k!}\EE[\inj^*(F',k,\mathbb{G})]\\
&=\sum_{k=1}^\infty \frac{(-1)^{k-1}}{k!}\EE[\inj^*(F',k,\mathbb{G}')]
=\EE[\inj(F,\mathbb{G}')]
. \qedhere
\end{align*}
\end{proof}

\begin{lemma} \label{lemmainjequalsamedist}
Suppose we have two random finite graphs, $\mathbb{G}$ and $\mathbb{G}'$.
Suppose that for every finite graph $F$,
$\EE[\inj(F,\mathbb{G})]=\EE[\inj(F,\mathbb{G}')]$. Let $X$ be the random
variable that gives the number of vertices in $\mathbb{G}$, and suppose that
for some $\varepsilon>0$, $\EE[(2+\varepsilon)^X]$ is finite. Then
$\mathbb{G}$ and $\mathbb{G}'$ give rise to the same distribution on graphs
(up to isomorphism).
\end{lemma}

We would like to emphasize that in these two lemmas, it suffices to assume
the finiteness condition for $\mathbb{G}$, not $\mathbb{G}'$ (for which it
follows).

\begin{proof}
For graphs $F$ and $G$, let $X(F,G)$ be equal to the random variable which is
equal to $\inj(F,G)$ if $G$ has the same number of vertices as $F$, and gives
$0$ otherwise. Fix a graph $F$ with $k$ vertices, and let $F_i$ be the graph
obtained by adding $i$ isolated vertices to $F$. We can then write
\[X(F,G)=\sum_{i=0}^\infty \frac{(-1)^i}{i!} \inj(F_i,\mathbb{G}).
\]
Indeed, if $G$ has the same number of vertices as $F$, then each term with
$i>0$ is zero, and the $i=0$ term gives $\inj(F_0,G)=\inj(F,G)$. If $G$ has
fewer vertices, then the entire expression is zero. If $G$ has $k+\ell$
vertices where $\ell>0$, then each injective copy of $G$ contributes (since
$\binom{\ell}{i}=0$ if $i>\ell$)
\[\sum_{i=0}^{\ell} (-1)^i \binom{\ell}{i}=0
;\]
therefore the entire expression is $0$.
Next, we claim that
\[\sum_{i=0}^\infty \frac{\EE[\inj(F_i,\mathbb{G})]}{i!}<\infty.
\]
To see this, note that for every $i$,
\[\EE[\inj(F_i,\mathbb{G})]\le \sum_n \PP(X=n) (n)_{k+i}.
\]
By the condition on $X$, there exists an $\varepsilon>0$ and $c$ such that
$\PP(X=n) \le c (2+\varepsilon)^{-n}$. Therefore,
\begin{align*}
\sum_{i=0}^\infty \frac{\EE[\inj(F_i,\mathbb{G})]}{i!} &\le \sum_i \sum_n c(2+\varepsilon)^{-n}(n)_{k}\binom{n-k}{i}\\
&=\sum_n c(2+\varepsilon)^{-n}2^{n-k}(n)_k<\infty.
\end{align*}

We therefore have that
\[\EE[X(F,\mathbb{G})]=\sum_{i=0}^\infty \frac{(-1)^i}{i!} \EE[\inj(F_i,\mathbb{G})]=\sum_{i=0}^\infty \frac{(-1)^i}{i!} \EE[\inj(F_i,\mathbb{G'})]=\EE[X(F,\mathbb{G'})]
.\] Thus, by an inclusion-exclusion formula, we can express the probability that
$\mathbb{G}$ and $\mathbb{G}'$ is isomorphic to a graph $F$ for any $F$, and
the two probabilities must also be equal. This completes the proof.
\end{proof}

\begin{proof}[Proof of Theorem \ref{theoremconveq}]
We first show that (\ref{convdist}) implies (\ref{convexpinjany}). Recall
that for any graph $F$, $\EE[\inj(F,\mathbb{G}_n)]$ and
$\EE[\inj(F,\mathbb{G})]$ are uniformly bounded. Fix a graph $F$ on $k$
vertices. We claim that $\EE[\inj(F,\mathbb{G}_n)^2]$ and
$\EE[\inj(F,\mathbb{G})^2]$ are uniformly bounded. This follows from the fact
that for a graph $G$, $\inj(F,G)^2$ is a linear combination of the form
$\sum_{F'} c_{F'} \inj(F',G)$, where $c_{F'}$ is independent of $G$ and zero
for all but a finite number of graphs $F'$. We also know that
$\inj(F,\mathbb{G}_n)$ converges to $\inj(F,\mathbb{G})$ in distribution.
Since their second moments are uniformly bounded, their expectations must
converge as well.

It is clear that (\ref{convexpinjany}) implies (\ref{convexpinjnoisol}). Let
us show that (\ref{convexpinjnoisol}) implies (\ref{convdist}). Assume
$\mathbb{G}_n$ satisfies (\ref{convexpinjnoisol}). We first claim that the
sequence $\mathbb{G}_n$ is tight. That is, we claim that for every
$\varepsilon$, there exists a finite set of graphs such that for each $n$,
with probability at least $1-\varepsilon$, $\mathbb{G}_n$ is in this set.
Indeed, the expected number of edges $\EE[\inj(K_2,\mathbb{G}_n)]$ is
uniformly bounded, which means that for any $\varepsilon$, there exists an
$M$ such that the probability of having more than $M$ edges is at most
$\varepsilon$. But the number of graphs with $M$ edges and no isolated
vertices is finite; therefore $\mathbb{G}_n$ is tight.

This means that there is a subsequence that converges to a random graph
$\mathbb{H}$ in the sense of (\ref{convdist}), which also has no isolated
vertices. First, we claim that for any $F$ with no isolated vertices,
$\EE[\inj(F,\mathbb{H})]$ is finite, and $\EE[\inj(F,\mathbb{G}_n)]$
converges to it. This again follows from the fact that
$\EE[\inj(F,\mathbb{G}_n)^2]$ is uniformly bounded. But then this implies
that $\EE[\inj(F,\mathbb{H})]=\EE[\inj(F,\mathbb{G})]$ for every graph $F$
with no isolated vertices. Therefore, the two distributions are equal by
Lemmas \ref{lemmanoisolallequal} and \ref{lemmainjequalsamedist}.
\end{proof}

\begin{lemma}
Given $t>0$, let $t'=2t+\log 4$. Then the following holds for any positive
integer $n$ and any graph $G$ on $n$ vertices with no isolated vertices.
Suppose we randomly color the vertices red and blue, and let $X$ be the
number of vertices that are colored red, and have at least one blue neighbor.
Then
\[\EE[e^{t' X}] \ge e^{tn}.
\]
\end{lemma}

\begin{proof}
Without loss of generality, we may assume $G$ is the disjoint union of stars,
since otherwise we can delete an edge and $G$ will still have no isolated
vertices. Suppose that $G$ is the disjoint union of stars with edge count
$s_1,s_2,\dots,s_k$, where each $s_i \ge 1$ and $\sum_i s_i=n-k$. Note that
$k \le n/2$. Then
\begin{align*}
\EE\left[ e^{t'X} \right]&=\prod_{i=1}^k \left( \frac{1}{2}\left(\frac{e^{t'}+1}{2}\right)^{s_k}+
\frac{1}{2}\left(\frac{e^{t'}+2^{s_k}-1}{2^{s_k}}\right) \right)\\
& \ge \left(\frac{e^{t'}+1}{2}\right)^{n-k}2^{-k}
\ge (e^{t'}+1)^{n/2}2^{-n}
 \ge e^{tn},
\end{align*}
given that our choice of $t'$ implies that $(e^{t'}+1)/4\geq e^{2t}$.
\end{proof}

\begin{lemma} \label{lemmaboundedgraphexmomentgenfinite}
For any $T,t,C,D \in \RR_+$, there exists a finite $B$ such that the
following holds. Suppose we have a graphex $\cW$, with $\|D_\cW\|_\infty \le
D$ and $\|\cW\|_1 \le C$. Let $X$ be the number of vertices (that are not
isolated) of $G_T(\cW)$. Then $\EE[e^{tX}] \le B$.
\end{lemma}

\begin{proof}
Note that since $X$ is nonnegative, we only need to worry about $t>0$. First,
let $X'$ be obtained by randomly coloring the vertices of $G_{T}$ red and
blue, and taking the red vertices with at least one blue neighbor. By the
above lemma, for $t'=2t+\log 4>0$, we have
\[\EE[e^{t'X'}] \ge \EE[e^{tX}].
\]
We claim that
\begin{equation} \label{eqnmomgenfuncbound}
\EE[e^{t'X'}] \le e^{\frac{T^2I}2 (e^{t'}-1)}
\exp\left( \frac T2\int_\Omega \left(e^{\frac T2 D_\cW(x)(e^{t'}-1)}-1\right) \,d\mu(x) \right).
\end{equation}
Let us show how this implies uniform boundedness. We know that $D_\cW$ is
bounded by $D$. Since the function $z \rightarrow e^z-1$ is $0$ at $z=0$ and
convex, there exists a constant $K$ depending only on $t'$, $T$, and $D$ such
that
\[e^{\frac T2D_\cW(x)(e^{t'}-1)}-1 \leq \frac{D_\cW(x)}{D}e^{\frac T2D(e^{t'}-1)}
= KD_\cW(x).
\]
Therefore,
\begin{align*}
\frac{T^2I}2 (e^{t'}-1)+&\frac T2\int_\Omega \left(e^{\frac T2 D_\cW(x)(e^{t'}-1)}-1\right) \,d\mu(x)\\
&\le\frac{T^2I}2 (e^{t'}-1)+\frac T2 \int_\Omega KD_\cW(x) \,d\mu(x)\\
& \le \frac{T^2C}4 (e^{t'}-1)+\frac{KTC}2
.
\end{align*}

In order to prove (\ref{eqnmomgenfuncbound}), we first show that it is enough
to consider the case where $\mu(\Omega)<\infty$. To see this, we write a
general $\sigma$-finite measure space as the union of finite spaces,
$\Omega=\bigcup_n\Omega_n$ where $\Omega_n$ is an increasing sequence with
$\mu(\Omega_n)<\infty$ for all $n$. For a fixed $T$, let $G_{T,n}$ consist of
the induced graph on those vertices which were born in $\Omega_n$, come from
a star of a vertex born in $\Omega_n$, or come from a dust edge, and define
$X_n'$ to be the number of red vertices in $G_{T,n}$ with at least one blue
neighbor in $G_{T,n}$. It is easy to see that $0 \le X_1' \le X_2' \le \dots
\le X_n' \le \dots$. By monotone convergence, the expectation of $e^{t'X_n'}$
converges to $\EE[e^{t'X'}]$. The bound we obtain also converges, giving the
required bound on $\EE[e^{t'X'}]$.

Assume thus that $\mu(\Omega)<\infty$. In this case, almost surely, a finite
number of blue points will be created. Suppose that these are
$x_1,x_2,\dots,x_k$. Conditioned on this, red vertices that have blue
neighbors can be created as follows. They can be created by the Poisson
process on $\Omega$, and then be connected to at least one of the $x_i$. If a
point is created at $x$, the probability that it is connected to at least one
of the $x_i$ is $1-\prod(1-W(x_i,x))$. It can also be created as a leaf of a
star created at one of the $x_i$. Finally, it can be created by a dust edge
being colored red and blue. Since the number of red vertices coming from each
of these cases is independent, the number of red vertices with at least one
blue neighbor is a Poisson distribution with expectation
\begin{align*}
f(x_1,\dots,x_k)&:=\frac T2 \int_{\Omega}\left(1-\prod_{i=1}^k (1-W(x_i,x))\right)\,d\mu(x)
+\frac T2\sum_{i=1}^k S(x_i)
+\frac {T^2 I}2
\\
&\le \frac T2\int_{\Omega} \left( \sum_{i=1}^k W(x_i,x)\right)\,d\mu(x)
+\frac T2\sum_{i=1}^kS(x_i)
+\frac {T^2 I}2\\
&=\frac T2 \sum_{i=1}^k D_\cW(x_i)
+\frac {T^2 I}2.
\end{align*}
In particular, this means that (for $t'>0$) we have
\[\EE[e^{t'X'}|x_1,x_2,\dots,x_k] = e^{f(x_1,\dots,x_k) (e^{t'}-1)}
\le e^{\frac T2\left(TI+\sum_{i=1}^k D_\cW(x_i)\right) (e^{t'}-1)}.
\]
Therefore
\begin{align*}\EE[e^{t'X'}]&=\sum_{k=0}^{\infty} e^{-\frac T2\mu(\Omega)} \frac{(T/2)^k\mu(\Omega)^k}{k!} \frac{1}{\mu(\Omega)^k} \int_{\Omega^k} e^{f(x_1,x_2,\dots,x_k)(e^{t'}-1)}\,d\mu(x_1)\dots\,d\mu(x_k) \\
&\le \sum_{k=0}^\infty e^{-\frac T2\mu(\Omega)}\frac{(T/2)^k}{k!} \int_{\Omega^k}e^{\frac T2\left(TI+\sum_{i=1}^k D_\cW(x_i) \right) (e^{t'}-1)} \,d\mu(x_1)\dots\,d\mu(x_k)\\
&= e^{\frac{T^2I}2 (e^{t'}-1)}\sum_{k=0}^\infty e^{-\frac T2\mu(\Omega)} \frac{(T/2)^k}{k!} \left( \int_\Omega e^{\frac T2D_\cW(x) (e^{t'}-1)}\,d\mu(x)\right)^k .
\end{align*}
Here we think of $\Omega^0$ as consisting of a single point on which $f$ is
$0$. We then have
\begin{align*}
\EE[e^{t'X'}] &\le e^{\frac{T^2I}2 (e^{t'}-1)}e^{\frac T2\int_\Omega e^{\frac T2 D_\cW(x)(e^{t'}-1)}\,d\mu(x)-\frac T2\mu(\Omega)}\\
&=e^{\frac{T^2I}2 (e^{t'}-1)}e^{\frac T2\int_\Omega \left(e^{\frac T2D_{\cW}(x)(e^{t'}-1) }-1\right)\, d\mu(x)}.
\end{align*}
So we know that (\ref{eqnmomgenfuncbound}) is true for any $\Omega$ with
finite measure. This completes the proof of the lemma.
\end{proof}

After these preparations, the proof of Theorem~\ref{theoremgraphexconveq} is
straightforward.

\begin{proof}[Proof of Theorem~\ref{theoremgraphexconveq}]
We first note that it is enough to prove the lemma for the case that $\cW_n$
and $\cW$ are $(C,D)$-bounded for some finite $C,D<\infty$. Indeed, both (3)
and (4) clearly imply a bound on the $\|\cdot\|_1$-norm, but also (2) (and
therefore (1)) does, since (2) implies that the random graphs $G_T(\cW_n)$
are tight, which implies that the set of graphexes is tight, which by
Corollary~\ref{cor:C-bounded-tight} (2) implies uniform boundedness of the
$\|\cdot\|_1$-norms.

Assume thus that $\cW_n$ and $\cW$ are $(C,D)$-bounded for some finite
$C,D<\infty$. The equivalence of (\ref{graphexconvnoisol}) and
(\ref{graphexconvconn}) follows from the fact that $t$ is multiplicative over
components of $F$. $(\ref{graphexconvallT}) \Rightarrow
(\ref{graphexconvoneT})$ is obvious. Using Lemma
\ref{lemmaboundedgraphexmomentgenfinite}, we can apply the equivalence in
Theorem \ref{theoremconveq} and Proposition~\ref{prop:t(F,W)} to show that
$(\ref{graphexconvoneT}) \Rightarrow (\ref{graphexconvnoisol})$ and
$(\ref{graphexconvnoisol}) \Rightarrow (\ref{graphexconvallT})$.
\end{proof}

A slight modification of the above proof gives the following theorem.

\begin{theorem} \label{theoremgraphexexpsameequiv}
Given two graphexes $\cW,\cW'$ with bounded marginals, the following are
equivalent:
\begin{enumerate}
\item $G_T(\cW)$ and $G_T(\cW')$ have the same distribution for every $T$. \label{graphexallTequal}
\item $G_T(\cW)$ and $G_T(\cW')$ have the same distribution for some $T$. \label{graphexoneTequal}
\item For every graph $F$ with no isolated vertices, $t(F,\cW)= t(F,\cW')$. \label{graphexsamecountnoisol}
\item For every connected graph $F$, $t(F,\cW) =t(F,\cW')$. \label{graphexsamecountconn}
\end{enumerate}
\end{theorem}
\begin{proof}
As before, the equivalence of (\ref{graphexsamecountconn}) and
(\ref{graphexsamecountnoisol}) follows from the product property of $t$. The
implication $(\ref{graphexallTequal}) \Rightarrow (\ref{graphexoneTequal})$
is obvious. To prove $(\ref{graphexoneTequal}) \Rightarrow
(\ref{graphexsamecountnoisol})$, we use the fact that
$t(F,\cW)=T^{-|V(F)|}\EE[\inj(F,G_T(\cW))]$ and the same holds for $\cW'$.
Since $\inj(F,G_T(\cW))$ and $\inj(F,G_T(\cW'))$ have the same distribution,
their expectations must be equal. With the help of
Proposition~\ref{prop:t(F,W)}, this implies (\ref{graphexsamecountnoisol}).
$(\ref{graphexsamecountnoisol}) \Rightarrow (\ref{graphexallTequal})$ follows
from Proposition~\ref{prop:t(F,W)}, the observation that graphexes with
bounded marginals are integrable, and Lemmas
\ref{lemmaboundedgraphexmomentgenfinite}, \ref{lemmanoisolallequal}, and
\ref{lemmainjequalsamedist}.
\end{proof}

\subsection{Metric convergence implies GP-convergence} \label{sec:metric-implies-GP}

We close this section by proving that under the assumption of uniformly
bounded marginals, $\deltt$-convergence implies GP-convergence. We then use
this result to show that without any assumptions on the marginals,
$\delGP$-convergence implies GP-convergence.

\begin{theorem}\label{thm:deltt-implies-GPconv}
Suppose $\cW_n$ and $\cW$ have uniformly bounded marginals, and
$\deltt(\cW_n,\cW) \rightarrow 0$. Then $\cW_n$ is GP-convergent to $\cW$.
\end{theorem}

\begin{proof}
Note that $\deltt$ convergence implies that $\|\cW_n\|_1 \rightarrow
\|\cW\|$; therefore the sequence is $(C,D)$-bounded for some $C,D$. If $F$ is
a graph without isolated vertices, then $t(F,\cW_n) \rightarrow t(F,\cW)$ by
Corollary \ref{corrolaryFclose}. Therefore, by Theorem
\ref{theoremgraphexconveq}, for any $T$, $G_T(\cW_n)$ converges to $G_T(\cW)$
in distribution.
\end{proof}

\begin{theorem} \label{theoremdimpliessampling}
Suppose that graphexes $\cW$ and $(\cW_n)_{n=1}^\infty$ have the property
that $\delGP(\cW_n,\cW) \rightarrow 0$. Then $\cW_n$ is GP-convergent to
$\cW$.
\end{theorem}

\begin{proof}
By Proposition \ref{propgeneralconvergenceequiv}, the sequence is tight, and
for all $D$ such that $\mu(\{D_\cW=D\})=0$, we have that
$\mu_n(\Omega_{n,>D}) \rightarrow \mu(\Omega_{>D})$ and $\deltt(\cW_{n,\leq
D},\cW_{\leq D}) \rightarrow 0$.

Fix $T$ and $\varepsilon>0$, and take $\delta$ small enough so that for all
sets $\Omega_\delta$ of measure at most $\delta$ the probability that any of
the vertices in $G_T$ has a feature in $\Omega_\delta$ is at most
$\varepsilon/3$. Take $D$ large enough so that for all $n$,
$\mu_n(\Omega_{n,>D})$ and $ \mu(\Omega_{>D})$ are at most $\delta$. Then the
total variation distance between $G_T(\cW_{n,\leq D})$ and $G_T(\cW_n)$ is at
most $\varepsilon/3$, and the same is true for $G_T(\cW_{\leq D})$ and
$G_T(\cW)$. We also know that $\deltt(\cW_{n,\leq D},\cW_{\leq D})
\rightarrow 0$, which in particular implies that the sequence is uniformly
$(C,D)$-bounded for some $C$. Therefore, it is GP-convergent. In particular,
for $n$ large enough, the total variation distance between $G_T(\cW_{n,\leq
D})$ and $G_T(\cW_{\leq D})$ is at most $\varepsilon/3$. This implies that
for $n$ large enough, the total variation distance between $G_T(\cW_n)$ and
$G_T(\cW)$ is at most $\varepsilon$, which shows that the sequence is
GP-convergent.
\end{proof}

\section{Sampling} \label{sec:sampling}

In this section, we prove that GP-convergence implies convergence in the weak
kernel metric, completing the proof of the equivalence of convergence in the
metric $\delGP$ and GP-convergence (Theorem~\ref{thm:delGP-GP}). The main
technical tool to establish this will be a ``sampling lemma'', showing that
as $T\to\infty$, the graphs $G_T(\cW)$ sampled from a graphex $\cW$ converge
to the generating graphex according to $\delGP$.

To make this precise, we need a way to compare graphs to graphexes. As in
\cite{BCCH16} and \cite{BCCV17}, we do this by transforming the graph into a
suitable ``empirical graphon'' and corresponding ``empirical graphex''.
Differing slightly from both \cite{BCCH16} and \cite{BCCV17}, where the
empirical graphon was a graphon over $\RR_+$, here we define it to be a
graphon over the vertex set of the graph. Explicitly, given a finite graph
$G$ and $\rho>0$, we define the graphon $W(G,\rho)$ as follows. Let
$\bOmega=(\Omega,\cF,\mu)$, where $\Omega$ is the set of vertices, $\cF$ is
the $\sigma$-algebra consisting of all subsets, and $\mu$ is the measure
where each vertex has weight $\rho$. Set $W(x,y)$ to be $1$ if there is an
edge between the corresponding vertices, and $0$ otherwise. This gives us the
graphon $W(G,\rho)$. We then set $\cW(G,\rho)=(W,0,0,\bOmega)$. Similarly, if
$H$ is a weighted graph with countably many vertices, we define $\Omega$ to
be the set of vertices, $\cF$ to be the $\sigma$-algebra consisting of all
subsets of $\Omega$, and $\mu$ to be the $\sigma$-finite measure which gives
weight $\rho$ to each vertex; $W(H,\rho)$ and $\cW(H,\rho)$ are then the
graphon and graphex obtained by taking $W$ according to edge weights.

With these definitions, we are ready to state the sampling lemma.

\begin{theorem} \label{theoremsamplesconverge}
For every graphex $\cW$ and $\varepsilon>0$,
\[\lim\limits_{T \rightarrow \infty}\PP[\delGP(\cW(G_{T}(\cW),1/T),\cW)>\varepsilon]= 0.
\]
For a set of graphexes that is tight, the convergence is uniform.
\end{theorem}

\begin{remark}
The above theorem only claims convergence in probability. However, once we
establish equivalence of GP-convergence and convergence in the weak kernel
norm, the results of \cite{JANSON17} imply convergence with probability one
(since there convergence with probability one is proved for GP-convergence).
Nevertheless, to \emph{establish} the equivalence, all we need is convergence
in probability, so this is all we will prove here.
\end{remark}

\subsection{Closeness of graphexes implies closeness of samples}
\label{sec:close-GR-close-Smpl}

In order to prove the sampling lemma, we will first prove that two graphexes
with bounded marginals that are close in the kernel metric lead to samples
that are close. This is formalized in the following theorem.

\begin{theorem} \label{thm:sampleddistanceclose}
Suppose $\cW_1,\cW_2$ are two $(C,D)$-bounded graphexes on the same space
$\bOmega$, and suppose that $\d22(\cW_1,\cW_2)\leq c$ for some $0<c<1$. Then
there exists a $T_0$ (depending only on $c$, $C$, and $D$) such that for any
$T>T_0$, there exists a coupling of the random graphs $G_T(\cW_1)$ and
$G_T(\cW_2)$ so that
\[\PP\left[\deltt(\cW(G_T(\cW_1),1/T),\cW(G_T(\cW_2),1/T))>
\min\left((31cC)^{1/4}, 2c^{3/4},\sqrt[3]{3}c\right)
\right] < c
.\]
\end{theorem}

For \emph{graphons}, or graphexes with only a graphon part, we can think of
obtaining $G_T$ as having two phases: first we sample the set of vertices,
and then we sample the edges according to the edge probability. If we do not
do the second phase, we obtain a weighted graph. We will work with this
intermediate graph in this section. To make this precise, given a graphon
$(W,\bOmega)$, define $H_T(W)$ as the random weighted graph where we take a
Poisson process on $\bOmega \times [0,T]$, set these to be the vertices of
$H_T(W)$, and for each pair of vertices $(x_i,t_i)$ and $(x_j,t_j)$, put a
weighted edge with weight $W(x_i,x_j)$ (with $0$ weights on the diagonals).

In order to prove the theorem, we need to find a coupling of the random
processes that provide $G_T(\cW_1)$ and $G_T(\cW_2)$. Since $\cW_1$ and
$\cW_2$ have the same underlying space, it is natural to couple the Poisson
processes that generate the vertices into a single Poisson process.
Conditioned on this, we generate the two random graphs independently (this is
not optimal but it is satisfactory for our purposes). Let
$\cW_i'=\cW_i(G_T(\cW_i),1/T)=(W_i',S_i',I_i',\bOmega_i')$. The underlying
space of $\cW_i'$ consists of the vertex set of $G_T(\cW_i)$, everything with
weight $1/T$. We couple the two underlying spaces by matching vertices that
correspond to the same point in $\Omega$, and couple the other vertices
arbitrarily (adding points with degree $0$ if necessary). We will show that
in this way, all three components of our distance will be close

Let us first show that $\|W_1'-W_2'\|_{2 \rightarrow 2}$ is small, with high
probability. Note that $G_T(\cW)$ consists of the edges in $G_T(W)$, and the
edges generated by the stars and the independent edges. In the following
lemma, we show that the extra edges generated have a small effect on this
distance

\begin{lemma} \label{lem:2to2closewithwithoutstars}
Let $\cW=(W,S,I,\bOmega)$ be a $(C,D)$-bounded graphex, and $T>1/D$. Let
$G_T(\cW)$ be the usual sample at time $T$, and let $\widetilde{G}_T(\cW)$
consist of only those edges which come from $I$ or $S$. Then
\[
\PP\left[\|W(\widetilde{G}_T(\cW),1/T)\|_{2 \rightarrow 2}>\left(\frac{2CD}{\sqrt{T}}\right)^{1/4}\right] < \frac{1}{\sqrt{T}}.
\]
\end{lemma}

\begin{proof}
Suppose we have sampled stars with $s_1,s_2,\dots,s_\ell$ leaves, and we have
sampled $m$ isolated edges. Let $U=W(\widetilde{G}_T(\cW),1/T)$. Then
\[t(C_4,U)=\frac{1}{T^4} \left(2\sum_i s_i^2 + 2m \right)
.\]
Therefore,
\[\EE[t(C_4,U)]=\frac{2T\int_\Omega (T^2S(x)^2 + T S(x)) \,d\mu(x)}{T^4} + \frac{2T^2I}{T^4} \le \frac{C(D+1/T)}{T} \le \frac{2CD}{T},
\]
and hence
\[\PP\left[t(C_4,U)>\frac{2CD}{\sqrt{T}}\right] < \frac{1}{\sqrt{T}}
.\] Using the fact that $\|U\|_{2 \rightarrow 2} \le t(C_4,U)^{1/4}$ (Lemma
\ref{lem:2t2C4equiv}), the lemma follows.
\end{proof}

This lemma implies that for the $2 \rightarrow 2$ component of the distance,
we can compare $G_T(W_1)$ and $G_T(W_2)$ instead of $G_T(\cW_1)$ and
$G_T(\cW_2)$. The following lemma will imply that it in fact suffices to
compare $H_T(W_1)$ and $H_T(W_2)$, because $G_T$ is close to $H_T$, as long
as $H_T(W_i)$ satisfies certain boundedness conditions (which, by the
boundedness of the $W_i$, will be satisfied with high probability).

\begin{lemma} \label{lemmaweightunweightedclose}
Suppose $H$ is a weighted graph on $\NN$ with weights $H_{i,j}\in[0,1]$, and
$H_{i,i}=0$. Suppose that $G$ is generated by taking an edge between $i$ and
$j$ with probability $H_{i,j}$, independently for every pair of vertices.
Suppose that $\sum_{i,j} H_{i,j}\leq E$ and $\sum_{i,j,k}H_{i,j}H_{j,k}\leq
F$ where the sum goes over pairwise distinct vertices. Let $0<\rho$. Then
\[\PP[\|\cW(G,\rho)-\cW(H,\rho)\|_{2 \rightarrow 2}>\rho^{7/8}(E+2F)^{1/4}] < \sqrt{\rho}
.\]
\end{lemma}

\begin{proof}
Let $X_{i,j}=G_{i,j}-H_{i,j}$. Notice that $\EE X_{i,j}=0$ and $X_{i,j}$ over
different pairs are independent. Also, each $X_{i,i}=0$. Therefore,
\begin{align*}
\EE[t(C_4,\cW(G,\rho)-\cW(H,\rho))] =\rho^4\EE \bigg[&\sum_{i,j}
X_{i,j}^4+2\sum_{i,j,k} X_{i,j}^2X_{j,k}^2\\
& \phantom{} +
\sum_{i,j,k,l}X_{i,j}X_{j,k}X_{k,l}X_{l,i} \bigg],
\end{align*}
where in each of the
sums, all indices are pairwise distinct. Here
\[\sum_{i,j} \EE[X_{i,j}^4] =
\sum_{i,j}\Bigl( H_{i,j}(1-H_{i,j})^4+(1-H_{i,j})H_{i,j}^4 \Bigr) \le \sum_{i,j} H_{i,j}\leq E.
\]
Also,
\begin{align*}
\sum_{i,j,k} \EE \left[X_{i,j}^2X_{j,k}^2\right]&=
\sum_{i,j,k} \left(H_{i,j}-2H_{i,j}^2+H_{i,j}^2 \right)\left(H_{j,k}-2H_{j,k}^2+H_{j,k}^2 \right)\\
&=\sum_{i,j,k} H_{i,j}(1-H_{i,j})H_{j,k}(1-H_{j,k})
\le \sum_{i,j,k}H_{i,j}H_{j,k}\\
&\leq F.
\end{align*}
Finally, for any pairwise distinct $i,j,k,\ell$,
\[\EE[X_{i,j}X_{j,k}X_{k,\ell}X_{\ell,i}]=0.
\]
Therefore,
\begin{equation}
\label{EtC4-bd}
0\leq\EE[t(C_4,\cW(G,\rho)-\cW(H,\rho))]\le \rho^4 (E +2F).
\end{equation}
This implies that
\[\PP\left[t(C_4,\cW(G,\rho)-\cW(H,\rho))> \rho^{7/2}(E+2F)\right] < \sqrt{\rho}.
\]
Using the fact that $\|U\|_{2 \rightarrow 2} \le \left(t(C_4,U)\right)^{1/4}$
(Lemma \ref{lem:2t2C4equiv}), the lemma follows.
\end{proof}

\begin{proof}[Proof of Theorem~\ref{thm:sampleddistanceclose}]
We are now ready to show that with high probability, $\|W_1'-W_2'\|_{2
\rightarrow 2}$ is small. Recall that we have a coupling of $H_T(W_1)$ and
$H_T(W_2)$ such that $H_T(W_1)-H_T(W_2)=H_T(W_1-W_2)$ with probability one.
Let $U=W_1-W_2$, so that $H_T(W_1)-H_T(W_2)=H_T(U)$.

Let us first show that $\|W(H_T(U),1/T)\|_{2 \rightarrow 2}$ is small. First,
suppose that $\mu=\mu(\Omega)$ is finite. Let
\[X=\int_{\Omega^2} U(x_1,x_2)^4 \,d\mu(x_1) \,d\mu(x_2) \le 2C,\]
\[Y= \int_{\Omega^3}U(x_1,x_2)^2U(x_2,x_3)^2\,d\mu(x_1)\,d\mu(x_2)\,d\mu(x_3) \le 4CD,\]
and
\begin{align*}
Z&=\int_{\Omega^4}U(x_1,x_2)U(x_2,x_3)U(x_3,x_4)U(x_4,x_1)=t(C_4,U)
\\
&\le \|W_1-W_2\|_{2 \rightarrow 2}^2 \|W_1-W_2\|_{2}^2\\
& \le \|W_1-W_2\|_{2
\rightarrow 2}^2 \|W_1-W_2\|_1\\
&\le 2c^2 C, \end{align*} where we used Lemma \ref{lem:2t2C4equiv}, the fact
that both graphexes are $(C,D)$-bounded, and the fact that
$\d22(\cW_1,\cW_2)\leq c$. If $T>\max\{16\frac{D}{c^2},2/c\}$, then
\begin{align*}
\EE[t(C_4,H_T(U))]&=
\sum_{n=0}^{\infty} e^{-T\mu} \frac{(T\mu)^n}{n!} \bigg(\frac{n (n-1)}{T^4\mu^2} X+2\frac{n(n-1)(n-2)}{T^4\mu^3}Y\\
& \qquad \qquad \qquad \qquad \quad \phantom{} +\frac{n(n-1)(n-2)(n-3)}{T^4\mu^4}Z\bigg)\\
&\le \frac{2C}{T^2}
+\frac{8CD}{T}
+2c^2C
\le 3c^2C.
\end{align*}
In general, we can take a sequence of finite measure subsets $\Omega_1
\subseteq \dots \subseteq \Omega_n \subseteq \dots$ with $\bigcup_n
\Omega_n=\Omega$ to show that the above bound on the expectation holds for
general $\Omega$. Therefore,
\[\PP[\|H_T(U)\|_{2 \rightarrow 2}> (30cC)^{1/4} ] \le\PP[t(C_4,H_T(U))>30 c C]<\frac{c}{10}
.\] Next, let $P_2$ be the star with two leaves. If $T$ is large enough, then
\[\EE[t(P_2,H_T(W_1))]= \left(t(P_2,W_1)T^3+ T^2 \|W_1\|_1\right) \le CDT^3 + CT^2 \le 2CDT^3
.\] Therefore,
\[\PP\left[t(P_2,H_T(W_1))>\frac{20CDT^3}{c}\right] \le \frac{c}{10}
.\]
Also, since
\[\EE[\|H_T(W_1)\|_1] \le CT^2
,\]
we also have
\[\PP\left[\|H_T(W_1)\|_1>\frac{10CT^2}{c}\right] \le \frac{c}{10}
.\]
Conditioned on neither of these happening, we can apply Lemma \ref{lemmaweightunweightedclose} with
\[E+2F \le \frac{10CT^2}{c}+ 2 \frac{20CDT^3}{c}
\le \frac{{50}CDT^3}{c}
.\]
This means that
\[\PP\left[\|\cW(G_T(W_1),1/T)-\cW(H_T(W_1),1/T)\|_{2 \rightarrow 2} > \left(\frac{{50}CD}{c\sqrt{T}}\right)^{1/4}\right] \le \frac{1}{\sqrt{T}}
.\] Clearly the analogous statements hold for $H_T(W_2)$. Let
$\widetilde{W}_i=W(\widetilde{G}_T(\cW_i),1/T)$ (i.e., the part consisting of
edges generated by the stars and independent edges). Also, let
$\widehat{W}_i=W(G_T(W_i),1/T)-W(H_T(W_i),1/T)$. Assuming none of the bad
events happen, if $T$ is large enough, then
\begin{align*}
\|W_1'-W_2'\|_{2 \rightarrow 2} &\le
\|\widetilde{W}_1\|_{2 \rightarrow 2}+\|\widetilde{W}_2\|_{2 \rightarrow 2}\\
& \qquad \qquad \phantom{} + \|\widehat{W}_1\|_{2 \rightarrow 2} + \|\widehat{W}_2\|_{2 \rightarrow 2}+\|W(H_T(U),1/T)\|_{2 \rightarrow 2}
\\
&\le 2\left(\frac{2CD}{\sqrt{T}}\right)^{1/4}+2 \left(\frac{({50}CD)^{1/4}}{c\sqrt{T}}\right)^{1/4}
+ (30cC)^{1/4}
\le (31cC)^{1/4} .
\end{align*}
The probability of one of
the bad events happening is at most
\[\frac{2}{\sqrt{T}}+4 \frac{c}{10}+\frac{2}{\sqrt{T}}+\frac{c}{10} \le \frac{6c}{10}
.\]
Let us now bound the probability that $\|D_{\cW_1'}-D_{\cW_2'}\|_2$ is large.
For $x \in \Omega$, let
\[D_{\cW_1\cW_2}(x)=\int_{\Omega}W_1(x,y)W_2(x,y) \,d\mu(y) .\]
With our coupling,
\begin{align*}
\EE\left[\sum_{v \in V_T} d_{G_T(\cW_1)}(v)^2\right]&=\int_\Omega T
\left((TD_{\cW_1}(x))^2 + TD_{\cW_1}(x) \right) \,d\mu(x),\\
\EE\left[\sum_{v \in V_T} d_{G_T(\cW_1)}(v)d_{G_T(\cW_2)}(v)\right]&=\int_\Omega T \left((TD_{\cW_1}(x))(TD_{\cW_2}(x))+TD_{\cW_1\cW_2}(x) \right) \,d\mu(x),\\
\EE\left[\sum_{v \in V_T} d_{G_T(\cW_2)}(v)^2\right]&=\int_\Omega T \left((TD_{\cW_2}(x))^2 + TD_{\cW_2}(x) \right) \,d\mu(x).
\end{align*}
Therefore,
\begin{align*}
\EE\Bigg[\sum_{v \in V_T} &\left(d_{G_T(\cW_1)}(x)-d_{G_T(\cW_2)}(x)\right)^2\Bigg]\\
& = T^3 \int_\Omega \left(D_{\cW_1}(x)-D_{\cW_2}(x)\right)^2 \,d\mu(x)
\\
& \qquad \phantom{} + T^2 \int_\Omega \left(D_{\cW_1}(x) + D_{\cW_2}(x) - 2D_{\cW_1\cW_2}(x) \right) \,d\mu(x)\\
&\leq T^3\|D_{\cW_1}-D_{\cW_2}\|_2^2+T^2\|\cW_1\|_1+T^2\|\cW_2\|_1,
\end{align*}
This means that if $T$ is large enough,
\[\EE\left[\|D_{\cW_1'}-D_{\cW_2'}\|_2^2\right]\le
c^4+2C/T\leq 2c^4
.\]
Therefore,
\[\PP[\|D_{\cW_1'}-D_{\cW_2'}\|_2>4c^{3/2}] \le \frac{c}{8}
.\] Finally, recall that by Lemma \ref{lem:edgebound}, the
number of edges of $G_T(\cW_i)$ has expectation $T^2 \|\cW_i\|_1/2$ and
variance $T^2\|\cW_i\|_1/2+T^3\|D_{\cW_i}\|_2^2$. Therefore, the probability
that $G_T(\cW_i)$ has more than $T^2(\|\cW_i\|_1+c^3)/2$ or less than
$T^2(\|\cW_i\|_1-c^3)/2$ edges is less than
\[\frac{T^2\|\cW_i\|_1/2+T^3\|D_{\cW_i}\|_2^2}{c^6T^4/4} \le \frac{2C + 4TCD}{c^6T^2} \le \frac{c}{8}.
\]
Here we used the fact that $\|D_{\cW_i}\|_2^2 \le
\|D_{\cW_i}\|_1\|D_{\cW_i}\|_\infty= \|\cW_i\|_1\|D_{\cW_i}\|_\infty$, and we
are assuming that $T$ is large. Assuming neither of these events happens,
$\|\cW_i'\|_1$ is between $\|\cW_i\|_1-c^3$ and $\|\cW_i\|_1+c^3$. Since
$|\|\cW_1\|_1-\|\cW_2\|_1| \le c^3$, we have that
$|\|\cW_1'\|_1-\|\cW_2'\|_1| \le 3c^3$.

To summarize, we have that with high probability,
\[\d22(\cW_1',\cW_2') \le \min(31cC)^{1/4}, 2c^{3/4},\sqrt[3]{3}c)
.\]
The probability that this does not happen is at most
\[\frac{6c}{10}+\frac{c}{8}+2\frac{c}{8} \le c
.\]
This completes the proof.
\end{proof}

\subsection{Samples converge to graphex} \label{sec:samplesconvtographex}

In this subsection, we prove the sampling lemma,
Theorem~\ref{theoremsamplesconverge}. To this end, we will first establish
two lemmas. The first one states that each $(C,D)$-bounded graphex can be
approximated by a step graphon, i.e., a graphex where the star and dust part
is zero, and the graphon part is a step graphon.

\begin{lemma}\label{lem:graphon-approx-of-graphex}
For every $\varepsilon$, $C$, and $D$, there exist $M$, $N$, and $\rho$ such
that the following holds. For every $(C,D)$-bounded graphex $\cW$, there
exists a $(C,D)$-bounded graphex
$\cW_\varepsilon=(W_\varepsilon,0,0,\bOmega_\varepsilon)$, where
$\bOmega_\varepsilon=(\Omega_\varepsilon,\cF_\varepsilon,\mu_\varepsilon)$
and $\mu_\varepsilon(\Omega_\varepsilon) \le N$, and furthermore the graphon
$W_\varepsilon$ is a step function with at most $M$ steps, with each part
having size equal to $\rho$, and $\deltt(\cW,\cW_\varepsilon) \le
\varepsilon$.
\end{lemma}

\begin{proof}
By Remark~\ref{rem:DC-mon}, we may assume that $\cW$ is a graphex over an
atomless measure space $\bOmega=(\Omega,\cF,\mu)$. By Theorem
\ref{regularityequalparts}, there exists $M(\varepsilon)$ and $\rho$ such
that there is a partition $\cP=\{P_1,P_2,\dots,P_m\}$ of $\Omega_\sP
\subseteq \Omega$ with $m \le M(\varepsilon)$ such that $\deltt(\cW_\sP,\cW)
\le \varepsilon/2$ and each part has size $\rho$. Let
$\Omega_\varepsilon=\Omega _\sP \cup Q$ where $Q$ is any set disjoint from
$\Omega_\sP$, and obtain $\mu_\varepsilon$ by extending $\mu$ to $Q$ (with
measure to be determined later). Let
$\cW_\varepsilon=(W_\varepsilon,0,0,\bOmega_\varepsilon)$ with
\[
W_\varepsilon(x,y) = \begin{cases}
W_\sP(x,y) &\text{ if $x \in P_i, y \in P_j$,}\\
\frac{S_\sP(x)}{\mu(Q)} &\text{ if $x \in P_i, y \in Q$,}\\
\frac{S_\sP(y)}{\mu(Q)} &\text{ if $x \in Q, y \in P_i$, and}\\
\frac{2I_\sP}{\mu(Q)^2} &\text{ if $x,y \in Q$.}
\end{cases}
\] Extend $\cW_\sP$ by $0$ to $Q$. Since $\cW_\sP$ is $(C,D)$-bounded, there
exists $K$ depending only on $\varepsilon,C,D$ such that if $\mu(Q) \ge K$,
then $W_\varepsilon-W_{\sP}$ is at most $\varepsilon^2/(4C)$ everywhere,
which implies that
\begin{align*}
\|W_\varepsilon-W_\sP\|_{2 \rightarrow 2} &\le \|W_\varepsilon-W_\sP\|_{2}\\
& \le \sqrt{\|W_\varepsilon-W_\sP\|_{1} \|W_\varepsilon-W_\sP\|_{\infty}}
 \le \sqrt{C\varepsilon^2/(4C)}
 = \varepsilon/2
.
\end{align*}
For $x \in \Omega_{\cP}$,
\[D_{\cW_\varepsilon}(x)=D_{W_{\cP}}(x)+ \mu(Q) \frac{S_{\cP}(x)}{\mu(Q)}=D_{\cW_{\cP}}(x)
.\]
We also have that for $x \in Q$,
\[D_{\cW_\varepsilon}(x)=\mu(Q) \frac{2I_\sP}{\mu(Q)^2}+ \sum_{i} \int_{P_i} \frac{S_\sP(y)}{\mu(Q)}\,d\mu(y)=\frac{D_{\cW_\sP}(\infty)}{\mu(Q)}.
\]
Therefore, there exists a $K'$ depending only on $\varepsilon$, $C$, and $D$
such that if $\mu(Q) \ge K'$, then
\begin{align*}
\|D_{\cW_\varepsilon}-D_{\cW_\sP}\|_2^2 &=\int_{\Omega_\sP \cup Q} \left(D_{\cW_\varepsilon}(x)-D_{\cW_\sP}(x)\right)^2 \,d\mu(x)\\
&=\int_Q \left(\frac{D_{\cW_\sP}(\infty)}{\mu(Q)}\right)^2
= \frac{D_{\cW_\sP}(\infty)^2}{\mu(Q)} \le \varepsilon^4/16
.
\end{align*}
Also, by construction, $\|\cW_\varepsilon\|_1=\|\cW_\sP\|_1$. Therefore
$\deltt(\cW_\varepsilon,\cW_\sP) \le \varepsilon/2$, and hence
$\deltt(\cW_\varepsilon,\cW) \le \varepsilon$.
\end{proof}

\begin{remark}\label{rem:W-approx-tildeW}
Using the ideas of the previous proof, it is not hard to see that in
distribution, the  process generated from the graphex
$\widetilde\cW_Q=(\widetilde W_Q,0,0,\widetilde\bOmega_D)$ constructed in
Remark~\ref{rem:graphex-process} (3) converges to the one generated from
$\cW$. Indeed, we claim that
\[
\delGP(\widetilde\cW_Q,\cW)\to 0\qquad\text{as}\qquad Q\to\infty.
\]
 To see this, fix $\eps>0$ and choose $D$ in such a way that the set
 $\Omega_{>D}=\{D_\cW>D\}$ has measure at most $\eps^2$. Let
 $\Omega_{\leq D}=\Omega\setminus\Omega_{>D}$
 and
 $\widetilde\Omega_{\leq D}
    =\Omega_{\leq D}\cup\{\infty\}
    =\widetilde\Omega\setminus\Omega_{>D}$.
Setting
  $\widetilde \cW_{Q,\leq D}=(\widetilde\cW_Q)_{|\widetilde\Omega\setminus\Omega_{>D}}$
  and
$\cW_{\leq D}=\cW_{|\Omega\setminus\Omega_{>D}}$ and defining
$\widetilde\cW_{\leq D}$ as the trivial extension of $\cW_{\leq D}$ to
$\widetilde\Omega_{\leq D}$, we will want to show that for $Q$ large enough,
 $\d22(\widetilde\cW_{\leq D},\widetilde \cW_{Q,\leq D})\leq\eps$,
since this implies that $\deltt(\cW_{\leq D},\widetilde \cW_{Q,\leq
D})\leq\eps$ and hence $\delGP(\widetilde\cW_Q,\cW)\leq \eps$. But this
follows by essentially the same argument as the one in the previous proof;
all that is needed is that by Proposition~\ref{prop:local-finite},
$\widetilde\cW_{\leq D}$ is $(C,D)$-bounded for some $C<\infty$.
\end{remark}

Our second lemma estimates the distance between the empirical graphex
corresponding to a weighted graph $H$ with weights in $[0,1]$ and the one
corresponding to the graph $G$ obtained from $H$ by choosing the edge in $G$
randomly according to $H$. More precisely, given a finite weighted graph $H$
with weights $H_{i,j}\in[0,1]$ and $H_{i,i}=0$, define $G(H)$ as the graph
generated by taking an edge between $i$ and $j$ with probability $H_{i,j}$,
independently for every pair of vertices. Our next lemma estimates the
distance between the empirical graphon of $H$ and the empirical graphon of
$G(H)$.

\begin{lemma}\label{lem:H-G-distance}
For every $N_0$, $\varepsilon$, and $\delta$, there exists $n_0$ such that
the following holds. For any weighted graph $H$ on $n \ge n_0$ vertices with
weights in $[0,1]$, and any $N \le N_0$, the probability that
$\deltt(\cW(H,N/n),\cW(G(H),N/n))>\varepsilon$ is at most $\delta$.
\end{lemma}

\begin{proof}
We first extend both $H$ and $G$ trivially to $\NN$, and then define $U$ as
the graphon $U=W(G(H),N/n)-W(H,N/n)$. Then $\|U\|_{2 \rightarrow 2} \le
\left(t(C_4,U)\right)^{1/4}$ by Lemma \ref{lem:2t2C4equiv}. As a consequence,
the probability that $\|U\|_{2 \to 2}>\varepsilon$ is bounded by
$\eps^{-4}\EE[t(C_4,U)]$. Using the bound \eqref{EtC4-bd} from the proof of
Lemma~\ref{lemmaweightunweightedclose} with $E=n^2$, $F=n^3$, and $\rho=N/n$,
we get that the probability that $\|U\|_{2 \to 2}>\varepsilon$ is bounded by
\[\eps^{-4}\EE[t(C_4,U)] \le \eps^{-4} \frac{N^4}{n^4}\left( n^2 + 2 n^3\right)
\leq 3 \eps^{-4}n^3\frac{N^4}{n^4}\leq 3\eps^{-4}\frac{N_0^4}{n_0}\leq\delta/2,
\]
provided $n_0\geq6\eps^{-4}\delta^{-1} N_0^4$. Let us now bound the other two
components of $\Delta_{2\to 2}$. For a fixed vertex $v$, by Hoeffding's
inequality \cite{H63},
\[\PP[\lvert d_{G(H)}(v)-d_H(v)\rvert > \varepsilon' n] \le 2e^{-{2} {\varepsilon'}^2 n}.
\]
Therefore, by a union bound,
\[\PP[\text{there exists a vertex $v$ such that }\lvert d_{G(H)}(v)-d_H(v)\rvert > \varepsilon' n] \le 2ne^{- {2}{\varepsilon'}^2 n}.
\]
For any fixed $\varepsilon'$, if $n$ is large enough, this probability is
less than $\delta/2$. If this does not happen, then for every vertex $v$,
\[
\left| D_{W(G(H),\frac{N}{n})}(v)-D_{W(H,\frac{N}{n})}(v)\right| \le
N\varepsilon' .\] Therefore, for $\varepsilon'$ small enough,
\[\|D_{W(G(H),N/n)}-D_{W(H,N/n)}\|_2 \le {N^{3/2}{\eps'}\leq}\varepsilon^2
,\] and
\begin{align*}\left|\|W(G(H),N/n)\|_1-\|W(H,N/n)\|_1\right|
&=
\left|\|D_{W(G(H),N/n)}\|_1-\|D_{W(H,N/n)}\|_1\right|\\
&\le N^2\eps'\le \varepsilon^3.
\end{align*}
This completes the proof of the lemma.
\end{proof}

With these preparations, we are ready to prove the sampling lemma.

\begin{proof}[Proof of Theorem~\ref{theoremsamplesconverge}]
Fix $\varepsilon>0$. We know from the definition of tightness that there
exist $C$ and $D$ so that we can remove a set $\Omega_\varepsilon$ of measure
at most $\varepsilon^2/2$ to obtain a $(C,D)$-bounded graphex. Then the
expected number of points in $G_T$ whose feature lies inside
$\Omega_\varepsilon$ is $\varepsilon^2T/2$. Therefore, since it is a Poisson
distribution, the probability that $G_T(\cW)$ has more than $\varepsilon^2 T$
points in $\Omega_\varepsilon$ is at most
\[ e^{\frac{\varepsilon^2T}{2}(1-2\log 2)}
.\] This converges to $0$ as $T \rightarrow \infty$. If $G_T(\cW)$ does not
have more than $\varepsilon^2 T$ points, then we can remove those points from
$G_T(\cW)$ and the sample is equivalent to a sample from the graphex
restricted to $\Omega \setminus \Omega_\varepsilon$. (It may have isolated
vertices but this does not affect our distance.) Since a set of
$\varepsilon^2 T$ points in $G_T(\cW)$ corresponds to a set of measure
$\varepsilon^2$ in $\cW(G_T(\cW,1/T))$, this shows that we may assume without
loss of generality that the original set is $(C,D)$-bounded, and prove
Theorem \ref{theoremsamplesconverge} for $\deltt$ instead of $\delGP$.

Choose $\cW_\delta$ as in Lemma~\ref{lem:graphon-approx-of-graphex} (with
$\delta$ taking the role of $\varepsilon$) so that in particular
$\deltt(\cW,\cW_\delta) \le \delta$. For sufficiently small $\delta$, Theorem
\ref{thm:sampleddistanceclose} then implies that there exists a $T_0$ such
that if $T>T_0$, then the samples from $\cW$ and $\cW_\delta$ can be coupled
so that
\[\PP[\deltt(\cW(G_T(\cW),1/T),\cW(G_T(\cW_\delta),1/T))> (31C\delta)^{1/4}] < \delta
.\] This means that it suffices to prove Theorem \ref{theoremsamplesconverge}
for step function graphons with equal size parts, uniformly over any set of
graphons with a bounded number of parts with the same size. Indeed, for any
$\varepsilon>0$, let $\delta>0$ be such that
\[2\delta+(31\delta C)^{1/4}<\varepsilon
.\] If we then take $\cW_\delta$ as above, then $\deltt(\cW_\delta,\cW) \le
\delta$, so by Theorem \ref{thm:sampleddistanceclose}, for large enough $T$,
we can couple $\cW(G_T(\cW_\delta,1/T))$ and $\cW(G_T(\cW,1/T))$ so that the
probability that they have $\deltt$ distance more than $(31\delta C)^{1/4}$
is at most $\delta$. Furthermore, we can take $T$ large enough so that the
probability that $\deltt(\cW(G_T(\cW_\delta),1/T),\cW_\delta)>\delta$ is at
most $\delta$ (detailed below). Overall, by the triangle inequality, this
implies that the probability that $\deltt(\cW(G_T(\cW),1/T),\cW) \ge
\varepsilon$ is at most $2\delta$. Since this works for arbitrarily small
$\delta$, the theorem follows.

Suppose therefore that $\cW=(W,0,0,\bOmega)$, where $W$ is a step graphon
with step size $\rho$ and $m$ steps total. Fix $\varepsilon>0$ and
$\delta>0$. For a fixed part $P_i$ and $T$, let $X_{T,i}$ be the number of
points in $P_i$ in the Poisson process. The expectation of each $X_{T,i}$ is
$\rho T$. For $\varepsilon'>0$, we have
\[\PP\left[X_{T,i}>(1+\varepsilon ')\rho T \right] <
e^{\rho T (\varepsilon '-(1+\varepsilon ')\log(1+\varepsilon '))}
=e^{-\rho T c(\varepsilon ')}\] for a nonnegative number $c(\varepsilon')$.
We also have
\[\PP\left[X_{T,i}<(1-\varepsilon')\rho T \right]<e^{\rho T (-\varepsilon'+(1-\varepsilon')\log(\frac{1}{1-\varepsilon'})}=e^{-\rho T c'(\varepsilon')}
\]
for a nonnegative number $c'(\varepsilon')$. Therefore, if $T$ is large
enough, then the probability that any part $P_i$ has more than
$(1+\varepsilon')\rho T$ or less than $(1-\varepsilon')\rho T$ points is less
than $\delta/2$. Note that in particular this means that the total measure of
nonzero points is at most $(1+\varepsilon')\rho m$. Therefore, with
probability at least $1-\delta/2$, we can add or delete points with total
measure at most $\varepsilon'\rho m$ to obtain $W$ from $W(H_T(W),\rho)$.
This means that we can couple $\cW(H_T(W),\rho)$ and $\cW$ so that they
differ on points with total measure at most $\varepsilon'\rho m$, and hence
\[\|W(H_T(W),\rho)-W\|_1 \le 2\varepsilon'(1+\varepsilon')\rho^2m^2
.\] Therefore, we have the same bound for $|\|W(H_T(W),\rho)\|_1-\|W\|_1|$.
Since both graphons are between $0$ and $1$, we also have
\begin{align*}
\|W(H_T(W),\rho)-W\|_{2 \rightarrow 2}
&\le \|W(H_T(W),\rho)-W\|_2\\
&\le \sqrt{\|W(H_T(W),\rho)-W\|_1}
\le \sqrt{2\varepsilon'(1+\varepsilon')}\rho m.
\end{align*}
Finally, we have
\[\|D_{\cW(H_T(W),\rho)}-D_{\cW}\|_2^2 \le \rho m (\varepsilon' \rho m)^2+\varepsilon' \rho m ((1+\varepsilon')\rho m)^2=({\varepsilon'}^3+3{\varepsilon'}^2+\varepsilon')\rho^3m^3
.\] We can therefore take $\varepsilon'$ small enough that
$\deltt(\cW(H_T(W),\rho),\cW)<\varepsilon/2$ with probability at least
$1-\delta/2$.

Using Lemma~\ref{lem:H-G-distance} for $\varepsilon/2$ and $\delta/2$, we
have that with probability at least $1-\delta/2$,
\[\deltt(\cW,H_T(W,1/T))\le \frac{\varepsilon}{2}
,\] and with probability at least $1-\delta/2$,
\[\deltt(\cW(H_T(W),1/T),\cW(G_T(\cW),1/T)) \le \frac{\varepsilon}{2} .\]
Therefore, with probability at least $1-\delta$,
\begin{align*}
\deltt(\cW,&\cW(G_T(\cW),1/T))\\
 &\le \deltt(\cW,H_T(W,1/T))+\deltt(\cW(H_T(W),1/T),\cW(G_T(\cW),1/T))\\
& \le \frac{\varepsilon}{2}+\frac{\varepsilon}{2}=\varepsilon
.
\end{align*}
This completes the proof of Theorem \ref{theoremsamplesconverge}.
\end{proof}

\subsection{Proofs of Theorem~\ref{thm:delGP-GP}, Proposition~\ref{thm:deltt-GP}, and Theorem~\ref{thm:D-bounded-convergence}}

Having completed the proof of Theorem~\ref{theoremsamplesconverge}, we are
finally ready to establish that $\delGP$ convergence is equivalent to
GP-convergence, together with several of the other equivalences stated in
Section~\ref{sec:defs}. To this end, we first prove the following theorem.

\begin{theorem} \label{theoremsamedistribution}
Given a pair of graphexes $\cW,\cW'$, we have $\delGP(\cW,\cW')=0$ if and
only if for every $T>0$, $G_T(\cW)$ and $G_T(\cW')$ have the same
distribution.
\end{theorem}

\begin{proof}
If $\delGP(\cW',\cW)=0$, then taking $\cW_n=W'$ for each $n$, Theorem
\ref{theoremdimpliessampling} implies that $G_T(\cW)$ and $G_T(\cW')$ must
have the same distribution for every $T$. Suppose now that $G_T(\cW)$ and
$G_T(\cW')$ have the same distribution for every $T$. By Theorem
\ref{theoremsamplesconverge}, we can choose $T$ such that with probability at
least $0.99$, $\delGP(G_T(\cW),\cW)<\varepsilon/2$ and
$\delGP(G_T(\cW'),\cW')<\varepsilon/2$. Since $G_T(\cW)$ and $G_T(\cW')$ have
the same distribution, the two graphexes have distance at most $\varepsilon$.
Since this holds for every $\varepsilon$, the lemma follows.
\end{proof}

\begin{proof}[Proof of Theorems~\ref{thm:delGP-GP}, Proposition~\ref{thm:deltt-GP}, and Theorem~\ref{thm:D-bounded-convergence}]
We start with the proof of Theorem~\ref{thm:delGP-GP}. One direction follows
from Theorem~\ref{theoremdimpliessampling}. Suppose now that $\cW_n $ is
GP-convergent to $\cW$. We know by Theorem \ref{theoremtightequiv} that then
the set $\cW_n$ is tight. By Theorem~\ref{thm:complete}, $\cW_n$ therefore
has a subsequence that converges according to $\delGP$ to a graphex $\cW'$,
which in turn implies the subsequence is GP-convergent to $\cW'$. This
implies that for any $T>0$, $G_T(\cW)$ and $G_T(\cW')$ have the same
distribution. By Theorem \ref{theoremsamedistribution}, $\delGP(\cW,\cW')=0$,
so $\delGP(\cW_n,\cW) \rightarrow 0$. Next recall that by Proposition
\ref{propboundedequivmetrics}, the distances $\delGP$ and $\deltt$ give
equivalent topologies on sets with uniformly bounded marginals, showing that
Theorem~\ref{thm:delGP-GP} implies Proposition~\ref{thm:deltt-GP}. We
conclude by noting that Theorem~\ref{thm:D-bounded-convergence} follows from
Corollary~\ref{corrolaryFclose} and Proposition~\ref{thm:deltt-GP}.
\end{proof}

\section{Identifiability}
\label{sec:idnetify}

In this section, we prove Theorem~\ref{thm:identify}. In fact, we will prove
the following version, which by Theorem \ref{theoremsamedistribution} is
equivalent.

\begin{theorem} \label{forwardequiv}
Let $\cW_1=(W_1,S_1,I_1,\bOmega_1)$ and $\cW_2=(W_2,S_2,I_2,\bOmega_2)$ be graphexes,
where $\bOmega_i=(\Omega_i,\cF_i,\mu_i)$ are $\sigma$-finite spaces.
Suppose $\delGP(\cW_1,\cW_2)=0$. Then there exists a third graphex
$\cW=(W,S,I,\bOmega)$ over a $\sigma$-finite measure space
$\bOmega=(\Omega,\cF,\mu)$ and measure preserving maps $\phi_i\colon
\dsupp W_i \rightarrow \Omega$ such that ${\cW_i}|_{\phi^{-1}(\Omega)}=\cW^{\phi_i}$ (and $W_i,S_i=0$ everywhere else) for $i=1,2$.
\end{theorem}

To prove the theorem, we will first prove the following theorem, which may be
of independent interest. We recall that a Borel measure space is a measure
space that is isomorphic to a Borel subset of a complete separable metric
space equipped with a Borel measure, where, as usual, two measure spaces
$\bOmega=(\Omega,\cF,\mu)$ and $\bOmega=(\Omega,\cF,\mu)$ are called
isomorphic if there exists a bijective map $\phi\colon\Omega\to\Omega'$ such
that both $\phi$ and its inverse are measure preserving.

\begin{theorem} \label{coupling}
Let $\cW_1=(W_1,S_1,I_1,\bOmega_1)$ and $\cW_2=(W_2,S_2,I_2,\bOmega_2)$ be
graphexes, where $\bOmega_i=(\Omega_i,\cF_i,\mu_i)$ are $\sigma$-finite Borel
spaces. Suppose further that $D_{\cW_i}>0$ everywhere for $i=1,2$, and
$\delGP(\cW_1,\cW_2)=0$. Then $\mu_1(\Omega_1)=\mu_2(\Omega_2)$, $I_1=I_2$,
and there exists a coupling of $\bOmega_1$ and $\bOmega_2$, that is, a
measure $\nu$ on $(\Omega_1 \times \Omega_2,\cF_1 \times \cF_2)$ with
marginals $\mu_1$ and $\mu_2$, such that if $\pi_i\colon \Omega_1 \times
\Omega_2 \rightarrow \Omega_i$ is the projection map, then
$W_1^{\pi_1}=W_2^{\pi_2}$ $\nu$-almost-everywhere, and
$S_1^{\pi_1}=S_2^{\pi_2}$ $\nu$-almost-everywhere.
\end{theorem}

Theorem~\ref{coupling} should be compared to Proposition 8 from \cite{BCCH16}
which states the analogous result for integrable Borel graphons that have cut
distance zero (without the assumption that $W_1$ and $W_2$ are non-negative),
using a different proof technique. Using still different proof techniques,
Janson proved a similar result (again without assuming non-negativity),
showing that after trivially extending two integrable Borel graphons with cut
distance zero they can be coupled so that the projections are equal almost
everywhere; see \cite{JANSON16}. We will prove Theorem~\ref{coupling} in
Section~\ref{sec:infmin}.

\begin{remark}\label{rem:coupling-signed-generalization}
Throughout this paper, we have considered graphexes where all three parts are
non-negative. While this makes sense when considering graphexes as generators
of a graphex process, from an analytical point of view, it is less natural.
Indeed, it is easy to define the kernel and weak kernel distance for
graphexes where the three parts take values in $\RR$. Taking, e.g., the
kernel distance $\d22$ defined in \eqref{d22-def}, all we need to do is
replace the $L^1$ norms in the third part by a signed ``edge density''
$\rho(\cW_i)=\int W \,d\mu\times d\mu +2 \int S \,d\mu + 2I$, and then use
the third root of $|\rho(\cW_1)-\rho(\cW_2)|$ instead of the third root of
$|\|\cW_1\|_1-\|\cW_2\||$. In particular in view of the just discussed
results from \cite{BCCH16} and \cite{JANSON16}, we conjecture that
Theorem~\ref{coupling} holds for signed graphexes as well, provided the
condition $D_{\cW_i}>0$ is replaced by the condition $D_{|\cW_i|}>0$, where
$|\cW_i|$ is obtained from $\cW_i$ by replacing all three components of
$\cW_i$ by their absolute values. We leave the proof of this conjecture as an
open problem.
\end{remark}

Once we have established Theorem~\ref{coupling}, we will then prove
Theorem~\ref{forwardequiv} by generalizing a construction which was developed
by Janson for the dense case in \cite{JANSON13}. To this end, we will assign
to each graphex $\cW$ a ``canonical version'' $\widehat{\cW}$ such that $\cW$
is a pullback of $\widehat{\cW}$ and show that if two graphexes are
equivalent, then their canonical versions are isomorphic up to measure zero
changes. This will be carried out in Section~\ref{sec:canonical}.

\begin{remark}\label{rem:forwardequivsigned-generalization}
Section~\ref{sec:canonical} does not use nonnegativity in any essential way,
and should easily generalizable to signed graphexes. This should give a
relatively straightforward proof of the analogue of
Theorem~\ref{forwardequiv} for graphons of cut distance zero, and also allow
for the more general setting of signed graphexes, once the above conjectured
generalization of Theorem~\ref{coupling} is established. Again we leave this
as an open problem.
\end{remark}

\subsection{Infimum is minimum} \label{sec:infmin}

In this subsection, we will prove Theorem~\ref{coupling}. The proof will be
based on a series of lemmas.

Let $\widetilde\cW_i=(\wW_i,\widetilde S_i,\widetilde
I_i,\widetilde\bOmega_i)$, for $i=1,2$, be trivial extensions of $\cW_i$ to
spaces of infinite measure, where $\widetilde\bOmega_i=(\widetilde
\Omega_i,\widetilde\cF_i,\widetilde\mu_i)$, and let $\eps>0$. By Proposition
\ref{propequivgraphexes} \eqref{graphexequivepsilonD}, there exist
$\widetilde\Omega_i^\varepsilon\subseteq\widetilde\Omega_i$ such that
$\widetilde\Omega_i \setminus \widetilde\Omega_i^\varepsilon$ has measure at
most $\varepsilon$, and a measure $\nu_\varepsilon$ on
$\widetilde\Omega_1^\varepsilon \times \widetilde\Omega_2^\varepsilon$ with
marginals $\widetilde\mu_1|_{\widetilde\Omega_1^\varepsilon}$ and
$\widetilde\mu_2|_{\widetilde\Omega_2^\varepsilon}$ such that for the
restricted graphexes $\widetilde\cW_{i,\varepsilon}$,
\[
\|\widetilde W_{1,\varepsilon}^{\pi_1}-\widetilde W_{2,\varepsilon}^{\pi_2}\|_{2 \rightarrow 2,\nu_\varepsilon} \le \varepsilon
\]
and
\[
\int_{\widetilde\Omega_1^\varepsilon \times \widetilde\Omega_2^\varepsilon} \left(D_{\widetilde\cW_{1,\varepsilon}}(x)-D_{\widetilde\cW_{2,\varepsilon}}(y)\right)^2\,d\nu_\eps(x,y) \le \varepsilon^{4}.
\]
With a slight abuse of notation, we extend $\nu_\varepsilon$ to
$\widetilde\Omega_1 \times \widetilde\Omega_2$ by zero. This means that in
fact
\[\|\widetilde W_{1}^{\pi_1}-\widetilde W_{2}^{\pi_2}\|_{2 \rightarrow 2, \nu_\varepsilon} \le \varepsilon
\]
and
\[\|D_{\widetilde \cW_1^{\pi_1,\nu_{\varepsilon}}} - D_{\widetilde \cW_2^{\pi_2,\nu_{\varepsilon}}}\|_2 \le \varepsilon^{{2}}.
\]

We first prove the following lemma.

\begin{lemma} \label{lemmadeg}
For any $c$,
\begin{enumerate}
\item $\lim_{\varepsilon\rightarrow
    0}\nu_\varepsilon(\Omega_{1,>c}\times\Omega_{2,>c})=\mu_i(\Omega_{i,>c})$
{for} $i=1,2$ (regardless of the choice of $\nu_\varepsilon$), {and}
\label{diffdegzero}
\item $\mu_1(\Omega_{1,>c})=\mu_2(\Omega_{2,>c})$. \label{degeq}
\end{enumerate}
\end{lemma}

\begin{proof}
We first prove (\ref{diffdegzero}). By symmetry, it suffices to prove it for
$i=1$. Note that for any $x \in \widetilde\Omega_i^\eps$,
\[D_{\widetilde\cW_{i,\varepsilon}}(x) \le D_{\widetilde\cW_i}(x) \le D_{\widetilde\cW_{i,\varepsilon}}(x)+\varepsilon
\]
and that $\widetilde\Omega_{i,>c}=\Omega_{i,>c}$. Therefore,
\begin{align*}
\varepsilon \nu_\varepsilon\bigg(\Omega_{1,>c}\times(\widetilde\Omega_2\setminus&\Omega_{2,>c-(\sqrt{\varepsilon}+\varepsilon)})\bigg)\\
&\le \int_{\Omega_{1,>c}\times(\widetilde\Omega_2\setminus\Omega_{2,>c-(\sqrt{\varepsilon}+\varepsilon)})}
\left( D_{\cW_{1,\varepsilon}}(x)-D_{\cW_{2,\varepsilon}}(y) \right)^2\le \varepsilon^{{4}},
\end{align*}
which implies that
\[
\nu_\varepsilon(\Omega_{1,>c}\times(\widetilde\Omega_2\setminus\Omega_{2,>c-(\sqrt{\varepsilon}+\varepsilon)}))
<\varepsilon^{{3}}.\] Now, for any $\varepsilon$,
\begin{align*}
|\mu_1(&\Omega_{1,>c})-\nu_\varepsilon(\Omega_{1,>c}\times\Omega_{2,>c})|
\leq \eps+|\nu_\varepsilon(\Omega_{1,>c} \times \widetilde\Omega_2)-\nu_\varepsilon(\Omega_{1,>c} \times \Omega_{2,>c})|\\
&=\varepsilon+\nu_\varepsilon(\Omega_{1,>c} \times (\widetilde\Omega_2 \setminus \Omega_{2,>c}))\\
&
=\eps+ \nu_\varepsilon(\Omega_{1,>c}\times(\widetilde\Omega_2\setminus\Omega_{2,>c-(\sqrt{\varepsilon}
+\varepsilon)}))
+\nu_\varepsilon(\Omega_{1,>c} \times (\Omega_{2,>c-(\sqrt{\varepsilon}+\varepsilon)}\setminus\Omega_{2,>c}))\\
& <{\varepsilon}+\eps^3 + \mu_2(\Omega_{2,>c-(\sqrt{\varepsilon}+\varepsilon)}\setminus\Omega_{2,>c}).
\end{align*}
This last expression is finite, and tends to $0$ as $\varepsilon \rightarrow
0$, so
\[\lim_{\varepsilon\rightarrow 0}\nu_\varepsilon(\Omega_{1,>c}\times\Omega_{2,>c})=\mu_1(\Omega_{1,>c}).\]
This proves (\ref{diffdegzero}). From this, (\ref{degeq}) is obvious.
\end{proof}

\begin{lemma} \label{lemmarestricted}
For any $n \ge 1$, there exists a measure $\nu_n$ on $\Omega_{1,>1/n} \times
\Omega_{2,>1/n}$ such that the following hold:
\begin{enumerate}
\item $\nu_n$ is a coupling of $\Omega_{1,>1/n}$ and $\Omega_{2,>1/n}$,
    \label{restrictedcoupling}
\item ${\nu_{n+1}}|_{\Omega_{1,>1/n} \times \Omega_{2,>1/n}}=\nu_{n}$,
    \label{restrictedmeasuresequal}
\item $W_1^{\pi_1}$ and $W_2^{\pi_2}$ are equal when restricted to
    $(\Omega_{1,>1/n} \times \Omega_{2,>1/n})^2$, $\nu_n\times\nu_n$-almost
everywhere, and   \label{restrictedgraphonequal}
\item $D_{\cW_1^{\pi_1}}$ and $D_{\cW_2^{\pi_2}}$ are equal when restricted
    to $\Omega_{1,>1/n} \times \Omega_{2,>1/n}$, $\nu_n$-almost everywhere.
\label{restrictedstarequal}
\end{enumerate}
\end{lemma}

\begin{proof}
It is well known that any two Borel measurable spaces with the same
cardinality are isomorphic; see, e.g., Theorem 8.3.6 in \cite{Cohn-book}. As
a consequence, each Borel measure space $(\Omega,\cF,\mu)$ with
$\mu(\Omega)<\infty$ is either empty or isomorphic to a finite set (with the
discrete topology), the countable set $\{0\} \cup \{1/n:n \in \NN\}$ (with
the induced topology from $\mathbb{R}$), or the Cantor cube
$\cC=\{0,1\}^\infty$ (with the product topology), equipped with the Borel
$\sigma$-algebras generated by the topologies, and a measure that is a finite
Borel measure with full support.

We can therefore assume without loss of generality that for each $i$,
$\Omega_{i,>1}$ and each set $\Omega_{i,>1/(n+1)} \setminus \Omega_{i,>1/n}$
are of this form. This means we may without loss of generality assume the
following properties:

\begin{enumerate}
\item Each $\Omega_{i,>1/n}$, and thus each $\Omega_{i,>1/n}^2$, is compact.
\item For any $i_1,i_2=1,2$ and $n_1,n_2 \in \NN^+$, and any finite Borel
    measure $\nu$ on $\Omega_{i_1,>1/n_1} \times \Omega_{i_2,>1/n_2}$, the
set of all step functions on $\Omega_{i_1,>1/n_1} \times
 \Omega_{i_2,>1/n_2}$ corresponding to partitions of
$\Omega_{i_j,>1/n_j}$ into clopen sets for $j=1,2$ is dense in
$L^1(\Omega_{i_1,>1/n_1} \times \Omega_{i_2,>1/n_2})$.
\end{enumerate}

Now, take a sequence $\varepsilon_k \rightarrow 0$, and recall that we have
an almost coupling measure $\nu_{\varepsilon_k}$ on $\widetilde\Omega_1
\times \widetilde\Omega_2$ with
\[
\|\widetilde W_1^{\pi_1}-\widetilde W_2^{\pi_2}\|_{2 \rightarrow 2,\nu_{\varepsilon_k}} \le \varepsilon_k
\]
and
\[\|D_{\widetilde \cW_1^{\pi_1,\nu_{\varepsilon_k}}} - D_{\widetilde \cW_2^{\pi_2,\nu_{\varepsilon_k}}}\|_2 \le \varepsilon_k^{{2}}.
\]

We know that for each $n$, $\Omega_{1,>1/n} \times \Omega_{2,>1/n}$ is
compact, and for any $K>0$, the set of measures on it bounded by $K$ is
compact under the topology of weak convergence of measures. Since for any
$c$,
\[
\lim_{k\rightarrow \infty}\nu_{\varepsilon_k}(\Omega_{1,>c}\times\Omega_{2,>c})
=\mu_1(\Omega_{1,>c})=\mu_2(\Omega_{2,>c})<\infty,
\]
we can take a subsequence of $\nu_{\varepsilon_k}$ such that for each $n$,
the measure is convergent when restricted to $\Omega_{1,>1/n} \times
\Omega_{2,>1/n}$. Without loss of generality we assume that the original
sequence has this property. For each $n$, we then define $\nu_n$ as the limit
measure on $\Omega_{1,1/n} \times \Omega_{2,>1/n}$. Having defined $\nu_n$ we
now prove (1)--(4).

(\ref{restrictedcoupling}) We have seen in Lemma \ref{lemmadeg} that
\[\nu_n(\Omega_{1,>1/n}\times\Omega_{2,>1/n})=\mu_1(\Omega_{1,>1/n})=\mu_2(\Omega_{2,>1/n}).\]
For any clopen $F \subseteq \Omega_{1,>1/n}$, $F \times \Omega_{2,>1/n}$ is
clopen in $\Omega_{1,>1/n}\times\Omega_{2,>1/n}$. Therefore,
\[
\lim_{k \rightarrow \infty} \nu_{\varepsilon_k}(F \times \Omega_{2,>1/n}) =
\nu_n(F \times \Omega_{2,>1/n}) .\] On the other hand, by Lemma
\ref{lemmadeg}{,}
\[\nu_{\varepsilon_k}(F \times (\widetilde\Omega_2 \setminus \Omega_{2,>1/n})) \le \nu_{\varepsilon_k}(\Omega_{1,>1/n}
\times (\widetilde\Omega_2 \setminus \Omega_{2,>1/n})) \xrightarrow{k \rightarrow \infty} 0.\] We also have that
\[|\nu_{\varepsilon_k}(F \times \widetilde\Omega_2)-\mu_1(F)| \le \varepsilon_k \xrightarrow{k \rightarrow \infty} 0 .
\]
Therefore for any clopen set, $\mu_1(F) = \nu_n(F \times \Omega_{2,>1/n})$.
This implies that $\mu_1$ and the projection of $\nu_n$ onto
$\Omega_{1,>1/n}$ are the same, which proves (\ref{restrictedcoupling}).

(\ref{restrictedmeasuresequal}) Since $\Omega_{1,>1/n} \times
\Omega_{2,>1/n}$ is clopen in $\Omega_{1,>1/(n+1)} \times
\Omega_{2,>1/(n+1)}$,
\[\nu_{n+1}(\Omega_{1,>1/n} \times \Omega_{2,>1/n})=\lim_{k \rightarrow \infty} \nu_{\varepsilon_k}(\Omega_{1,>1/n}\times \Omega_{2,>1/n})=\nu_n(\Omega_{1,>1/n}\times \Omega_{2,>1/n})
.\] Furthermore, for any closed $F \subseteq \Omega_{1,>1/n} \times
\Omega_{2,>1/n}$, $F$ is also closed in $\Omega_{1,1/(n+1)} \times
\Omega_{2,1/(n+1)}$, which implies that
\[\limsup_{k \rightarrow \infty} \nu_{\varepsilon_k}(F) \le \nu_{n+1}(F).
\] This implies that $\nu_{\varepsilon_k}$ converges
{weakly} to ${\nu_{n+1}}|_{\Omega_{1,>1/n} \times \Omega_{2,>1/n}}$, but
since it also converges to $\nu_n$, the two must be equal.

(\ref{restrictedgraphonequal}) Let $W_{i,n}={\widetilde
W_i}|_{(\Omega_{i,>1/n})^2}$. Since
\[ \|\widetilde W_1^{\pi_1}-\widetilde W_2^{\pi_2}
\|_{2 \rightarrow 2,\nu_{\varepsilon_k}} \le \varepsilon_k
,\]
we have that in particular, for any $n$,
\[\|W_{1,n}^{\pi_1}-W_{2,n}^{\pi_2}\|_{2 \rightarrow 2,\nu_{\varepsilon_k}} \le
\varepsilon_k.
\]
This implies that
\[\|W_{1,n}^{\pi_1}-W_{2,n}^{\pi_2}\|_{\square,\nu_{\varepsilon_k}} \le
\varepsilon_k \mu_1(\Omega_{1,>1/n}).
\]

Since $\Omega_{i,>1/n}$ each have finite measure, we can use Janson's
argument in \cite{JANSON13}. We present the argument for completeness. Fix
$\varepsilon$. We can find step graphons $U_{1,n}$ and $U_{2,n}$ on
$(\Omega_{1,>1/n})^2$ and $(\Omega_{2,>1/n})^2$, with each part in the
partition being a clopen set, such that
\[\|U_{i,n}-W_{i,n}\|_1 \le {\varepsilon}.
\]
This means that for any coupling measure on $\Omega_{1,>1/n} \times \Omega_{2,>1/n}$,
\[
\|
U_{i,n}^{\pi_i}-W_{i,n}^{\pi_i}\|_1 \le \varepsilon.
\]
Now, $U_{1,n}^{\pi_1}-U_{2,n}^{\pi_2}$ is a step function on
$(\Omega_{1,>1/n} \times \Omega_{2,>1/n})^2$ with a partition into clopen
parts. This means that since the restrictions of $\nu_{\varepsilon_k}$ weakly
converge to $\nu_n${,}
\[
\|U_{1,n}^{\pi_1}-U_{2,n}^{\pi_2}\|_{\square,\nu_{\varepsilon_k}}
\xrightarrow{k \rightarrow \infty}
\|U_{1,n}^{\pi_1}-U_{2,n}^{\pi_2}\|_{\square,\nu_{n}} .\] Take $k$ large
enough so that $\varepsilon_k\mu_1(\Omega_{1,>1/n}) \le \varepsilon$ and
\[\bigg|\|U_{1,n}^{\pi_1}-U_{2,n}^{\pi_2}\|_{\square,\nu_{\varepsilon_k}} - \|U_{1,n}^{\pi_1}-U_{2,n}^{\pi_2}\|_{\square,\nu_{n}}\bigg| \le \varepsilon.\]
We then have
\[
\begin{split}
\|W_{1,n}^{\pi_1}&-W_{2,n}^{\pi_2}\|_{\square,\nu_n}
\\
&\le
\|W_{1,n}^{\pi_1}-U_{1,n}^{\pi_1}\|_{1,\nu_n}+\|U_{1,n}^{\pi_1}-U_{2,n}^{\pi_2}\|_{\square,\nu_{n}} +\|U_{2,n}^{\pi_2}-W_{2,n}^{\pi_2}\|_{1,\nu_n}\\
&\le 3\varepsilon + \|U_{1,n}^{\pi_1}-U_{2,n}^{\pi_2}\|_{\square,\nu_{\varepsilon_k}}\\
& \le 3\varepsilon +
\|W_{1,n}^{\pi_1}-U_{1,n}^{\pi_1}\|_{1,\nu_{\varepsilon_k}}
+\|W_{1,n}^{\pi_1}-W_{2,n}^{\pi_2}\|_{\square,\nu_{\varepsilon_k}}
+\|U_{2,n}^{\pi_2}-W_{2,n}^{\pi_2}\|_{1,\nu_{\varepsilon_k}}\\
& \le 6\varepsilon.
\end{split}
\]
Since this holds for any $\varepsilon$, this proves
(\ref{restrictedgraphonequal}).

(\ref{restrictedstarequal}) Fix $\eps>0$ and assume that $\eps_k\leq\eps$.
Since $D_{\widetilde\cW_{i,\varepsilon_k}}(x) \le D_{\cW_i}(x) \le
D_{\widetilde\cW_{i,\varepsilon_k}}(x)+\varepsilon_k$ for all $x \in
\Omega_i\cap\widetilde\Omega_i^\eps$,
\begin{align*}
&\nu_\eps\left(\left\{
x\in\Omega_{1,>1/n}\times\Omega_{2,>1/n} :
\left|D_{\cW_1^{\pi_1}}(x)-D_{\cW_2^{\pi_2}}(x)\right|\geq 2\eps
\right\}\right)
\\
&\qquad\leq
\nu_\eps\left(\left\{
x\in\Omega_{1,>1/n}\times\Omega_{2,>1/n} :
\left|D_{\widetilde \cW_1^{\pi_1,\nu_{\varepsilon_k}}}(x) - D_{\widetilde \cW_2^{\pi_2,\nu_{\varepsilon_k}}}(x)\right|\geq \eps
\right\}\right)\\
&\qquad\leq \eps^{-2}\|D_{\widetilde \cW_1^{\pi_1,\nu_{\varepsilon_k}}} - D_{\widetilde \cW_2^{\pi_2,\nu_{\varepsilon_k}}}\|_2^2 \le {\eps^{-2}}{\varepsilon_k^4}.
\end{align*}
Since for all $\eps>0$ the right side converges to $0$ as $k\to\infty$ this
shows that $\nu_n$ is supported on $\{x \in \Omega_{1,>1/n} \times
\Omega_{2,>1/n}: D_{\cW_1^{\pi_1}}(x)=D_{\cW_2^{\pi_2}}(x)\}$.
\end{proof}

\begin{proof}[Proof of Theorem~\ref{coupling}]
{After these preparations, we are ready to} define the measure $\mu$ on
$\Omega_1 \times \Omega_2$. Note that since before the extensions
$D_{\cW_i}>0$ almost everywhere, we have that
$\bigcup_{n}\Omega_{i,>1/n}=\Omega_i$. For $A \subseteq \Omega_1 \times
\Omega_2$, let
\[
\begin{split}
\nu(A)&=\lim_{n \rightarrow \infty} \nu_n(A \cap (\Omega_{1,>1/n} \times
\Omega_{2,>1/n}))\\
&=\sum_{n=1}^\infty \nu_n(A \cap ((\Omega_{1,>1/n} \times
\Omega_{2,>1/n}) \setminus (\Omega_{1,>1/(n-1)} \times \Omega_{2,>1/(n-1)}))).
\end{split}
\]
Here with a slight abuse of notation we think of $\Omega_{i,1/0}$ as the
empty set. To show that this a coupling, note that for any measurable set $X
\subseteq \Omega_1$,
\[
\begin{split}
\mu_1(X) &= \lim_{n \rightarrow \infty} \mu_1(X \cap \Omega_{1,>1/n})\\
&=\lim_{n \rightarrow \infty} \nu_n((X \cap \Omega_{1,>1/n}) \times \Omega_{2,>1/n})\\
&=
\lim_{n \rightarrow \infty} \nu_n((X \times \Omega_2) \cap (\Omega_{1,>1/n} \times \Omega_{2,>1/n}{))}
=\nu(X \times \Omega_2),
\end{split}
\]
where in the second step we used that for each $n$, the restriction of $\nu$
to $\Omega_{1,>1/n} \times \Omega_{2,>1/n}$ is equal to $\nu_n$, which is a
coupling. Clearly the analogous argument works for subsets
$Y\subseteq\Omega_2$.

Now, let
\[
N=\{(x,y)=((x_1,x_2),(y_1,y_2)) \in (\Omega_1 \times \Omega_2)^2:W_1^{\pi_1}(x,y) \neq W_2^{\pi_2}(x,y)\}
\]
and
\[
M=\{x=(x_1,x_2)\in (\Omega_1 \times \Omega_2):
D_{\cW_1^{\pi_1}}(x)\neq D_{\cW_2^{\pi_2}}(x)
\}
.
\]
Let
\[N_n=\{(x,y) \in (\Omega_{1,>1/n} \times \Omega_{2,>1/n})^2:W_1^{\pi_1}(x,y) \neq W_2^{\pi_2}(x,y)\}
\]
and
\[
M_n=\{x=(x_1,x_2)\in (\Omega_{1,>1/n} \times \Omega_{2,>1/n}):
D_{\cW_1^{\pi_1}}(x)\neq D_{\cW_2^{\pi_2}}(x)
\}
.
\]
Since ${\bigcup_n}\Omega_{i,>1/n}=\Omega_i$ and $\Omega_{i,>1/n}\subseteq
\Omega_{i,1/(n+1)}$, we have that $N=\bigcup_n N_n$ and $M=\bigcup_n M_n$. By
Lemma \ref{lemmarestricted}
(\ref{restrictedgraphonequal},\ref{restrictedstarequal}), $(\nu \times
\nu)(N_n)=(\nu_n \times \nu_n) (N_n) = 0$ and $\nu(M_n)=0$, which implies
that $(\nu \times \nu)(N)=0$ and $\nu(M)=0$, and hence
$W_1^{\pi_1}=W_2^{\pi_2}$ and $D_{\cW_1^{\pi_1}}= D_{\cW_2^{\pi_2}}$ almost
everywhere. Since $W_1^{\pi_1}=W_2^{\pi_2}$ almost everywhere implies that
$D_{W_1^{\pi_1}}= D_{W_2^{\pi_2}}$ almost everywhere, this in turn implies
that $S_1^{\pi_1}=S_2^{\pi_2}$ $\nu$-almost everywhere.

To prove that $I_1=I_2$, we again use Proposition~\ref{propequivgraphexes},
but instead of \eqref{graphexequivepsilonD} we this time use
\eqref{graphexequiveachDsame}. Fix $D>0$. Since $\deltt(\cW_{1,{\leq
D}},\cW_{2,{\leq D}})=0$, we in particular have that $\|\cW_{1,\leq
D}\|_1=\|\cW_{2,\leq D}\|_1$. But since $W_1^{\pi_1}=W_2^{\pi_2}$ and
$D_{\cW_1^{\pi_1}}= D_{\cW_2^{\pi_2}}$ almost everywhere,
\begin{align*}
\int_{(\Omega_1,\leq D)^2}W_1\,d\mu_1\times d\mu_1
&=\int_{(\Omega_1\times\Omega_2)^2}W_1^{\pi_1}1_{D_{\cW_1^{\pi_1}}\leq D}\,d\nu\times d\nu\\
&=\int_{(\Omega_1\times\Omega_2)^2}W_2^{\pi_2}1_{D_{\cW_2^{\pi_2}}\leq D}\,d\nu\times d\nu\\
&=\int_{(\Omega_2,\leq D)^2}W_2\,d\mu_2\times d\mu_1.
\end{align*}
In a similar way, $\int_{\Omega_1,\leq D}S_1=\int_{\Omega_2,\leq D}S_2$.
Therefore $\|\cW_{1,\leq D}\|_1=\|\cW_{2,\leq D}\|_1$ implies $I_1=I_2$.
\end{proof}

\begin{corollary} \label{chain}
Let $\cW=(W,S,I,\bOmega)$ and $\cW'=(W',S',I',\bOmega')$ be graphexes that
are equivalent, and suppose that $D_\cW,D_{\cW'}>0$ everywhere. Then there
exist a positive integer $n$ and a chain of graphexes
$\cW_i=(W_i,S_i,I_i,\bOmega_i)$ for $i=0,\dots,n$, with $D_{\cW_i}>0$
everywhere for each $i=0,\dots,n$, $\cW_0=\cW$, $\cW_n=\cW'$, and for each $i
\ge 1$, either $\cW_{i-1}=\cW_i^{\phi_i}$ almost everywhere for some measure
preserving map $\phi_i$ from $\bOmega_{i-1}$ to $\bOmega_i$, or
$\cW_{i}=\cW_{i-1}^{\phi_i}$ almost everywhere for some $\phi_i$ from
$\bOmega_i$ to $\bOmega_{i-1}$. In fact, we can take $n=4$.
\end{corollary}

\begin{proof}
By the construction in the next section (which itself does not use
Corollary~\ref{chain}), there exists $\cW_1=(W_1,S_1,I_1,\bOmega_1)$ such
that $\bOmega_1$ is Borel and a measure-preserving map $\phi_1$ from
$\bOmega_0$ to $\bOmega_1$ with $W=W_0=W_1^{\phi_1}$ almost everywhere. Then
$\Omega_0=\dsupp W_0=\phi^{-1}(\dsupp W_1)$, so by replacing $\bOmega_1$ by
its restriction to $\dsupp W_2$, we may assume that $D_{W_1}>0$ everywhere.
Similarly, there exists $(W_3,\bOmega_3)$ with $\bOmega_3$ Borel and a
measure preserving map $\phi_4$ from $\bOmega_4$ to $\bOmega_3$ with
$W'=W_4=W_3^{\phi_4}$ almost everywhere and $D_{W_3}>0$ everywhere. Now, we
can apply Theorem \ref{theoremsamplesconverge} to show that
$\delGP(\cW_1,\cW_3)=0$. We can then apply Theorem \ref{coupling} to
$(\cW_1,\bOmega_1)$ and $(\cW_3,\bOmega_3)$ to find a Borel space $\bOmega_2$
and measure preserving maps $\phi_2$ from $\bOmega_2 $ to $\bOmega_1$ and
$\phi_3$ from $\bOmega_2$ to $\bOmega_3$ such that
$W_3^{\phi_3}=W_1^{\phi_2}$ almost everywhere. We can then take $W_2$ to be,
say, $W_1^{\phi_2}$.
\end{proof}

\subsection{Canonical graphex} \label{sec:canonical}
In this section, we prove Theorem~\ref{forwardequiv}. We follow the approach
of Janson in \cite{JANSON13}, based on the construction of Lov\'asz and
Szegedy in \cite{LS10}.

Concretely, given a graphex $\cW=(W,S,I,\bOmega)$, where
$\bOmega=(\Omega,\cF,\mu)$, define a map $\psi_W\colon \Omega \rightarrow
L^1(\Omega,\cF,\mu)$ by $x \mapsto W(x,\cdot)$, and define a map
$\psi_\cW\colon \Omega \rightarrow L^1(\Omega,\cF,\mu) \times \RR$ by $x
\mapsto (W(x,\cdot),S(x))$, where we equip $L^1(\Omega,\cF,\mu)$ with the
standard Borel $\sigma$-algebra and $L^1(\Omega,\cF,\mu) \times \RR$ with the
standard product Borel $\sigma$-algebra. Note that in general, we only know
that $\psi_W(x)\in L^1(\Omega,\cF,\mu)$ for almost all $x\in\Omega$, but by
changing $W$ on a set of measure zero, we may assume that this holds for all
$x\in\Omega$; we will assume that throughout this section. We will see in
Lemma~\ref{lemmamapmeasurable} that $\psi_W$, and thus $\psi_\cW$, is
measurable. Defining $\mu_W$ and $\mu_\cW$ as the pushforward of $\mu$ under
$\psi_W$ and $\psi_\cW$ respectively, and $\Omega_W$ and $\Omega_\cW$ as the
corresponding supports, we then construct a graphex $\widehat \cW$ over
$\Omega_\cW$ (equipped with the Borel $\sigma$-algebra and the measure
$\mu_\cW$) such that $\cW$ is almost everywhere equal to a pullback of
$\widehat \cW$. Furthermore, we will show that if $\cW'$ is a.e.\ equal to a
pullback of $\cW$, then we can find a measure preserving bijection $\phi^*$
from $\Omega_\cW$ to $\Omega_\cW'$ such that $\widehat{\cW}=(\widehat
{\cW'})^{\phi^*}$ a.e. Combined with Corollary~\ref{chain}, this will
establish Theorem~\ref{forwardequiv}.

We first state some preliminary lemmas. Recall that if $(\Omega,\cF,\mu)$ is
a measure space, then $L^1(\Omega,\cF,\mu)$ is a Banach space where each
point is an equivalence class consisting of integrable measurable functions
from $\Omega$ to $\mathbb{R}$, where two functions are equivalent if they are
equal almost everywhere (or equivalently their $L^1$ distance is $0$). Note
that in general, the space $L^1(\Omega,\cF,\mu)$ is not separable, a fact
which will lead to technical complications when considering measurable
functions into $L^1(\Omega,\cF,\mu)$ (e.g., the sum of two such functions is
in general not measurable). As in \cite{JANSON13}, we will avoid these
difficulties by carefully constructing separable subspaces of
$L^1(\Omega,\cF,\mu)$ such that the functions of interest take values in
these subspaces. See Lemma~\ref{lemmamapseparable} below.

Throughout this section, we will frequently consider two measure spaces
$(\Omega,\cF,\mu)$ and $(\Omega',\cF',\mu')$, where $\Omega=\Omega'$,
$\cF'\subseteq \cF$, and $\mu'$ is the restriction of $\mu$ to $\cF'$. With a
slight abuse of notation, we will often denote the second space by
$(\Omega,\cF',\mu)$, rather than $(\Omega,\cF',\mu|_{\cF'})$.

As already noted, the space $L^1(\Omega,\cF,\mu)$ is in general not separable
(if $\Omega$ is a Borel space, it is, but we want to define this construction
for general $\Omega$). We will therefore need the following lemma.

\begin{lemma} \label{lemmamapseparable}
Suppose that $(\Omega,\cF,\mu)$ and $(\Omega',\cF',\mu')$ are $\sigma$-finite
measure spaces, $W\colon \Omega \times \Omega' \rightarrow {\RR}$ is
measurable, and for all $x\in\Omega$, the function $W(x,\cdot)$ is
integrable. Let $\psi_W\colon \Omega \rightarrow L^1(\Omega,\cF,\mu)$ be the
map $x \mapsto W(x,\cdot)$. Then we can find a separable closed subspace $B$
of $L^1(\Omega',\cF'{,\mu'})$ such that $\psi_W(x)\in B$ for all $x\in
\Omega$.
\end{lemma}

\begin{proof}
First, assume that $W$ is bounded and both $\mu(\Omega)<\infty$ and
$\mu(\Omega')<\infty$. The statement of the lemma then clearly holds for all
step functions, and by a monotone class argument it holds for all bounded
$W$.

Next, relax the condition that $\mu'(\Omega')<\infty$. Let $\Omega_1'
\subseteq \Omega_2' \subseteq \dots \subseteq \Omega_n' \subseteq \dots$ be a
sequence of measurable subsets of $\Omega'$ with finite measure and
$\bigcup_{n=1}^\infty \Omega_n'=\Omega'$. Then for every $n$ we can find a
separable closed subspace $B_n \subseteq
L^1(\Omega_n',{\cF'|_{\Omega_n'},\mu'|_{\Omega_n'})}$ such that for every $x
\in \Omega$, $W(x,\cdot)|_{\Omega_n'} \in B_n$. Let $\widetilde{B}_n$ consist
of those $f \in L^1(\Omega',\cF'{,\mu'})$ that have $f|_{\Omega_n'} \in B_n$
and $f|_{\Omega'-\Omega_n'} \equiv 0$. Clearly $\widetilde{B}_n$ is
isomorphic to $B_n$, and thus separable. Let $B$ be the closure of the space
generated by $\bigcup_n \widetilde{B}_n${;} this is separable. We claim that
for any $x \in \Omega$, the function $W(x,\cdot)$ is contained in $B$. It
suffices to show that for any $\varepsilon$, there is an $n\in\NN$ and a $g
\in \widetilde{B}_n$ such that $\|W(x,\cdot)-g\|_1 <\varepsilon$. Since
$W(x,\cdot)\in L^1(\Omega',\cF'{,\mu'})$, we can take $n$ large enough that
$\|W(x,\cdot)-W(x,\cdot)\chi(\Omega_n')\|_1<\varepsilon$. But then by the
definition of $\widetilde{B}_n$, we have $W(x,\cdot)\chi(\Omega_n') \in
\widetilde{B}_n$, so taking $g=W(x,\cdot)\chi(\Omega_n')$, we are done.

In a similar way, we can approximate an unbounded $W$ by the function
$W1_{|W|\leq n}$ to relax the condition that $W$ is bounded.

Finally, for general $\Omega$ and $\Omega'$, we can write $\Omega$ as the
disjoint union of finite measure sets
$\Omega_1,\Omega_2,\dots,\Omega_n,\dots$. We know that the image of each
$\Omega_n$ is contained in a separable closed subspace $B_n \subseteq
L^1(\Omega',\cF')$. Therefore, taking $B$ to be the closure of the subspace
generated by $\bigcup_n B_n$, the image of the map $\psi_W$ is contained in
$B$.
\end{proof}

\begin{lemma} \label{lemmamapmeasurable}
Suppose $(\Omega,\cF,\mu)$ and $(\Omega',\cF',\mu')$ are $\sigma$-finite
measure spaces, $W\colon \Omega \times \Omega' \rightarrow {\RR}$ is $\cF
\times \cF'$-measurable, and that for all $x\in\Omega$, the function
$W(x,\cdot)$ is integrable. Then the map $\psi_W\colon \Omega \rightarrow
L^1(\Omega',\cF'{,\mu'})$ with $x \mapsto W(x,\cdot)$ is measurable with
respect to the standard Borel $\sigma$-algebra on $L^1(\Omega,\cF,\mu)$. .
\end{lemma}

\begin{proof}
Our goals is to show that $c\in [0,\infty)$ and any $f \in
L^1(\Omega',\cF'{,\mu'})$, the set
\[
B(f,W,c)=\{x \in \Omega:
\|W(x,\cdot)-f\|_1 \le c\}
\]
is measurable.

Assume that $x\in B(f,W,c)$, and let $F_W\subset L^1(\Omega',\cF',\mu')$ be a
countable set such that $W(x,\cdot)$ lies in the closure of $F_W$ for all
$x\in \Omega$ (the existence of such a set follows from Lemma
\ref{lemmamapseparable}). Given any $\varepsilon>0$, we can then find an
$\widehat f\in F_W$ such that $x\in B(\widehat f,W,\eps)$, and $\|f-\widehat
f\|_1\leq c+\eps$. If, on the other hand, $\|f-\widehat f\|_1\leq c+\eps$ and
$x\in B(\widehat f,W,\eps)$ then $x\in B( f,W,c+2\eps)$. Since $ B(
f,W,c)=\bigcap_i B( f,W,c+\eps_i)$ whenever $\eps_i\to 0$, this proves that
it is enough to prove measurability of $B(\widehat f,W,\eps_{i})$ for all
$\widehat f\in F_W$ and an arbitrary sequence $\eps_i\in (0,\infty)$ such
that $\eps_i\to 0$.

Using this observation, it is easy to see that if $W_1$ and $W_2$ obey the
conclusions of the lemma, then so does any linear combination. The lemma is
clearly also true for all step functions. A standard monotone class argument
then implies that the lemma holds for all bounded, measurable $W$.

If $W$ is unbounded, we use that by assumption, $\psi_W(x)\in
L^1(\Omega',\cF',\mu')$ for all for all $x\in\Omega$. Using this fact, one
easily shows that
\[
B(f,W,c)=\bigcap_{n{=}1}^\infty B(f,W(x,\cdot)1_{|W(x,\cdot)|\leq n},c).
\]
Since $B(f,W(x,\cdot)1_{|W(x,\cdot)|\leq n},c)$ is measurable, this proves
the statement for unbounded $W$,
\end{proof}

We will also use the following technical lemma, which is Lemma~G.1 in
\cite{JANSON13} (the proof also works for $\sigma$-finite measures):

\begin{lemma} \label{lemmaevaluationmap}
Let $(\Omega,\cF,\mu)$ be any $\sigma$-finite measure space, and $B
\subseteq L^1(\Omega,\cF{,\mu})$ a closed separable subspace. Then
there exists a measurable evaluation map
\[
\Phi\colon B \times \Omega \rightarrow \mathbb{R}
\]
such that for any $f \in B$, for almost every $x \in \Omega$,
$f(x)=\Phi(f,x)$ (note that $f(x)$ is only defined almost everywhere). In
particular, if $L^1(\Omega,\cF{,\mu})$ is separable, we can take
$B=L^1(\Omega,\cF,\mu)$.
\end{lemma}

We need one more lemma.

\begin{lemma} \label{lemmacontainedinFW}
Suppose that $(\Omega,\cF,\mu)$ and $(\Omega',\cF',\mu')$ are $\sigma$-finite
measure spaces, $W\colon \Omega \times \Omega' \rightarrow {\RR}$ is
measurable, and for all $x\in \Omega$ and $x'\in\Omega'$, both $W(x,\cdot)$
and $W(\cdot,x')$ are integrable. Let $\psi_W$ be the map $\Omega \rightarrow
L^1(\Omega',\cF'{,\mu'})$ with $x \mapsto W(x,\cdot)$, and $\psi_W'$ be the
map $\Omega' \rightarrow L^1(\Omega,\cF{,\mu})$ {with} $y \mapsto
W(\cdot,y)$. Further, let $\cG$ be the Borel $\sigma$-algebra on
$L^1(\Omega,\cF,\mu)$, and let
\[\cF'_W={\psi'}_W^{-1}({\cG}).
\]
Then for almost every $x \in \Omega$,
$\psi_W(x)\in L^1(\Omega',\cF'_W{,\mu'})$.
\end{lemma}

\begin{proof}
By Lemma \ref{lemmamapseparable}, {the} image of $\psi'_W$ is contained in a
separable subspace $B \subseteq L^1(\Omega,\cF{,\mu})$. By Lemma
\ref{lemmaevaluationmap}, there exists a measurable map $\Phi\colon B \times
\Omega \rightarrow \mathbb{R}$ such that for any $f \in B$ and almost any $x
\in \Omega$, $f(x)=\Phi(f,x)$. Define $\widetilde{W}$ on $\Omega \times
\Omega'$ by
\[\widetilde{W}(x,y)=\Phi(\psi_{W}'(y),x).\]
Then for every $y$ and almost every $x$, $\widetilde{W}(x,y)=W(x,y)$, so
$\widetilde{W}=W$ almost everywhere on $\Omega \times \Omega'$. Thus, if we
define $\psi_{\widetilde{W}}$ analogously to $\psi_W$, then for almost every
$x\in\Omega$,
$\|\psi_{\widetilde{W}}(x)-\psi_W(x)\|_{L^1(\Omega',\cF',\mu')}=0$; that is,
for almost all $x$, $\psi_W(x)=\psi_{\widetilde W}(x)$ as elements of
$L^1(\Omega',\cF'{,\mu'})$. However, since $\Phi$ is $B \times
\Omega$-measurable, and $\psi_W'\colon \Omega' \rightarrow B$ is measurable,
$\widetilde{W}$ is $\cF \times \cF'_W$-measurable, so in particular
$\psi_{\widetilde W}(x)$ is $\cF'_W$-measurable for all $x$. Since
$\psi_W(x)=\psi_{\widetilde{W}}(x)\in L^1(\Omega',\cF'{,\mu'})$ for almost
every $x$, it follows that $\psi_W(x) \in L^1(\Omega',\cF'_W{,\mu'})$ for
almost every $x\in\Omega$.
\end{proof}

Let $(W,S,I,\bOmega)$ be a graphex over $\bOmega=(\Omega,\cF,\mu)$ such that
$W(x,\cdot)$ is integrable for all $x$. For each $x \in \Omega$, we have the
section $W_x \in L^1(\Omega,\cF{,\mu})$ defined by $W_x(y)=W(x,y)$, giving us
the map
\[
\psi_W\colon \Omega \rightarrow
L^1(\Omega,\cF{,\mu})
\]
defined by $x\mapsto W_x$. Let
\[
\psi_\cW \colon \Omega \rightarrow L^1(\Omega,\cF{,\mu}) \times \RR
\]
be defined by $ x\mapsto (W_x,S(x))$. By Lemma \ref{lemmamapmeasurable},
$\psi_W$, and thus $\psi_\cW$, is measurable. Let $\mu_W=\mu^{\psi_W}$ and
$\mu_\cW=\mu^{\psi_\cW}$, and let $\Omega_W\subseteq L^1(\Omega,\cF,\mu)$ and
$\Omega_\cW\subseteq L^1(\Omega,\cF,\mu) \times \RR$ be the supports of
$\mu_W$ and $\mu_\cW$, respectively, i.e.,
\[ \Omega_W = \{f \in L^1(\Omega,\cF,\mu): \mu_W(U)>0
\text{ for every open $U \subseteq L^1(\Omega,\cF{,\mu})$ with $f \in U$}\},
\] and
\begin{align*}
\Omega_\cW = \{(f,c) &\in L^1(\Omega,\cF{,\mu}) \times \RR: \\
&\text{$\mu_\cW(U)>0$ for every open $U \subseteq L^1(\Omega,\cF{,\mu}) \times \RR$ with
$(f,c) \in U$}\}.
\end{align*}
Alternatively, we can also define $\psi_\cW$ and $\mu_\cW$ as follows. Let
$(\wOmega, \wcF,\wmu)$ be defined as $\wOmega=\Omega \cup \{\Omega_\infty\}$,
where $\Omega_\infty$ is an atom with measure $1$. Then we can think of
$\psi_\cW(x)$ as a function in $L^1(\wOmega, \wcF,\wmu)$, with
$\psi_\cW(x)(y)=W(x,y)$ if $y\in \Omega$ and $\psi_\cW(x)(y)=S(x)$ if
$y=\Omega_\infty$, and $\|\psi_\cW(x)\|_1=D_\cW(x)$. This gives a bijection
$L^1(\wOmega, \wcF,\wmu) \equiv L^1(\Omega,\cF,\mu) \times \RR$ and
$\psi_\cW$ as a map from $\Omega$ to $L^1(\wOmega, \wcF,\wmu)$. Note that
$\mu_W$ is the projection of $\mu_\cW$, and thus $\Omega_W$ is the closure of
the projection of $\Omega_\cW$. Equipping $\Omega_W$ with the standard Borel
$\sigma$-algebra $\cG_W$, this gives us a measure space
$(\Omega_W,\cG_W,\mu_W)$, and similarly we obtain
$\Omega_\cW=(\Omega_\cW,\cG_\cW,\mu_\cW)$.

Let $\cG$ and $\widetilde \cG$ be the Borel $\sigma$-algebra on
$L^1(\Omega,\cF,\mu)$ and $L^1(\wOmega, \wcF,\wmu)$, respectively. Via the
maps $\psi_W$ and $\psi_\cW$ they induce two different $\sigma$-algebras on
$\Omega$, the $\sigma$-algebras
\[
\cF_W=\psi_W^{-1}(\cG)
\]
and
\[
\cF_\cW=\psi_\cW^{-1}(\widetilde\cG).
\]
We also define $\wcF_W=\cF_W \times \mathcal{B}$ and $\wcF_\cW=\cF_\cW \times
\mathcal{B}$. Note that $\cF_W \subseteq \cF_\cW$, with the example of a zero
graphon but a nonconstant $S$ function showing that strict inequality is
possible.

It is easy to see that if $\cW'$ is equal to $\cW$ almost everywhere, then
$\mu_\cW=\mu_{\cW'}$ and hence $\Omega_\cW=\Omega_{\cW'}$. Indeed, if
$\cW=\cW'$ a.e., then for almost all $x$, $\psi_\cW(x)=\psi_{\cW'}(x)$ when
viewed as vectors in $L^1$. This implies that there exists a set $N
\subseteq\Omega$ of measure zero such that for all $A\in \widetilde\cG$, the
symmetric difference of $\psi_\cW^{-1}(A)$ and $\psi^{-1}_{\cW'}(A)$ lies in
$N$, which shows that $\mu_\cW=\mu_{\cW'}$. Furthermore, under the same
change, $\cF_W,\cF_\cW$ and $\cF_{W'},\cF_{\cW'}$ only change on a set of
measure zero, implying that the Banach spaces $L^1(\Omega,\cF_W,\mu)$ and
$L^1(\Omega,\cF_\cW,\mu)$ remain unchanged.

Note that in general, $\mu_\cW$ is not $\sigma$-finite. Indeed, choosing $W$
to be the graphon $0$ over any space of infinite measure, and $I,S$ to be
$0$, we have that $\psi_\cW^{-1}(A)=\Omega$ for every $A$ containing the
origin $(0)\in L^1(\wOmega,\wcF,\wmu)$, so $\mu_\cW(A)=\mu(\Omega)=\infty$ if
$0\in A$, and $\mu_\cW(A)=0$ otherwise. So in particular $\Omega_\cW=\{(0)\}$
and $\mu_\cW(\Omega_\cW)=\infty$. This also means $\mu_W$ is not
$\sigma$-finite.

It turns out, however, that this problem can be avoided if we require that
the set where $D_\cW=0$ has measure zero; see Lemma~\ref{lem:sigma-finite}
below. Before stating the lemma, we prove the following:

\begin{proposition}
\label{prop:measurepres} The space $\Omega_\cW$ as defined above is a
complete, separable metric space, with the metric induced by
$L^1(\wOmega,\wcF,\wmu)$, and $\mu_\cW$ has full support in $\Omega_\cW$.
Furthermore, $\Omega_\cW \subseteq L^1(\wOmega,\wcF_W,\wmu) \subseteq
L^1(\wOmega,\wcF_\cW,\wmu)$. Finally, after modifying $W$ and $S$ on a set of
measure zero $\psi_\cW$ becomes an everywhere defined, measure{-}preserving
map from {$(\Omega,\cF,\mu)$} to $(\Omega_\cW,{\cG}_\cW,\mu_\cW)$, and we
furthermore have $\psi_\cW^{-1}(\cG_\cW)=\cF_\cW$.
\end{proposition}

\begin{proof}
By Lemma~\ref{lemmamapseparable}, the image of $\psi_W$ is contained in a
closed separable subspace $B$ of $L^1(\Omega,\cF{,\mu})$. This means that the
image of $\psi_\cW$ is contained in $\wB=B \times \RR$, which is a closed
separable subspace of $L^1(\wOmega,\wcF,\wmu)$. We will show that in fact
\[
\Omega_W = \{f \in B: \mu_W(U)>0 \text{ for {every} open $U \subseteq B$ with $f \in U$}\}.
\]
and
\[
\Omega_\cW = \{(f,c) \in \wB: \mu_\cW(U)>0 \text{ for {every} open
$U \subseteq \widetilde B$
with $(f,c) \in U$}\}.
\]
Indeed, since $\psi_W^{-1}(L^1(\Omega,\cF{,\mu}) \setminus B)=\emptyset$, we
have $\mu_W(L^1(\Omega,\cF{,\mu}) \setminus B)=0$, which in turn implies that
support of $\mu_W$ is contained in $B$. Let, for a moment, the above defined
set be $\Omega_W'$. First, if $f \in \Omega_W'$, then for any open set $U
\subseteq L^1(\Omega,\cF{,\mu})$ with $f \in U$, we have $\mu_W(U) \ge
\mu_W(U \cap B)>0$. Conversely, if $f \in \Omega_W$, we know we must have $f
\in B$, and for any open $U \subseteq B$ with $f\in U$, we can find an open
$V \subseteq L^1(\Omega,\cF{,\mu})$ with $U= V \cap B$, so in particular
$f\in V$. Then $\mu_W(V) >0$, and since $\mu_W(L^1(\Omega,\cF{,\mu})
\setminus B)=0$, we have $\mu_W(U)=\mu_W(V)>0$, showing that $f\in
\Omega_W'$. A similar, argument shows the second claim, noting that
$\Omega_\cW \subseteq \widetilde B=B \times \RR$.

Now take the union $V$ of all open sets $U \subseteq \wB$ with
$\mu_\cW(U)=0$. Then $\Omega_\cW=\wB \setminus V$. Since $\wB$ is separable
and thus second countable, we can find a countable collection
$U_1,U_2,\dots,U_n,\dots$ with $\mu_\cW(U_n)=0$ and $\bigcup_n U_n=V$. This
means that $\mu_\cW(V)=0$, so $\mu(\psi_\cW^{-1}(V))=0$, and thus almost
every point in $\Omega$ is mapped to $\Omega_\cW$. Since $V$ is open,
$\Omega_\cW$ is closed in $\wB$, so it is a closed subset of a separable
Banach space{;} thus it is a complete, separable metric space.

To see that $\mu_\cW$ has full support, consider an open subset
$U\subseteq\Omega_\cW$ with $\mu_\cW(U)=0$. Since $\Omega_\cW$ is closed in
$\wB$ and $\mu_\cW(\wB \setminus \Omega_\cW)=0$, we can find an open subset
$\tilde U\subseteq \wB$ such that $U=\tilde U\cap \Omega_\cW$ and
$\mu_\cW(\tilde U)=\mu_\cW(U)=0$, so in particular $\tilde U\subseteq
V=\widetilde B\setminus \Omega_\cW$. But this implies $U=\tilde U\cap
\Omega_{\cW}=\emptyset$, as required.

Next, we use Lemma~\ref{lemmacontainedinFW} to infer that $\psi_W(x)\in
L^1(\Omega,\cF_W,\mu)$ for almost all $x\in\Omega$, which in turn implies
that $\psi_\cW(x)\in L^1(\widetilde\Omega,\widetilde\cF_W,\widetilde\mu)$ for
almost all $x\in\Omega$. As a consequence, the open set
$U=L^1(\wOmega,\wcF,\wmu)\setminus L^1(\wOmega,\wcF_W,\wmu)$ has measure
zero:
\[
\mu_\cW(U)=\mu(\psi_\cW^{-1}(U)) =
\mu(\{x\in\Omega : \psi_\cW(x)\notin L^1(\wOmega,\wcF_W,\wmu)\}) =0.
\]
Since $\mu_\cW$ has full support on $\Omega_\cW$, the open set
$U\cap\Omega_\cW\subseteq\Omega_\cW$ is empty, showing that
$\Omega_\cW\subseteq L^1(\wOmega,\wcF_W,\wmu)$.

Let $N=\psi_\cW^{-1}(V)$, where as above $V=\widetilde B\setminus\Omega_\cW$.
We have seen that $\mu(N)=0$. Furthermore, for $A\in {\cG}_\cW$, $A \subseteq
\Omega_\cW$ and hence $\psi_\cW^{-1}(A) \subseteq \Omega\setminus N$.
Finally, $\mu(\psi_{{\cW}}^{-1}(A))=\mu_{{\cW}}(A)$ by the definition of
$\mu_\cW$. Fix some $f \in \Omega_\cW$. On $N \times (\Omega \setminus N)$,
change $W(x,y)$ to $f(y)$, and change it on $(\Omega \setminus N) \times N$
to make it symmetric. Finally, change it to $0$ on $N \times N$, and on $N$,
change $S$ to $f(\Omega_\infty)$. Clearly it is still the case that $W$ and
$S$ are measurable and $W(x,\cdot)$ integrable for every $x$. We have changed
$W$ and $S$ on a set of measure zero, so $\Omega_\cW$ and $\mu_\cW$ did not
change, and we now have that $\psi_\cW$ is an everywhere defined,
measure-preserving map from $(\Omega,\cF,\mu)$ to
$(\Omega_\cW,{\cG}_\cW,\mu_\cW)$.

To complete the proof, we need to show that $\psi_\cW^{-1}(\cG_\cW)=\cF_\cW$.
To this end, we note that for each open $A \subseteq
L^1(\wOmega,\wcF_W,\wmu)$, $\psi_\cW^{-1}(A)=\psi_\cW^{-1}(A\cap\Omega_\cW)$,
which shows that $\psi_\cW^{-1}(\widetilde\cG)=\psi_\cW^{-1}(\cG_\cW)$. This
proves the last claim.
\end{proof}

\begin{lemma}\label{lem:sigma-finite}
Let $\Omega_\cW$ and $\mu_\cW$ be as defined above, and let ${\cG}_\cW$ be
the Borel $\sigma$-algebra over $\Omega_\cW$. If $\cW$ is locally finite,
$W(x,\cdot)$ is integrable for all $x\in \Omega$, and $\mu(\Omega\setminus
\dsupp \cW)<\infty$, then $(\Omega_\cW,{\cG}_\cW,\mu_\cW)$ is
$\sigma$-finite.
\end{lemma}

\begin{remark}
Note that the condition $\mu(\Omega\setminus \dsupp \cW)<\infty$ is
necessary, since otherwise the point $\{(0,0)\}$ becomes an atom with
infinite measure.
\end{remark}

\begin{proof}
By Proposition~\ref{prop:measurepres}, we can change $W$ on a set of measure
zero such that it still satisfies the conditions of the lemma, and $\psi_\cW$
becomes a measure-preserving map from $(\Omega,\cF,\mu)$ to
$(\Omega_\cW,{\cG}_\cW,\mu_\cW)$. Let $\Omega_0\subset L^1(\Omega,\cF,\mu)
\times \RR$ be the subspace consisting of just the origin, i.e.,
$\Omega_0=\{(0,0)\}$. Then
\[\mu_\cW(\Omega_0)=\mu(\psi_\cW^{-1}(\Omega_0))
=\mu(\{x\in\Omega : \|\psi_\cW(x)\|_1=0 \})
=\mu(\Omega\setminus\dsupp \cW)
<\infty.
\]
Next define $\Omega_n=\{(f,c)\in\Omega_\cW : \|f\|_1+c\geq 1/n\}$ for $n \ge
1$. Then
\[
\mu_\cW(\Omega_n)
=\mu(\{x\in\Omega : \|\psi_\cW(x)\|_1+S(x) \geq 1/n \})= \mu(\{x \in \Omega: D_\cW(x) \ge 1/n\})
<\infty.
\]
Here the last inequality follows from Proposition \ref{prop:local-finite}.
Since $\Omega_\cW=\Omega_0\cup\Omega_1\cup\dots$, this proves that $\mu_\cW$
is $\sigma$-finite.
\end{proof}

Next, we would like to show that we can define a graphex $\widehat{\cW}$ on
$\Omega_\cW $ such that its pullback is equal almost everywhere to $\cW$.
Using Proposition~\ref{prop:measurepres}, we can without loss of generality
assume that $\psi_\cW$ is a measure-preserving map from $(\Omega,\cF,\mu)$ to
$(\Omega_\cW,{\cG}_\cW,\mu_\cW)$. By Lemma \ref{lemmapullback}, this implies
that the map $\psi_\cW^*\colon L^1(\Omega_\cW,{\cG}_\cW,\mu_\cW) \rightarrow
L^1(\Omega,{\cF}_\cW,\mu)$ with $f\mapsto f^{\psi_\cW}$ and
${\cF}_\cW=\psi_\cW^{-1}(\cG_\cW)$ is an isometric isomorphism. Note that
this in particular implies that $\psi_\cW^*$ and $(\psi_\cW^*)^{-1}$ are
continuous, and hence measurable.

Now, since $\Omega_\cW$ is a separable metric space,
$L^1(\Omega_\cW,{\cG}_\cW,\mu_\cW)$ is separable, so there exists an
evaluation map
\[
\Phi\colon L^1(\Omega_\cW,{\cG}_\cW,\mu_\cW) \times \Omega_\cW \rightarrow \mathbb{R}
\]
such that for every $\alpha \in L^1(\Omega_\cW,{\cG}_\cW,\mu_\cW)$ and almost
every $g \in \Omega_\cW$, $\alpha(g)=\Phi(\alpha,g)$. Note that by
definition, we also have that for every fixed $\alpha$ and almost every $y
\in \Omega$,
$\psi_\cW^*(\alpha)(y)=\alpha(\psi_\cW(y))=\Phi(\alpha,\psi_\cW(y))$.

By Proposition~\ref{prop:measurepres}, $\Omega_\cW \subseteq L^1(\wOmega,
\wcF_\cW, \wmu)=L^1(\Omega,\cF_\cW,\mu) \times \RR$, which means that
$(\psi_\cW^*)^{-1}(f|_{\Omega})$ is well defined for all $f\in \Omega_\cW$.
We therefore may define
\[
\widehat{W}_{0}(f,g)=\Phi\left((\psi_\cW^*)^{-1}(f|_{\Omega}),g\right)
\]
and
\[
\widehat W(f,g)=\frac 12\left(\widehat{W}_{0}(f,g)+\widehat{W}_{0}(g,f)\right).
\]
Since $(\psi_W^*)^{-1}$ is measurable, $\widehat W_0$ and hence $\widehat W$
is measurable.

Suppose $\psi_\cW(x)\in \Omega_\cW$. Then, noting that for all $x$,
$\psi_\cW(x)|_{\Omega}=\psi_W(x)$, we have for almost all $y$,
\begin{align*}
\widehat W_0(\psi_\cW(x),\psi_\cW(y))&=\Phi(\left(\psi_\cW^*)^{-1}(\psi_\cW(x)|_{\Omega}),\psi_\cW(y)\right)\\
&=
\psi_\cW^*\Bigl((\psi_\cW^*)^{-1}(\psi_\cW(x)|_{\Omega})\Bigr)(y)
=\psi_W(x)(y)=W(x,y),
\end{align*}
where the third and the fourth terms are only defined for almost all $y$.
Thus $\widehat{W}_0^{\psi_\cW}$ and hence $\widehat{W}^{\psi_\cW}$ is equal
to $W$ almost everywhere on $\Omega \times \Omega$. We also define
$\widehat{S}(f)=f(\Omega_\infty)$, which gives us that for $x \in \Omega$,
\[\widehat{S}(\psi_\cW(x))=\psi_\cW(x)(\Omega_\infty)=S(x)
,\] implying that $S^{\psi_\cW}=S$. Finally, we take $\widehat{I}=I$, giving
us a graphex $\widehat{\cW}=(\widehat{W},\widehat{S},\widehat{I},\Omega_\cW)$
such that $\widehat{\cW}^{\psi_\cW}=\cW$ almost everywhere.

Note that this implies in particular that $\widehat \cW$ inherits the local
finiteness property from $\cW$, so $\widehat \cW$ is a bona fide graphex over
the $\sigma$-finite Borel space $(\Omega_W,\cG_W,\mu_W)$.

Note that the requirement that $\widehat{W}^{\psi_\cW}=W$ and
$\widehat{S}^{\psi_\cW}=S$ almost everywhere uniquely determines
$\widehat{\cW}$ up to changes on a set of measure zero. Indeed, if
$\widehat{W}'$ is another graphon with
$\widehat{W}'^{\psi_\cW}=\widehat{W}^{\psi_\cW}$ $(\mu\times\mu)$-almost
everywhere, then by the definition of pullbacks and the definition of
$\mu_\cW$, the equality $\widehat{W}'=\widehat{W}$ must hold
$(\mu_\cW\times\mu_\cW)$-almost everywhere. Similarly, if $\widehat{S}'$ is
another function with $\widehat{S}'^{\psi_\cW}=\widehat{S}^{\psi_\cW}$
$\mu$-almost everywhere, then $\widehat{S}'=\widehat{S}$ $\mu_\cW$-almost
everywhere. Also by definition we must have $\widehat{I}'=\widehat{I}$.

On the other hand, suppose we have two graphexes $\cW_1$ and $\cW_2$ on the
same space $\Omega$ with $W_1=W_2$ almost everywhere, $S_1=S_2$ almost
everywhere, and $I_1=I_2$. We have seen that $\Omega_{\cW_1}=\Omega_{\cW_2}$
and $\mu_{\cW_1}=\mu_{\cW_2}$. Since their pullbacks are equal almost
everywhere, we must have $\widehat{W}_1=\widehat{W}_2$ almost everywhere for
any choices of $\widehat{W}_1$ and $\widehat{W}_2$, and
$\widehat{S}_1=\widehat{S}_2$ and $\widehat{I}_1=\widehat{I}_2$ by
definition.

Finally, if the graphex $\cW$ only has the property that $W(x,\cdot)$ is
integrable for almost every $x$, we can still define $\Omega_\cW$ and
$\mu_\cW$ in the same way, and find a $\widehat{W}$ such that the pullback is
defined almost everywhere on $\Omega \times \Omega$ and equal to $W$ almost
everywhere. Again, it is easy to see that we obtain the same $\Omega_\cW$ and
$\mu_\cW$ if we first modify $W$ on a set of measure zero to make
$W(x,\cdot)$ integrable for every $x$, and any choice of $\widehat{W}$ will
be equal almost everywhere. Therefore, this construction gives a graphex
$\widehat{\cW}$ on $(\Omega_\cW,\cG_\cW,\mu_\cW)$ for any graphex $\cW$.

Next, we show the following:

\begin{lemma}
For $i=1,2$, let $\cW_i=(W_i,S_i,I_i,\bOmega_i)$ be graphexes with
$\bOmega_i=(\Omega_1,\cF_i,\mu_i)$ and $\mu_i(\Omega_i\setminus\dsupp
W_i)=0$. Suppose that there exists a measure-preserving map $\phi\colon
\Omega_1 \rightarrow \Omega_2$ such that $\cW_1=\cW_2^\phi$ almost
everywhere. Extend $\phi$ to $\widetilde{\phi}\colon\wOmega_1 \rightarrow
\wOmega_2$ by $\widetilde{\phi}(\Omega_{1,\infty})=\Omega_{2,\infty}$. Then
the map $\widetilde \phi^*\colon L^1(\widetilde{\bOmega}_2) \rightarrow
L^1(\widetilde{\bOmega}_1)$ defined by $f\mapsto f\circ{\widetilde\phi}$
restricts to a map $\Omega_{\cW_2} \rightarrow \Omega_{\cW_1}$, which is an
isometric measure-preserving bijection between
$(\Omega_{\cW_2},\cG_{\cW_2},\mu_{\cW_2})$ and
$(\Omega_{\cW_1},\cG_{\cW_1},\mu_{\cW_1})$, and
$\widehat{\cW}_2=\widehat{\cW}_1^{\phi^*}$ almost everywhere, for any choices
of $\widehat{\cW}_1$ and $\widehat{\cW}_2$.
\end{lemma}

\begin{proof}
By the remarks before the lemma, we may assume that $\cW_1=\cW_2^\phi$
everywhere, not just almost everywhere, and $W_1(x,\cdot)$ and
$W_2(x',\cdot)$ are always integrable. Since $\phi$ and thus
$\widetilde{\phi}$ is measure preserving, $\widetilde{\phi}^*$ is isometric
and injective from $L^1(\wOmega_2,\wcF_2,\wmu_2)$ to
$L^1(\wOmega_1,\widetilde\phi^{-1}(\wcF_2),\wmu_2)$ by
Lemma~\ref{lemmapullback}. If $x \in \Omega_1$, then for almost every $y \in
\Omega_1$ (note that the first two terms below are only defined for almost
every $y$),
\[
(\phi^* \circ \psi_{W_2} \circ \phi)(x)(y)={(}\psi_{W_2} \circ \phi)(x)(\phi(y))=W_2(\phi(x),\phi(y))=W_1(x,y).
\]
Therefore $\phi^* \circ \psi_{W_2} \circ \phi=\psi_{W_1}$ a.e.
Furthermore,
\[(\id \circ S_2 \circ \phi)(x)=S_2(\phi(x))=S_1(x),
\]
which implies that $(\id \circ S_2 \circ \phi)=S_1$. Since
$\widetilde{\phi}^*=\phi^* \times \id$ and $\psi_{\cW_i}=\psi_{W_i} \times
S$, this implies that $\widetilde{\phi}^* \circ \psi_{\cW_2} \circ
\phi=\psi_{\cW_1}$ almost everywhere. Now let $A \subseteq L^1(\wOmega_1)$ be
Borel measurable. Then
\[
\begin{split}
\mu_{\cW_2}&((\widetilde{\phi}^*)^{-1}(A))=\mu_2(\psi_{\cW_2}^{-1}((\widetilde{\phi}^*)^{-1}(A)))\\
&=\mu_1(\phi^{-1}(\psi_{\cW_2}^{-1}((\widetilde{\phi}^*)^{-1}(A))))
=\mu_1(\psi_{\cW_1}^{-1}(A))
=\mu_{\cW_1}(A).
\end{split}
\]
So $\phi^* \times \id\colon L^1(\bOmega_2) \times \RR \rightarrow
L^1(\bOmega_1) \times \RR$ is a measure preserving isometry. Since it is an
isometry, in particular, it is continuous. Thus, $(\phi^* \times
\id)^{-1}(L^1(\bOmega_1) \times \RR \setminus \Omega_{\cW_1})$ is an open set
with measure zero, so it is disjoint from $\Omega_{\cW_2}$. This implies that
$\phi^*$ restricts to a measure-preserving injection $\Omega_{\cW_2}
\rightarrow\Omega_{\cW_1}$. Since $\Omega_{\cW_2}$ is complete, $(\phi^*
\times \id)(\Omega_{\cW_2})$ is a complete subset of $\Omega_{\cW_1}$, which
is itself complete. Therefore $\phi^*(\Omega_{\cW_2})$ is a closed subset of
$\Omega_{\cW_1}$. However, we also have that
\[
\mu_{\cW_1}(\Omega_{\cW_1} \setminus \phi^*(\Omega_{\cW_2}))=
\mu_{\cW_2}((\phi^*)^{-1}(\Omega_{\cW_1} \setminus \phi^*(\Omega_{\cW_2})))=
\mu_{\cW_2}((\phi^*)^{-1}(\Omega_{\cW_1}) \setminus \Omega_{\cW_2})=0.
\]
But $\Omega_{\cW_1} \setminus \phi^*(\Omega_{\cW_2})$ is an open subset of
$\Omega_{\cW_1}$ of measure $0$, which means it must be the empty set because
$\mu_{\cW_1}$ has full support in $\Omega_{\cW_1}$. Therefore, $\phi^*\colon
\Omega_{\cW_2} \rightarrow \Omega_{\cW_1}$ is a measure-preserving isometry
of metric measure spaces.

Now, we want to show that $\widehat{W}_1^{\widetilde{\phi}^*}=\widehat{W}_2$.
We have that almost everywhere on $\Omega_1 \times \Omega_1$,
\[
((\widehat{W}_1^{\widetilde\phi^*})^{\psi_{\cW_2}})^{\phi}=\widehat{W}_1^{\widetilde\phi^* \circ \psi_{\cW_2} \circ \phi}=\widehat{W}_1^{\psi_{\cW_1}}=W_1=W_2^\phi.
\]
Therefore $(\widehat{W}_1^{\phi^*})^{\psi_{\cW_2}}=W_2$ almost everywhere,
but then $\widehat{W}_1^{\phi^*}=\widehat{W}_2$ almost everywhere. By
definition, we also have $\widehat{S}_1^{\phi^*}=\widehat{S}_2$, and
$\widehat{I}_1=I_1=I_2=\widehat{I}_2$.
\end{proof}

Now, suppose $\cW_1$ and $\cW_2$ are equivalent. Then their restrictions to
their respective degree supports are also equivalent. By Corollary
\ref{chain}, there exists a chain of pullbacks that link $\cW_1$ and $\cW_2$.
We have seen that if a graphex is a pullback of another, then the
construction above yields an isomorphism between the corresponding graphexes,
up to almost everywhere changes. This clearly extends to chains of pullbacks;
thus, we may find an isomorphism between $\Omega_{\cW_1}$ and
$\Omega_{\cW_2}$ so that $\widehat{\cW}_1$ and $\widehat{\cW}_2$ are equal
almost everywhere. We can extend the map $\psi_{\cW_i}\colon \Omega_i
\rightarrow \Omega_{\cW_i}$, which is defined almost everywhere, to be
defined everywhere, by mapping the rest of the points in $\Omega_i$ to an
arbitrary point.

\section{Uniform integrability and uniform tail regularity}
\label{sec:UIandUTR}

\subsection{Uniform integrability}

The goal of this subsection is to prove
Theorem~\ref{thm:UI-norm-convergence}. Before doing this we establish that
several alternative definitions of uniform integrability are equivalent to
Definition~\ref{def:UI}.

\begin{theorem} \label{thmunifintegequiv}
Given a set of integrable graphexes $\mathcal{S}$, the following are
equivalent.
\begin{enumerate}
\item $\cS$ is uniformly integrable. \label{conditionunifintegD}
 \item The graphexes in $\cS$ have uniformly bounded $\|\cdot\|_1$-norms,
     and for every $\eps>0$, there exists a $D$ such that for all $\cW
\in\mathcal{S}$, $\|\cW\|_1 - \|\cW_{\le D}\|_1<\varepsilon$.
\label{conditionunifintegW}
\item For any $T>0$, the random variables $E(G_T(\cW))$ with $\cW \in \cS$
    are uniformly integrable. \label{conditionunifintegallT}
\item There exists $T>0$ such that the random variables $E(G_T(\cW))$ with
    $\cW \in \cS$ are uniformly integrable. \label{conditionunifintegsomeT}
\end{enumerate}
\end{theorem}

\begin{proof}
Throughout this proof, let $\Omega_{>D}$, $\Omega_{\le D}$, $\cW_{>D}$,
$\cW_{\le D}$, etc, be defined as before. Let us first show
$(\ref{conditionunifintegD}) \Rightarrow (\ref{conditionunifintegW})$. We
have that
\begin{align*}
\|\cW_1\|_1 - \|\cW_{\le D}\|_1 &= 2\int_{\Omega_{>D}} S(x) \,d\mu(x)+2\int_{\Omega_{>D} \times \Omega_{\le D}}W(x,y) \,d\mu(x) \,d\mu(y)\\
& \qquad \qquad \phantom{} + \int_{\Omega_{>D} \times \Omega_{>D}} W(x,y) \,d\mu(x) \,d\mu(y)\\
&=2\int_{\Omega_{>D}}D_{\cW}(x) \,d\mu(x)-\int_{\Omega_{>D} \times \Omega_{>D}} W(x,y) \,d\mu(x) \,d\mu(y)\\
& \le 2\int_\Omega D_\cW1_{D_\cW>D} \,d\mu
.
\end{align*}
Therefore, taking $D$ for $\varepsilon/2$ in uniform integrability gives a
good $D$ for $\varepsilon$ in (\ref{conditionunifintegW}).

To show that $(\ref{conditionunifintegW}) \Rightarrow
(\ref{conditionunifintegallT})$, we first show that if a set of graphexes has
uniformly bounded marginals, then the set of random variables is uniformly
integrable. Let $E_T$ be the random variable for a fixed $\cW \in \cS$ that
gives the number of edges of $G_T(\cW)$. Recall that by Lemma
\ref{lem:edgebound}, $E_T$ has expectation $T^2\|\cW\|_1/2$ and variance
$T^2\|\cW\|_1/2+T^3\|D_\cW\|_2^2$. Let $C$ be a bound on $\|\cW\|_1$ for $\cW
\in \cS$. We then have that for any $K>T^2\|\cW\|_1/2$,
\[
\PP[E_T>K] \le
\frac{T^2\|\cW\|_1/2+T^3\|D_\cW\|_2^2}{(K-T^2\|\cW\|_1/2)^2} \le
\frac{T^2C/2+T^3CD}{(K-T^2C/2)^2}
.
\]
If $K\ge T^2C$, then this gives
\[\PP[E_T>K] \le \frac{T^2C/2+T^3CD}{(K-T^2C/2)^2} \le \frac{T^2C/2+T^3CD}{(K/2)^2} = \frac{2T^2C+4T^3CD}{K^2}
.
\]
Therefore, for $K_0 \ge T^2C$,
\begin{align*}
\EE[E_T1_{E_T> K_0}] &=\sum_{K=K_0+1}^\infty \PP[E_T \ge K]\\
& \le \sum_{K=K_0+1}^\infty \frac{2T^2C+4T^3CD}{K^2}
\le \frac{2T^2C+4T^3CD}{K_0}
.
\end{align*}
Suppose now that instead of uniformly bounded marginals, we have only
(\ref{conditionunifintegW}). For $D>0$, let $E_{T,D}$ be the number of edges
that either have both endpoints labeled with a vertex in $\Omega_{\le D}$,
one endpoint is labeled with a vertex in $\Omega_{\le D}$ and the edge is
generated as a star from that vertex, or the edge is a dust edge. We then
have that for all $D>0$,
\begin{align*}
\EE[E_T1_{E_T>2K_0}] &=
\EE[E_{T}1_{E_T>2K_0,E_{T,D}>K_0}]+ \EE[E_{T}1_{E_T>2K_0,E_{T,D}\le K_0}]
\\
&\le \EE[E_T-E_{T,D}]+ \EE[E_{T,D}1_{E_{T,D}>K_0}]\\
& \qquad \quad \phantom{}+ \EE[E_{T,D}1_{E_{T,D}\le K_0,E_T-E_{T,D}>K_0}]
\\
&\le \EE[E_T-E_{T,D}]+ \frac{2T^2C+4T^3CD}{K_0}+ K_0 \PP[E_{T}-E_{T,D}>K_0]
\\
&\le 2\EE[E_T-E_{T,D}]+ \frac{2T^2C+4T^3CD}{K_0}
,
\end{align*}
provided $K_0 \ge T^2C$. Condition (\ref{conditionunifintegW}) now implies
that for any $\eps>0$, there exists a $D$ such that
$\EE[E_{T,D}-E_T]<\varepsilon$. Given such a $D$, we choose $K_0$ in such a
way that the last term in the above bound is at most $\eps$, implying that
for each $\eps>0$ we can find a $K_0$ such that $\EE[E_T1_{E_T>2K_0}]\leq
3\eps$. This proves that the set of random variables $E_T$ are indeed
uniformly integrable.

It is clear that $(\ref{conditionunifintegallT})$ implies
$(\ref{conditionunifintegsomeT})$. Suppose now that
$(\ref{conditionunifintegsomeT})$ holds. Since the expectation of $E_T$ is
$T^2\|\cW\|_1/2$, $\|\cW\|_1$ must be uniformly bounded for $\cW \in \cS$.
Let $C$ be an upper bound. Suppose that (\ref{conditionunifintegD}) is false.
Then there exists a fixed $\varepsilon>0$, such that for any $D$, there
exists a graphex $\cW \in \cS$ such that
\[\int_{\Omega_{>D}} D_\cW 1_{D_\cW>D} \,d\mu \ge \varepsilon
.\] Since $\EE[D_\cW]\leq\|\cW\|_1 \le C$, we have that $\mu(\Omega_{>D}) \le
C/D$. By taking $D$ large enough, we can assume that $C/D \le D/2$. Let
$F_{T,D}$ be the number of edges in $G_T(\cW)$ that have exactly one endpoint
in $\Omega_{>D}$. Then
\[\EE[F_{T,D}] \ge \int_{\Omega_{>D}} T^2 \left(D_\cW(x)-C/D\right) \,d\mu(x) \ge T^2 \int_{\Omega_{>D}} \left(D_\cW(x)/2\right) \,d\mu(x) \ge T^2 \varepsilon/2
.\] If there are no points sampled in $\Omega_{>D}$, then $F_{T,D}=0$.
Conditioned on there being at least one point sampled in $\Omega_{>D}$, the
number of neighbors of a point whose feature is $x \in \Omega_{>D}$ is a
Poisson random variable with mean equal to $TD_{\cW}(x)/2 \ge TD/2$.
Therefore,
\[\EE[F_{T,D} | F_{T,D}>0] \ge TD/2
.\] We also have that
\[\EE[F_{T,D}1_{F_{T,D} \le TD/4}|F_{T,D}>0] \le TD/4
.\] Therefore,
\[\EE[F_{T,D}1_{F_{T,D}>TD/4}|F_{T,D}>0] \ge\frac 12 \EE[F_{T,D} | F_{T,D}>0]
.\] We then have
\begin{align*}
\EE[F_{T,D}1_{F_{T,D}>TD/4}]
&=\EE[F_{T,D}1_{F_{T,D}>TD/4}|F_{T,D}>0]\PP[F_{T,D}>0]
\\
&\ge \frac 12\EE[F_{T,D} | F_{T,D}>0]\PP[F_{T,D}>0]
= \frac 12\EE[F_{T,D}] \ge T^2\varepsilon/4
.
\end{align*}
Since $D$ can be arbitrary (above some $D_0$), this contradicts Condition
(\ref{conditionunifintegsomeT}).
\end{proof}

Theorem~\ref{thm:UI-norm-convergence} is an easy corollary of
Theorem~\ref{thmunifintegequiv}.

\begin{proof}[Proof of Theorem~\ref{thm:UI-norm-convergence}]
By Theorem~\ref{thm:delGP-GP}, $\cW_n$ is GP-convergent to $\cW$. Fix a
subsequence $n_i$ such that $
\liminf_{n\to\infty}\|\cW_n\|_1=\lim_{i\to\infty} \|\cW_{n_i}\|_1$, and fix
$T>0$. Let $e_i$ be the number of edges in $G_T(\cW_{n_i})$, and let $e$ be
the number of edges in $G_T(\cW)$. Following the proof of Corollary 3.10 in
\cite{BCCV17}, for $\lambda>0$ define $f_\lambda\colon\RR_+\to\RR_+$ by
$f_\lambda(x)=x 1_{x\leq \lambda}$. Then $\EE[f_\lambda(e_i)]\leq
\EE[e_i]=T^2\|\cW_{n_i}\|_1$. Since $e_i\to e$ in distribution,
$\EE[f_\lambda(e)]=\lim_{i\to\infty}\EE[f_\lambda(e_i)]\leq
T^2\lim_{i\to\infty}\|\cW_i\|_1=\liminf_{n\to\infty}\|\cW_n\|_1$. The
monotone convergence theorem then gives that
$T^2\|\cW\|_1=\EE[e]=\lim_{\lambda\to\infty}\EE[f_\lambda(e)]\leq
T^2\liminf_{n\to\infty}\|\cW_n\|_1$, proving the first part of the theorem.

To prove the second part, assume first that $\cW_n$ is uniformly integrable,
and fix $\varepsilon>0$. By Theorem \ref{thmunifintegequiv} (2), for every
$\varepsilon>0$, there exists a $D$ such that each $\cW_n$ has
\[
\left|\|\cW_n\|_1 - \|\cW_{n,\le D}\|_1\right| \le \varepsilon.
\]
Since $\cW$ is integrable, after possibly increasing $D$, we can also assume that
\[
\left|\|\cW\|_1 - \|\cW_{\le D}\|_1\right| \le \varepsilon.
\]
Increasing $D$ further, we may also assume that $\mu(\{D_{\cW}=D\})=0$. By
Proposition \ref{propgeneralconvergenceequiv}, $\deltt(\cW_{n, \le
D},\cW_{\le D}) \rightarrow 0$, which implies in particular that $\|\cW_{n,
\le D}\|_1 \rightarrow \|\cW_{\le D}\|_1$. Therefore, we can take $n_0$ so
that if $n \ge n_0$, then
\[
\left|\|\cW_{\le D}\|_1 - \|\cW_{n,\le D}\|_1\right| \le \varepsilon.
\]
These three inequalities imply that if $n \ge n_0$, then
\[
\left|\|\cW_n\|_1 - \|\cW\|_1\right| \le 3\varepsilon.
\]
Since $\varepsilon$ was arbitrary, this completes the proof of the first
direction.

For the other direction, fix $\varepsilon$. Since $\cW$ is integrable, there
exists $D>0$ such that
\[
\|\cW_{\le D}\|_1 \ge \|\cW\|_1 - \varepsilon/2.
\]
By increasing $D$, we can assume that $\mu(\{D_{\cW}=D\})=0$. By Proposition
\ref{propgeneralconvergenceequiv},
\[\|\cW_{n,\le D}\|_1 \to \|\cW_{\le D}\|_1.
\]
We then have that
\begin{align*}
\limsup_{n \to \infty}
\Bigl(\|\cW_n\|_1-\|\cW_{n,\le D}\|_1 \Bigr)
&= \limsup_{n \to \infty}
\Bigl( (\|\cW_n\|_1 - \|\cW\|_1)
+(\|\cW\|_1-\|\cW_{\le D}\|_1)\\
& \qquad \qquad \qquad \phantom{}
+(\|\cW_{\le D}\|_1-\|\cW_{n,\le D}\|_1) \Bigr)\le \varepsilon/2
.
\end{align*}
Therefore, there exists an $n_0$ such that if $n > n_0$, then
\[
\|\cW_n\|_1 - \|\cW_{n,\le D}\|_1 <\varepsilon.
\]
Since $\cW_1,\cW_2,\dots,\cW_{n_0}$ is a finite set of graphexes, each of
which is integrable, we can increase $D$ so that the above inequality holds
for each $n$, which by Theorem \ref{thmunifintegequiv} means that they are
uniformly integrable.
\end{proof}

\subsection{Uniform tail regularity}

The goal of this subsection is to prove
Theorem~\ref{thmcutmetricvsweakkernel}. Before doing this, we show that
uniform tail regularity implies uniform integrability.

\begin{lemma}\label{lem:tail-reg-implies-UI} Suppose that a set
of graphexes consisting only of graphons is uniformly tail regular. Then the
set is uniformly integrable.
\end{lemma}

\begin{proof}
Fix $\varepsilon>0$. By the definition of tail regularity, we can find an
$M<\infty$ such that for each graphon $W$ in the set there exists a subset
$\Omega_0$ of measure at most $M$ such that
$\|W\|_1-\|W|_{\Omega_0}\|_1\leq\varepsilon/3$. Note that clearly
$\|W\|_1\leq M^2+\varepsilon/3$, so in particular the set of graphons has
uniformly bounded $L^1$ norm. Let
\[A=\{x \in \Omega_0,D_W(x)>2M\}
.\] Note that for any $x\in A$,
\[\int_{\Omega \setminus \Omega_0} W(x,y)\,d\mu(y) \ge D_W(x)-\int_{\Omega_0}W(x,y)\,d\mu(y) \ge D_W(x)-M \ge M
,\] which implies that
\[
\mu(A) \le \frac 1M\int_{A\times\Omega\setminus\Omega_0}W \,d\mu^2\leq\frac 1M\int_{\Omega\setminus\Omega_0}D_W \,d\mu
\leq \frac{\varepsilon}{3M}.
\]
We then have
\begin{align*}
\int_\Omega D_{W}1_{D_{W}>2M} \,d\mu
&\leq\int_{\Omega\setminus\Omega_0} D_{W} \,d\mu +\int_A D_{W} \,d\mu\\
&= \int_{\Omega \setminus \Omega_0} D_W \,d\mu + \int_{A \times (\Omega \setminus \Omega_0)} W(x,y) \,d\mu(x) \,d\mu(y)\\
& \qquad \qquad \qquad \phantom{} + \int_{A \times \Omega_0} W(x,y) \,d\mu(x) \,d\mu(y)\\
&\le 2\int_{\Omega \setminus \Omega_0} D_W \,d\mu + \int_{A \times \Omega_0} W(x,y) \,d\mu(x) \,d\mu(y)\\
&\le 2\varepsilon/3 + \mu(A)\mu(\Omega_0)\le \varepsilon
.\qedhere\end{align*}
\end{proof}

Next, we show the following lemma. As before, $\Omega_{>\delta}$ is the set
$\{x \in \Omega: D_\cW(x) > \delta\}$.

\begin{lemma} \label{lemuniftailregequiv}
Given a set of graphons $\cS$, the following are equivalent:
\begin{enumerate}
\item The set of graphons is uniformly tail regular.
\item The set of graphons has a uniform bound on their $\|\cdot\|_1$-norm,
    and for every $\varepsilon$, there exists a $\delta$ such that $
 \|W\|_1 - \|W|_{\Omega_{> \delta}}\|_1 \le \varepsilon. $
\item The set of graphons has a uniform bound on their $\|\cdot\|_1$-norm,
    and for every $\varepsilon$, there exists a $\delta$ such that $
 \|D_W1_{D_W\leq \delta}\|_1\leq \varepsilon $.

\end{enumerate}
\end{lemma}

\begin{corollary} \label{coruniftailregpullback}
Given a set of graphons $\cS$, suppose that we replace each graphon with a
pullback. Let $\cS'$ be the new set. Then $\cS'$ is uniformly tail regular
if and only if $\cS$ is.
\end{corollary}

\begin{proof}
Property (2) in Lemma \ref{lemuniftailregequiv} is unaffected by taking
pullbacks.
\end{proof}

\begin{proof}[Proof of Lemma~\ref{lemuniftailregequiv}]
We first show that (1) implies (2). Fix $\varepsilon>0$. Take $M$ for
$\varepsilon/2$ as in the definition of uniform tail regularity, let
$\delta=\varepsilon/4M$.  Fix an arbitrary graphon $W \in \cS$, and let
$\Omega_0$ be a set of measure $M$ such that
\[\|W\|_1 - \|W|_{\Omega_0}\|_1 \le \varepsilon/2
.\] Note that $W$ has $L^1$ norm at most $M^2+\varepsilon/2$, which
proves that the graphs have a uniform bound on their $\|\cdot\|_1$-norm. Now,
we have
\[\|W\|_1 - \|W|_{\Omega_{> \delta}}\|_1 \le \|W\|_1 - \|W|_{\Omega_0}\|_1+\|W|_{\Omega_0}\|_1 - \|W|_{\Omega_0 \cap \Omega_{> \delta}}\|_1 \le \varepsilon/2 + 2\delta M \le \varepsilon
.\]
This shows that (1) implies (2).

The fact that (2) implies (1) follows from the observation that
\[
\mu(\Omega_{>\delta})=\int d\mu(x)1_{D_W(x)> \delta}\leq \int d\mu(x) \frac{D_W(x)}\delta
=\frac 1\delta\|W\|_1.
\]
Finally, (2) and (3) are equivalent by the fact that
\[
\int_{\Omega\times\Omega\setminus\Omega_{>\delta}}W
\leq \int_{\Omega\times\Omega} W - \int_{\Omega_{>\delta}\times\Omega_{>\delta}} W
\leq 2 \int_{\Omega\times\Omega\setminus\Omega_{>\delta}}W. \qedhere
\]
\end{proof}

To prove Theorem \ref{thmcutmetricvsweakkernel} we establish three more
lemmas.

\begin{lemma} \label{lemmacutmetricrestrictedtoD}
Suppose that a sequence of integrable graphons $W_n$ converge to a graphon
$W$ in the cut metric. Then for any $D>0$ such that $\mu(D_W=D)=0$, the
graphons $W_{n,\le D}$ converge to $W_{\le D}$ in the cut metric.
\end{lemma}

\begin{proof}
Let $\wmu_n$ be a coupling of trivial extensions of $W_n$ and $W$, and let
$\wOmega_n$ be the product space on which the coupling is defined. Let
$\wW_n$ and $\wW$ be the pullbacks of the trivial extensions to $\wOmega_n$,
and suppose that
\[\|\wW_n-\wW\|_\square < \varepsilon.\]
Defining $A=\{D_{\wW_n}-D_{\wW}>0\}$ and $B=\{D_{\wW}-D_{\wW_n}>0\}$, we then
have that
\begin{align*}
\wmu_n(\{|D_{\wW_n}-D_{\wW}|>\sqrt{\varepsilon}\})
&\leq \frac 1{\sqrt\eps}\|D_{\wW_n}-D_{\wW}\|_1\\
&= \frac
1{\sqrt\eps}\Bigl(\int_{A\times \wOmega_n}({\wW_n}-{\wW})+\int_{B\times
\wOmega_n}({\wW}-{\wW_n})\Bigr)
 <2\sqrt{\varepsilon} .
\end{align*}
As a consequence,
\begin{align*}
\wmu_n(\{D_{\wW_n}>D,&\ D_{\wW} \le D\})
\\
&\le
\wmu_n(\{|D_{\wW_n}-D_{\wW}|>\sqrt{\varepsilon}\})
 +\wmu_n(\{D-\sqrt{\varepsilon}<D_{\wW}\le D\})
\\
&<2\sqrt{\varepsilon}+\mu(\{D-\sqrt{\varepsilon}<D_W\le D\}).
\end{align*}
For any $\delta>0$, we can take $\varepsilon$ small enough so that this is at
most $\delta$. Similarly, we have
\begin{align*}
\wmu_n(\{D_{\wW_n}\le D,&\ D_{\wW} > D\})
\\
&\le \wmu_n(\{|D_{\wW_n}-D_{\wW}|>\sqrt{\varepsilon}\})
+\wmu_n(\{D<D_{\wW}\le D+\sqrt{\varepsilon}\})
\\
& < 2\sqrt{\varepsilon}+\mu(\{D<D_W\le D+\sqrt{\varepsilon}\}),
\end{align*}
which is also at most $\delta$ if $\varepsilon$ is small enough.

Next we trivially extend $W_{n,\le D}$ and $W_{\le D}$ first to the spaces
$W_n$ and $W$ are defined on, and then to the spaces used in the coupling
$\wmu_n$. Let $\wW_{n,\le D}$ and $\wW_{\le D}$ be the pullbacks, let
$\wW_{n,\le D}'$ be equal to $\wW_{n,\le D}$ on $\{D_{\wW} \le D \}^2$ and
$0$ otherwise, and let $\wW_{\le D}'$ be equal to $\wW_{\le D}$ on
$\{D_{\wW_n} \le D \}^2$ and $0$ otherwise. Then $\wW_{n,\le D}$ and
$\wW_{n,\le D}'$ differ only on $\{D_{\wW_n}\le D,D_{\wW} > D\}
\times\{D_{\wW_n} \le D\}$ and its transpose. Indeed, if $D_{\wW_n} > D$ in
either coordinate, then both graphons are zero, and if $D_{\wW} \le D$ in
both coordinates, then by the definition they are the same. Since $\wW_{n,\le
D}$ has maximum degree $D$, this implies that
\[
\|\wW_{n,\le D}-\wW_{n,\le D}'\|_\square \le \|\wW_{n,\le D}-\wW_{n,\le D}'\|_1 \le 2\wmu_n(\{D_{\wW_n}\le D,D_{\wW} > D\})D \le 2\delta D.
\]
Analogously,
\[
\|\wW_{\le D}-\wW_{\le D}'\|_\square \le 2\delta D.
\]
Note that $\wW_{n,\le D}'$ and $\wW_{\le D}'$ are equal to $\wW_n$ and $\wW$,
respectively, on $\{D_{{\wW_n}} \le D,D_{{\wW}}\le D\}^2$, and zero
everywhere else, which implies that $\wW_{n,\le D}'-\wW_{\le D}'$ is the
restriction of $\wW_n-\wW$ to $\{D_{\cW_n} \le D,D_{\cW}\le D\}^2$. This
implies that
\[
\|\wW_{n,\le D}'-\wW_{\le D}'\|_\square
\le \|\wW_n-\wW\|_\square < \varepsilon,
\]
which in turn implies that
\begin{align*}
\|\wW_{n,\le D}-\wW_{\le D}\|_\square
&\le \|\wW_{n,\le D}-\wW_{n,\le D}'\|_\square
+\|\wW_{n,\le D}'-\wW_{\le D}'\|_\square\\
& \qquad
\phantom{}+\|\wW_{\le D}'-\wW_{\le D}\|_\square\\
&\le 4\delta D + \varepsilon.
\end{align*}
Taking $\varepsilon$ small enough, this can be made arbitrarily small, which
completes the proof.
\end{proof}

We also have the following:

\begin{lemma} \label{lemmaboundedcutmetricimplieskernelmetric}
Suppose that a sequence of integrable graphons $W_n$ have uniformly bounded
marginals, and converge to a (necessarily integrable) graphon $W$ in the cut
metric. Then $W$ has the same bound on its marginals, and $\deltt(\cW_n,\cW)
\to 0$.
\end{lemma}

\begin{proof}
Suppose that for each $n$, $D_{W_n} \le D$ almost everywhere, but $D_W>D$ on
a set of positive measure. Then there exists $D' \ge D$ such that
$\mu(\{D_W>D'\})>0$ and $\mu(\{D_W=D'\})=0$. By Lemma
\ref{lemmacutmetricrestrictedtoD}, $W_{n,\le D'}$ converges to $W_{\le D'}$
in the cut metric, but the cut distance of $W_{n,\le D'}$ from $W_n$ is $0$,
since $D_{W_n} \le D$ almost everywhere. Therefore, the cut distance of
$W_{\le D'}$ and $W$ is $0$, which is a contradiction.

Now, we have the following. Since $\|W_n\|_1 \rightarrow \|W\|_1$, there
exists a uniform bound $C$ on $\|W_n\|_1$ and $\|W\|_1$, which implies that
$\|W_n-W\|_1 \le 2C$. Furthermore, for any $x$, $|D_{|W_n-W|}(x)| \le
D_{W_n}(x)+D_W(x) \le 2D$. By Lemma \ref{lemmasquare2to2equiv}, and recalling
that $\|U\|_{\jbl} \le \sqrt{\|U\|_\square\|U\|_\infty}$, we have
\begin{align*}
\|W_n - W\|_{2\to2} &\le \left(8\|W_n-W\|_{\jbl}\|W_n-W\|_\infty^{3/4}\|D_{|W_n-W|}\|_\infty^{3/2}\|W_n-W\|_1^{3/4}\right)^{1/4}\\
&\le
\left(8\|W_n-W\|_{\square}^{1/2}\|W_n-W\|_\infty^{5/4}\|D_{|W_n-W|}\|_\infty^{3/2}\|W_n-W\|_1^{3/4}\right)^{1/4}\\
&\le \left(100D^{3/2}C^{3/4}\|W_n-W\|_{\square}^{1/2}\right)^{1/4} \to 0
.
\end{align*}
Furthermore, note that
\[\|D_{W_n}-D_W\|_1 \le 2 \sup_S\left\{\left|\int_S D_{W_n}-D_W \right|\right\} \le 2\|W_n-W\|_\square
.\] Therefore,
\begin{align*}
\|D_{\cW_n}-D_{\cW}\|_2 &= \|D_{W_n}-D_W\|_2\\
& \le
\sqrt{\|D_{W_n}-D_W\|_1\|D_{W_n}-D_W\|_\infty}
 \le \sqrt{4D\|W_n-W\|_\square}\to 0 .
\end{align*}
Finally,
\[\|\cW_n\|_1 = \|W_n\|_1 \to \|W\|_1=\|\cW\|_1
. \qedhere\]
\end{proof}

The last lemma we need to prove Theorem \ref{thmcutmetricvsweakkernel} is the
following.

\begin{lemma} \label{lemmadelgp0cut0}
Suppose that $\delGP(\cW,\cW')=0$ for two integrable graphexes, and suppose
that $\cW=(W,0,0,\bOmega)$. Then $\cW'=(W',0,0,\bOmega')$ and
$\delta_\square(W,W')=0$.
\end{lemma}

\begin{proof}
Since $\delGP(\cW,\cW')=0$, $\xi(G(\cW))$ and $\xi(G(\cW'))$ have the same
distribution. But, as already observed in Remark 5.4 in \cite{JANSON16},
almost surely, the dust part of $\cW'$ generates edges which are isolated,
the star part generates edges with one vertex of degree one and a second
vertex of infinite degree, and the graphon part generates edges with two
endpoints of infinite degree. Since $\xi(G(\cW))$ has no star or dust edges,
$\xi(G(\cW'))$ doesn't have these either, showing that
$\cW'=(W',0,0,\bOmega')$. Finally, since the graphon process generated by $W$
and $W'$ have the same distribution, $\delta_\square(W,W')=0$ by Theorem 27
in \cite{BCCH16}.
\end{proof}

We are now ready to prove Theorem \ref{thmcutmetricvsweakkernel}.

\begin{proof}[Proof of Theorem \ref{thmcutmetricvsweakkernel}]
First, we show that if a sequence converges in the cut metric, then it
converges in $\delGP$. We show property (2) from Proposition
\ref{propgeneralconvergenceequiv}. Since the graphons converge in cut metric,
we must have in particular that $\|\cW_n\|_1=\|W_n\|_1 \to
\|W\|_1=\|\cW\|_1$, which implies that the set $\{\|\cW_n\|\}_{n}$ is
uniformly bounded; therefore, the sequence is tight by Corollary
\ref{cor:C-bounded-tight} (1). By Lemma \ref{lemmacutmetricrestrictedtoD},
for any $D>0$ with $\mu(\{D_W=D\})=0$, $W_{n,\le D}$ converges to $W_{\le D}$
in cut metric, and by Lemma \ref{lemmaboundedcutmetricimplieskernelmetric},
they must also converge in $\deltt$, which completes the proof that cut
metric convergence implies weak kernel convergence. Since we know that any
cut metric convergent sequence is uniformly tail regular, this completes that
proof that (1) implies (2). It is clear that (2) is stronger than (3), so it
remains to show that (3) implies (1).

To this end, we first note that uniform tail regularity implies uniformly
bounded $L^1$ norms, which by Theorem~\ref{thm:UI-norm-convergence} implies
that $\cW$ is integrable. Suppose that $\cW$ does not consist of only a
graphon part, or that $W_n$ does not converge to it in the cut metric. Since
the sequence $W_n$ is uniformly tail regular, we may choose a subsequence
that converges to an integrable graphon $W'$, such that either
$\delta_\square(W',W) \ne 0$, or $\cW$ is not just a graphon. In either case,
letting $\cW'=(W',0,0,\bOmega')$, we have by Lemma \ref{lemmadelgp0cut0} that
$\delGP(\cW,\cW') \ne 0$. However, since $\delta_\square(W_n,W') \to 0$, we
must have that $\delGP(\cW_n,\cW') \to 0$, which implies that
$\delGP(\cW',\cW)=0$, which is a contradiction. This completes the proof that
(3) implies (1), and thus we have proven the theorem.
\end{proof}

\begin{proof}[Proof of Theorem \ref{thmwhenconvergetographon}]
First, assume that $\cW=(W,0,0,\bOmega)$, and let us prove (2). Assume first
that the sequence has uniformly bounded marginals. In this case,
$\deltt(\cW_n,\cW) \to 0$. Take a sequence of couplings of trivial extensions
of $\cW_n$ and $\cW$ which show that their kernel distance goes to zero, let
$\bOmega_n'$ be the space for each $n$, and $\cW_n'$ and $\cW^n$ the pulled
back graphexes, and let $W_n'$ and $W^n$ be their graphon parts. By Corollary
\ref{coruniftailregpullback}, it is enough to prove uniform tail regularity
for $W_n'$. Given $\eps>0$, let $\delta>0$ be such that
\[
\|W\|_1-\|W|_{M_\delta}\|_1\leq\eps,
\]
where $M_\delta$ is the set
\[
M_\delta=\{x \in \Omega: D_{W}(x) \ge \delta\}.
\]
Let $M_\delta^n$ be the pullback to $\Omega_n'$. We then have that
\[
\|W_n'-W^n\|_{2\to2} \to 0.
\]
Since $M_\delta^n$ has finite measure, this implies that
\[
\int_{M_\delta^n \times M_\delta^n} W_n' \to \int_{M_\delta \times M_\delta}W.
\]
Since $\|\cW_n'\|_1= \|\cW_n\|_1 \to \|\cW\|_1=\|W\|_1$, this implies that
\[
\limsup_{n\to\infty} \Bigl(\|W_n'\|_1 - \int_{M_\delta^n \times M_\delta^n} W_n'\Bigr)
\leq\limsup_{n\to\infty} \Bigl(\|\cW_n'\|_1 - \int_{M_\delta^n \times M_\delta^n} W_n'\Bigr) \le \varepsilon.
\]
This can be made arbitrarily small by taking $\delta$ small enough, which
proves uniform tail regularity under the assumption of uniformly bounded
marginals. On the other hand,
\begin{align*}
2\limsup_{n\to\infty} \Bigl(\|S_n\|_1+I_n\Bigr)&=\limsup_{n\to\infty} \Bigl(\|\cW_n\|_1-\|W_n'\|_1\Bigr)\\
&\leq \|W\|_1-\lim_{n\to\infty}\int_{M_\delta^n \times M_\delta^n} W_n'\leq\eps,
\end{align*}
which implies that $\|S_n\|_1 \to 0$ and $I_n \to 0$, completing the proof of
(2) under the assumption of uniformly bounded marginals.

If instead of uniformly bounded marginals, we have uniform integrability,
then the claims follow from the fact that for each $D>0$, $\deltt(\cW_{n,\le
D},\cW_{\le D}) \to 0$, and we can take $D$ large enough so that each
$\int_{\Omega_{n,>D}}S_{n}$ is less than $\varepsilon$ and
\[
\int_{\Omega \times \Omega_{>D}} W_n < \varepsilon ,\] and $I_n$ is unaffected
by the restriction.

Conversely, assume (2). Let, for each $n$, $\cW_n'=(W_n,0,0,\bOmega)$ (so we
are replacing $S_n$ and $I_n$ with $0$). Since $\|S_n\|_1 \to 0$ and $I_n \to
0$,
\[\delGP(\cW_n,\cW_n') \to 0
\]
Indeed, clearly $|\|\cW_n\|_1 - \|\cW_n'\|_1| \to 0$, and for any $D>0$, we
have that $\int_{\Omega_{\le D}} S_n^2 \le D\|S_n\|_1 \to 0$. Taking $D$
large enough that $\Omega_{>D}$ has small measure, we can show that
$\delGP(\cW_n,\cW_n')$ is arbitrarily small for large enough $n$. Now, the
statement follows from Theorem \ref{thmcutmetricvsweakkernel}; specifically,
we have shown that (3) holds for the sequence $\cW_n'$, which by (2) implies
that the limit is a pure graphon.

The equivalence of (2) and (3) follows from Theorem
\ref{thmcutmetricvsweakkernel} applied to the sequence $\cW_n'$.
\end{proof}

\section*{Acknowledgements}

L\'aszl\'o Mikl\'os Lov\'asz thanks Microsoft Research New England for an
internship in the summer of 2016, when most of the research part of this work
was done. L\'aszl\'o Mikl\'os Lov\'asz was also supported by NSF Postdoctoral Fellowship Award DMS 1705204 for part of this work. All of us thank Svante Janson and Nina Holden for various
discussions about the work presented here.

\appendix

\section{Local finiteness}\label{app:local-finite}
In this appendix, we prove Proposition~\ref{prop:local-finite}.

Throughout this appendix, $\bOmega=(\Omega,\cF,\mu)$ will be a
$\sigma$-finite measure space, $S\colon\Omega\to\RR_+$ will be measurable,
$W\colon\Omega\times\Omega\to [0,1]$ will be a symmetric, measurable
function, $\eta=\sum_{i}\delta_{x_i}$ will be a Poisson point process on
$\Omega$ with intensity $\mu$, and
\[
\eta(S)=\sum_i S(x_i)
\qquad\text{and}\qquad
\eta^2(W)=\sum_{i\neq j} W(x_i,x_j).
\]
We start with the following lemma, which is the analogue of Lemma A.3.6 from
\cite{Kal05} for general measure spaces. We use $\EE$ to denote expectations
with respect to the Poisson point process and $W\circ W$ to denote the
function $(x,y)\mapsto\int W(x,z)W(z,y)\,d\mu(z)$.

\begin{lemma}\label{lem:PP-moments}
Let $\psi(x)=1-e^{-x}$. Then the following hold, with both side of the
various identities being possibly infinite:
\begin{enumerate}
\item $\EE[\eta(S)]=\|S\|_1$ and $\EE[\eta^2(W)]=\|W\|_1$,
    \label{EofEtaS+W}
\item $\EE[\psi(\eta(S))]=\psi(\|\psi(S)\|_1)$, and \label{EofpsiEtaS}
\item $\EE[(\eta^2(W))^2]=\|W\|_1^2+4\|W\circ W\|_2+2\|W^2\|_1$.
    \label{EofEtaW2}
\end{enumerate}
\end{lemma}
\begin{proof}
We first assume that $m=\mu(\Omega)$ is finite and $S$ is bounded. Then
$\eta$ can be generated by first choosing $N$ as a Poisson random variable
with rate $m$ and then choosing $x_1,\dots,x_N$ i.i.d.\ according to the
distribution $\frac 1m\mu$. Conditioned on $N$, the expectations of $\eta(S)$
and $\eta^2(W)$ are $\frac Nm\|S\|_1$ and $\frac{N(N-1)}{m^2}\|W\|_1$,
respectively, and the expectation of $\psi(\eta(S))$ is
\begin{align*}
\EE[\psi(\eta(S))\mid N]&= 1 - \EE[e^{-\sum_{i=1}^N S(x_i)}]\\
&=1-\prod_{i=1}^N \frac 1m\int_{\Omega} d\mu(x_i)e^{-S(x_i)}
=1-\Bigl(\frac 1m\int_{\Omega} d\mu(x) e^{-S(x)}\Bigr)^N.
\end{align*}
Therefore,
\[\EE[\eta(S)]=\sum_{N=0}^\infty e^{-m}\frac{m^N}{N!}\frac{N}{m}\|S\|_1=\sum_{N=1}^\infty
e^{-m}\frac{m^{N-1}}{(N-1)!}\|S\|_1=\|S\|_1
.\] Also,
\[\EE[\eta^2(W)]=\sum_{N=0}^\infty e^{-m}\frac{m^N}{N!}\frac{N(N-1)}{m^2}\|W\|_1=\sum_{N=,}^\infty
e^{-m}\frac{m^{N-2}}{(N-2)!}\|W\|_1=\|W\|_1
.\] Finally,
\begin{align*}
\EE[\psi(\eta(S))]&=\sum_{N=0}^\infty e^{-m}\frac{m^N}{N!}\left(1-\Bigl(\frac 1m\int_{\Omega} d\mu(x) e^{-S(x)}\Bigr)^N\right)\\
&=1-\sum_{N=0}^\infty e^{-m}\frac{m^N}{N!}\Bigl(\frac 1m\int_{\Omega} d\mu(x) e^{-S(x)}\Bigr)^N\\
&=1-\exp\left(\int_\Omega e^{-S(x)}\,d\mu(x)-m\right)
=1-\exp\left(\int_\Omega-\psi(S(x))\,d\mu(x)\right)\\
&=1-e^{\|\psi(S)\|_1}=\psi(\|\psi(S)\|_1)
.\end{align*}
To calculate the expectation of
\[
(\eta^2(W))^2=\sum_{i\neq j}\sum_{k\neq \ell}\EE[W(x_i,x_j)W(x_k,x_\ell)]
\]
we distinguish whether $\{i,j\}$ and $\{k,\ell\}$ intersect in $0$, $1$, or
$2$ elements, leading to the expression
\begin{align*}
\EE[(\eta^2(W))^2\mid N]
&= \frac{N(N-1)(N-2)(N-3)}{m^4}\|W\|_1^2\\
&\quad\phantom{}+\frac{4N(N-1)(N-2)}{m^3}\|W\circ W\|_1
+\frac{2N(N-1)}{m^2}\|W^2\|_1.
\end{align*}
Taking the expectation over $N$ gives the expression in the lemma similarly.
This completes the proof for spaces of finite measure and bounded functions
$S$. The general case follows by the monotone convergence theorem.
\end{proof}

Using Lemma~\ref{lem:PP-moments}, we now prove the following proposition,
which is the analogue of the relevant parts for us of Theorem A3.5 from
\cite{Kal05} for general $\sigma$-finite measure spaces.

\begin{proposition}\label{prop:as-finite}
Let $S\colon\Omega\to\RR_+$ be measurable, and let
$W\colon\Omega\times\Omega\to [0,1]$ be symmetric and measurable. Then the
following hold:
\begin{enumerate}
\item $\eta(S) <\infty$ a.s.\ if and only if $\|\min\{S,1\}\|_1<\infty$,
    and
\item $\eta^2(W)<\infty$ a.s.\ if and only if there exists a finite $D>0$
    such that the following three conditions hold:
\begin{enumerate}
\item $D_W<\infty$ almost surely,
\item $\mu(\{x\in\Omega: D_W(x)>D\})<\infty$, and
\item $\|W|_{\{x\in\Omega: D_W(x)\leq D\}}\|_1<\infty$.
\end{enumerate}
\end{enumerate}
\end{proposition}

\begin{proof}
Since $\frac 12\min\{1,x\}\leq \psi(x)\leq \min\{1,x\}$, the condition
$\|\min\{S,1\}\|_1<\infty$ in (1) is equivalent to the statement that
$\|\psi(S)\|_1<\infty$, which is equivalent to the statement that
$\psi(\|\psi(S)\|_1)<1$. By Lemma~\ref{lem:PP-moments} \eqref{EofpsiEtaS},
this is equivalent to saying that $\EE[\psi(\eta(S))]<1$, which holds if and
only if $\eta(S)<\infty$ with positive probability. By Kolmogorov's zero-one
law, we either have $\eta(S)<\infty$ almost surely, or $\eta(S)=\infty$
almost surely; therefore we have obtained that $\|\min\{S,1\}\|_1<\infty$ if
and only if $\eta(S)<\infty$ almost surely.

To prove the second statement, assume first that the conditions (a)--(c)
hold. Condition (a) then implies that a.s., no Poisson point falls into the
set $\{D_W=\infty\}$, which means we may replace $\Omega$ by a space such
that $D_W(x)<\infty$ for all $x\in \Omega$. Let $\Omega_{>D}=\{x\in\Omega:
D_W(x)>D\}$ and $\Omega_{\leq D}=\Omega\setminus\Omega_{>D}$. Since
$\Omega_{>D}$ has finite measure by assumption (b), we have that a.s., only
finitely many Poisson points fall into this set, which in particular implies
that the contribution of the points $x_i,x_j\in \Omega_{>D}$ to $\eta^2(W)$
is a.s.\ finite. Next let us consider the contributions to $\eta^2(W)$ from
pairs of points $x_i,x_j$ such that one lies in $\Omega_{>D}$ and the other
one lies in $\Omega_{\leq D}$. Observing that the Poisson process in
$\Omega_{>D}$ and $\Omega_{\leq D}$ are independent, and that a.s., there are
only finitely many points in $\Omega_{>D}$, it will clearly be enough to show
that for all $x\in \Omega_{>D}$, a.s.\ with respect to the Poisson process in
$\Omega_{\leq D}$,
\[
\sum_{j: x_j\in \Omega_{\leq D}} W(x,x_j)<\infty.
\]
But by Lemma~\ref{lem:PP-moments} \eqref{EofEtaS+W} applied to the function
$S'\colon\Omega_{\leq D} \rightarrow \RR_+$ defined by $S'(y)=W(x,y)$, the
expectation of this quantity is equal to
\[
\int_{\Omega_{\leq D}}S'(y) \,d\mu(y)=\int_{\Omega_{\leq D}}W(x,y)\,d\mu(y).
\]
This is bounded by $D_W(x)$ and hence finite, which proves that the sum is
a.s.\ finite. We are thus left with estimating $\eta^2(W|_{\Omega_{\leq
D}})$. Again by Lemma~\ref{lem:PP-moments} \eqref{EofEtaS+W}, we have that
$\EE[\eta^2(W|_{\Omega_{\leq D}})]=\|W|_{\Omega_{\leq D}}\|_1$ which is
finite by assumption (c), showing that $\eta^2(W|_{\Omega_{\leq D}})$ is
a.s.\ finite.

Conversely, let us assume that a.s., $\eta^2(W)<\infty$. First we will prove
that this implies $\mu(\{D_W=\infty\})=0$. Assume for a contradiction that
this is not the case. Since $\mu$ is $\sigma$-finite, we can find a
measurable set $N\subseteq\Omega$ such that $D_W(x)=\infty$ for all $x\in N$
and $0<\mu(N)<\infty$. Consider the contribution to $\eta^2(W)$ by all
Poisson points $(x_i,x_j)$ such that $x_i\in N$ and $x_j\in
N^c=\Omega\setminus N$. Since the Poisson processes on $N$ and $N^c$ are
independent, the finiteness of $\eta^2(W)$ implies that for almost all $x\in
N$, the sum $\sum_{j:x_j\in N^c }W(x,x_j)$ is a.s.\ finite. Applying
statement (1) of the current proposition to $W(x,\cdot)$ (and recalling that
$W$ is bounded by $1$), we conclude that for almost all $x\in N$,
$\int_{N^c}W(x,y)\,d\mu(y)<\infty$, which implies that for almost all $x\in
N$, $\int_NW(x,y)\,d\mu(y)=D_W(x)-\int_{N^c}W(x,y)\,d\mu(y)=\infty$. This is
a contradiction since $\mu(N)<\infty$ and $W\leq 1$.

We next prove (b) (for any value of $D$). Suppose for a contradiction that
$\mu(\{x \in \Omega:D_W(x)>D\})=\infty$. We then claim that almost surely,
$\eta^2(W)=\infty$. After obtaining the Poisson process, color each point
randomly red or blue, with equal probability, independently. We can then
obtain the red and blue points equivalently by taking two independent Poisson
processes, both with intensity $\mu/2$. We claim that almost surely, the sum
of $W(x,y)$ just over red-blue pairs is already $\infty$. We know that almost
surely, there are an infinite number of red points $x_i$ with $D_W(x_i)>D$.
Let $x_n$ be such a sequence, and given $y \in \Omega$, let
$S'(y)=\sum_{n=1}^\infty W(x_n,y)$. Then the sum of $W$ over red-blue edges
is equal to $\eta(S')$ for the Poisson process with intensity $\mu/2$.
Therefore, it suffices to prove that $\|\min\{S',1\}\|_{1,\mu/2}=\infty$.
First, note that if either $\mu(\{y \in \Omega:S'(y)=\infty\})>0$ or $\mu(\{y
\in \Omega:S'(y)>1\})=\infty$, then it clearly holds. Otherwise, we have that
as $D' \to \infty$, $\mu(\{y \in \Omega:S'(y)>D'\}) \to 0$; therefore, there
exists some $D'$ (without loss of generality, we may assume $D'\geq1$) such
that $\mu(\{y \in \Omega:S'(y)>D'\})<D/2$. Let $\Omega'$ be the complement of
$\{y \in \Omega:S'(y)>D'\}$. We then have that for each $x_n$,
\begin{align*}
\int_{\Omega'} W(x_n,y) \frac{d\mu(y)}{2}
&= \int_{\Omega} W(x_n,y) \frac{d\mu(y)}{2}-\int_{\Omega \setminus\Omega'} W(x_n,y) \frac{d\mu(y)}{2}\\
&\geq\frac 12 D_W(x_n) - \frac 12\mu(\Omega \setminus\Omega')
\ge \frac D2 - \frac D4.
\end{align*}
We also have that
\begin{align*}
\int_{\Omega'} S'(y)\frac{d\mu(y)}{2} &=\int_{\Omega'} \sum_{n=1}^\infty W(x_n,y) \frac{d\mu(y)}{2} \\
&= \sum_{n=1}^\infty\int_{\Omega'} W(x_n,y) \frac{d\mu(y)}{2}
 \ge \sum_{n=1}^{\infty} D/4= \infty.
\end{align*}
Therefore,
\begin{align*}
\int_\Omega \min\{S'(y),1\} \frac{d\mu(y)}{2}
&\ge \int_{\Omega'} \min\{S'(y),1\} \frac{d\mu(y)}{2}\\
&\ge
\frac 1{D'} \int_{\Omega'} \min\{S'(y),D'\} \frac{d\mu(y)}{2}
=\frac 1{D'} \int_{\Omega'} S'(y) \frac{d\mu(y)}{2}=\infty
.
\end{align*}
This contradiction completes the proof.

We are left with proving (c) (we will again prove it for any value of $D$).
Assume the opposite, and let $\Lambda_n\subseteq\Lambda$ be an increasing
sequence such that $\mu(\Lambda_n)<\infty$ and
$\bigcup_n\Lambda_n=\Omega_{\leq D}$. Let $U_n=W|_{\Lambda_n}$. Then
$\|U_n\|_1<\infty$, $\|U_n\|_1\uparrow \|W|_{\Omega_{\leq D}}\|_1=\infty$,
and $\|D_{U_n}\|_\infty\leq D$, implying in particular that $\|U_n\circ
U_n\|_1=\|D_{U_n}\|_2^2\leq D\|D_{U_n}\|_1=D\|U_n\|_1$. Given an arbitrary
constant $\lambda$, we claim that
\begin{equation}\label{Paley-Z}
\PP\Bigl(\eta^2(W|_{\Omega_{\leq D}})>\lambda\Bigr)
\geq \frac{(\|U_n\|_1-\lambda)^2}{\|U_n\|_1^2+(4D+2)\|U_n\|_1},
\end{equation}
provided $n$ is large enough to ensure that $\|U_n\|_1>\lambda$. Indeed,
writing
\[
\EE[\eta^2(U_n)]=\EE[\eta^2(U_n)1_{\eta^2(U_n)\leq\lambda}]
+\EE[\eta^2(U_n)1_{\eta^2(U_n)>\lambda}],
\]
we can bound the first term by $\lambda$ and the second by
$\sqrt{\EE[(\eta^2(U_{n}))^2]\PP[\eta^2(U_n)>\lambda]}$, using Cauchy's
inequality. We therefore obtain that
\[\EE[\eta^2(U_n)]\le \lambda
+\sqrt{\EE[(\eta^2(U_n))^2]\PP[\eta^2(U_n)>\lambda]}
.\]
Rearranging, we obtain the bound
\[
\PP\Bigl(\eta^2(U_n)>\lambda\Bigr)\geq \frac{(\EE[\eta^2(U_n)]-\lambda)^2}{\EE[(\eta^2(U_{n}))^2]}
=\frac{(\|U_n\|_1-\lambda)^2}{\|U_n\|_1^2+4\|U_n\circ U_n\|_1+\|U_n^2\|_1},
\]
where we used Lemma~\ref{lem:PP-moments} \eqref{EofEtaS+W}
and \eqref{EofEtaW2} in the last step. Observing
that
\[
\Pr(\eta^2(W|_{\Omega_{\leq D}})>\lambda)\geq \Pr(\eta^2(U_n)>\lambda)
\]
and bounding $4\|U_n\circ U_n\|_1+\|U_n^2\|_1$ by $(4D+2)\|U_n\|_1$, we
obtain \eqref{Paley-Z}. Since the right side of \eqref{Paley-Z} goes to $1$
as $n\to\infty$, we get that with probability one, $\eta^2(W|_{\Omega_{\leq
D}})>\lambda$ for all $\lambda$, which contradicts the assumption that
$\eta^2(W|_{\Omega_{\leq D}})<\infty$ a.s.
\end{proof}

\begin{proof}[Proof of Proposition~\ref{prop:local-finite}]
We first prove the equivalence of (A) -- (E). Clearly $(B) \Rightarrow
(C)\Rightarrow (A)$ and $(D) \Rightarrow (E)$. It is also not hard to see
that $(E) \Rightarrow (A)$. Indeed, note first that for any $D$, the
condition on $S$ is equivalent to the condition that $\min\{S,D\}$ is
integrable (which implies that $\mu(\{S>D\})<\infty$.) Set $\Omega'=\{D_W\leq
D\}\cap \{S\leq D\}$. Then (E) implies that
\begin{align*}
\|\cW|_{\Omega'}\|_1&\leq 2 I + \|W_{\{D_W\leq D\}}\|_1
+2\|S 1_{S\leq D}\|_1\\
& \le 2 I + \|W_{\{D_W\leq D\}}\|_1
+2\|\min\{S,D\}\|_1<\infty
\end{align*}
and $\mu(\Omega\setminus\Omega')\leq \mu(\{D_W>D\})+\mu(\{S>D\})<\infty$,
proving (A). So it will be enough to show $(A) \Rightarrow (B)$ and $(A)
\Rightarrow (D)$.

Suppose that (A) holds, and let $\Omega'$ be a set such that $\mu(\Omega
\setminus \Omega')<\infty$, $\cW' = \cW|_{\Omega'}$, and $\|\cW'\|_1 = C <
\infty$. Let $D>0$. First, assume that $D>D_0=\mu(\Omega \setminus \Omega')$.
Then
\[
\{x \in \Omega: D_\cW(x) > D\} \subseteq (\Omega \setminus \Omega') \cup \{x \in \Omega', D_{\cW'}(x) > D-D_0 \}.
\]
Since $\|D_{\cW'}\|_1 \le \|\cW'\|_1=C$, this set has measure at most
\[
D_0+ \frac{C}{D-D_0}.
\]
Now, let $\cW''=\cW|_{\{x:D_\cW(x) \le D\}}$. Then
\[
\|\cW''\|_1 \le \| \cW'\|_1
+ 2 \int_{\{ x \in \Omega \setminus \Omega': D_\cW(x) \le D\}} D_\cW(x)
\le \|\cW'\|_1 + 2DD_0.
\]
We have thus proven that (B) holds for all $D$ larger than some $D_0$, and
more generally for any $D$ for which there exists an $\Omega' \subseteq
\Omega$ with $\mu(\Omega \setminus \Omega')<D$ and $\|\cW|_{\Omega'}\|_1
<\infty$.

Note that if $\cW|_{\{x:D_\cW(x) \le D\}}$ is integrable for $D>D_0$, then it
must remain integrable if we decrease $D$, since that is just a restriction
to a subset. Therefore, this implies that $\cW|_{\{x:D_\cW(x) \le D\}}$ is
integrable for all $D$. Since $D_\cW<\infty$ almost everywhere, we further
have that $\mu(\{x \in \Omega: D_\cW(x) \ge \lambda\})$ tends to $0$ as
$\lambda$ tends to $\infty$ (since we at least know that it is finite for
large enough $\lambda$). Fixing $D>0$, we can therefore take $D'$ large
enough so that $\mu(\{x \in \Omega: D_\cW(x) \ge D'\})<D$. Taking
$\Omega':=\Omega \setminus \{x \in \Omega: D_\cW(x) \ge D'\}$, we get a set
$\Omega'$ such that $\mu(\Omega \setminus \Omega')<D$ and
$\|\cW|_{\Omega'}\|_1 <\infty$ proving that (B) holds for all $D>0$.

On the other hand if (A) holds for some $\Omega'$, then
$\|W|_{\Omega'}\|_1<\infty$ and $\|S 1_{\Omega'}\|_1<\infty$. Proceeding
exactly as above we conclude that for all $D$, $\mu(\{D_W>D\})<\infty$ and
$\|W|_{\{D_W\leq D\}}\|_1<\infty$, as well as $\mu(\{S>D\})<\infty$ and
$\|S1_{\{S\leq D\}}\|_1<\infty$. Since
\[\|\min\{S,D\}\|_1 = D\mu(\{S>D\}) + \|S1_{\{S\leq D\}}\|_1<\infty
,\] the latter condition is equivalent to $\|\min\{S,D\}\|_1<\infty$, as
required.

We are left with proving that the local finiteness conditions in
Definition~\ref{def:graphex} are necessary and sufficient for the almost sure
finiteness of $G_T(\cW)$ for all $T<\infty$. It is easy to check that the
local finiteness conditions are not affected if we multiply the underlying
measure by $T$ and $S$ by $T$. We therefore assume that $T=1$. Let
$\eta=\sum_{i}\delta_{x_i}$ be a Poisson process of intensity $\mu$ on
$\Omega$, let $Y_i$ be Poisson random variable with mean $S(x_i)$, and let
$Y_{ij}$ be Bernoulli with mean $W(x_i,x_j)$, all of them independent of each
other. We will have to show that the local finiteness conditions on $\cW$ are
equivalent to the a.s.\ finiteness of the sums
\[
e_S=\sum_i Y_i
\qquad\text{and}\qquad
e_W=\sum_{i > j} Y_{ij}.
\]
We next use the fact that a sum of independent, non-negative random variables
$\sum_k Z_k$ is a.s.\ finite if and only if
$\sum_i\EE[\min\{Z_i,1\}]<\infty$. In the case of $e_W$, $Y_{i,j}$ is
bounded, and therefore we immediately have that $e_W$ is a.s.\ finite if and
only if $\eta^2(W)$ is a.s.\ finite. Proposition \ref{prop:as-finite} (b)
then proves this case. In the case of $e_S$, setting $S'=\min\{S,1\}$,
applying Proposition \ref{prop:as-finite} to $S'$, and noting that $S'$ is
bounded, we have that $\sum_i S'(x_i)$ is almost surely finite if and only if
$\|S'\|_1<\infty$. This is exactly the condition on $S$.
\end{proof}

\section{Sampling with loops} \label{sec:samplingwithloops}
In this section, we discuss how to handle samples with loops. The sampling
process is adjusted as follows. We follow the same process as for
$\cG_T(\cW)$ and $\cG_\infty(\cW)$; however, for each vertex labeled as
$(t,x)$, with probability $W(x,x)$, we add a loop to the vertex. Deleting
isolated vertices as before, and then removing the feature labels from the
vertices, we obtain a family $(\widetilde{\cG}_T(\cW))_{T\geq 0}$ of labelled
graphs with loops, as well as the infinite graph
$\widetilde{\cG}_\infty(\cW)=\bigcup_{T\geq 0}\widetilde{\cG}_T(\cW)$.

Note that a vertex that was previously isolated may not be isolated anymore
if it receives a loop, so a vertex may have been deleted from $\cG_T(\cW)$
but not from $\widetilde{\cG}_T(\cW)$. We add a further condition for local
finiteness: \[\int_\Omega W(x,x) \,d\mu(x) < \infty.\] Note that if $\cW$ is
atomless, then the values $W(x,x)$ do not have an effect on $\cG_T$ and
$\cG_\infty$, and the diagonal constitutes a zero measure set in $\Omega
\times \Omega$.

As stated, Theorem \ref{thm:identify} is false for sampling with loops. Since
the diagonal may be a zero measure set, almost everywhere equal pullbacks do
not imply having the same looped samples. We could further add the condition
that $W(x,x)$ is equal to the pullback almost everywhere, but the theorem
would still be false. This is demonstrated by the following example. Let
$\bOmega_1=\bOmega_2=[0,1]$. Take $W_1$ to be constant $1/2$ on $[0,1] \times
[0,1]$, and let $W_2$ be constant $1/2$ off the diagonal, $0$ if $x<1/2$, $1$
otherwise. Let $\cW_i=(W_i,0,0,\bOmega_i)$. Then we claim that
$\widetilde{\cG}_T(\cW_1)$ and $\widetilde{\cG}_T(\cW_2)$ have the same
distribution. Indeed, both are equivalent to taking $\Poisson(T)$ vertices,
adding a loop to each vertex with probability $1/2$, independently, and also
taking an edge between each pair of vertices with probability $1/2$,
independently over different pairs.

It turns out that in general, allowing diagonal values strictly between $0$
and $1$ is not necessary, because we could extend the feature space to
determine whether each vertex has loops. For graphexes where the diagonal is
$0$ or $1$, we can then conclude an analogous theorem from
Theorem~\ref{thm:identify}.

We first show the following:

\begin{proposition}
For any graphex $\cW=(W,S,I,\bOmega)$, there exists a graphex
$\widetilde{\cW}=(\widetilde{W},\widetilde{S},\widetilde{I},\widetilde{\bOmega})$
on an atomless space $\widetilde\bOmega$ such that on the diagonal,
$\widetilde W$ is $\{0,1\}$ valued and such that
$\widetilde{\cG}_\infty(\widetilde{\cW})$ and
$\widetilde{\cG}_T(\widetilde{\cW})$ are equivalent to
$\widetilde{\cG}_\infty(\cW)$ and $\widetilde{\cG}_T(\cW)$, respectively.
\end{proposition}

\begin{proof}
Let $\widetilde{\bOmega}=\bOmega \times [0,1]$, and let $\pi_1,\pi_2$ be the
projection maps. Note that we can obtain a Poisson process on
$\widetilde{\bOmega} \times \RR_+$ by taking a Poisson process on $\bOmega
\times \RR$, and independently labeling each point with a uniform random real
number from $[0,1]$, which becomes the second coordinate. Clearly
$\widetilde{\bOmega}$ is atomless, so the diagonal values only affect the
generation of the loops. Define $\widetilde{I}=I$, $\widetilde{S}=S \circ
\pi_1$, $\widetilde{W}(x,y)=W(\pi_1(x),\pi_1(y))$ if $x \ne y$, and
\[\widetilde{W}(x,x)= \begin{cases}
1, & \mbox{if } \pi_2(x) \le W(\pi_1(x),\pi_1(x)) \\
0, & \mbox{otherwise}.
\end{cases}
.\] Then the sampling of edges between vertices is not affected by the second
coordinate of a vertex. Note that the probability that there exist two
vertices corresponding to the same point in $\widetilde{\bOmega}$ is zero,
since $\widetilde{\bOmega}$ is atomless. For the loops, since we can obtain
the vertices by first taking the Poisson process on $\bOmega \times \RR$ and
then randomly labeling each vertex with a $[0,1]$ real number, we can see
that for a point $y \in \Omega$, if it ends up as a point, there is a
$W(y,y)$ probability that the point $x$ corresponding to it has
$\widetilde{W}(x,x)=1$, and $1-W(y,y)$ that $\widetilde{W}(x,x)=0$, and this
is independent over different points. Therefore, the distribution of loops is
the same.
\end{proof}

Using this proposition, sampling loops according to the diagonal is
equivalent to the following theory. The objects are graphexes with special
subsets $\cW=(W,S,I,\bOmega,A)$ where $W$, $S$, $I$, and $\bOmega$ are as
before, and the \emph{special set} $A \subseteq \Omega$ is a measurable
subset with finite measure. We sample $\widetilde{\cG}_\infty(\cW)$ in the
same way as $\cG_\infty(\cW)$, except that we add a loop to each vertex with
a feature label in $A$. We then take the non-isolated vertices with time
label at most $T$ for $\widetilde{\cG}_T(\cW)$. We can extend the definition
of measure-preserving map by requiring that points in the special set be
mapped to points in the special set, and points not in the special set be
mapped to points not in the special set. We also define $\dsupp$ as earlier,
except it contain all points in $A$ (even if otherwise they would not be
included).

\begin{theorem}
Let $\cW_1$ and $\cW_2$ be graphexes with special subsets as above. Then
$\widetilde{\cG}_T(\cW_1)$ and $\widetilde{\cG}_T(\cW_2)$ have the same
distribution for all $T \in \RR_+$ if and only if there exists a third
graphex with special subset $\cW$ such that $\cW_1$ and $\cW_2$ are pullbacks
of $\cW$.
\end{theorem}

\begin{proof}
It is clearly enough to prove the only if direction. Suppose therefore that
$\cW_1$ and $\cW_2$ have the same distribution. Then for any $0<c<1$,
$c\cW_1$ and $c \cW_2$ have the same distributions (i.e., $W,S,I$ are all
multiplied by $c$, and the special set stays the same). Then let
$\widetilde{\cW_i}$ be obtained by taking $ \cW_i/2$, adding a set $B_i$ of
measure $1$ to $\Omega_i$, and extending $W_i$ to be $1$ on $B_i \times B_i$,
$1$ between $B_i$ and $A_i$, and $0$ between $B_i$ and $\Omega_i \setminus
A_i$. Then we can obtain $G_T(\widetilde{\cW_i})$ from
$\widetilde{G}_T(\cW_i)$ by the following process. We first keep each edge
that is not a loop with probability $1/2$, and delete it otherwise,
independently. We keep all the loops. Then we take $\Poisson(T)$ new
vertices, put an edge between every pair, and put an edge between each new
vertex and each vertex that had a loop (and delete loops). It is clear that
in this way, the distributions $G_T(\widetilde{\cW}_1)$ and
$G_T(\widetilde{\cW}_2)$ are the same for every $T$. Therefore, there exists
a graphex
$\widetilde{\cW}=(\widetilde{W},\widetilde{S},\widetilde{I},\widetilde{\bOmega})$
such that $\widetilde{\cW}_1$ and $\widetilde{\cW}_2$ are both pullbacks of
$\widetilde{\cW}$. It is clear that $\widetilde{\cW}$ must have a set of
measure $1$, call it $B$, which has $\widetilde{W}(x,y)=1$ if $x,y \in B$,
and $\widetilde{W}(x,y)$ is either $0$ or $1$ if $x \in B, y \notin B$, and
only depends on $y$, and $\widetilde{W}(x,y) \le 1/2$ if $x,y \notin B$, and
$B$ must pullback to exactly $B_1$ and $B_2$. If we let $A$ be the set of
points $x$ with $\widetilde{W}(x,y)=1$ for any and all $y \in B$, then $A$
must pullback to $A_1$ and $A_2$. If we therefore let $\cW$ have underlying
set $\widetilde{\Omega} \setminus B$, and be equal to $2\widetilde{\cW}$
restricted to this set, and special set $A$, then $\cW$ pulls back to both
$\cW_1$ and $\cW_2$.
\end{proof}

\providecommand{\bysame}{\leavevmode\hbox to3em{\hrulefill}\thinspace}
\providecommand{\MR}{\relax\ifhmode\unskip\space\fi MR }
\providecommand{\MRhref}[2]{\href{http://www.ams.org/mathscinet-getitem?mr=#1}{#2}}
\providecommand{\href}[2]{#2}

\end{document}